\documentclass[11pt]{book}

\usepackage{authblk}

\usepackage{bookmark}
\usepackage{amsmath}
\usepackage{amssymb}
\usepackage{amsthm}
\usepackage{bbm}
\usepackage{enumerate}
\usepackage{epstopdf}
\usepackage{longtable}
\usepackage{setspace}
\usepackage{graphicx}
\usepackage{subcaption}
\usepackage{nicematrix}
\usepackage[normalem]{ulem}
\usepackage{tikz}
\usepackage[utf8]{inputenc}
\usepackage{hyperref}
\hypersetup{
    colorlinks=false,
    linkcolor=blue,
    filecolor=magenta,      
    urlcolor=cyan,
}
\usepackage{makeidx}

\usetikzlibrary{arrows,decorations.markings}
\usetikzlibrary{arrows.meta}

\newtheorem{theorem}{Theorem}[chapter]

\newtheorem{lemma}[theorem]{Lemma}

\newtheorem{exercise}{Exercise}[chapter]

\theoremstyle{definition}
\newtheorem{problem}{Problem}[chapter]

\theoremstyle{definition}
\newtheorem{definition}{Definition}[chapter]
\newtheorem{example}{Example}[chapter]

\newtheorem{assumption}{Assumption}

\newcommand{\nR}{{\mathbb R}}

\newcommand{\nP}{{\mathbb P}}
\newcommand{\nE}{{\mathbb E}}
\newcommand{\I}{{\mathbbm{1}}}

\newcommand{\Norm}[1]{\ensuremath{\left \Vert  #1 \right \Vert}}

\newcommand{\mb}[1]{\ensuremath{\mathbf{#1}}}
\newcommand{\tr}{\ensuremath{t_{\text{rel}}}}
\newcommand{\tm}{\ensuremath{t_{\text{mix}}}}
\newcommand{\ra}{\rightarrow}

\DeclareMathOperator{\Var}{Var}

\DeclareMathOperator{\Unif}{Unif}
\DeclareMathOperator{\Bin}{Bin}

\newcommand{\eq}{\begin{equation}}
\newcommand{\en}{\end{equation}}
\newcommand{\rr}{\mathbb{R}}
\newcommand{\NN}{\mathbb{N}}

\newcommand{\norm}[1]{\left\lVert #1 \right\rVert}
\newcommand{\abs}[1]{\left\lvert #1 \right\rvert}

\newcommand{\nin}{\noindent}
\newcommand{\tbf}{\textbf}

\newcommand{\E}{\mathrm{E}}

\begingroup
\def\2#1{\ifnum#1<10 0\fi\the#1}
\xdef\isodayandtime%
{\the\year-\2\month-\2\day\space\2{\count0}:\2{\count2}}

\endgroup

\renewcommand{\arraystretch}{1}

\allowdisplaybreaks
\title{Finite Markov Chains and Monte-Carlo Methods: \\ An Undergraduate Introduction}
\author[1]{Soumik Pal}
\author[2]{Tim Mesikepp}

\affil[1]{University of Washington \\ \texttt{soumik@uw.edu}}
\affil[2]{\texttt{tmesikepp@gmail.com}}

\makeindex
\begin{document}
\maketitle

\tableofcontents

\chapter*{Preface}
\addcontentsline{toc}{section}{Preface}

Around 2018, a new undergraduate course was introduced at the University of Washington, Seattle. Offered jointly by the Departments of Mathematics and Statistics, this course, titled math/stat 396, was designed to introduce undergraduates to finite state space Markov chains and common Monte Carlo methods. This is one of the most elegant areas of modern probability theory—deep, rich, and full of surprises.

My own introduction to the subject came through William Feller’s legendary textbook, \textit{An Introduction to Probability Theory and Its Applications} \cite{feller1957introduction}, which devotes Chapter XV to Markov chains. Written in the 1950s, Feller’s text presented the nascent theory through an array of dazzling examples. Since then, many excellent books have appeared. Among those that have shaped my understanding are Hoel, Port, and Stone’s \textit{Introduction to Stochastic Processes} \cite{hoel1986introduction}, Aldous and Fill’s classic \textit{Reversible Markov Chains and Random Walks on Graphs} \cite{AldousFillBook}, Persi Diaconis’s monograph \textit{Group Representations in Probability and Statistics} \cite{diaconis1988group}, and Levin and Peres’s modern standard \textit{Markov Chains and Mixing Times} \cite{Peres}. Each of these resources is spectacular, to be read and re-read, for knowledge and for pleasure. 

Yet, when we discussed which textbook to adopt for the new course, we faced a dilemma. Feller and Hoel–Port–Stone are too dated and do not cover Monte Carlo methods adequately. The others are too advanced for most undergraduates. We also needed a text with a large and varied supply of exercises for homework and exams. In the end, I decided to write my own lecture notes and build a problem bank. Tim, then a teaching assistant for the course, joined me in this effort, which eventually led to this book. The theory of this book can be mostly found in the first four chapters of the textbook by Levin and Peres \cite{Peres}. We selected parts that can be easily explained to undergraduates and expanded with our own examples.

Our bigger contribution arguably is the collection of more than a hundred exercises and problems in this textbook. Some are classics, drawn from sources such as Feller, but many are original. Each chapter ends with a problem set; the more challenging problems include hints. Routine exercises are embedded within the text, and suggested homework sets appear in the appendix. Students are strongly encouraged to attempt exercises as they read, both to reinforce understanding and to prepare for deeper problems at the chapter’s end.

For longer or more intricate proofs, we often refer readers to the textbook by Levin and Peres \cite{Peres}, both to keep the material accessible and to encourage engagement with that excellent source. A few proofs, however—such as that of Theorem \ref{Thm:ReturnTimeLite}—are original to this book.  

This book is limited to finite state space Markov chains. Countable state spaces are mentioned but not covered in any detail. Concepts of recurrence and transience are not touched. Measure theory is eschewed. The focus has been on families of concrete Markov chains, such as random walks on cycles, to drive the basic theory. Martingales are introduced, but solely as functions of finite Markov chains. For a one quarter course, as we do in Washington, we recommend chapters 1--4. Completing chapter 5 requires more time. The entirety of this book can be covered within one semester.  

Over the years, many instructors have taught the course (Math/Stat 396) using these notes. We gratefully acknowledge their edits and suggestions, especially our former postdoc Andrea Ottolini. We also owe a great debt to former students Jamie Forschmiedt and Celeste Zeng, who carefully read through the manuscript, corrected errors, and created many of the figures in this book in \LaTeX.

This book is dedicated to all the probabilists whose creativity and insight shaped this beautiful subject.

\begin{flushright}
Soumik Pal \\
Seattle, 2025
\end{flushright}

As Soumik's co-author, several additional prefatory remarks come to mind.  Firstly, the reader who notices a rather distinct change in voice in Chapter \ref{Ch:MC} will not be mistaken.  I would have preferred to have Soumik's beautiful writing for the entirety of the text, but I took the writing lead outside of chapter four, while Soumik did for that chapter.  

Secondly, our approach is to build up the theory from a basic starting point, which we hope is conducive for the reader without an exhaustive background in probability and network theory.  Those starting with more robust knowledge and who find the development too tedious can focus more on the problem sets, or to go deeper with companion resources like the outstanding Levin-Peres text \cite{Peres} mentioned above (as alluded to, our text takes as its basic framework the initial four chapters of that study, expanding as appropriate for our setting).  Those encountering the material for the first time should read the text carefully and do the in-text exercises.  All readers should wrestle through the proofs, explore their own examples, argue with the text, find counter-examples when assumptions are omitted, and ask lots of questions.  The more active you are in studying, after all, the deeper your understanding will be.

I would love to hear your feedback or be alerted to any typos; please email me directly at \texttt{tmesikepp@gmail.com}.  

\begin{flushright}
Tim Mesikepp\\
Westborough, MA, 2025
\end{flushright}

% The text contains \total{figure} figures.

%\chapter*{Preface}
%\bigskip

%Include:
%\begin{itemize}
 %    \item Homage to Peres' book
  %   \item The "exercises" are typically routine and should be done by everyone as reading to make sure they are tracking with the text.  ``Problems'' typically require more thought; there are suggested homework sets in the appendix. 
   %  \item Proof of Theorem \ref{Thm:ReturnTimeLite} is original.
%\end{itemize}

% \section*{To discuss in meetings}

% \begin{itemize}
%     \item definition of $\hat{P}(x,y)$ when $\pi(x)=0$.
% \end{itemize}

%%%% Lecture 1 %%%%%%%%%%%%%%%%%%%%%%
\chapter{Markov Chains}\label{Ch:Markov}

% Need some sort of introduction with motivation and examples.
% %Markov chains are random processes that make make a movement at each interval of time, and where the choice of movement is dependent only on the current location.

% \begin{exercise}  
%     Are the following Markov chains?  Explain your reasoning.
%     \begin{enumerate}[$(i)$]
%         \item Some examples
%     \end{enumerate}
% \end{exercise}

\section{Random walks on graphs}\label{Sec:RWG}

In this chapter we introduce \emph{Markov chains} and their fundamental properties.  Before we give the general definition, however, it will be helpful to begin with a very concrete and important class of examples, \emph{random walks on graphs}. To understand a random walk on a \emph{graph}, though, we need to understand what a graph is, and so we start by reviewing the very basics of graphs.  Afterwards we describe two kinds of random walks on the vertices of graphs, both of which will turn out to be Markov chains.  These starting examples will help us begin to see what Markov chains are all about.

\subsection{What are graphs?}

A \emph{graph} $G=(V,E)$ is just a collection of vertices $V$ and edges $E$.  Both of these are sets, and $E$ is a subset of unordered pairs of vertices. An edge $\{u, v\}\in E$ iff (if and only if) the vertices $u$ and $v$ are ``connected'' in the graph $G$. When $\{u,v\} \in E$, we say that the vertices $u$ and $v$ are \emph{adjacent} or that they are \emph{neighbors}, and write $u \sim v$. A \emph{loop} is an edge $\{u,u\} \in E$ beginning and ending at the same vertex. 
% Our graphs will typically not have loops.

%Recall that $V \times V$ is the set of pairs $(v_1,v_2)$ where both $v_1 \in V$ and $v_2 \in V$. Thus, symbolically $E \subset V \times V$ note that by the subset $\subset$ symbol, we include the possibility of equality. 

%Soumik's comments: V \times V consists of ordered pairs, while edges are unordered. 

\begin{figure}
    \centering
    \scalebox{0.75}{
    \begin{tikzpicture}
    [node/.style={circle, draw=black, very thick, minimum size=7mm}]
    \node[node] (one) at (0, 0) {1};
    \node[node] (two) at (1, 1) {2};
    \node[node] (three) at (2, 0) {3};
    \node[node] (four) at (3, 1) {4};
    \node[node] (five) at (4, 0) {5};

    \draw[black, very thick] (one) -- (two);
    \draw[black, very thick] (two) -- (three);
    \draw[black, very thick] (two) -- (four);
    \draw[black, very thick] (three) -- (four);
    \draw[black, very thick] (three) -- (five);
    \draw[black, very thick] (four) -- (five);
    \end{tikzpicture}
    }
    \caption{\small A graph with five vertices}
    \label{Fig:FirstGraph1}
\end{figure}
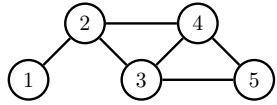

This is all best seen in examples.  Take a look at the graph in Figure \ref{Fig:FirstGraph1}, for instance.  Here $V=\{1,2,3,4,5\}$ and 
\begin{align*}
    E = \{\{1,2\},\{2,3\},\{2,4\},\{3,4\},\{3,5\},\{4,5\} \}.
\end{align*}
Note that, as \emph{sets}, $\{2,3\}=\{3,2\}$, and so we don't have to list this single edge twice. 

Another example is the \textbf{6-cycle} in Figure \ref{Fig:Hexagon}.  Here $V=\{1,2,3,4,5,6\}$ and
\begin{align*}
    E = \{\{1,2\},\{2,3\},\{3,4\},\{4,5\},\{5,6\},\{6,1\} \}.
\end{align*}
In Figure \ref{Fig:Hexagon} we have arranged the vertices in a hexagonal shape, but keep in mind that the layout, or \emph{embedding}, of the graph in Euclidean space doesn't matter in terms of this bare definition.  Writing $G = (V,E)$, after all, pays no attention to how we draw the graph; all that matters is what the vertices are and how they are connected.  So, for instance, the picture of a trip across the US in Figure \ref{Fig:America} gives exactly the same graph as Figure \ref{Fig:Hexagon}, even though the embeddings appear quite different.

\begin{figure}[b]
    \centering
   
   \scalebox{0.75}{
    \begin{tikzpicture}
    \tikzset{
        node/.style={circle, draw=black, very thick, minimum size=7mm},
        line/.style={black, very thick}
    }
    
    \node[node] (one) at (1.732,1) {1};
    \node[node] (two) at (1.732,-1) {2};
    \node[node] (three) at (0,-2) {3};
    \node[node] (four) at (-1.732,-1) {4};
    \node[node] (five) at (-1.732,1) {5};
    \node[node] (six) at (0,2) {6};
    
    \draw[line] (one) -- (two);
    \draw[line] (two) -- (three);
    \draw[line] (three) -- (four);
    \draw[line] (four) -- (five);
    \draw[line] (five) -- (six);
    \draw[line] (six) -- (one);
    \end{tikzpicture}
    }

    \caption{\small The 6-cycle graph.  In general, the \textbf{$n$-cycle} has $n$ vertices, $n$ edges, and each vertex is connected to two neighbors, forming one ``loop'' around the entire graph.}
    \label{Fig:Hexagon}
\end{figure}
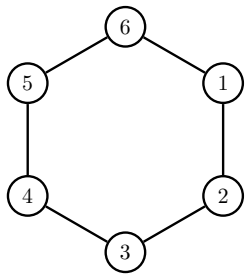

Sometimes our vertex set will be infinite, as with the integer graph $(V,E)$ in Figure \ref{Fig:NumberLine}. Here the vertices are $V = \{\ldots,-1,0,1,2,\ldots\}$ and edges connect adjacent integers, $E= \{\, \{j,j+1\} \; : \; j \in \mathbb{Z}\,\}$.  Often we will restrict ourselves to a finite portion of this graph, like $V_n=\{0,1,\ldots, n\}$ and the associated edges $E_n = \{\, \{ j,j+1 \} \;: \; j \in \{0,1, \ldots, n-1\} \,\}$.

\begin{figure}
    \centering
    \includegraphics[width=10cm]{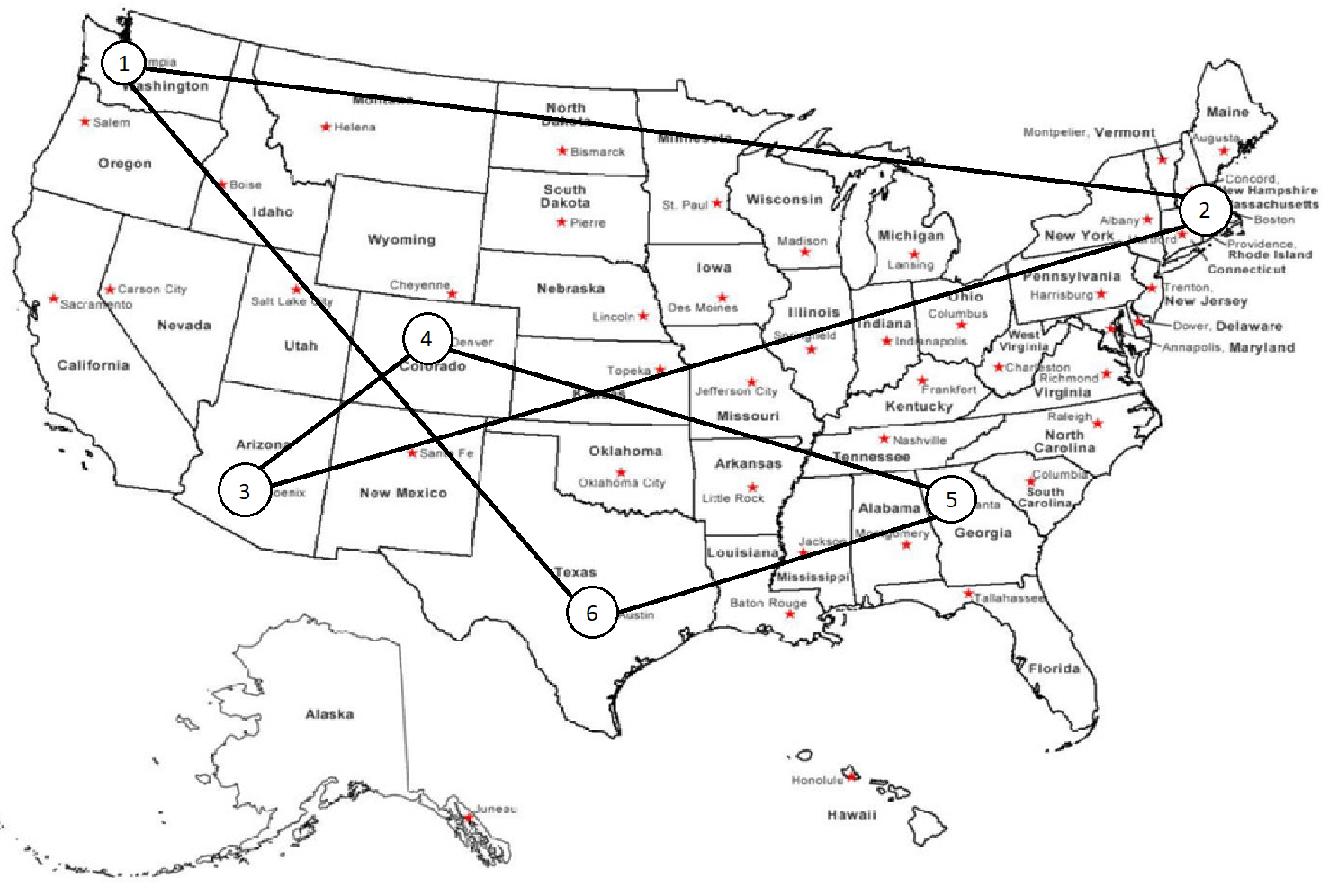}
    \caption{\small Another embedding of the 6-cycle.  As a \emph{graph}, this is identical to the hexagonal embedding of the 6-cycle in Figure \ref{Fig:Hexagon}.}
    \label{Fig:America}
\end{figure}

Graphs can also be created in non-deterministic ways. In Figure \ref{Fig:ErdosRenyi} we have a sample of a \emph{Erd\H{o}s-R\'{e}nyi} \emph{random graph}, which is generated via two parameters: the number of vertices $n$, and a probability $0 \leq p \leq 1$ for edge inclusion.  For each pair $1 \leq i \neq j \leq n$ of the vertices, the edge $\{i,j\}$ is included with probability $p$, independent of all the other edges.  In Figure \ref{Fig:ErdosRenyi}, there are $n=50$ vertices and $p=0.2$.

\begin{figure}[b]
    \centering
    \scalebox{0.75}{
    \begin{tikzpicture}
    \tikzset{
        node/.style={circle, draw=black, very thick},
        arrow/.style={->, black, very thick},
        line/.style={black, very thick}
        }
        
        \node[] (oneminus) at (-5,0) {};
        \node[node, scale=0.85] (one) at (-4,0) {$-2$};
        \node[node, scale=0.85] (two) at (-2,0) {$-1$};
        \node[node] (three) at (0,0) {0};
        \node[node] (four) at (2,0) {1};
        \node[node] (five) at (4,0) {2};
        \node[] (fivep) at (5,0) {}; 
        \node[] (jminus) at (7,0) {};
        \node[node] (j) at (8,0) {$n$};
        \node[] (jplus) at (9,0) {};
    
        \draw[line] (oneminus) -- (one);
        \draw[line] (one) -- (two);
        \draw[line] (two) -- (three);
        \draw[line] (three) -- (four);
        \draw[line] (four) -- (five);
        \draw[line] (five) -- (fivep);
        \draw[line] (jminus) -- (j);    
        \draw[line] (j) -- (jplus);
        
        \path (-6,0) -- node[auto=false]{\ldots} (oneminus);
        \path (fivep) -- node[auto=false]{\ldots} (jminus);
        \path (jplus) -- node[auto=false]{\ldots} (10,0);
        
        % \draw[black, very thick, dashed] (five) -- (j);
        % \draw[black, very thick, dashed] (j) -- (8,0);
        % \draw[black, very thick, dashed] (-6,0) -- (one);
        
    \end{tikzpicture}
    }
    \caption{\small The integers $\mathbb{Z}$ as an infinite graph.}
    \label{Fig:NumberLine}
\end{figure}
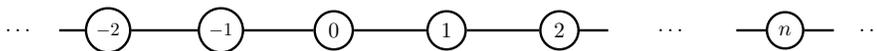

Random graphs are rather ubiquitous in applications, and can be used to analyze communication networks, model groupings of friends on social media, and describe how a respiratory illness spreads through a population, to give several examples.  They also prove to be useful tools in some areas of pure mathematics, such as combinatorics and graph theory.

    % n=50, p=0.2
\begin{figure}
    \centering
    \includegraphics[trim=140 70 100 40, clip, width=0.6\textwidth]{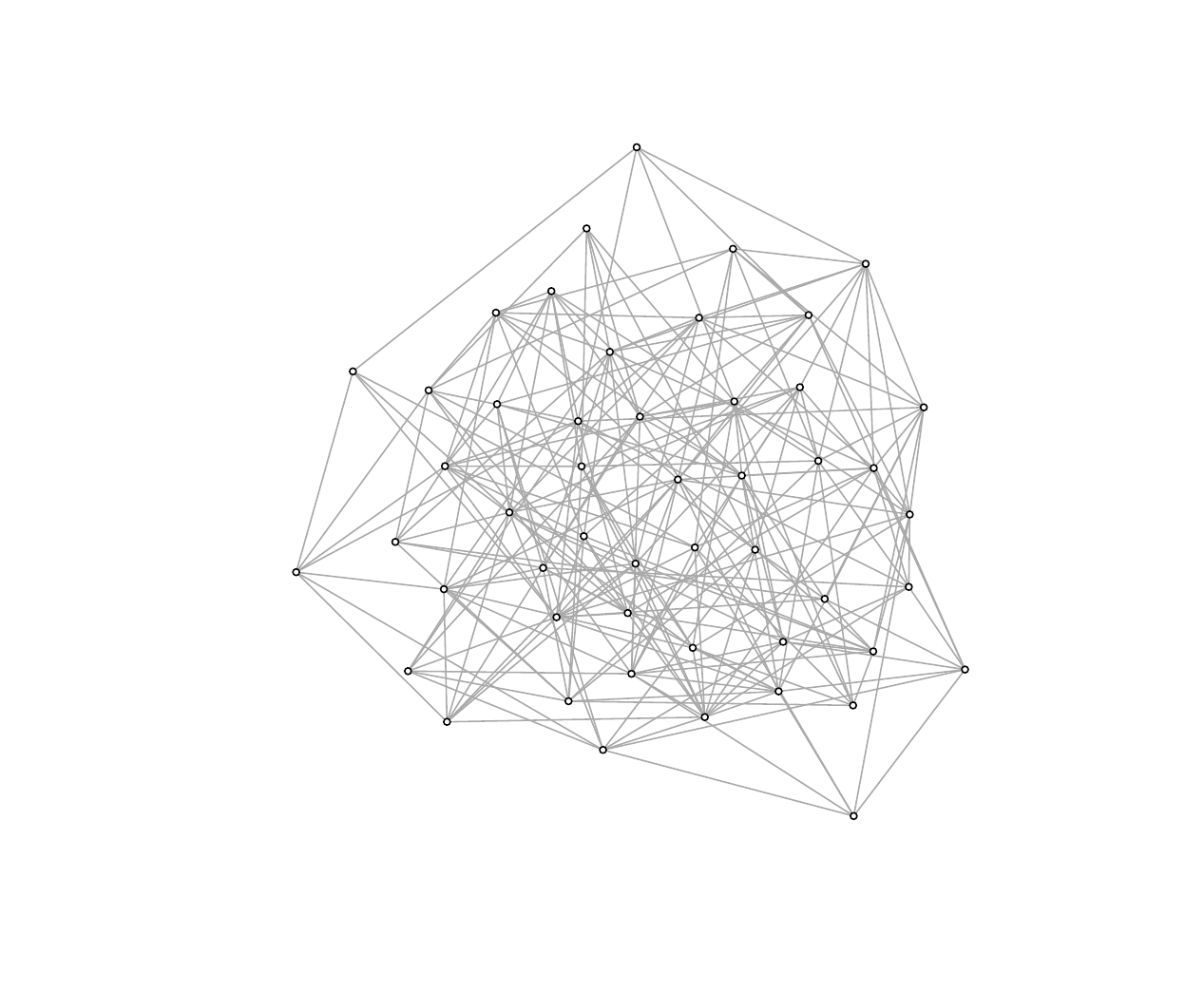}
    \caption{\small An Erd\H{o}s-R\'{e}nyi random graph with $n=50$ vertices and edge inclusion probability $p=0.2$.}
\label{Fig:ErdosRenyi}
\end{figure}

\subsection{Simple symmetric random walks on graphs}\label{SSec:SRW}

We are interested in graphs because we want to move around them \emph{randomly}.  The idea is simple: start at some vertex $X_0 \in V$ at time $t=0$, and at time $t=1$ choose a new vertex $X_1 \in V$, and at $t=2$ a new vertex $X_2 \in V$, and so on, generating a random sequence $X_0, X_1, X_2, \ldots \subset V$.  The key question is how to choose $X_{n+1}$ given that our current vertex is $X_n=v$.  We could just select $X_{n+1}$ uniformly at random among all vertices.  That is not terribly interesting, however, because we are not respecting the geometry of the graph.  So a natural way is to move to any of the \emph{adjacent} vertices of $X_n$, each chosen with equal probability. 

Let us introduce some notation to make things more precise. We say the \textbf{degree}  $\deg(v)$ of a vertex $v$ is the number of \textbf{adjacent vertices} to $v$.  This is the number of edges $\{v,w\} \in E$, or the number of edges coming out of $v$.  In Figure \ref{Fig:FirstGraph1}, $\deg(1)=1, \deg(2)=\deg(3)=\deg(4)=3$, and $\deg(5)=2$.  So if our current position is $X_n =v$, then our next vertex $X_{n+1}$ will be one of the $\deg(v)$ adjacent vertices to $X_n$, each chosen with equal probability.  We thus have the transition probabilities
\begin{align}\label{Eq:SRWProb}
    \nP(X_{n+1} = w \, | \, X_n = v) = \frac{1}{\deg(v)} \qquad \text{for each} \; w \sim v.
\end{align}
Once we select a starting vertex $X_0$, the resulting sequence of random vertices $X_0, X_1, X_2, \ldots$ is a Markov chain which is called the \textbf{simple symmetric random walk} on the graph.  Here \emph{simple} means that we only jump to neighboring vertices, and \emph{symmetric} means that each neighbor is chosen with equal probability. Often in discussion we will be informal and just call this the \textbf{simple random walk}, although for exercises and theorem statements we will use the full name.

If, at each step, we had chosen $X_{n+1}$ uniformly at random in $V$, irrespective of the value of $X_n$, then all the $X_n$'s would be independent of each other.  This is certainly not the case for the simple random walk, though: the next step $X_{n+1}$ heavily depends on our current location $X_n$.

Going back to our examples, suppose $X_n=2$ in Figure \ref{Fig:FirstGraph1}. Then $X_{n+1}$ will be either 1,3 or 4, each chosen with probability $1/3$. If we are currently at vertex 5, we next jump to either 3 or 4, each with probability $1/2$.  If $X_n=1$, then $X_{n+1}=2$ because that is the only option. We observe that, because the vertices have differing degrees,
\begin{align*}
    \nP(X_{n+1} = 2 \, | \, X_n = 1) = 1 \neq 1/3 = \nP(X_{n+1} = 1 \, | \, X_n = 2).
\end{align*}

For the 6-cycle in Figure \ref{Fig:Hexagon}, we flip a fair coin at each step to determine if we move clockwise (CW) or counter-clockwise (CCW).

Note that if we sum the transition probability \eqref{Eq:SRWProb} over \emph{all} $w \sim v$, we have
\begin{align}\label{Eq:SRWTransProb}
    \sum_{w \sim v} \nP(X_{n+1} = w \, | \, X_n = v) = \sum_{w \sim v} \frac{1}{\deg(v)} = \deg(v) \cdot \frac{1}{\deg(v)} = 1,
\end{align}
as we would expect: our probabilities for how we take the next step should add up to 1.  Note also that how we take the next step is determined \emph{entirely} from our current position; the earlier history has no influence.  If we are currently at $v=5$ on the 6-cycle, our position two steps ago, or two thousand steps ago, has no say in how we determine the next step: we move to either 4 or 6 with equal probability. This ``forgetfullness'' is called the \textbf{Markov property} and is why the simple random walk on a graph is a particular case of a Markov chain.\footnote{We'll formally define the Markov property below in Definition \ref{Def:MC}.}

\begin{exercise}
    For an Erd\H{o}s-R\'{e}nyi random graph with $n$ vertices and edge probability $p$, find the expected degree $\nE(\deg(v))$ of a vertex $v$.
\end{exercise}
    % \item For a fixed number $n$ of vertices, how many possible graphs are there with exactly $m$ edges, $1 \leq m \leq \binom{n}{2}$?  (Here, as in Erd\H{o}s-R\'{e}nyi model, we only allow edges between distinct vertices $i \neq j$; no loops $\{i,i\}$.)
\begin{exercise}
    \begin{enumerate}[$(i)$]
        \item For the simple symmetric random walk on the graph in Figure \ref{Fig:FirstGraph1}, compute $\nP(X_{n+1} = 2 \, | \, X_n=2)$, $\nP(X_{n+2} = 2 \, | \, X_n=2)$ and $\nP(X_{n+3} = 2 \, | \, X_n=2)$.
        \item Find the same probabilities for the simple symmetric random walk on the graph in Figure \ref{Fig:Hexagon}.
    \end{enumerate}
\end{exercise}
    
\subsection{Weighted walks on graphs}\label{Sec:NonSRW}

Sometimes we may wish to bias moving one direction over another.  In general, for each vertex $v$ and adjacent vertex $w \sim v$, we can assign a probability $p_{vw} \geq 0$ for how likely it is to move from $v$ to $w$, where
\begin{align}\label{Eq:WeightedWalk}
     \sum_{w \sim v} \nP(X_{n+1} = w \, | \, X_n = v) = \sum_{w \sim v} p_{vw} = 1,
\end{align}
as for the simple walk \eqref{Eq:SRWTransProb}.
% Often our random walks on graphs will be the simple, as described above, but not always; we may wish to bias moving one direction over another, and hence not chose the among the adjacent vertices with equal probability.

In the 6-cycle, for example, we may want to move clockwise with probability $p$, and counterclockwise with probability $1-p$, for some $0 \leq p \leq 1$.  See Figure \ref{Fig:SimpleVsBiased}.  Similarly with the number line: perhaps we want to go right with probability $p$, and left with probability $1-p$. Such walks thus need not be symmetric, are called in general \textbf{weighted} or \textbf{biased random walks}.  In this more general case we don't have to have any fixed ``rule'' of how to assign the probabilities, such as saying the probability of going CW is always $p$.  The only requirement is that \eqref{Eq:WeightedWalk} holds at each vertex $v$.  And as in the case of simple random walk, we may or may not have $p_{vw} = p_{wv}$.

% Note that we need not have $p_{vw} = p_{wv}$, as in the 6-cycle, $p_{65}= \frac{9}{10} \neq \frac{1}{10} = p_{56}$.  
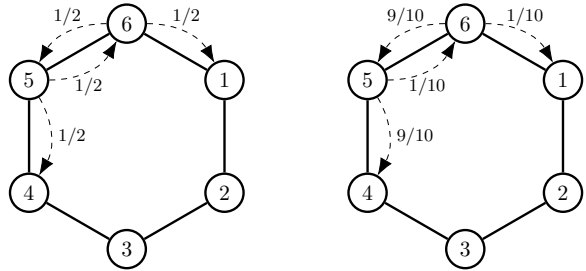
\begin{figure}
    \centering
    \scalebox{0.75}{
    \begin{tikzpicture}
    \tikzset{
        node/.style={circle, draw=black, very thick, minimum size=4mm},
         arrow/.style={-{Latex[length=3mm]}, black},
        line/.style={black, very thick}
    }
    
    %First hexagon
    \node[node] (one) at (-1.268,1) {1};
    \node[node] (two) at (-1.268,-1) {2};
    \node[node] (three) at (-3,-2) {3};
    \node[node] (four) at (-4.732,-1) {4};
    \node[node] (five) at (-4.732,1) {5};
    \node[node] (six) at (-3,2) {6};
    
    \draw[line] (one) -- (two);
    \draw[line] (two) -- (three);
    \draw[line] (three) -- (four);
    \draw[line] (four) -- (five);
    \draw[line] (five) -- (six);
    \draw[line] (six) -- (one);

    \draw [arrow, dashed] (six) edge [bend left] node[above]{\footnotesize $1/2$}  (one);
    \draw [arrow, dashed] (six) edge [bend right] node[above]{\footnotesize $1/2$}  (five);
    
    \draw [arrow, dashed] (five) edge [bend right] node[below]{\footnotesize $1/2$}  (six);
    \draw [arrow, dashed] (five) edge [bend left] node[right]{\footnotesize $1/2$}  (four);
    
     %Second hexagon
    \node[node] (one2) at (4.732,1) {1};
    \node[node] (two2) at (4.732,-1) {2};
    \node[node] (three2) at (3,-2) {3};
    \node[node] (four2) at (1.268,-1) {4};
    \node[node] (five2) at (1.268,1) {5};
    \node[node] (six2) at (3,2) {6};
    
    \draw[line] (one2) -- (two2);
    \draw[line] (two2) -- (three2);
    \draw[line] (three2) -- (four2);
    \draw[line] (four2) -- (five2);
    \draw[line] (five2) -- (six2);
    \draw[line] (six2) -- (one2);

    \draw [arrow, dashed] (six2) edge [bend left] node[above]{\footnotesize $1/10$}  (one2);
    \draw [arrow, dashed] (six2) edge [bend right] node[above]{\footnotesize $9/10$}  (five2);
    
    \draw [arrow, dashed] (five2) edge [bend right] node[below]{\footnotesize $1/10$}  (six2);
    \draw [arrow, dashed] (five2) edge [bend left] node[right]{\footnotesize $9/10$}  (four2);
    
    \end{tikzpicture}}
    \caption{\small A simple and biased walk on the 6-cycle.}
    \label{Fig:SimpleVsBiased}
\end{figure}

\section{Transition matrices}

Whether we do the simple or the biased walk on our graph, we need some way of keeping track of the probabilities of moving between the vertices.  For a graph of $n$ vertices, we do this through an $n \times n$ matrix $P$, where entry $P_{ij}$ of $P$ is the probability $\nP(X_{k+1} = j \, | \, X_k = i)$.  For us that this probability is independent of the time $k$, so we could also say $P_{ij} = \nP(X_{1} = j \, | \, X_0 = i)$, for instance. The resulting \emph{transition matrix} $P = (P_{ij})$ will be a key tool in our analysis of Markov chains.
\begin{definition}
    A \textbf{transition matrix}, or \textbf{transition probability matrix}, is an $n\times n$ matrix $P$ satisfying the following two conditions.
    \begin{enumerate}[$(i)$]
        \item All the entries are non-negative: $P_{jk} \geq 0$ for all $1 \leq j,k \leq n$.
        \item Each row sums to one: $\sum_{k=1}^n P_{jk} =1$ for each $j$.
    \end{enumerate}
\end{definition}
Every random walk (simple, or with weights) on a graph corresponds to a transition matrix $P$, and conversely each transition matrix defines a random walk.  For example, the matrix for the simple walk on the graph in Figure \ref{Fig:FirstGraph1} is
\begin{align}\label{Eq:TransMat1}
   P_1= \begin{bmatrix}
        0 & 1 & 0 & 0 & 0\\
        1/3 & 0 & 1/3 & 1/3 & 0\\
        0 & 1/3 & 0 & 1/3 & 1/3\\
        0 & 1/3 & 1/3 & 0 & 1/3\\
        0 & 0 & 1/2 & 1/2 & 0
    \end{bmatrix}.
\end{align}
Entry $P_{23}$ is $1/3$ because the probability of moving to 3, starting from vertex 2, is 1/3.  $P_{14}=0$ because there is no edge from vertex 1 to 4.  And so on. The lack of symmetry across the diagonal shows instances where $p_{ij}$ is not the same as $p_{ji}$.  The transition matrix for the 6-cycle is
\begin{align}\label{Eq:TransMatHexagon}
   P_2= \begin{bmatrix}
        0 & 1/2 & 0 & 0 & 0 & 1/2\\
        1/2 & 0 & 1/2 & 0 & 0 & 0\\
        0 & 1/2 & 0 & 1/2 & 0 & 0\\
        0 & 0 & 1/2 & 0 & 1/2 & 0\\
        0 & 0 & 0 & 1/2 & 0 & 1/2\\
        1/2 & 0 & 0 & 0 & 1/2 & 0
    \end{bmatrix}.
\end{align}
As you proceed in our study you may need to review your class notes on matrices from time to time; we will eventually see that surprisingly deep and rich properties of our Markov chains are encoded in $P$.

\begin{exercise}\label{Ex:FirstTransMatrix}
Consider the transition matrix $P_1$ in \eqref{Eq:TransMat1}.
    \begin{enumerate}[$(i)$]
        \item What do the rows add up to?  Explain what this means intuitively.
        \item What do the columns add up to? 
        \item Give a (right) eigenvector for $P_1$ corresponding to eigenvalue $\lambda=1$.
        \item Answer the same three questions for $P_2$ in \eqref{Eq:TransMatHexagon}.
        \item What does it intuitively mean that the diagonal of both matrices have only zeros?
        \item Use a computer to compute the eigenvalues and their multiplicities for either $P_1$ or $P_2$. What do you notice?
    \end{enumerate}
\end{exercise}

\begin{exercise}
    Draw a graph $G$ for which the matrix
    \begin{align}
   P= \begin{bmatrix}
        0 & 1 & 0 & 0 & 0 \\
        1 & 0 & 0 & 0 & 0\\
        0 & 0 & 0 & 1/2 & 1/2\\
        0 & 0 & 1/2 & 0 & 1/2\\
        0 & 0 & 1/2 & 1/2 & 0
    \end{bmatrix}
\end{align}
is the transition matrix for the simple symmetric random walk on $G$.  Use a computer to compute the eigenvalues of $P$ along with their multiplicities.  What do you notice?  Compare and contrast with Exercise \ref{Ex:FirstTransMatrix} $(vi)$.
\end{exercise}

% Lecture 2 %%%%%%%%%%%%%%%%%%%%%%%%%%%%%%%%%%%%%%%%

\section{Markov chains}

We are now ready to define the leading actor of our book, a Markov chain. A Markov chain is, first of all, a \emph{stochastic process}.

\begin{definition}\label{Def:StochasticProcess}
Given a state space $\Omega$, a discrete-time\footnote{It is possible to consider stochastic processes $(X_t)$ with continuous time $t \in \mathbb{R}$, but the Markov chains we consider here always have discrete time.} \textbf{stochastic process} on $\Omega$ is a sequence of random variables $(X_0, X_1, X_2, \ldots)$, all defined on the same probability space, such that every $X_j$ takes values in $\Omega$. 
\end{definition}

So a stochastic process is a random sequence $(X_0, X_1, X_2, \ldots)$ of elements of $\Omega$. We typically think of a stochastic process having infinite length, as opposed to the random vectors of finite length that you might have encountered before.  We may think of a stochastic process as a ``movie'': time is unfolding in discrete intervals $t=0,1,2,\ldots$, and each new step in time yields a new random element of $\Omega$.

While each of the $X_k$ have the same codomain $\Omega$, they also have the same \emph{domain}, which is the underlying probability space in Definition \ref{Def:StochasticProcess}.  That is, there is a collection of outcomes $\{\omega\}$, and for every outcome $\omega$, each of the random variables $X_k$ takes a value $X_k(\omega) \in \Omega$.  So, every $X_k$ is a \emph{function} from the \emph{same} underlying probability space to the state space $\Omega$. This is important but will be largely swept under the rug in our considerations\footnote{The reader may consult \cite[Chapter 1]{ASV}, for instance, to learn more about the formalism behind probability spaces and random variables.}
% The underlying probability space (and its associated $\sigma$-algebra, etc) will largely lie the in shadows for us; our concern will be with the \emph{distributions} of the random variables $X_k$ on our state space $\Omega$. Put another way, we will not be very worried with measure-theoretic details of probability theory, although they are implicitly in the background.}

Not all stochastic processes are Markov chains. To be a Markov chain, we additionally need the ``forgetfulness'' property that we saw with random walks on graphs: our transition probabilities must depend only upon our current state.  
% The rigorous formulation is as follows.

\begin{definition}\label{Def:MC}
    Given a set $\Omega$ with $n$ elements labeled by $\{1,2,\ldots, n\}$ and an $n \times n$ transition matrix $P$, a \textbf{Markov chain} with state space $\Omega$ and transition matrix $P$ is a stochastic process $X=(X_0, X_1, X_2, \ldots)$ on $\Omega$ such that $X$ satisfies \emph{the Markov property}
    \begin{multline}\label{Eq:MarkovProperty}
        \nP(X_{k+1}=y \, | \, X_0 = j_0, X_1 = j_1, \ldots, X_{k-1}=j_{k-1}, X_k=x\}\\
        = \nP(X_{k+1}=y \, | \, X_k = x) = P_{xy}
    \end{multline}
    for any $k\ge 1$ and any $j_0, j_1, \ldots, j_{k-1},x,y \in \Omega$.
\end{definition}

At first glance the definition is rather unwieldy, and so let's make slow down to make several observations.  First, pay special attention to what the Markov property \eqref{Eq:MarkovProperty} is saying: the history of how we arrived at state $x$ does not matter for the next jump - only the fact that we are now at $x$. This is what we saw with our random walks on graphs, and this is what makes Markov chains special among stochastic processes. 

The elements of $\Omega$ are called the \emph{states} of the Markov chain. For a random walk on a graph, the states are the vertices of the graph. 
We will see later that, if we know the distribution of $X_0$, from the above definition we get the joint distribution of the random vector $(X_0, X_1, \ldots, X_k)$, for every $k=1,2,\ldots$. It follows that the probabilistic properties of a Markov chain is completely determined by three things:  its state space $\Omega$, the transition matrix $P$, and the distribution of the initial state $X_0$, which is called the \textit{initial distribution} of the chain. 

In Definition \ref{Def:MC}, the \emph{state space} $\Omega$ is finite. We will describe a few examples of Markov chains for infinite but discrete $\Omega$ (say the set of natural numbers) later in the text. Markov chains with finite state space are sometimes called finite Markov chains to distinguish from the infinite case. Whether or not $\Omega$ is finite, keep in mind that the number of steps $k$ in the chain need not be bounded. Most of the interesting questions, in fact, deal with the chain's behavior as $k \rightarrow \infty$, and this will be the focus of much of our attention.

Some authors describe the the second equality in \eqref{Eq:MarkovProperty} as the \emph{time homogeneous Markov property} and the corresponding stochastic process as a \emph{time-homogeneous Markov chain}. There are Markov chains that are not time-homogeneous, in the sense that the transition matrix varies in time. However, our Markov chains will always be time-homogeneous, so we will ignore this additional qualifier.

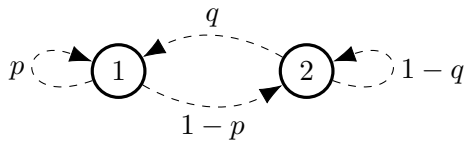
\begin{figure}[b]
    \centering
    \scalebox{1}{
    \begin{tikzpicture}
    \tikzset{
        node/.style={circle, draw=black, very thick, minimum size=7mm},
        arrow/.style={-{Latex[length=3mm]}, black},
        line/.style={black, very thick}
    }
    
    \node[node] (one) at (0,0) {1};
    \node[node] (two) at (2.5,0) {2};
    
    \draw [arrow, dashed] (one) to [out=-30,in=210, bend right] node[below]{$1-p$}  (two);
    \draw [arrow, dashed] (two) to [out=150,in=30, bend right] node[above]{$q$}  (one);
    
    \draw[thin, arrow, dashed] (two) to [out=340,in=20,looseness=11] node[midway, right] {$1-q$} (two) ;
    \draw[thin, arrow, dashed] (one) to [out=200,in=160,looseness=11] node[midway, left] {$p$} (one);

    \end{tikzpicture}}
    
    \caption{\small A diagram of the Markov chain in Example \ref{Eg:2by2pq}.  Note that we could also draw this as a graph with three edges (namely, $\{1,1\}$, $\{1,2\}$, and $\{2,2\}$) and keep in mind that the probability of traversing the edge $\{1,2\}$ depends on the direction of travel (as also happens in general for simple random walks).}
    \label{Fig:Lect2Graph}
\end{figure}

\begin{example}\label{Eg:2by2pq}
    Let's consider a simple example.  Let $\Omega= \{1,2\}$ and 
    \begin{align}\label{Eq:Lect2Eg}
        P = \begin{bmatrix}
                p & 1-p\\
                q & 1-q
            \end{bmatrix}
    \end{align}
for some $0<p<1$ and $0<q<1$. This corresponds to a weighted random walk on a graph with two vertices, labelled 1 and 2, as in Figure \ref{Fig:Lect2Graph}. Note that this graph has loops, or edges $\{i,i\}$ that connect the same vertex to itself. Our transition matrix  \eqref{Eq:Lect2Eg} says
\begin{align*}
    \nP(X_{j+1} = 1 \, | \, X_j = 1) &= P_{11} = p,\\
     \nP(X_{j+1} = 2 \, | \, X_j = 1) &= P_{12} = 1-p,\\
      \nP(X_{j+1} = 1 \, | \, X_j = 2) &= P_{21} = q,\\
       \nP(X_{j+1} = 2 \, | \, X_j = 2) &= P_{22} = 1-q.
\end{align*}
Drawing a \textit{graphical representation} of the Markov chain, as in Figure \ref{Fig:Lect2Graph}, is often a very helpful way of keeping track of these probabilities and developing intuition for the chain's behavior. All finite time-homogeneous Markov chains have a graphical representation as a weighted random walk on a graph with loops. 

\begin{exercise}
    Give a graphical representation of the Markov chain associated to the transition matrix
    \begin{align*}
        P = \begin{bmatrix}
                0 & 0 & 1 & 0 & 0 & 0\\
                0 & 1/2 & 0 & 0 & 2/3 & 0\\
                0 & 0 & 1 & 0 & 0 & 0\\
                0 & 0 & 0 & 0 & 0 & 1\\
                0 & 2/3 & 0 & 0 & 1/3 & 0\\
                1/5 & 1/5 & 1/5 & 1/5 & 1/5 & 0
            \end{bmatrix}.
    \end{align*}
\end{exercise}

Continuing with Example \ref{Eg:2by2pq}, we know that if we start at vertex 1, the probability of moving to state 2 is $1-p$.  But how do we compute probabilities for more than one step?  What is 
\begin{align*}
    \nP(X_2 = 2 \, | \, X_0 = 1),
\end{align*}
for instance?  This is where the Markov property \eqref{Eq:MarkovProperty} comes into play.  Recalling $\nP(A \text{ and }B) = \nP(A|B) \nP(B)$, we see
\begin{align}
    \nP(X_2 = 2 | X_0 = 1) &= \nP(X_2=2, X_1=2 | X_0=1) + \nP(X_2=2, X_1=1| X_0=1) \notag\\
    &= \nP(X_2=2 \, | \, X_1=2, X_0=1)\nP(X_1=2 \, |\, X_0=1) \notag\\
   &\qquad  + \nP(X_2=2 \,|\, X_1=1, X_0=1)\nP(X_1=1 \, | \, X_0=1)\notag\\
   &= \nP(X_2=2 \, | \, X_1=2)\nP(X_1=2 \, |\, X_0=1) \label{Eg:FirstMPUsage}\\
   &\qquad  + \nP(X_2=2 \,|\, X_1=1)\nP(X_1=1 \, | \, X_0=1) \notag \\
   &= (1-q)(1-p) + (1-p)p,\label{Eg:FirstMCAns1}
\end{align}
where we used the Markov property in \eqref{Eg:FirstMPUsage} to eliminate conditioning on the first step.

Similarly, we find
\begin{align}
    \nP(X_2 = 1 | X_0 = 1) &= \nP(X_2=1, X_1=2 | X_0=1) + \nP(X_2=1, X_1=1 | X_0=1) \notag\\
    &= \nP(X_2=1 \, | \, X_1=2, X_0=1)\nP(X_1=2 \, |\, X_0=1) \notag\\
   &\qquad  + \nP(X_2=1 \,|\, X_1=1, X_0=1)\nP(X_1=1 \, | \, X_0=1)\notag\\
   &= \nP(X_2=1 \, | \, X_1=2)\nP(X_1=2 \, |\, X_0=1) \notag\\
   &\qquad  + \nP(X_2=1 \,|\, X_1=1)\nP(X_1=1 \, | \, X_0=1) \notag \\
   &= q(1-p) + p^2\label{Eg:FirstMCAns2} 
\end{align}
\qed

\end{example}

\begin{exercise}
    Compute $\nP(X_2 = 2 \,|\, X_0 =2)$. Suppose $X_0$ has the following initial distribution $\nP(X_0=1)=\nP(X_0=2)=1/2$. Find the joint probability $\nP(X_0=2, \; X_2=2)$.
\end{exercise}

%Tim to here 10-15-25

\section{A first look at asymptotic behavior}\label{Sec:FirstLook}

Obviously the approach outlined above will get tedious very quickly.  What if we wish to find $\nP(X_3 = 2 \,|\, X_0 =2)$ or $\nP(X_5 = 2 \,|\, X_0 =2)$ or 
\begin{align}\label{Eq:HighNProb}
    \nP(X_{1000} = 2 \,|\, X_0 =2)?
\end{align}
This sort of question is what we mean by the \textbf{asymptotic behavior} of the Markov chain.  That is, if we run the chain for a \emph{very large} number of steps, what is the probability it is at any given state?  Are some vertices more likely than others?  Does a limiting distribution exist? If we change the value of $X_0$ do the above questions give a different answer? The goal of this text is to answer some of these questions.

Thinking concretely about \eqref{Eq:HighNProb}, it is clear that working out all the possible chains starting and ending at 2 through 1000 steps is not feasible.  Fortunately, matrix algebra comes in to save the day in a remarkable way.  Notice that the square of the matrix $P$ in \eqref{Eq:Lect2Eg} is
\begin{align*}
    P^2 = \begin{bmatrix}
            p^2 + q(1-p) & p(1-p)+ (1-p)(1-q)\\
            pq + q(1-q) & (1-p)q + (1-q)^2
        \end{bmatrix}.
\end{align*}
We observe that $P^2_{12}$ is exactly our answer \eqref{Eg:FirstMCAns1} for $\nP(X_2 = 2 \,|\, X_0 = 1)$, while $P^2_{11}$ is what we found in \eqref{Eg:FirstMCAns2} for $\nP(X_2 = 1 \,|\, X_0 = 1)$.  This is not a coincidence: the $(j,k)$ entry $P^n_{jk}$ of the $n$th power of $P$ is exactly what you get for summing the probabilities of all possible paths from $j$ to $k$ in $n$ steps, using the Markov property.  We formalize this useful fact in the following theorem.

%%% Lecture 3 %%%%%%%%%%%%%%%%%%%%%%%%%%%%%%%%%%%%

\begin{theorem}\label{Thm:PowersOfP}
    For any $k \in \mathbb{N} = \{1,2,3, \ldots\}$ and $i,j \in \Omega$,
    \begin{align}\label{Eq:PowersOfP}
        \nP(X_k = j \,|\, X_0 = i) = P^k_{ij},
    \end{align}
    the $(i,j)$-entry of the $k$th power of the transition matrix $P$.
\end{theorem}
We will use the notations $P_{ij}^k$ and $P^k(i,j)$ interchangeably.  This theorem tells us that $P_{ij}^k = P^k(i,j)$ is the probability of moving from $i$ to $j$ in exactly $k$ steps.
\begin{proof}
    We use induction on $k$.  By the definition of $P$, \eqref{Eq:PowersOfP} holds for $k=1$.  Now suppose \eqref{Eq:PowersOfP} holds for some $k=m \geq 1$.  We need to show it holds for $k=m+1$.  We condition on the $m$th step and observe
    \begin{align}
        \nP(X_{m+1} = j& \,|\, X_0 = i) \notag\\
        &= \sum_{l=1}^n \nP(X_{m+1} = j, X_m=l \,|\, X_0 = i) \notag\\
        &= \sum_{l=1}^n \nP(X_{m+1} = j \, |\, X_m=l, X_0 = i)\nP(X_m=l \,|\, X_0=i) \notag\\
        &= \sum_{l=1}^n \nP(X_{m+1} = j \, |\, X_m=l)\nP(X_m=l \,|\, X_0=i) \label{Eq:PowersOfP1}\\
        &= \sum_{l=1}^n P_{lj}P^m_{il}= \sum_{l=1}^n P^m_{il}P_{lj}  \label{Eq:PowersOfP2}\\
        &= (P^m \cdot P)_{ij} = P^{m+1}_{ij},  \label{Eq:PowersOfP3}
    \end{align}
    where we used the Markov property in \eqref{Eq:PowersOfP1} and the inductive hypothesis in the first equality of \eqref{Eq:PowersOfP2}, and \eqref{Eq:PowersOfP3} is just the definition of matrix multiplication.  This proves our desired formula for the case $m+1$, completing the induction argument.
\end{proof}

Note that \eqref{Eq:PowersOfP} is stated in terms of the number of steps from the initial position.  But what about probabilities like
\begin{align*}
    \nP(X_{100} = y \, | \, X_{90} = x)?
\end{align*}
By the Markov property, it is as if the entire chain starts afresh at step 90.  So the only thing that matters is the number of steps between the two times, and so $\nP(X_{100} = y \, | \, X_{90} = x) = \nP(X_{10} = y \, | \, X_{0} = x) = P^{10}_{xy}$.

\begin{example}[The 6-cycle, revisited]\label{Eg:Hexagon2}
    Let's go back to our simple symmetric random walk on the 6-cycle $\Omega = \{1,2,3,4,5,6\}$.  We found its transition matrix $P_2$ above in \eqref{Eq:TransMatHexagon}.  Suppose we wanted to find $\nP(X_{k} = 5 \,|\, X_0 =1)$ for large $k$.  That is, we start our walk at vertex 1, and want to know the likelihood it is at vertex 5 after a large number of steps.  Using a computer, we have
    \begin{align}\label{Eq:P_2^50}
P_2^{50} &\approx        \left[
\begin{array}{cccccc}
 0.333 & 0 & 0.333 & 0 & 0.333 & 0 \\
 0 & 0.333 & 0 & 0.333 & 0 & 0.333 \\
 0.333 & 0 & 0.333 & 0 & 0.333 & 0 \\
 0 & 0.333 & 0 & 0.333 & 0 & 0.333 \\
 0.333 & 0 & 0.333 & 0 & 0.333 & 0 \\
 0 & 0.333 & 0 & 0.333 & 0 & 0.333 \\
\end{array}
\right],
\end{align}
while
\begin{align}\label{Eq:P_2^51}
P_2^{51} &\approx \left[
\begin{array}{cccccc}
 0 & 0.333 & 0 & 0.333 & 0 & 0.333 \\
 0.333 & 0 & 0.333 & 0 & 0.333 & 0 \\
 0 & 0.333 & 0 & 0.333 & 0 & 0.333 \\
 0.333 & 0 & 0.333 & 0 & 0.333 & 0 \\
 0 & 0.333 & 0 & 0.333 & 0 & 0.333 \\
 0.333 & 0 & 0.333 & 0 & 0.333 & 0 \\
\end{array}
\right].
    \end{align}
So, by Theorem \ref{Thm:PowersOfP}, we have $\nP(X_{50} = 5 \,|\, X_0 =1) \approx 1/3$, while, interestingly, $\nP(X_{51} = 5 \,|\, X_0 =1) \approx 0$.  

What is happening?  Note that if we start at vertex 1, then after an \emph{even} number of steps, we must be at an odd-numbered vertex.  Similarly, after an \emph{odd} number of steps, we must be at an even-numbered vertex.  Our $P_2^{50}$ matrix \eqref{Eq:P_2^50} suggests that, after a large even number of steps, our chain is equally likely to be at any of the odd vertices, while \eqref{Eq:P_2^51} suggests that, after a large odd number of steps, we are equally likely to be at any even vertex.  In particular, it appears that
\begin{align}\label{Lim:AsymptoticBehavior5Given1}
    \lim_{k \rightarrow \infty} \nP(X_{k} = 5 \,|\, X_0 =1)
\end{align}
does not exist, since the value alternates between 1/3 and 0 for large values of $k$ - that is, assuming the pattern we see in \eqref{Eq:P_2^50} and \eqref{Eq:P_2^51} continues, which we would need to verify.
% \footnote{On the other hand, supposing that this pattern continues, it appears we could take the limits $\lim_{k \rightarrow \infty} \nP(X_{2k} = 5 \,|\, X_0 =1)$ and $\lim_{k \rightarrow \infty} \nP(X_{2k-1} = 5 \,|\, X_0 =1)$ for taking even and odd numbers of steps, respectively.  However, we will usually just be interested in the limit over all steps, as in \eqref{Lim:AsymptoticBehavior5Given1}.}
\end{example}

\begin{exercise}\label{Ex:MatrixPowers1}
    Find the transition matrix $P$ for the simple symmetric random walk on the 5-cycle $\{1,2,3,4,5\}$.  Use a computer to compute $P^{50}$ and $P^{51}$.  Compare and contrast to the case of the 6-cycle above.
\end{exercise}

\begin{exercise}\label{Ex:MatrixPowers2}
    Recall the transition matrix $P_1$ in \eqref{Eq:TransMat1} for the graph in Figure \ref{Fig:FirstGraph1}.  Use a computer to compute $P_1^{50}$ and $P_1^{51}$.  What do you find?  How does this compare with the 5-cycle example in Exercise \ref{Ex:MatrixPowers1}?
\end{exercise}

While a computer can typically compute the 50th power of a small transition matrix, it can be computationally expensive.  And sometimes it is just not possible, as with the matrix \eqref{Eq:TransMat1} in Example \ref{Eg:2by2pq} for general $p$ and $q$ (the expressions for each entry would be prohibitively nasty; try, if you dare, in Mathematica or Maple). So while we know that 
\begin{align*}
    \nP(X_{1000} = 2 \,|\, X_0 =2) = P_{22}^{1000},
\end{align*}
we still do not necessarily know how to practically get our hands on this probability.

Let's step back and think more generally. Suppose we have a Markov chain on $\Omega=\{1,2,\ldots, n\}$ with $n\times n$ transition matrix $P$.  Let's start our Markov chain at $X_0=1$, which we could describe as beginning with the probability distribution
\begin{align*}
    \mu_0 = \underbrace{(1,0,0, \ldots, 0)}_n
\end{align*}
on the state space $\Omega$. For $k=1,2,3\ldots$, let 
\begin{align*}
    \mu_k = (\mu_{k1}, \mu_{k2}, \ldots, \mu_{kn})
\end{align*}
be the row vector which is the probability distribution of where our chain is at after $k$ steps.  In other words,
\begin{align*}
    (\mu_{k})_j = \mu_{kj} = \nP(X_k = j \, | \, X_0 =1).
\end{align*}
Do we have a limit $\lim_{k \rightarrow \infty} \mu_{kj}$ for each entry, or a limiting distribution $\lim_{k \rightarrow \infty} \mu_k$ as a whole?

Observe that
\begin{align*}
    \mu_1 &= \big(\nP(X_1 = 1 \, | \, X_0 =1),\nP(X_1 = 2 \, | \, X_0 =1), \ldots, \nP(X_1 = n \, | \, X_0 =1) \big)\\
    &= (P_{11}, P_{12}, \ldots, P_{1n}) = \mu_0 P,
\end{align*}
where $\mu_0 P$ is the row vector $\mu_0$ times the matrix $P$ (which, in this case, just picks out the top row of $P$).  Similarly,
\begin{align*}
    \mu_2 &= \big(\nP(X_2 = 1 \, | \, X_0 =1),\nP(X_2 = 2 \, | \, X_0 =1), \ldots, \nP(X_2 = n \, | \, X_0 =1) \big)\\
    &= \left(P^2_{11}, P^2_{12}, \ldots, P^2_{1n}\right)\\
    &= \mu_0 P^2 = (\mu_0P)P = \mu_1P,
\end{align*}
since we found $\mu_0P = \mu_1$.  Continuing this pattern, we see
\begin{align}\label{Eq:mu_k1}
    \mu_{k+1} = \mu_k P
\end{align}
for all $k$.  Now \emph{suppose} a limiting vector $\pi = \lim_{k \rightarrow \infty} \mu_k$ exists.  Then taking limits on both sides of \eqref{Eq:mu_k1} gives
\begin{align*}
    \pi = \lim_{k \rightarrow \infty} \mu_{k+1} &= \lim_{k \rightarrow \infty} (\mu_k P)
    = (\lim_{k \rightarrow \infty} \mu_k )P = \pi P,
\end{align*}
and so such a limit would satisfy
\begin{align}\label{Eq:pi}
    \boxed{\pi = \pi P}
\end{align}
That is, $\pi$ would be a row vector that is a \emph{left eigenvector} of $P$ with eigenvalue 1.  This equation \eqref{Eq:pi} will prove very significant in our investigation of Markov chains and their limiting behavior. We are immediately led to a number of questions:
\begin{enumerate}
    \item When does such a solution $\pi$ of \eqref{Eq:pi} exist? Can $\pi$ be taken to be a probability vector? That is, is there a solution $\pi$ to \eqref{Eq:pi} that has nonnegative coordinates that add up to one? Finally, if such a solution exists, is it unique?
    \item Does 
    \begin{align}\label{Eq:muLimit1}
        \lim_{k \rightarrow \infty} \mu_{kj} = \lim_{k \rightarrow \infty} \nP(X_k=j \, | \, X_0=1) = \pi_j?
    \end{align}
    Or, more generally, does the initial value matter? Is
    \begin{align}\label{Eq:muLimit2}
        \lim_{k \rightarrow \infty} \nP(X_k=j \, | \, X_0=i) = \pi_j
    \end{align}
    for any starting point $i \in \{1,2,\ldots,n\}$?
    
    \item If \eqref{Eq:muLimit1} or \eqref{Eq:muLimit2} holds, what is the rate of convergence?  In other words, how many steps $k$ does it take until the probability we are at vertex $j$ is within a given magnitude of error of $\pi_j$?
\end{enumerate}

\begin{exercise}\label{Ex:Firstpi}
    Show that $\pi=(\frac{1}{6}, \frac{1}{6}, \frac{1}{6}, \frac{1}{6}, \frac{1}{6}, \frac{1}{6})$ is a left eigenvector for the 6-cycle transition matrix \eqref{Eq:TransMatHexagon}.
\end{exercise}

Exercise \ref{Ex:Firstpi} shows that answering question 1 affirmatively above is not the same as that for question 2.  Indeed, we have such a $\pi$ for the simple random walk on the 6-cycle, but our computations in Example \ref{Eg:Hexagon2} suggest that a limiting probability distribution doesn't exist, since the walk alternates between even and odd vertices. 

In order to give a positive answer to both questions 1 and 2, we will see that our Markov chain needs the properties of \emph{irreducibility} and \emph{aperiodicity}, which we introduce in the next section.  In Chapter \ref{Ch:Asymptotics}, we will see that irreducible and aperiodic Markov chains converge exponentially fast to their limiting distributions $\pi$ (see Theorem \ref{Thm:ConvToStationary}), answering question 3.

\section{Irreducibility and aperiodicity}

\begin{definition}
    A Markov chain $X$ on state space $\Omega$ with transition matrix $P$ is \textbf{irreducible} if for any two states $x,y \in \Omega$, there exists $k \in \mathbb{N}$ such that $(P^k)_{x,y}>0$.  In other words, \textit{for any pair of states} $x$ and $y$, there is a positive probability that the chain can move from $x$ to $y$ in some finite number of steps.
\end{definition}

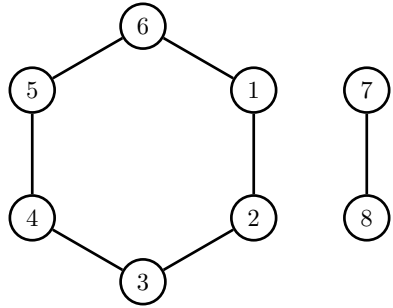
\begin{figure}
    \centering
    \scalebox{0.85}{
    \begin{tikzpicture}
    \tikzset{
        node/.style={circle, draw=black, very thick, minimum size=7mm},
        line/.style={black, very thick}
    }
    
    \node[node] (one) at (1.732,1) {1};
    \node[node] (two) at (1.732,-1) {2};
    \node[node] (three) at (0,-2) {3};
    \node[node] (four) at (-1.732,-1) {4};
    \node[node] (five) at (-1.732,1) {5};
    \node[node] (six) at (0,2) {6};
    \node[node] (seven) at (3.5,1) {7};
    \node[node] (eight) at (3.5,-1) {8};
    
    \draw[line] (one) -- (two);
    \draw[line] (two) -- (three);
    \draw[line] (three) -- (four);
    \draw[line] (four) -- (five);
    \draw[line] (five) -- (six);
    \draw[line] (six) -- (one);
    \draw[line] (seven) -- (eight);
    \end{tikzpicture}}
    \caption{\small Disconnected graphs are not irreducible.}
    \label{Fig:HexagonPlusLine}
\end{figure}

\begin{example}\label{Eg:Irreducibility}
    The random walk on the six-cycle is irreducible.  If $x=1$ and $y=4$, for instance, we have a path $1 \rightarrow 2 \rightarrow 3 \rightarrow 4$ which has probability $(1/2)^3 >0$.  Similarly, any two vertices have a path of positive probability connecting them.
    
    However, if we add a disconnected line segment to the graph, as in Figure \ref{Fig:HexagonPlusLine}, then the random walk is no longer irreducible: there is no positive-probability path from vertex 1 to 7, for example.  
    
     The random walk on the graph in Figure \ref{Fig:FirstGraph1} is also irreducible, as is the weighted walk on the graph in Figure \ref{Fig:Lect2Graph}, so long as $0<p<1$ and $0<q<1$.  However, if either $p \in \{0,1\}$ or $q \in \{0,1\}$, the graph fails to be irreducible.  For instance, if $p=1$, then we cannot get from $x=1$ to $y=2$. \qed
\end{example}

From the point of view of the graph $G=(V,E)$, irreducibility of the simple symmetric random walk on $G$ is equivalent to the connectivity of the graph. That is, there must be a path on the graph from any one vertex to any other. This also works for weighted random walks if all edges have positive probability either way. Inuitively speaking, the graph must consist of only a single piece. 

\begin{exercise}
    Is the Markov chain in Problem \ref{Prob:AbsorbingBarriers1} irreducible?  Why or why not?
\end{exercise}

%%% Lecture 4 %%%%%%%%%%%%%%%%%%%%%%%%%%%%%%%%%

Aperiodicity is a slightly more subtle property.  Start with an irreducible Markov chain $X$, and for a fixed state $x \in \Omega$ set
\begin{align*}
    T(x) := \{\, k \; : \; P^k(x,x) >0 \,\}.
\end{align*}
We say that the \textbf{period of the state $x$} is $\gcd(T(x))$, the greatest common divisor of all the integers in $T(x)$.  The striking thing about this number is that it does not depend on the state $x$.

\begin{theorem}
    If $X$ is an irreducible Markov chain, then
    \begin{align}\label{Eq:PeriodInvariant}
        \gcd(T(x)) =  \gcd(T(y))
    \end{align}
    for any two states $x,y \in \Omega$.
\end{theorem}

Although we omit the proof, the idea is to use the irreducibility of the chain: since $X$ is irreducible/connected, any information about one vertex $x$ can be ``communicated'' to any other vertex $y$ by using a path from $x$ to $y$.  See \cite[Lemma 1.6]{Peres} for proof details.

Given \eqref{Eq:PeriodInvariant}, we can define the \emph{period} of an irreducible Markov chain.

\begin{definition}
    The \textbf{period } of an irreducible Markov chain $X$ is $\gcd(T(x))$, where $x \in \Omega$ is any state. The Markov chain is \textbf{aperiodic} if the period is 1.
\end{definition}

Some exercises and examples will help us start to get a sense of what this definition is getting at.

\begin{exercise}
    Suppose there is a state $x\in \Omega$ such that $P(x,x)>0$. Argue that the Markov chain must be aperiodic. 
\end{exercise}

The previous exercises can be generalized with the help of a fact from number theory to give a helpful interpretation of the period.  If the period of a chain is $d$, then for each state $x$, there exists an $N$ (that may depend on $x$) such that, for all $n \geq N$,
\begin{align*}
    P^{nd}(x,x) > 0.
\end{align*}
In other words, it is possible to return to the same state $x$ in all sufficiently large multiples of the period.    And if the chain is aperiodic, i.e. $d=1$, this says that $P^n(x,x) >0$ for all sufficiently large $n$: it is possible to return to $x$ in \emph{all} sufficiently large numbers of steps.

\begin{example}\label{Eg:WalkOn4Cycle}
    The simple symmetric random walk on the 4-cycle $\{1,2,3,4\}$, as in Figure \ref{Fig:4cycle}, has transition matrix
    \begin{align}\label{Eq:SquareTransMatrix}
        P = \begin{bmatrix}
            0 & 1/2 & 0 & 1/2\\
            1/2 & 0 & 1/2 & 0\\
            0 & 1/2 & 0 & 1/2\\
            1/2 & 0 & 1/2 & 0
        \end{bmatrix}.
    \end{align}
    
\begin{figure}
    \centering
    \scalebox{0.85}{
        \begin{tikzpicture}
    \tikzset{
        node/.style={circle, draw=black, very thick, minimum size=7mm},
        line/.style={black, very thick}
    }
    
    \node[node] (one) at (0,0) {4};
    \node[node] (two) at (2,0) {1};
    \node[node] (three) at (2,-2) {2};
    \node[node] (four) at (0, -2) {3};
    
    \draw[line] (one) -- (two);
    \draw[line] (two) -- (three);
    \draw[line] (three) -- (four);
    \draw[line] (four) -- (one);s
    \end{tikzpicture}}
    \caption{\small Is the walk on the 4-cycle aperiodic?}
    \label{Fig:4cycle}
\end{figure}
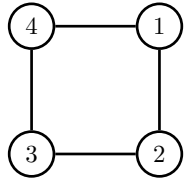

    Since the 4-cycle is a connected graph, the chain is irreducible. Is it aperiodic?  Let's consider $T(1) = \{\, k \; : \; P^k(1,1)>0\,\}$, the set of possible return times to 1.  Some possible paths from 1 to 1 include
    \begin{align*}
        &1 \ra 2 \ra 1  &k=2\\
        &1 \ra 2 \ra 3 \ra 2 \ra 1 &k=4\\
        &1 \ra 2 \ra 3 \ra 4 \ra 1     &k=4\\
        &1 \ra 2 \ra 3 \ra 2 \ra 3 \ra 4 \ra 1      &k=6
    \end{align*}
    We claim that \emph{all} possible paths from 1 to 1 require an even number of steps.  Indeed, given a path $p$, write the length $k_p$ as 
    \begin{align*}
        k_p = a_p+b_p,
    \end{align*}
    where $a_p$ is the number of clockwise steps in the path, and $b_p$ the number of counter-clockwise.  Numbering the paths above as 1 to 4, respectively, we have
    \begin{align*}
        &a_1=b_1 = 1\\
        &a_2=b_2 = 2\\
        &a_3=4, \; b_3=0\\
        &a_4=5, \; b_4=1.
    \end{align*}
    In general, we claim that
    \begin{align}\label{Eg:Periodic1}
        a_p = b_p +4n
    \end{align}
    for some $n \in \mathbb{Z}$, as we can begin to see from our examples.  Why \eqref{Eg:Periodic1}?  As we are travelling from 1 to 1, if we do a ``sub-cycle'' like $3 \ra 2 \ra 3$ in the fourth path, we add the same number of clockwise and counter-clockwise steps.  The remaining difference $4n$ accounts for how many times we loop around the square.

    Thus, the total number of steps is 
    \begin{align*}
        k_p &= a_p + b_p\\
        &= 2b_p + 4n,
    \end{align*}
    which is always even.  Since a return of length 2 is possible, $\gcd(T(1))=2$.  So we see that the walk on the 4-cycle is not aperiodic, but is rather periodic with period 2. 
\end{example}

\begin{exercise}
    For the random walk on the 4-cycle, compute
    \begin{align*}
        \nP(X_{1000}=3 \, | \, X_0=1) \quad \text{ and } \quad \nP(X_{1000}=2 \, | \, X_0=1).
    \end{align*}
\end{exercise}

\begin{example}\label{Eg:WalkOn5Cycle}
    What about a random walk $X$ on the 5-cycle $\{1,2,3,4,5\}$?  This is again clearly irreducible, but it is also periodic? Since paths from 1 to 1 include
    \begin{align*}
       &1 \ra 2 \ra 1    &k=2\\
       &1 \ra 2 \ra 3 \ra 4 \ra 5 \ra 1     &k=5,
    \end{align*}
    we see $T(1)$ contains co-prime integers $2$ and $5$, yielding $\gcd(T(x))=1$.  
\end{example}

So while the 4-cycle walk is periodic, the 5-cycle walk is aperiodic.  These arguments easily generalize to walks on $n$-cycles.

\begin{theorem}
    Let $X$ be a random walk on an $n$-cycle.  If $n$ is even, the period of $X$ is 2.  If $n$ is odd, $X$ is aperiodic.
\end{theorem}

\begin{proof}
    Repeat the logic of examples \ref{Eg:WalkOn4Cycle} and \ref{Eg:WalkOn5Cycle} for the even and odd cases, respectively.
\end{proof}

\begin{exercise}
    Suppose that every non-diagonal element of the transition probability matrix $P$ is strictly positive. That is, $P(x,y)>0$ for all $x\neq y\in \Omega$. If $\#\Omega=2$, additionally assume that $P(x,y) <1$ for some $x \neq y \in \Omega$.  Other than this, we don't assume anything about the diagonal elements (so there are no assumptions on the diagonal when $\#\Omega \geq 3$). Argue that the Markov chain is aperiodic.   
\end{exercise}

\begin{exercise}
    Let $G$ be a connected graph.  Suppose there is a path of odd-length $2k+1$, for some $k$, from one given vertex $x$ to itself.  Show that the random walk on $G$ is aperiodic.
\end{exercise}

%% Lecture 5 %%%%%%%%%%%%%%%%%%%%%%%%%%%%%%%%%%%

\section{Stationary distributions $\pi$}\label{Sec:Stationary}

%Suppose we are at state $X_n=x$ in a simple random walk on the graph $G=(V,E)$.  Recall that to choose the next state, we pick any of the adjacent vertices $y \sim x$ to $x$ with equal probability.  Thus, the entries of the transition matrix for the walk are
%\begin{align*}
%    P_{xy} = \begin{cases}
%        \frac{1}{d_x} & \text{if } y \sim x\\
%        0 & \text{otherwise}
%    \end{cases}
%\end{align*}
%where $d_x = \# \{\, y \in V \; : \; y \sim x \}$.

So far our calculations in terms of probabilities for Markov chains have all been conditional.  We know how to compute things like
\begin{multline*}
    \nP(X_{1} = y \, | \, X_{0} = x) \; \text{ or } \; \nP(X_{14} = y \, | \, X_{0} = x)\; 
    \text{or } \; \nP(X_{350} = y \, | \, X_{125} = x)
\end{multline*}
using the transition matrix $P$ and its powers. But what if we want the \emph{unconditional} probability $\nP(X_1 =y)$?  

Suppose the \textbf{initial distribution} of $X_0$ is given by the vector 
\begin{align*}
    \mu_0 =& (q_1, q_2, \ldots, q_n) := ( \nP(X_0=1), \nP(X_0=2), \ldots, \nP(X_0=n)).
\end{align*}
By the law of total probability,
\begin{align}
    \nP(X_1 = y) &= \sum_{j=1}^n \nP(X_1 = y \, | \, X_0 = j) \nP(X_0=j) 
    = \sum_{j=1}^n P_{jy} q_j \label{Eq:muP}.
\end{align}

So, if we have an explicit initial distribution $\mu_0 = (q_1, q_2, \ldots, q_n)$, we can compute unconditional probabilities. Note that we can view \eqref{Eq:muP} as matrix multiplication
\begin{align*}
    \sum_{j=1}^n P_{jy} q_j = (\mu_0P)_y,
\end{align*}
the $y$th entry of row-matrix product $\mu_0 P$.  Thus, denoting the \emph{unconditional} probability row vector for the first step as $\mu_1$,
\begin{align*}
    \mu_1 :=& ( \nP(X_1=1), \nP(X_1=2), \ldots, \nP(X_1=n)),
\end{align*}
we have
\begin{align}
    \mu_1 =& ( (\mu_0P)_1,  (\mu_0P)_2, \ldots,  (\mu_0P)_n) \notag\\
    =& \mu_0P. \label{Eq:DistributionAfter1Step}
\end{align}
Similarly, by the law of total probability and \eqref{Eq:DistributionAfter1Step},
\begin{align*}
    \mu_2 :=& ( \nP(X_2=1), \nP(X_2=2), \ldots, \nP(X_2=n))\\
    =& \mu_1P = (\mu_0 P)P = \mu_0 P^2,
\end{align*}
and, continuing this logic, we have
\begin{align}
    \mu_k :=& ( \nP(X_k=1), \nP(X_k=2), \ldots, \nP(X_k=n)) =\mu_0 P^k. \label{Eq:mu_k}
\end{align}
So, matrix powers also give us \emph{unconditional} probabilities after multiple steps, so long as we \textit{left-multiply} by the initial distribution.

We may have a starting distribution $\mu_0$ that is \emph{deterministic}, i.e. that always begins the chain at some fixed vertex $j$, and so $\mu_0$ is all zeros except for a single 1 at entry $j$. This is indeed what we did in the derivation leading up to \eqref{Eq:pi}. But it doesn't have to be the case. We can also randomly choose our initial position.  

\begin{example}\label{Eg:HexagonStationary}
    Consider the random walk on the 6-cycle (hexagon) with $X_0 \sim \text{Unif}(1,2,\ldots, 6)$.  That is,
    \begin{align}\label{Eq:HexagonStationary}
        \mu_0 = \left(\frac{1}{6}, \frac{1}{6}, \ldots, \frac{1}{6} \right).
    \end{align}
    What is $\nP(X_1 = 2)$?  If $X_1=2$, clearly $X_0$ could only be $1$ or $3$. Repeating our calculation with the law of total probability, we find
    \begin{align*}
        \nP(X_1 = 2) &= \sum_{j=1}^6 \nP(X_1 = 2 \, | \, X_0 = j) \nP(X_0=j)\\
        &= \nP(X_1 = 2 \, | \, X_0 = 1) \nP(X_0=1) + \nP(X_1 = 2 \, | \, X_0 = 3) \nP(X_0=3)\\
        &= \frac{1}{2} \cdot \frac{1}{6} + \frac{1}{2} \cdot \frac{1}{6} = \frac{1}{6}.
    \end{align*}
So, note that we interestingly have
\begin{align*}
    \nP(X_1 = 2) = \nP(X_0 = 2) = \frac{1}{6}
\end{align*}
under the initial distribution \eqref{Eq:HexagonStationary}.  But a moments thought shows this computation would equally apply to \emph{any} vertex in the 6-cycle, and so
\begin{align}\label{Eq:Hexagon1StepStationary}
    \nP(X_1 = j) = \nP(X_0 = j) = \frac{1}{6},
\end{align}
for \emph{every} $j$.  Thus if we start the chain with the distribution $\mu_0$, \eqref{Eq:Hexagon1StepStationary} says that our distribution after one step is still $\mu_0$.  But then we can repeat this process: taking another step is just like starting over by the Markov property, and so the distribution of $X_k$ after \emph{any} number of steps is still $\mu_0$!
\end{example}

This is rather remarkable.  Even though our state $X_n$ varies over time, from a probabilistic point of view our chain is \emph{stationary}: the probability that we are at any given vertex $j$ after $k$ steps is always $q_j = (\mu_0)_j$. Such initial distributions, when they exist, will prove to be significant, and we give them the obvious name.

\begin{definition}
    For a Markov chain $X$ on $\Omega = \{1,2,\ldots, n\}$ with transition matrix $P$, a \textbf{stationary distribution} 
    \begin{align*}
        \pi = (\pi_1, \pi_2, \ldots, \pi_n)
    \end{align*}
    is a probability distribution on $\Omega$ (arranged in a row vector) such that
    \begin{align}\label{Eq:Stationary}
        \pi P = \pi.
    \end{align}
    In other words, $\pi$ is a left eigenvector for $P$ with eigenvalue $\lambda=1$.
\end{definition}

The observation we made above bears repeating.  If we have $\pi$ as in \eqref{Eq:Stationary} and we start the chain at a random vertex distributed according to $\pi$, what is the distribution of $X_k$?  From \eqref{Eq:mu_k} and the eigenvector property \eqref{Eq:Stationary}, we have
\begin{align*}
    \mu_k :=& ( \nP(X_k=1), \nP(X_k=2), \ldots, \nP(X_k=n)) \notag\\
    =& \pi P^k = (\pi P)P^{k-1} = \pi P^{k-1} = (\pi P)P^{k-2} = \pi P^{k-2} = \cdots = \pi P = \pi.
\end{align*}
So, if we start the chain with the stationary distribution, the distribution of the chain is stationary: $X_k$ is \emph{always} distributed according to $\pi$; $X_k$'s pdf is always $\pi$.

The following exercise is a helpful observation that we will repeatedly use, and will be a good tool for you to keep in mind.

\begin{exercise}\label{Ex:ColsAddToOne}
(Characterizing chains which have uniform $\pi$)
\begin{enumerate}[$(a)$]
    \item Suppose a Markov chain on $\Omega = \{1,2,\ldots, n\}$ has a transition matrix $P$ such that every column of $P$ sums to 1.  Show that the uniform distribution $\pi=(\frac{1}{n}, \frac{1}{n}, \ldots, \frac{1}{n})$ on $\Omega$ is a stationary distribution for $P$.
    \item Now suppose that $\pi =(\frac{1}{n}, \frac{1}{n}, \ldots, \frac{1}{n})$ is a stationary distribution for some transition matrix $P$.  Show that each column of $P$ sums to 1.
    \item Build a Markov chain on the state space $\Omega = \{1,2\}$ for which the uniform distribution is not stationary.
\end{enumerate}
\end{exercise}

\subsection{$\pi$ for simple symmetric random walks on graphs}
When does such a distribution $\pi$ exist?  If it exists, is it unique?  These are natural questions that we seek to answer in subsequent sections.  There is one class of Markov chains for which the answer is simple. 
    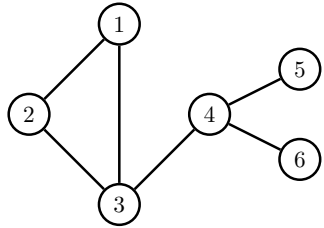
\begin{figure}
        \centering
            \scalebox{0.8}{
    \begin{tikzpicture}
    \tikzset{
        node/.style={circle, draw=black, very thick, minimum size=4mm},
        line/.style={black, very thick}
    }
    
    \node[node] (one) at (0,1.5) {1};
    \node[node] (two) at (-1.5,0) {2};
    \node[node] (three) at (0,-1.5) {3};
    \node[node] (four) at (1.5,0) {4};
    \node[node] (five) at (3,0.75) {5};
    \node[node] (six) at (3,-0.75) {6};
    
    \draw[line] (one) -- (two);
    \draw[line] (two) -- (three);
    \draw[line] (three) -- (four);
    \draw[line] (three) -- (one);
    \draw[line] (four) -- (five);
    \draw[line] (four) -- (six);
    \end{tikzpicture}}
        \caption{\small What is $\pi$ for the simple random walk on this graph?}
        \label{Fig:Eg1Lecture5}
    \end{figure}

\begin{theorem}\label{Thm:SRWpi}
    A simple symmetric random walk on a connected graph $G = (V,E)$ has a unique stationary distribution $\pi = (\pi_v)_{v \in V}$, given by the formula
    \begin{align}\label{Eq:SRWStationaryDistribution}
        \pi_v = \frac{\deg(v)}{2|E|}
    \end{align}
    for each $v \in V$. Here $|E|$ is the number of edges of $G$.
\end{theorem}

Note that the assumption that $G$ is connected is equivalent to saying that the random walk is irreducible - it is possible to travel from any one vertex to any other.  Note also there is no assumption on periodicity.

Before we prove this, let's work through some exercises and examples.

\begin{exercise}
    Theorem \ref{Thm:SRWpi} says that $\pi$ is unique when the graph is connected.  Consider the disconnected graph $G = (V,E)$ given by
    \begin{align*}
        V = \{1,2,3,4\} \quad \text{ and } \quad E = \{\{1,2\}, \{3,4\} \}.
    \end{align*}
    Show there are infinitely-many distinct stationary distributions for the simple symmetric random walk on $G$.
\end{exercise}

\begin{exercise}
    For a connected graph $G= (V,E)$ and $\pi_v$ given by \eqref{Eq:SRWStationaryDistribution}, show that 
    \begin{align*}
        \sum_{v\in V} \pi_v = 1.
    \end{align*}
\end{exercise}

\begin{example}
    Consider the simple symmetric random walk on the 6-cycle again.  Here, $\deg(v) = 2$ for each vertex, and $|E|=6$.  Hence
    \begin{align*}
        \pi_v = \frac{2}{12} = \frac{1}{6}
    \end{align*}
    for each vertex $v$.  So, as we saw above, $\pi$ is the uniform distribution on $V$.
\end{example}

\begin{example}\label{Example:FirstGraphStationaryDist}
    The graph in Figure \ref{Fig:FirstGraph1} has six edges, and thus we have
    \begin{align*}
        \pi = \Big( \frac{1}{12}, \frac{3}{12}, \frac{3}{12}, \frac{3}{12}, \frac{2}{12} \Big) = \Big( \frac{1}{12}, \frac{1}{4}, \frac{1}{4}, \frac{1}{4}, \frac{1}{6} \Big).
    \end{align*}
\end{example}

% \begin{exercise}
%     Give an intuitive explanation for why vertices of higher degree must have more probability in $\pi$.
% \end{exercise}

\begin{exercise}
    Find the stationary distribution of the graph in Figure \ref{Fig:Eg1Lecture5}.
\end{exercise}

% Lecture 6 %%%%%%%%%%%%%%%%%%%%%%%%%%%%%%%%%%%%

Note that the hypothesis that the Markov chain on the graph G in Theorem \ref{Thm:SRWpi} is \emph{simple and symmetric} is also essential.  Biased random walks on graphs as in \S \ref{Sec:NonSRW} are, of course, completely legitimate, but Theorem \ref{Thm:SRWpi} doesn't apply to them, and they may not have such an elegant formula for a stationary distribution. The simple symmetric random walk is in many senses the simplest and cleanest example.   

\begin{example}
\begin{figure}
    \centering
    \scalebox{0.95}{
    \begin{tikzpicture}
    \tikzset{
        node/.style={circle, draw=black, very thick, minimum size=4mm},
         arrow/.style={-{Latex[length=3mm]}, black},
        line/.style={black, very thick}
    }
    
    \node[node] (one) at (-2.5,2.17) {1};
    \node[node] (two) at (-3.75,0) {2};
    \node[node] (three) at (-1.25,0) {3};
    
    \node[node] (four) at (2.5, 2.17) {1};
    \node[node] (five) at (1.25,0) {2};
    \node[node] (six) at (3.75,0) {3};
    
    \draw[line] (one) -- (two);
    \draw[line] (two) -- (three);
    \draw[line] (three) -- (one);
    \draw [arrow, dashed] (one) to [looseness=6] node[above]{$1/2$}  (one);
    \draw [arrow, dashed] (one) to [bend right] node[left]{$1/4$}  (two);
    \draw [arrow, dashed] (one) to [bend left] node[right]{$1/4$}  (three);
    
    \draw [arrow, dashed] (two) to [out=-20,in=200] node[below]{$1$}  (three);
    
    \draw [arrow, dashed] (three) to [out=140, in=-80] node[left]{{\scriptsize $1/2$}}  (one);
    \draw [arrow, dashed] (three) to [out=160, in=20] node[above]{{\scriptsize $1/2$}}  (two);
    
    %Second triangle
    \draw [arrow, dashed] (four) edge [bend right] node[left]{$1/2$}  (five);
    \draw [arrow, dashed] (four) edge [bend left] node[right]{$1/2$}  (six);
        
    \draw[line] (four) -- (five);
    \draw[line] (five) -- (six);
    \draw[line] (six) -- (four);
    \end{tikzpicture}}
    \caption{\small A non-simple walk and a simple random walk on a 3-cycle.  Their differing transition probabilities yield different stationary distributions.}
    \label{Fig:TrianglesInLecture6}
\end{figure}
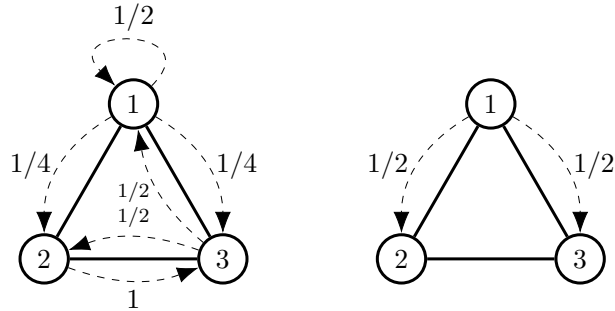
Consider both a biased and simple walk on the 3-cycle, as depicted in Figure \ref{Fig:TrianglesInLecture6}.  The transition matrices are
\begin{align*}
    P_1 = \begin{bmatrix}
        \frac{1}{2} & \frac{1}{4} & \frac{1}{4} \\
        0 & 0 & 1\\
        \frac{1}{2} & \frac{1}{2} & 0\\
    \end{bmatrix} \quad \text{ and } \quad P_2 = \begin{bmatrix}
        0 & \frac{1}{2} & \frac{1}{2} \\
        \frac{1}{2} & 0 & \frac{1}{2}\\
        \frac{1}{2} & \frac{1}{2}  & 0\\
    \end{bmatrix}
\end{align*}
The latter matrix $P_2$ defines the simple random walk, and each entry of its stationary distribution $\pi_2$ has the same formula $\pi_2(v) = \frac{2}{2 \cdot 3} = \frac{1}{3}$.  However, we find that $\pi_1 = (\frac{4}{11}, \frac{3}{11}, \frac{4}{11})$ satisfies
\begin{align*}
    \pi_1 = \pi_1 P_1,
\end{align*}
and so $\pi_1$ is a stationary distribution for the non-simple walk.  In particular, the formula \eqref{Eq:SRWStationaryDistribution} does not apply to $P_1$.
\end{example}

\begin{proof}[Proof of Theorem \ref{Thm:SRWpi}]
    Theorem \ref{Thm:SRWpi} consists of two parts: one is that the formula \eqref{Eq:SRWStationaryDistribution} gives \emph{a} stationary distribution, and another is that this is the only one. We here only prove the first part, and we take up the question of uniqueness below in a more general setting in Theorem \ref{Thm:ReturnTime}. 
    
    Recall that the entries for the transition matrix $P$ are $P_{ij} = \frac{1}{\deg(i)}$ if $i \sim j$ and 0 otherwise.  We need to show $\pi P = \pi$, and we can do this component-wise by showing 
    \begin{align}\label{Eq:SRWPiProof1}
        (\pi P)_j = \pi_j
    \end{align}
    for each $j \in \{1,2,\ldots, n\}$, where $n = |V|$. Indeed, for a fixed $j$, $(\pi P)_j$ is the result of multiplying $\pi$ with the $j$th column of $P$, and hence
    \begin{align*}
        (\pi P)_j &= \sum_{i=1}^n \pi_i P_{ij}\\
        &= \sum_{i \sim j} \frac{\deg(i)}{2|E|} \cdot \frac{1}{\deg(i)}\\
        &= \sum_{i \sim j} \frac{1}{2|E|} = \frac{\deg(j)}{2|E|} = \pi_j,
    \end{align*}
    giving \eqref{Eq:SRWPiProof1}.  (The second-to-last equality holds because the number of vertices $i$ adjacent to $j$ is, by definition, $\deg(j)$.)
\end{proof}

\begin{example}\label{Eg:StationaryForBiasedWalk}
    We study another biased random walk as a final example.  Consider the $n$-cycle where we move clockwise with probability $p$ and counter-clockwise with probability $1-p$, $0\leq p \leq 1$ (of course we recover the simple symmetric walk when $p=1/2$).  What is the stationary distribution $\pi_p$?  Does it depend on $p$? Surprisingly it does not, and we find
    \begin{align*}
        \pi_p = \Big( \frac{1}{n}, \frac{1}{n} \ldots, \frac{1}{n} \Big) = \text{Unif}\{1,2,\ldots,n\}
    \end{align*}
    for all $p$.  One way to see this is to try a specific case and then extrapolate.  If we do the biased walk on the 5-cycle, for instance, we have the transition matrix
    \begin{align*}
        P = \begin{bmatrix}
                0 & p & 0 & 0 & 1-p\\
                1-p & 0 & p & 0 & 0\\
                0 & 1-p & 0 & p & 0 \\
                0 & 0 & 1-p & 0 & p\\
                p & 0 & 0 & 1-p & 0
            \end{bmatrix},
    \end{align*}    
    and we can simply compute that 
    \begin{align*}
        \Big( \frac{1}{5}, \frac{1}{5}, \frac{1}{5}, \frac{1}{5}, \frac{1}{5} \Big)P = \Big( \frac{1}{5}, \frac{1}{5}, \frac{1}{5}, \frac{1}{5}, \frac{1}{5} \Big).
    \end{align*}
    The underlying reason is that the columns still add to 1 (why?), which yields a uniform $\pi$ by Exercise \ref{Ex:ColsAddToOne} $(a)$. 
\end{example}

%%%%%%%%%%%%%%%%%%%%%%%%%%%%%%
\section{Hitting times and return times}

Given a stationary distribution $\pi$, we now know how to run a Markov chain in such a way that we know the exact distribution of $X_n$ for all $n$: we simply start with $X_0 \sim \pi$, and then $X_n \sim \pi$ for all $n$.  

There are other important matters to settle, though.  For instance, how long will it take for a chain to reach a certain state?  Or how many steps, on average, does it take to return to a given state?  These turn out to be important questions in furthering our analysis, and these are the questions we take up in this section. 

% As always, let $X = (X_k)$ be a Markov chain on the finite state space $\Omega = \{1,2,\ldots, n\}$.

% We now have some groundwork of vocabulary and fundamentals and can go a bit deeper in our study of Markov chains. 

\begin{definition}
    The \textbf{hitting time} $\tau_x$ of a state $x \in \Omega$ is the first nonnegative time that the Markov chain is in state $x$.  That is,
    \begin{align}\label{Def:tau_x}
        \tau_x := \min\big\{\, k\in \{0,1,2,\ldots\} \; : \; X_k=x \,\big\}.
    \end{align}
\end{definition}
\noindent Note that $\tau_x$ is a random amount of time, and thus a \emph{random variable}.  Different runs of the chain will result in different hitting times $\tau_x$, and we can ask questions like, ``What is the average $\nE(\tau_x)$ of $\tau_x$?''  Or, ``What is $\tau_x$'s distribution?''  Note also that if we start the chain at $X_0=x$, then \eqref{Def:tau_x} says $\tau_x=0$; a hitting time of zero \emph{is} allowed.

On the other hand, sometimes we wish to know when we first \emph{return} to a given state.  Then we exclude the possibility $k=0$ in \eqref{Def:tau_x}.

\begin{definition}
    The (first) \textbf{return time} $\tau_x^+$ of $x$ is the first time after the chain has started that we assume the state $x$.  That is,
    \begin{align}\label{Def:tau_x+}
        \tau_x^+ := \min\big\{\, k\in \{1,2,\ldots\} \; : \; X_k=x \,\big\}.
    \end{align}
\end{definition}
\noindent This is, of course, another random amount of time, just like $\tau_x$.  While $\tau_x \geq 0$, however, we have $\tau_x^+ \geq 1$.  So if we start the chain at $x$, then $\tau_x=0$ but we do not automatically know anything about $\tau_x^+$, other than $\tau_x^+ \geq 1$.  If we start the chain at some $y \neq x$, then  $\tau_x = \tau_x^+$.

We can extend these definitions in a natural way to collections of states $A \subset \Omega$.  We define the \textbf{hitting and return times to a set $A$} as, respectively,
\begin{align*}
     \tau_A &:= \min\big\{\, k\in \{0,1,2,\ldots\} \; : \; X_k\in A \,\big\},\\
     \tau_A^+ &:= \min\big\{\, k\in \{1,2,\ldots\} \; : \; X_k\in A \,\big\}.
\end{align*}
These are the first time, and the first time after starting, that the chain enters the set $A$, respectively.

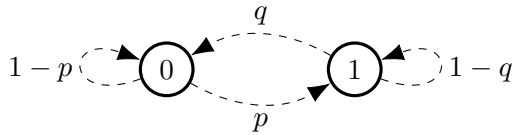
\begin{figure}[b]
        \centering
    \scalebox{1}{\begin{tikzpicture}
    \tikzset{
        node/.style={circle, draw=black, very thick, minimum size=7mm},
        arrow/.style={-{Latex[length=3mm]}, black},
        line/.style={black, very thick}
    }
    
    \node[node] (one) at (0,0) {0};
    \node[node] (two) at (2.5,0) {1};
    
    \draw [arrow, dashed] (one) to [out=-30,in=210, bend right] node[below]{$p$}  (two);
    \draw [arrow, dashed] (two) to [out=150,in=30, bend right] node[above]{$q$}  (one);
    
    \draw[arrow, dashed] (two) to [out=340,in=20,looseness=11] node[midway, right] {$1-q$} (two) ;
    \draw[arrow, dashed] (one) to [out=200,in=160,looseness=11] node[midway, left] {$1-p$} (one);

    \end{tikzpicture}}
        \caption{\small What are the hitting times for this Markov chain?}
        \label{Fig:ReturnTimeEg}
\end{figure}
\begin{example}\label{Eg:ReturnTime}
    Consider the walk on $\Omega = \{0,1\}$ defined by the transition matrix
    \begin{align*}
        P = \begin{bmatrix}
                1-p & p\\
                q & 1-q
            \end{bmatrix}, \qquad 0<p,q<1.
    \end{align*}
See Figure \ref{Fig:ReturnTimeEg}. Suppose $X_0=0$ (equivalently, $\mu_0 = (1,0)$).  Then the hitting time $\tau_1$ of state 1 is a random variable.  What is its distribution?  Let us write $\nP_0(E)$ for the probability of the event $E$ given that the chain starts at $X_0$. We compute
\begin{align*}
    \nP_0(\tau_1=0) &= 0\\
    \nP_0(\tau_1=1) &= p\\
    \nP_0(\tau_1=2) &= \nP_0(0 \rightarrow 0 \rightarrow 1) = (1-p)p\\
    \nP_0(\tau_1=3) &=  \nP_0(0 \rightarrow 0 \rightarrow 0 \rightarrow 1) =(1-p)^2p\\
    &\vdots\\
    \nP_0(\tau_1=k) &= (1-p)^{k-1}p, \qquad k=1,2,\ldots,
\end{align*}
since to hit 1 for the first time in $k$ steps, we must stay at vertex $0$ for $k-1$ steps, and then move to vertex 1.  This is the familiar pmf of the geometric random variable for the number of trials to have the first success; we thus see $\tau_1 \sim \text{Geo}(p)$.

What about $\tau_1^+$?  Well, since we are starting with $X_0=0 \neq 1$, the minima in \eqref{Def:tau_x} and \eqref{Def:tau_x+} give exactly the same numbers, and so $\tau_1^+ \sim \text{Geo}(p)$ too. 

Similarly, if $X_0=1$, then $\tau_0 = \tau_0^+ \sim \text{Geo}(q)$.

How about the distribution of $\tau_0^+$, when $X_0=0$?  We find
\begin{align*}
    \nP_0(\tau_0^+=1) &= 1-p\\
    \nP_0(\tau_0^+=2) &= \nP_0(0 \rightarrow 1 \rightarrow 0) = pq\\
    \nP_0(\tau_0^+=3) &= \nP_0(0 \rightarrow 1 \rightarrow 1 \rightarrow 0) = p(1-q)q\\
    \nP_0(\tau_0^+=4) &= \nP(0 \rightarrow 1 \rightarrow 1 \rightarrow 1 \rightarrow 0) = p(1-q)^2q\\
        \nP_0(\tau_0^+=5) &= \nP_0(0 \rightarrow 1 \rightarrow 1 \rightarrow 1 \rightarrow 1 \rightarrow 0) = p(1-q)^3q\\
    &\vdots\\
    \nP_0(\tau_0^+=k) &= p(1-q)^{k-2}q, \qquad k \geq 2.
\end{align*}
In particular, we see $\tau_0^+$ is \emph{not} a geometric random variable, even though it is similar to one.
% ); it is not one of our familiar named distributions.
\end{example}

%% Lecture 7 %%%%%%%%%%%%%%%%%%%%%%%%%%%%%%%%%

\subsection{Expected hitting times}

To understand a random variable we need to know its average.  Can we say anything about the expectations of our new random variables $\tau_x, \tau_x^+, \tau_A$ and $\tau_A^+$?  On average, how long does it take to hit or return to state $x$ or a collection of states $A$?  
\begin{example}
For example, suppose the vertices of our graph were all airports in the world, with $x \sim y$ if there is a flight between airport $x$ and $y$.  Now, imagine a virus arises in some city $x$.  How much time would a country, on average, have to prepare before the virus reaches them via flights?  Using a Markov chain model, this is simply $\nE_x(\tau_A)$ for the collection of airports/vertices $A$ in the country.  (The subscript of ``$x$'' in the expectation tells us that the chain starts at state $x$.)
\end{example}

We will be able to explicitly compute the averages using linear algebra.  Let's first establish some notation, however.  As usual, our state space is $\Omega = \{1,2,\ldots, n\}$, and as above in Example \ref{Eg:ReturnTime}, we continue to write
\begin{align}
    \nP_i(X_j = k) &:= \nP(X_j = k \, | \, X_0=i) \quad \text{ and}\label{Notation:ConditionalSubscriptP}
\end{align}
Similarly, 
\begin{align}
    \nE_i(\tau_A) &:= \nE(\tau_A \, | \, X_0=i)\label{Notation:ConditionalSubscriptE}
\end{align}
abbreviates the average of $\tau_A$ given the starting position.  We also want notation for starting the chain according to a given distribution $\pi$, instead of just at a fixed state $x$. We write this informally as $\nP_\pi(X_j = k) := \nP(X_j = k \, | \, X_0 \sim \pi)$ and similarly for the expectation, but the rigorous definitions are
\begin{align}
    \nP_\pi(X_j = k) &:= \sum_{i=1}^n\nP(X_j = k \, | \, X_0 =i)\pi(i)\label{Notation:ConditionalSubscriptPiP}
\end{align}
and
\begin{align}
    \nE_\pi(\tau_A) &:= \sum_{i=1}^n \nE(\tau_A \, | \, X_0=i)\pi(i).\label{Notation:ConditionalSubscriptPiE}
\end{align}
% We will not be perfectly consistent in using these abbreviated notations, but will alternate between using them and writing out the entire conditional expressions as we see fit. 
\begin{exercise}
    One possible initial distribution $\pi$ is the \emph{Dirac mass} $\delta_i$ at state $i$, which is the probability distribution on $\Omega$ defined by
    \begin{align*}
        \delta_i(j) = \begin{cases}
            1 & j=i,\\
            0 & j \neq i.
        \end{cases}
    \end{align*}
    Use the Dirac mass to show that the formulas \eqref{Notation:ConditionalSubscriptP} and \eqref{Notation:ConditionalSubscriptE} are special cases of \eqref{Notation:ConditionalSubscriptPiP} and \eqref{Notation:ConditionalSubscriptPiE}, respectively.
\end{exercise}

Notation in hand, we return to the problem of calculating the average hitting time $\nE_i(\tau_A)$.  Without loss of generality  we may assume that the set $A$ in question consists of the last $k$ elements $\{n-k+1, n-k+2, \ldots, n\}$ of $\Omega$ for some $1 \leq k < n$, since if not, we can just relabel the states to make it so.  Now, partition the transition matrix $P$ as
\begin{align*}
    P = \left[ \begin{array}{c|c}
            P_{(n-k) \times (n-k)} & P_{(n-k) \times k}\\
            \hline
            P_{k \times (n-k)} & P_{k \times k}
        \end{array} \right].
\end{align*}
Here $P_{(n-k) \times (n-k)}$ is the top left $(n-k) \times (n-k)$ submatrix, giving all probabilities of moving from a state in $\{1,2, \ldots, n-k\}$ to a state in $\{1,2, \ldots, n-k\}$.  $P_{(n-k)\times k}$ is the top right submatrix, corresponding to the probabilities of moving from states $\{1,2, \ldots, n-k\}$ to states $\{n-k+1, n-k+2, \ldots, n\}$.  And so on for the bottom two blocks.  

As it turns out, the submatrix that will be most important for us is the top left block $P_{(n-k) \times (n-k)} =: Q$, which gives the probabilities of remaining \textit{outside} of $A$, given one starts \emph{outside} of $A$.   Define the $(n-k) \times (n-k)$ matrix $M$ as 
\begin{align}\label{Eq:ReturnTimeMMatrix}
    M := (I - Q)^{-1}, 
\end{align}
where $I$ is the $(n-k) \times (n-k)$ identity matrix.  Of course, to make this definition one has to show that $I-Q$ is invertible, but we will do that during the proof of the following theorem.  This theorem says that, perhaps surprisingly, our expected hitting time comes from summing the row of $M$ corresponding to our starting state. 

\begin{theorem}\label{Thm:ExpectedHittingTimes}
    Let $X$ be an irreducible Markov chain on $\Omega$ with transition matrix $P$, and let 
    \begin{align*}
        i \in \{1,2, \ldots, n-k\} = \Omega \backslash A
    \end{align*}  
    be any state outside of $A$.  Then
    \begin{align}\label{Eq:ExpectedHittingTime}
        \nE_i(\tau_A) = \sum_{j=1}^{n-k} M_{ij},
    \end{align}
    where $M$ is given by \eqref{Eq:ReturnTimeMMatrix}, and the $\nE_i$ notation is as in \eqref{Notation:ConditionalSubscriptE}.
\end{theorem}
So to find $\nE(\tau_A \, | \, X_0 = i)$, we sum the elements of the $i$th row of the matrix $M$.  In matrix $Q$, the $i$th row gives the probabilities of staying outside of the set $A$, starting from state $i \notin A$.  We will see that the $i$th row of $M$ gives the expected number of \emph{visits} to each of the states outside of $A$ before hitting $A$, starting at $i$.  So \eqref{Eq:ExpectedHittingTime} says that the average time to reach $A$ is the total average number of visits to all of the states outside of $A$.
% ; while it is not immediately clear why this should be true, that much is simple.

\begin{example}\label{Eg:ExpectedHittingTimes1}
    Let's consider the simple symmetric random walk on the graph in Figure \ref{Fig:ExpectedHittingTimes1}, which has transition matrix
    \begin{align*}
        P = \begin{bmatrix}
                0 & \frac{1}{3} & \frac{1}{3} & \frac{1}{3}\\
                \frac{1}{2} & 0 & \frac{1}{2} & 0\\
                \frac{1}{3} & \frac{1}{3} & 0 & \frac{1}{3}\\
                \frac{1}{2} & 0 & \frac{1}{2} & 0
            \end{bmatrix}.
    \end{align*}
    Suppose we want to find the expected hitting time to $A = \{3,4\}$ starting from either vertex 1 or 2.  We partition the transition matrix $P$ as
    \begin{align*}
            P = \left[\begin{array}{c c | c c}
                0 & \frac{1}{3} & \frac{1}{3} & \frac{1}{3}\\
                \frac{1}{2} & 0 & \frac{1}{2} & 0\\
                \hline
                \frac{1}{3} & \frac{1}{3} & 0 & \frac{1}{3}\\
                \frac{1}{2} & 0 & \frac{1}{2} & 0
            \end{array} \right],
    \end{align*}
    isolating the vertices corresponding to $A$.  Our matrix $Q$ is then the top-left corner
    \begin{align*}
        Q = \begin{bmatrix}
             0 & \frac{1}{3}\\
              \frac{1}{2} & 0
        \end{bmatrix},
    \end{align*}
    the probabilities that we start and remain outside of $A$ (if our vertices were labelled differently, $Q$ would not be in the top-left corner; we then simply select the submatrix with these probabilities).  Thus
    \begin{align*}
        I-Q = \begin{bmatrix}
             1 & -\frac{1}{3}\\
              -\frac{1}{2} & 1
        \end{bmatrix} \quad \text{and} \quad (I-Q)^{-1} = \begin{bmatrix}
             \frac{6}{5} & \frac{2}{5}\\
              \frac{3}{5} & \frac{6}{5}
        \end{bmatrix},
    \end{align*}
    and \eqref{Eq:ExpectedHittingTime} says
    \begin{align*}
        \nE_1(\tau_A) &= \frac{6}{5} + \frac{2}{5} = \frac{8}{5},\\
        \nE_2(\tau_A) &= \frac{3}{5} + \frac{6}{5} = \frac{9}{5}.
    \end{align*}
So, if we begin at vertex 1, on average it takes $8/5$ steps to enter  $\{3,4\}$.  If we start at vertex 2, it takes $9/5$ steps, on average. We are not surprised that $\nE(\tau_A) < \nE(\tau_A)$, as there are more routes to $A$ from vertex 1 than from vertex 2.  \qed
\end{example}

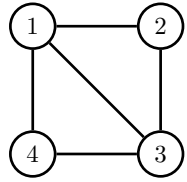
\begin{figure}
    \centering
    \scalebox{0.85}{
        \begin{tikzpicture}
    \tikzset{
        node/.style={circle, draw=black, very thick, minimum size=7mm},
        line/.style={black, very thick}
    }
    
    \node[node] (one) at (0,0) {1};
    \node[node] (two) at (2,0) {2};
    \node[node] (three) at (2,-2) {3};
    \node[node] (four) at (0, -2) {4};
    
    \draw[line] (one) -- (two);
    \draw[line] (two) -- (three);
    \draw[line] (three) -- (four);
    \draw[line] (four) -- (one);
    \draw[line] (one) -- (three);
    \end{tikzpicture}}
    \caption{\small Computing $\nE(\tau_A \, | \, X_0=i)$ for $A=\{3,4\}$, $i \in \{1,2\}$.}
    \label{Fig:ExpectedHittingTimes1}
\end{figure}

\begin{proof}[Proof of Theorem \ref{Thm:ExpectedHittingTimes}]
    We first claim that, for any $j \notin A$,
    \begin{align}\label{Eq:HittingTimePfClaim1}
        M_{ij} &= \nE(\, \#\{ \text{visits to $j$ before hitting $A$} \} \; | \; X_0 = i\,)
    \end{align}
    %  &= \nE \Big(\, \sum_{\ell=1}^{\tau_A} \I_{j}(X_\ell)   \; | \; X_0=i \,\Big), \notag
    % where the second line is just a mathematical formulation of the English in the expectation in first line.  Here the indicator function is defined as
    % \begin{align*}
    %     \I_j(X_\ell) := \begin{cases}
    %                 1 & X_\ell = j,\\
    %                 0 & \text{otherwise}.
    %         \end{cases}
    % \end{align*}
    % (So adding this up over all steps $X_\ell$ counts how many times $X_\ell = j$.)
    Before we prove this, let's see that this would give our result \eqref{Eq:ExpectedHittingTime}.  Indeed, we can decompose the hitting time $\tau_A$ into the sum
    \begin{align*}
        \tau_A &= \sum_{j \notin A} \#\{ \text{visits to $j$ before hitting $A$} \}
    \end{align*}
    and thus taking expectations yields
    \begin{align}
        \nE(\tau_A \, | \, X_0 = i) &= \nE \Big( \sum_{j \notin A}\#\{ \text{visits to $j$ before hitting $A$}\}  \, \Big| \, X_0 = i\Big) \notag\\
        &= \sum_{j \notin A} \nE \big(\#\{ \text{visits to $j$ before hitting $A$}\}\, | \, X_0 = i\big) \notag\\
        &= \sum_{j \notin A} M_{ij},\label{Eq:HittingTimePfGoal}
    \end{align}
    assuming \eqref{Eq:HittingTimePfClaim1} holds.  But \eqref{Eq:HittingTimePfGoal} is precisely the claimed result \eqref{Eq:ExpectedHittingTime}.  So it is enough to prove \eqref{Eq:HittingTimePfClaim1}.
    
    Recall $Q= P_{(n-k) \times (n-k)}$ gives the probabilities of moving around in states outside of $A$, and thus
    \begin{align*}
        (Q^\ell)_{ij} = \nP(\, X_\ell = j, \tau_A >\ell \, |\, X_0=i \,).
    \end{align*}
    That is, $Q^\ell$ gives the probabilities of moving around $\Omega$ \emph{outside} of the set $A$ for $\ell$ steps (just as $P^\ell$ gives the probabilities for moving anywhere in $\Omega$ in $\ell$ steps).  Recalling the expectation formula for indicator random variables
    \begin{align*}
        \nE(\I_B) = 1 \cdot \nP(B) + 0 \cdot \nP(B^c) = \nP(B),
    \end{align*}
    we see
    \begin{align}\label{Eq:ExpectedHittingTimeQ^l}
        (Q^\ell)_{ij} &= \nP(\, X_\ell = j, \tau_A >\ell \, |\, X_0=i \,) = \nE\big(\, \I_{\{X_\ell=j, \tau_A > \ell\}} \, |\, X_0=i \,\big),
    \end{align}
    and so, noting that 
    \begin{align*}
        \#\{ \text{visits to $j$ before hitting $A$}\} = \sum_{\ell=0}^\infty \I_{\{X_\ell=j, \ell < \tau_A \}},
    \end{align*}
    we have
    \begin{align}
        \nE(\, \#\{&\text{visits to $j$ before hitting $A$} \} \; | \; X_0 = i\,) \notag\\
        &= \nE\Big(\, \sum_{\ell=0}^\infty \I_{\{X_\ell=j, \ell <\tau_A\}}  \; \Big| \; X_0 = i\, \Big) \notag\\
        &= \sum_{\ell=0}^\infty  \nE\big(\, \I_{\{X_\ell=j, \ell < \tau_A\}} \, |\, X_0=i \,\big) 
        =\sum_{\ell=0}^\infty (Q^\ell)_{ij}, \label{Eq:HittingTimeAlmostThere}
    \end{align}
    where the last equality is by \eqref{Eq:ExpectedHittingTimeQ^l}, and we are allowed to move the expectation through the infinite sum in the previous equality because each term of the sum is non-negative.\footnote{This is \emph{Tonelli's theorem} from measure theory.}  Now, the matrix $Q$ is \emph{sub-stochastic}, which is to say each row of $Q$ sums to at most one (instead of exactly one for the entire \emph{stochastic} matrix $P$).  Furthermore, one of the rows of $Q$ must sum to strictly less than one, since if every row summed to one, it would be impossible to reach a state in $A$ from a state in $\Omega\backslash A$, and the chain would not be irreducible.   
    
    To finish the proof, we claim that for such sub-stochastic matrices $Q$, we have the identity 
    \begin{align}\label{Eq:MatrixGeoSum}
        \sum_{\ell=0}^\infty Q^\ell = (I-Q)^{-1} = M,
    \end{align}
    where $I$ is the $(n-k) \times (n-k)$ identity matrix.  In essence, this is saying that the familiar geometric series formula 
    \begin{align*}
        \sum_{\ell=0}^\infty x^\ell = (1-x)^{-1} = \frac{1}{1-x}
    \end{align*}
which holds for \emph{numbers} $|x|<1$, also holds for \emph{matrices} $Q$ which satisfy ``$|Q|<1$,'' which is a beautiful fact.  We prove this below, but let us first see how the argument finishes assuming \eqref{Eq:MatrixGeoSum}.  Indeed, we have
    \begin{align*}
        M_{ij} = (I-Q)^{-1}_{ij} &= \Big(  \sum_{\ell=0}^\infty Q^\ell \Big)_{ij}\\
        &= \sum_{\ell=0}^\infty (Q^\ell)_{ij}\\
        &=  \nE\big(\, \#\{ \text{visits to $j$ before hitting $A$} \}\,\big)
    \end{align*}
    by \eqref{Eq:HittingTimeAlmostThere}.  This is exactly \eqref{Eq:HittingTimePfClaim1}, what we needed to prove.

    The rest of the proof is to show \eqref{Eq:MatrixGeoSum} holds. Note that it is equivalent to saying 
    \begin{align}\label{Eq:MatrixGeoSumStep}
        I = (I-Q) \sum_{\ell=0}^\infty Q^\ell = \lim_{N \rightarrow \infty} (I-Q) \sum_{\ell=0}^N Q^\ell.
    \end{align}
    Now, we have that
    \begin{align*}
        (I-Q) \sum_{\ell=0}^N Q^\ell = \Big(\sum_{\ell=0}^N Q^\ell\Big) - \Big( \sum_{\ell=0}^N Q^{\ell+1} \Big) = I - Q^{N+1},
    \end{align*}
    and so, going back to \eqref{Eq:MatrixGeoSumStep}, we need to show that 
    \begin{align*}
        I = \lim_{N \rightarrow \infty} (I - Q^{N+1}) = I - \lim_{N \rightarrow \infty} Q^{N+1}.
    \end{align*}
    It suffices, therefore, to prove that $\lim_{N \rightarrow \infty} Q^{N} = \mathbf{0}$, the $(n-k) \times (n-k)$ zero matrix.

    We will first show that there is some power $M>0$ such that $\lim_{k \rightarrow \infty} Q^{kM} = \mathbf{0}$.  Indeed, fix a state $i \in \{1, \ldots, n-k\}$, which is to say, a state $i \notin A$.  For another fixed state $j \in A$, by irreducibility we have an $m = m_i$ such that $\nP(X_{n+m}=j \, | \, X_n = i) >0$.  Thus the probability that we remain in states outside of $A$ after $m$ steps, starting from $i$, is strictly less than 1.  Another way to say this is that the row sum
    \begin{align*}
        \sum_{j=1}^{n-k} Q^m_{ij} \leq \Lambda_i < 1
    \end{align*}
    for some number $\Lambda_i$. Therefore, for any $\ell \geq 0$, and $p \in \{1, \ldots, n-k\}$, the $(i,p)$-entry of $Q^{m+\ell}$ satisfies
    \begin{align}\label{Ineq:SubstochasticEntryBound}
        Q^{m+\ell}_{ip} = (Q^mQ^\ell)_{ip} = \sum_{j=1}^{n-k} Q^m_{ij}Q^\ell_{jp} \leq \sum_{j=1}^{n-k} Q^m_{ij} \cdot 1 \leq \Lambda_i,
    \end{align}
    where the first inequality is because $Q^\ell$ is still a (sub)stochastic matrix.

    We can do this for each state $j \in \{1, \ldots, n-k\}$, getting a number $m_j$ and bound $\Lambda_j <1$ for the $j$th row sum of $Q^{m_j}$.  Setting $M:= \max \{m_1, \ldots, m_{n-k}\}$ and $\Lambda := \max \{\Lambda_1, \ldots, \Lambda_{n-k})$, we have that, exactly as in \eqref{Ineq:SubstochasticEntryBound}, $Q^{M+\ell}_{ij} \leq \Lambda <1$ for all $i, j \in \{1, \ldots, n-k\}$ and all $\ell \geq 0$.    

    We can now that $\lim_{k \rightarrow \infty} Q^{kM} = \mathbf{0}$.  In fact, we claim that for any entry $(i,p)$ of $Q^{kM}$, we have 
    \begin{align}\label{Ineq:SubstochasticInduction1}
        Q^{kM}_{ip} \leq \Lambda^k.
    \end{align}
    By the previous paragraph we have this when $k=1$.  Let us assume \eqref{Ineq:SubstochasticInduction1} holds for some fixed $k \geq 1$; we need to show it also holds for $k+1$.  And indeed, estimating as in \eqref{Ineq:SubstochasticEntryBound} again, we find
    \begin{align*}
        Q^{(k+1)M}_{ip} = (Q^MQ^{kM})_{ip} = \sum_{j=1}^{n-k} Q^M_{ij}Q^{kM}_{jp} &\leq \sum_{j=1}^{n-k} Q^M_{ij} \cdot \Lambda^k\\ &= \Lambda^k \sum_{j=1}^{n-k} Q^M_{ij} \leq \Lambda^{k+1},
    \end{align*}
    as claimed.  We conclude by induction that \eqref{Ineq:SubstochasticInduction1} holds for all $k$.  Since $\Lambda <1$, this shows that each entry of $Q^{kM} \rightarrow 0$ (exponentially fast!) as $k \rightarrow \infty$, which is what we wanted to show:
    \begin{align*}
        \lim_{k \rightarrow \infty} Q^{kM} = \mathbf{0}.
    \end{align*}
    
    The argument for why $\lim_{N \rightarrow \infty} Q^{N} = \mathbf{0}$ is nearly complete.  For large $N$, we simply write $N = kM + r$ for some $r \in \{0, \ldots, M-1\}$ and estimate any fixed $(i,p)$-entry of $Q^N$ similar to the above.  We leave this detail as the next exercise, and conclude the proof.
\end{proof}
% The linear algebra fact that we need to finish the proof is the following.  Recall that a \emph{substochastic} matrix $Q$ is a stochastic matrix where at least one row sum is strictly less than one.  Irreducibility for a substochastic matrix is the same as for stochastic matrices: in the corresponding graph with probabilities given by $Q$, there is a path of positive probability between any two vertices.
% \begin{lemma}\label{Lemma:Substochastic}
%     Let $Q$ be an irreducible $m \times m$ substochastic matrix with $I$ the $m \times m$ identity matrix.  Then $I-Q$ is invertible and
%     \begin{align}\label{Eq:MatrixGeoSum2}
%         (I-Q)^{-1} = \sum_{\ell=0}^\infty Q^\ell.
%     \end{align}
% \end{lemma}
% The underlying idea for the proof is simple, although it takes some patience to work out the details.  We will see it suffices to show that $Q^\ell \rightarrow \mathbf{0}$ as $\ell \rightarrow \infty$, where $\mathbf{0}$ is the zero matrix.  The reason this holds is intuitively a geometric random variable argument:  

\begin{exercise}
    Finish the proof by showing that for a substochastic matrix $Q$, $\lim_{k \rightarrow \infty} Q^{kM} = \mathbf{0}$ implies $\lim_{N \rightarrow \infty} Q^{N} = \mathbf{0}$.
\end{exercise}

\begin{exercise}
    For the simple symmetric random walk on the graph in Figure \ref{Fig:ExpectedHittingTimes1}, compute $\nE(\tau_A \; | \; X_0=4)$, where $A = \{2,3\}$.
\end{exercise}

\subsection{Expected return times}

Theorem \ref{Thm:ExpectedHittingTimes} tells us how to find the average hitting time for $A$ when we start outside of $A$.  What about for the expected \emph{return} time $\nE(\tau_A^+ \, | \, X_0 = i)$ to $A$?  We can start with an obvious observation: if $i \notin A$, then ``returning'' is the same as hitting, and so $\tau_A^+ = \tau_A$ and we can use Theorem \ref{Thm:ExpectedHittingTimes}.  However, this is still rather cumbersome if we wish to focus on an individual state $A = \{x\}$, since we would need to invert an $(n-1)\times (n-1)$ matrix $I- Q$, and then sum up the row corresponding to our starting vertex.  The following theorem says there is a simpler way.

\begin{theorem}\label{Thm:ReturnTime}
    Let $\pi$ be a stationary distribution of an irreducible Markov chain $X$ on $\Omega = \{1,2,\ldots, n\}$.  Then, for any $x \in \Omega$,
    \begin{align}\label{Eqn:ReturnTime}
        \nE_x(\tau_x^+) = \frac{1}{\pi(x)}.
    \end{align}
    
    Conversely, given such an irreducible chain $X$, the formula 
    \begin{align}\label{Eq:piFormula}
        \pi(x) = \frac{1}{ \nE_x(\tau_x^+)}, \qquad x \in \Omega
    \end{align}
    gives a probability distribution $\pi$ on $\Omega$ that is a stationary distribution.  
\end{theorem}
\noindent An important consequence of Theorem \ref{Thm:ReturnTime} is that
\begin{align*}
    \text{\textit{irreducible Markov chains always have a unique stationary distribution}.}
\end{align*}
Indeed, the second half of the theorem says that we obtain at least one stationary distribution through the formula \eqref{Eq:piFormula}.  Furthermore, if $\pi_1$ and $\pi_2$ are two stationary distributions for $X$, then by \eqref{Eqn:ReturnTime},
\begin{align*}
    \frac{1}{\pi_1(x)} = \nE(\tau_x^+ \, | \, X_0 = x) = \frac{1}{\pi_2(x)},
\end{align*}
and so $\pi_1(x) = \pi_2(x)$ for each $x \in \Omega$.  Thus stationary distributions on irreducible chains are unique and are always given by \eqref{Eq:piFormula}.

Before we discuss the proof of Theorem \ref{Thm:ReturnTime}, we argue that the denominator in \eqref{Eq:piFormula} is finite.  That is, the stationary distribution defined by \eqref{Eq:piFormula} has no zero entries.  We will use this in our proof of Theorem \ref{Thm:ReturnTime} and it will also be important when we later divide by $\pi(x)$ when computing transition probabilities for the \emph{time reversal} of irreducible Markov chains in \S \ref{Sec:Reverse}.
\begin{lemma}\label{Lemma:StationaryPositive}
    For an irreducible Markov chain $X$ on $\Omega = \{1,2,\ldots, n\}$, $1 \leq \nE_x(\tau_x^+)<\infty$ for every $x \in \Omega$.  Consequently $0 < \pi(x)$ for every $x \in \Omega$ by \eqref{Eq:piFormula}.
\end{lemma}
\begin{proof}
    The lower bound on $\nE_x(\tau_x^+)$ is obvious since $1 \leq \tau_x^+$.  Note that with $A=\{x\}$ and $y\neq x$, equation \eqref{Eq:ExpectedHittingTime} in Theorem \ref{Thm:ExpectedHittingTimes} shows
    \begin{align}\label{Eq:ExpectedReturnFinite}
        \nE_y(\tau_x^+) = \nE_y(\tau_x) = \nE_y(\tau_A) < \infty.
    \end{align}
   The result follows by direct computation.  We condition on $X_1$ and use the tower property of conditional expectation to see
\begin{align}
    \nE_x(\tau_x^+) &= \nE_x(\nE_x(\tau_x^+\, |\, X_1)) \notag\\
    &= \nE_x(\tau_x^+\, |\, X_1=x)P_{xx} + \sum_{y \neq x}\nE_x(\tau_x^+\, |\, X_1=y)P_{xy} \notag\\
    &= 1 \cdot P_{xx} + \sum_{y \neq x}\nE_x(\tau_x^+\, |\, X_1=y)P_{xy}. \label{Eq:ExpectedReturnFinite2}
    % &= \nE_x(1+\tau_x^+-1\, |\, X_1=x)P_{xx} + \sum_{y \neq x}\nE_x(1+\tau_x^+-1\, |\, X_1=y)P_{xy} \notag\\
    % &= 1 + \nE_x(\tau_x^+-1\, |\, X_1=x)P_{xx} + \sum_{y \neq x}\nE_x(\tau_x^+-1\, |\, X_1=y)P_{xy}, \notag
\end{align}
At this point we note by the Markov property that, for $y \neq x$, 
$$\nE_x(\tau_x^+\, |\, X_1=y) = \nE_y(\tau_x) +1,$$ since the number of steps to return to $x$, given we first visit $y$, is one more than the number of steps to $x$ starting from $y$. Thus \eqref{Eq:ExpectedReturnFinite2} says
\begin{align}
   \nE_x(\tau_x^+) &= P_{xx} + \sum_{y \neq x}\nE_y(\tau_x)P_{xy} + \sum_{y \neq x} P_{xy} \notag\\
   &= 1 + \sum_{y \neq x}\nE_y(\tau_x)P_{xy},\label{Eq:AverageReturnTimeIdentity2}
\end{align}
which is finite by \eqref{Eq:ExpectedReturnFinite}.
\end{proof}

Note that we have actually proved more than just $\nE_x(\tau_x^+) <\infty$. Since each $\nE_y(\tau_x)$ is explicitly computable through the identity \eqref{Eq:ExpectedHittingTime}, \eqref{Eq:AverageReturnTimeIdentity2} gives us an additional way to compute the average return time.  Sometimes this is useful, and other times it will be easier to use the formula \eqref{Eqn:ReturnTime} in Theorem \ref{Thm:ReturnTime}. 

This detail under our belts, we return to proving Theorem \ref{Thm:ReturnTime}.  A complete proof of the theorem requires some tools from linear algebra that we currently lack, unfortunately.  What we can do, however, is immediately prove the second half of the theorem, which we will use along with the \emph{Perron-Frobenius theorem} to finish the entire proof later in \S \ref{Sec:PF}.\footnote{Even there, however, we will use a part of the Perron-Frobenius theorem that we state without proof, which leaves something to be desired.  We will completely tie this down by using \emph{harmonic functions} in \S \ref{Sec:Harmonic} to complete the proof of Theorem \ref{Thm:ReturnTime} only using Theorem \ref{Thm:ReturnTimeLite}.  See the end of \S \ref{Sec:Harmonic}.}  We isolate the second half of Theorem \ref{Thm:ReturnTime} as a separate theorem statement and then proceed with its proof.  

\begin{theorem}\label{Thm:ReturnTimeLite}
    Let $X$ be an irreducible Markov chain on $\Omega = \{1,2,\ldots, n\}$ with transition matrix $P$.  Then the row vector $\pi \in \mathbb{R}^n$ defined by \eqref{Eq:piFormula} is a probability distribution on $\Omega$ and satisfies
    \begin{align}\label{Eq:StationaryGoal}
        \pi P = \pi.
    \end{align}
    In particular, an irreducible Markov chain always has a stationary distribution.
\end{theorem}

%% Lecture 8 %%%%%%%%%%%%%%%%
\begin{proof}[Proof of Theorem \ref{Thm:ReturnTimeLite}]
    Fix a state $x$ where we will start our chain and define
    \begin{align}
        M_j:=\#\{ \text{visits to $j$ before returning to $x$}\} &= \sum_{k=0}^\infty \I_{\{X_k=j, \, k < \tau_x^+ \}} \notag\\
        &= \sum_{k=0}^\infty \I_{\{X_k=j\}} \I_{\{k < \tau_x^+ \}}\label{Eq:StationaryM_jDefb} \\
        &= \sum_{k=1}^\infty \I_{\{X_{k} = j\}} \I_{\{k \leq \tau_x^+\}}.\label{Eq:StationaryM_jDef}
    \end{align}
    This last line requires some explanation, and there are two cases to consider.  First, if $j=x$, then $M_j = M_x =1$, because only the initial step is counted for $M_x$. In this case the sum \eqref{Eq:StationaryM_jDef} is also 1, because it only counts the step $X_{\tau_x^+}$ when we return to $x$.  The second case is when $j \neq x$, and here both the first step $X_0 =x$ and the last step $X_{\tau_x^+} =x$ are not $j$, and so we can omit the first and include the last and preserve equality.  Thus \eqref{Eq:StationaryM_jDef} holds in either case.  
    
    Let $\tilde{\pi}(j) := \nE_x(M_j)$ be the average of $M_j$ (starting from $x$) and define the row vector $\tilde{\pi} := (\tilde{\pi}(1), \tilde{\pi}(2), \ldots, \tilde{\pi}(n))$. \emph{A priori} the $\tilde{\pi}(j)$ could be infinite, but we can use Lemma \ref{Lemma:StationaryPositive} to show each is actually finite.  Indeed, observe that by \eqref{Eq:StationaryM_jDefb} a fixed entry $\tilde{\pi}(\ell)$ satisfies
    \begin{align}
        &\tilde{\pi}(\ell) \leq \sum_{j=1}^n \tilde{\pi}(j) = \sum_{j=1}^n \nE_x(M_j) = \nE_x\Big( \sum_{j=1}^n M_j \Big) \notag\\
        &= \nE_x\Big(\sum_{j=1}^n \sum_{k=0}^\infty \I_{\{X_k=j\}} \I_{\{k < \tau_x^+ \}} \Big) 
        = \nE_x\Big(\sum_{k=0}^\infty \sum_{j=1}^n  \I_{\{X_k=j\}} \I_{\{k < \tau_x^+ \}} \Big) \label{Eq:ExpectedVisitsFinite}\\
        &= \nE_x\Big(\sum_{k=0}^\infty 1 \cdot \I_{\{k < \tau_x^+ \}} \Big) = \nE_x(\tau_x^+) < \infty \notag
    \end{align}
    by Lemma \ref{Lemma:StationaryPositive}. Here in \eqref{Eq:ExpectedVisitsFinite} we use the fact that each entry in the double sum is non-negative, and thus we can freely interchange the order of summation.\footnote{This is a version of the \emph{Fubini-Tonelli theorem} in real analysis.}  The next equality follows from observing that the sum of the indicator $\I_{\{X_k=j\}}$ over all $j$ is, of course, always exactly one.     
    
    Thus our $\tilde{\pi}$ vector is a well-defined element of $\mathbb{R}^n$.  It is not necessarily a \emph{probability} distribution, though.  Nevertheless, our first goal will be to show
    \begin{align}\label{Eq:ReturnTimesGoal}
        \tilde{\pi} P = \tilde{\pi},
    \end{align}
    since if this is the case, we only need to divide out by $\sum_{\ell=1}^n \tilde{\pi}(\ell)$ and set
    \begin{align}\label{Eq:piFormulaAlmost}
        \pi := \Big(\frac{1}{\sum_{\ell=1}^n \tilde{\pi}(\ell)}\Big) \tilde{\pi}
    \end{align}
    to obtain a stationary distribution.
    
    In fact, let us finish the proof of Theorem \ref{Thm:ReturnTimeLite} assuming that \eqref{Eq:piFormulaAlmost} is a stationary distribution, and then we will come back to prove \eqref{Eq:ReturnTimesGoal}.  Given $\pi$ is stationary, all that is left is to show $\pi(x) = 1/\nE_x(\tau_x^+)$, the identity \eqref{Eq:piFormula}.  For any state $j$, \eqref{Eq:piFormulaAlmost} currently says
    \begin{align}\label{Eq:Generalpi(j)}
        \pi(j) = \frac{\tilde{\pi}(j)}{\sum_{\ell=1}^n \tilde{\pi}(\ell)},
    \end{align}
    and so, in particular,
    \begin{align}\label{Eq:piFormulaAlmost2}
        \pi(x) = \frac{\tilde{\pi}(x)}{\sum_{\ell=1}^n \tilde{\pi}(\ell)} = \frac{1}{\sum_{\ell=1}^n \tilde{\pi}(\ell)},
    \end{align}
    where we use $\tilde{\pi}(x) = \nE_x(M_x) = \nE_x(1) = 1$ for the numerator.  With respect to the denominator, \eqref{Eq:StationaryM_jDefb} yields
    \begin{align}
        \sum_{\ell=1}^n \tilde{\pi}(\ell) &= \sum_{\ell=1}^n \nE_x(M_\ell) \notag\\
        &= \sum_{\ell=1}^n \nE_x \Big( \sum_{k=0}^\infty \I_{\{X_k=\ell\}} \I_{\{k < \tau_x^+\}}  \Big) \notag\\
        &= \nE_x \Big( \sum_{\ell=1}^n \sum_{k=0}^\infty \I_{\{X_k=\ell\}} \I_{\{k < \tau_x^+\}}   \Big) \notag\\
         &= \nE_x \Big( \sum_{k=0}^\infty \Big(\sum_{\ell=1}^n  \I_{\{X_k=\ell\}}\Big) \I_{\{k < \tau_x^+\}}   \Big) \notag\\
         &= \nE_x \Big( \sum_{k=0}^\infty 1 \cdot  \I_{\{k < \tau_x^+\}}   \Big) \notag\\
        &= \nE_x(\tau_x^+).\label{Eq:pi(j)denominator}
    \end{align}
    Again, the non-negativity of the terms enables us to interchange the order of summation as we please.  Thus \eqref{Eq:piFormulaAlmost2} says $\pi(x) = 1/\nE_x(\tau_x^+)$, as claimed, completing the proof of Theorem \ref{Thm:ReturnTimeLite} given \eqref{Eq:ReturnTimesGoal}.

    To prove \eqref{Eq:ReturnTimesGoal} it suffices to show each component on either side of the equation is identical.  The $j$th component is
    \begin{align}\label{Eq:ReturnTimesGoal1}
        \sum_{\ell=1}^n \nE_x(M_\ell) P_{\ell j} = \nE_x(M_j).
    \end{align}
    Proving this will take some effort, but the basic idea will be to substitute the expression \eqref{Eq:StationaryM_jDef} for $M_j$ into the expectation on the right-hand side and then manipulate everything until we obtain the left-hand side. We will again use the Fubini-Tonelli theorem to move expectations through sums of non-negative terms, and we also recall that $\nE_x(\I_A) = \nP_x(A)$.  We have
    \begin{align}
        \nE_x(M_j) &= \nE_x\Big( \sum_{k=1}^\infty \I_{\{X_{k} = j,\, k \leq \tau_x^+\}}\Big) \notag\\
        &= \sum_{k=1}^\infty \nE_x \big( \I_{\{X_{k} = j,\, k \leq \tau_x^+\}}\big) \notag\\
        &= \sum_{k=1}^\infty \nP_x (X_{k} = j, k \leq \tau_x^+) \notag\\
        &= \sum_{k=1}^\infty \sum_{\ell=1}^n \nP_x (X_{k} = j, X_{k-1}=\ell, k \leq \tau_x^+).\label{Eq:NjSum0}
    \end{align}
    The last equality holds because we are summing over all possibilities for $X_{k-1}$.  At this point we would like to separate out the cases $k=1$ and $k>1$ and also remove terms that we know are zero.  We thus write
    \begin{align}
    \sum_{k=1}^\infty \sum_{\ell=1}^n \nP_x (X_{k} = j, X_{k-1}=\ell, k \leq &\tau_x^+)\notag\\
        =\sum_{\ell=1}^n\nP_x (X_{1} = j, X_{0}=\ell, 1 \leq \tau_x^+) &+ \sum_{k=2}^\infty \sum_{\ell=1}^n \nP_x (X_{k} = j, X_{k-1}=\ell, k \leq \tau_x^+)\notag\\
        =\nP_x (X_{1} = j, X_{0}=x, 1 \leq \tau_x^+) &+ \sum_{k=2}^\infty \sum_{\substack{\ell=1\\\ell\neq x}}^n \nP_x (X_{k} = j, X_{k-1}=\ell, k \leq \tau_x^+),\label{Eq:NjSum1}
    \end{align}
   where for the last line we used $\nP_x (X_{1} = j, X_{0}=\ell, 1 \leq \tau_x^+) =0$ when $\ell \neq x$ (the chain starts at $x$), and also that
\begin{align*}
    \nP_x (X_{k} = j, X_{k-1}=x, k \leq \tau_x^+) = 0, \qquad k \geq 2,
\end{align*}
since we cannot have both $X_{k-1} =x$ and $k \leq \tau_x^+$.  For the $k=1$ term in \eqref{Eq:NjSum1}, we observe that $\nP_x(1 \leq \tau_x^+) =1$, and thus
\begin{align*}
    \nP_x (X_{1} = j, X_{0}=x, 1 \leq \tau_x^+) &=\nP_x (X_{1} = j, X_{0}=x)\\ &= P_{xj}\nP_x(X_0=x) = P_{xj}\nP_x(X_0=x, 1 \leq \tau_x^+).
\end{align*}
For the $k\geq 2, \ell \neq x$ terms in \eqref{Eq:NjSum1}, we have
    \begin{align*}
        \nP_x (&X_{k} = j, X_{k-1}=\ell, k \leq \tau_x^+)\notag\\ &= \nP_x (X_{k} = j \,|\, X_{k-1}=\ell, k \leq \tau_x^+)\nP_x (X_{k-1}=\ell, k \leq \tau_x^+).
    \end{align*}
    Noting that the event
    \begin{align*}
        \{X_{k-1}=\ell, k \leq \tau_x^+\} = \{X_{k-1}=\ell, X_{k-2} \neq x, X_{k-3} \neq x, \ldots, X_1 \neq x\},
    \end{align*}
    we have
    \begin{align}
    \nP_x (&X_{k} = j \,|\, X_{k-1}=\ell, k \leq \tau_x^+)\nP_x (X_{k-1}=\ell, k \leq \tau_x^+)\notag\\
        &=\nP_x (X_{k} = j \,|\, X_{k-1}=\ell, X_{k-2} \neq x, \ldots, X_1 \neq x)\nP_x (X_{k-1}=\ell, k \leq \tau_x^+) \notag\\
        &= \nP_x(X_k=j \,|\, X_{k-1}=\ell)\nP_x (X_{k-1}=\ell, k \leq \tau_x^+) \notag\\
        &= P_{\ell j} \nP_x (X_{k-1}=\ell, k \leq \tau_x^+),\label{Eq:NjSum2}
    \end{align}
    where we have used the Markov property, in the sense of an extension of Problem \ref{Prob:MarkovPropertyNegate}, to move from the second to third lines.  Putting these expressions for the two cases back into \eqref{Eq:NjSum1} (and recalling that the sum is $\nE_x(M_j)$ by \eqref{Eq:NjSum0}), we have
\begin{align}
    \nE_x(M_j)&= \nP_x (X_{1} =  P_{x j} \nP_x(X_0=x, 1 \leq \tau_x^+) + \sum_{k=2}^\infty \sum_{\substack{\ell=1\\\ell\neq x}}^n P_{\ell j}\nP_x (X_{k-1}=\ell, k \leq \tau_x^+)\notag\\
        &=  \sum_{\ell=1}^nP_{\ell j} \nP_x(X_0=\ell, 1 \leq \tau_x^+) + \sum_{k=2}^\infty \sum_{\ell=1}^n P_{\ell j} \nP_x (X_{k-1}=\ell, k \leq \tau_x^+) \notag\\
        &=\sum_{k=1}^\infty \sum_{\ell=1}^n P_{\ell j} \nP_x (X_{k-1}=\ell, k \leq \tau_x^+).\notag
    \end{align}
Notice that for the second equality here we have simply re-introduced a bunch of zero terms.  Now we are nearly done; we just examine this last sum and note
\begin{align*}
\sum_{k=1}^\infty \sum_{\ell=1}^n P_{\ell j} \nP_x (X_{k-1}=\ell, k \leq \tau_x^+) &=\sum_{k=1}^\infty \sum_{\ell=1}^n P_{\ell j} \nE_x\big( \I_{\{X_{k-1}=\ell\}} \I_{\{k-1<\tau_x^+\}}  \big)\\
    &= \sum_{\ell=1}^n P_{\ell j} \sum_{k=1}^\infty \nE_x\big( \I_{\{X_{k-1}=\ell\}} \I_{\{k-1<\tau_x^+\}}  \big)\\
    &= \sum_{\ell=1}^n P_{\ell j} \sum_{m=0}^\infty \nE_x\big( \I_{\{X_{m}=\ell\}} \I_{\{m<\tau_x^+\}}  \big)\\
    &= \sum_{\ell=1}^n P_{\ell j} \nE_x\Big(\sum_{m=0}^\infty  \I_{\{X_{m}=\ell\}} \I_{\{m<\tau_x^+\}}  \Big)\\
    &= \sum_{\ell=1}^n P_{\ell j} \nE_x(M_\ell)
\end{align*}
by \eqref{Eq:StationaryM_jDefb}.  Recalling the entire sum is $\nE_x(M_j)$, we have thus shown that
\begin{align*}
    \nE_x(M_j) = \sum_{\ell=1}^n P_{\ell j} \nE_x(M_\ell).
\end{align*}
which is exactly \eqref{Eq:ReturnTimesGoal1}, what we wanted to show!  That is, we have proven that our $\tilde{\pi}$ vector satisfies \eqref{Eq:ReturnTimesGoal}.  As explained at the beginning of the argument, this finishes the proof.
\end{proof}
    
Theorem \ref{Thm:ReturnTimeLite} says that the formula \eqref{Eq:piFormula} always gives us a stationary distribution.  Notice, however, that we also have another useful equation: combining \eqref{Eq:Generalpi(j)} with the expression \eqref{Eq:pi(j)denominator} for the denominator yields
\begin{align}\label{Eq:piFormula2}
    \pi(j) = \frac{\tilde{\pi}(j)}{\nE_x(\tau_x^+)} = \frac{\nE_x \big( \#\{ \text{visits to $j$ before $\tau_x^+$} \} \big)}{\nE_x(\tau_x^+)}.
\end{align}
Both formulas \eqref{Eq:piFormula} and \eqref{Eq:piFormula2} will prove useful in different contexts.

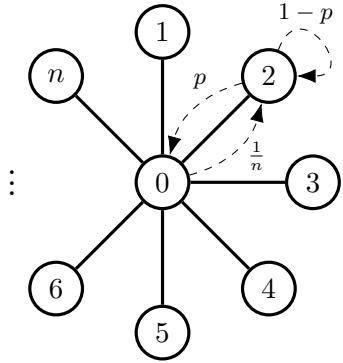
\begin{figure}
    \centering
        \begin{tikzpicture}
    \tikzset{
        node/.style={circle, draw=black, very thick, minimum size=7mm},
         arrow/.style={-{Latex[length=2.5mm]}, black},
        line/.style={black, very thick}
    }
    
    \node[node] (zero) at (0,0) {0};
    \node[node] (one) at (0,2) {1};
    \node[node] (two) at (1.414,1.414) {2};
    \node[node] (three) at (2, 0) {3};
    \node[node] (four) at (1.414, -1.414) {4};
    \node[node] (five) at (0, -2) {5};
    \node[node] (six) at (-1.414, -1.414) {6};
    \node[] (dots) at (-2, 0.1) {\Large{$\vdots$}};
    \node[node] (n) at (-1.414, 1.414) {$n$};
    \node[] (text1) at (1.9,2.25) {\footnotesize{$1-p$}};
    \node[] (text2) at (0.5,1.35) {\footnotesize{$p$}};
    \node[] (text3) at (1.25,0.35) {\footnotesize{$\frac{1}{n}$}};
    
    \draw[line] (one) -- (zero);
    \draw[line] (two) -- (zero);
    \draw[line] (three) -- (zero);
    \draw[line] (four) -- (zero);
    \draw[line] (five) -- (zero);
    \draw[line] (six) -- (zero);
    \draw[line] (n) -- (zero);
    
     \draw[arrow, dashed] (two) to [out=70,in=0,looseness=5] node[anchor=north, sloped] {} (two);
     \draw[arrow, dashed] (two) to [bend right] node[anchor=north, sloped] {} (zero);
    \draw[arrow, dashed] (zero) to [bend right] node[anchor=north, sloped] {} (two);
     
    \end{tikzpicture}
    \caption{\small What is $\pi$ for this chain?}
    \label{Fig:StarGraph}
\end{figure}

\begin{example}
    To see \eqref{Eq:piFormula} in action, consider the ``star graph'' of Figure \ref{Fig:StarGraph}.  Here there is one central vertex, labelled 0, connected to $n$ ``leaf'' vertices $1,2, \ldots, n$.  If we are at 0, we go to any of the leaves with equal probability.  If we are at a leaf, we move to the center with probability $p$, and stay at the same leaf with probability $1-p$, for some fixed $0< p<1$.  What is the stationary distribution for this chain?
    
    To find $\pi$, let's compute the return times and use \eqref{Eq:piFormula}. We could use \eqref{Eq:AverageReturnTimeIdentity2} to find each $\nE_x(\tau_x^+)$, but here arguing directly from the chain is much easier.  We observe that if we start with $X_0=0$, $\tau_0^+ = 1 + Y$, where $Y$ is a $\text{Geo}(p)$ random variable, since we immediately step away to a leaf, and then have independent trials with probability of success $p$ to see if we return.  Hence 
    \begin{align*}
        \nE_0(\tau_0^+) = 1 + \nE(Y) = 1 + \frac{1}{p} = \frac{1+p}{p},
    \end{align*}
    and so $\pi(0) = \frac{p}{1+p}$.  By symmetry, the $\nE_j(\tau_j^+)$ are all identical for $j=1,2,\ldots, n$, and so $\pi(1)=\pi(2) = \cdots = \pi(n)$.  We therefore have
    \begin{align*}
        1 = \sum_{j=1}^n \pi(j) = \pi(0) + n\pi(1) = \frac{p}{p+1} + n \pi(1),
    \end{align*}
    and so $\pi(1) = \frac{1}{n(1+p)}$.  We conclude
    \begin{align*}
        \pi = \Big( \frac{p}{p+1}, \frac{1}{n(1+p)}, \frac{1}{n(1+p)}, \ldots, \frac{1}{n(1+p)} \Big).
    \end{align*}
So, for instance, if there are $n=10$ leaves and we move to the center with probability $p=1/5$.  Then our stationary distribution is
    \begin{align}\label{Eq:StarStationaryEg}
        \pi = \Big( \frac{1}{6}, \frac{1}{12}, \frac{1}{12}, \ldots, \frac{1}{12} \Big).
    \end{align}
    Note two things: 
    \begin{itemize}
        \item If we start the Markov chain with our initial position $X_0$ randomly selected according to \eqref{Eq:StarStationaryEg}, then at \emph{every} subsequent step $k$, the distribution of $X_k$ on the vertices is still exactly $\pi$.  For example, $\nP_\pi(X_{1,000,000}=5 ) = 1/12$.
        \item Secondly, Theorem \ref{Thm:ReturnTime} tells us that if we start at vertex 0, on average it will take us 6 steps to return back to 0.  Similarly, if we start at any vertex $1,2, \ldots, 10$, it will take us 12 steps, on average, to return.
    \end{itemize}
\end{example}

\begin{figure}
    \centering
        \begin{tikzpicture}
    \tikzset{
        node/.style={circle, draw=black, very thick, minimum size=7mm},
        line/.style={black, very thick}
    }
    
    \node[node] (four) at (2,0) {4};
    \node[node] (one) at (4,0) {1};
    \node[node] (two) at (4,-2) {2};
    \node[node] (three) at (2, -2) {3};
    \node[node] (five) at (0, -0) {5};
    \node[node] (six) at (0, -2) {6};
    
    \draw[line] (one) -- (two);
    \draw[line] (two) -- (three);
    \draw[line] (three) -- (four);
    \draw[line] (four) -- (one);
    \draw[line] (one) -- (three);
    \draw[line] (six) -- (three);
    \draw[line] (five) -- (three);
    \draw[line] (five) -- (six);
    \end{tikzpicture}
    \caption{}
    \label{Fig:ReturnTimeEx}
\end{figure}
\begin{exercise}
    Find $\nE_1(\tau_1^+)$ and $\nE_5(\tau_5^+)$ for the graph in Figure \ref{Fig:ReturnTimeEx}.
\end{exercise}

% Lecture 9 %%%%%%%%%%%%%%%%%%%%%%%%%%%%%%%

\section{Time-reversibility}\label{Sec:Reverse}

Consider a biased and an unbiased walk on the 5-cycle, as illustrated in Figure \ref{Fig:ReversibleMotivation}. As discussed above in Example \ref{Eg:StationaryForBiasedWalk}, both walks have the same stationary distribution 
\begin{align*}
    \pi = \Big(\frac{1}{5}, \frac{1}{5}, \frac{1}{5}, \frac{1}{5}, \frac{1}{5}   \Big).
\end{align*}
Suppose we had movies of our Markov chains unfolding on each of the pentagons.  What would happen if we played the movies in reverse?  

For the unbiased walk, we wouldn't see any difference: when reversed, the walk would still jump CW or CCW with equal probability, just as in the forward direction. The biased walk would look different, though: in the forward direction it prefers jumping CW, so when reversed it prefers CCW steps. 

In this section we ``reverse the movie'' of a Markov chain and study when the chain looks the same when viewed in either direction.
\begin{figure}
    \centering
    
    \scalebox{0.9}{
    \begin{tikzpicture}
    \tikzset{
        node/.style={circle, draw=black, very thick, minimum size=7mm},
        arrow/.style={-{Latex[length=2mm]}, black},
        line/.style={black, very thick}
        }
    
    \node[node] (one) at (0,1) {1};
    \node[node] (two) at (1.376,0) {2};
    \node[node] (three) at (0.85065,-1.618) {3};
    \node[node] (four) at (-0.85065,-1.618) {4};
    \node[node] (five) at (-1.376,0) {5};
    
    \draw[line] (one) -- (two);
    \draw[line] (two) -- (three);
    \draw[line] (three) -- (four);
    \draw[line] (four) -- (five);
    \draw[line] (five) -- (one);
    
    \draw[arrow, dashed] (one) to[bend left] node[anchor=south, sloped] {0.5} (two);
    \draw[arrow, dashed] (one) to[bend right] node[anchor=south, sloped] {0.5} (five);
    
    \node[node] (oneb) at (6,1) {1};
    \node[node] (twob) at (7.376,0) {2};
    \node[node] (threeb) at (6.85065,-1.618) {3};
    \node[node] (fourb) at (5.14935,-1.618) {4};
    \node[node] (fiveb) at (4.624,0) {5};
    
    \draw[line] (oneb) -- (twob);
    \draw[line] (twob) -- (threeb);
    \draw[line] (threeb) -- (fourb);
    \draw[line] (fourb) -- (fiveb);
    \draw[line] (fiveb) -- (oneb);
    
    \draw[dashed, arrow] (oneb) to[bend left] node[above,sloped] {0.8} (twob);
    \draw[dashed, arrow] (oneb) to[bend right] node[above, sloped] {0.2} (fiveb);
    \end{tikzpicture}}
    
    \caption{\small An unbiased walk and a biased walk on the 5-cycle.}
    \label{Fig:ReversibleMotivation}
\end{figure}
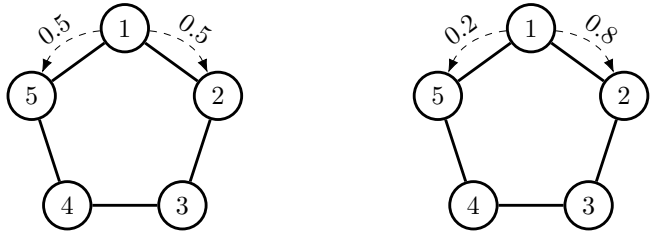

% First, some notation: Recall that we defined $\nP_x(\cdot)$ and $\nE_x(\cdot)$ above in \eqref{Notation:ConditionalSubscriptP} and \eqref{Notation:ConditionalSubscriptE} as the conditional probability and expectation, given $X_0=x$ (the dot just means that we can plug in any valid event or function, respectively).  We similarly define
% \begin{align}\label{Notation:ConditionalPSubscriptMu}
%     \nP_\mu (\cdot) := \nP( \, \cdot \, |\, X_0 \sim \mu),
% \end{align}
% where $X_0 \sim \mu$ means that $X_0$ is randomly distributed on $\Omega$, with pdf given by $\mu$.  Similarly,
% \begin{align}\label{Notation:ConditionalESubscriptMu}
%     \nE_\mu (\cdot) :=& \nE( \, \cdot \, |\, X_0 \sim \mu)\\
%     =& \sum_{x \in \Omega} \nE_x(\cdot) \mu(x). \notag
% \end{align}

\subsection{Bayes' rule and time-reversal}

The first thing to settle is how to rigorously reverse time, or ``reverse the movie.''  Suppose our chain $X$ has stationary distribution $\pi$ and transition matrix $P$.  If we start with $X_0$ distributed according to $\pi$, then $X_k$ is still distributed as $\pi$ for each $k = 1,2,\ldots$, and
\begin{align}\label{Eq:TimeReversalForward}
    \nP_\pi(X_{k+1} = y \,|\, X_k = x) = P_{xy}
\end{align}
for any $k \geq 1$.  To reverse time, we want to instead compute
\begin{align}\label{Eq:ReversedProbDef}
    \hat{P}_{yx} := \nP_\pi (X_k = x \, | \, X_{k+1} = y).
\end{align}
That is, the probability of going from $y$ to $x$ under time reversal is the probability of arriving at $y$ from state $x$ in the original chain.  This, of course, is backwards from what we are used to computing: now we are conditioning on our position at a \emph{later} time (in the original chain) and asking about the position one step \emph{earlier} (in the original chain).  This is what it means to ``go forward'' in reversed time.

This definition  \eqref{Eq:ReversedProbDef} of the transition probabilities of the reversed chain begs two questions, however.  First, is it really independent of $k$?  While the Markov property tells us that \eqref{Eq:TimeReversalForward} does not depend on $k$, it is not immediately obvious that the right-hand side of \eqref{Eq:ReversedProbDef} is also independent of $k$, as we are conditioning in the opposite order. By Bayes' theorem, however, we see
\begin{align}
    \hat{P}_{yx} &:=\nP_\pi(X_k=x \,|\, X_{k+1}=y) \notag\\
    &= \frac{ \nP_\pi(X_{k+1} = y \,|\, X_k = x)\nP_\pi(X_k = x)}{\nP_\pi(X_{k+1} = y)} = P_{xy} \cdot \frac{\pi(x)}{\pi(y)}. \label{Eq:FindTimeReversal}
\end{align}
\noindent Thus $\hat{P}_{yx}$ is independent of $k$, and we are right to omit $k$ in its notation.  As the above computation shows, this hinges upon having a stationary distribution, which implies that $\nP_\pi(X_k = x)$ and $\nP_\pi(X_{k+1} = y)$ are the same for all $k$.  Of course it is also important that we are not dividing by zero in \eqref{Eq:FindTimeReversal}, but Lemma \ref{Lemma:StationaryPositive} shows that each $\pi(y) >0$.

The second question about \eqref{Eq:ReversedProbDef} is whether or not it truly defines a transition matrix on $\Omega \times \Omega$.  Clearly each $\hat{P}_{yx} \geq 0$, and for a fixed $y \in \Omega$, the corresponding row sum is
    \begin{align*}
        \sum_{x \in \Omega} \hat{P}_{yx} = \frac{1}{\pi(y)}\sum_{x \in \Omega}P_{xy} \pi(x) = \frac{1}{\pi(y)} \pi(y) = 1,
    \end{align*}
    where the second equality follows from the fact that $\pi P = \pi$.  So the transition probabilities given by \eqref{Eq:FindTimeReversal} define a legitimate transition matrix, and we thus have a (potentially new) Markov chain $\hat{X}$, the time reversal of our original chain $X$.

\begin{definition}
    The \textbf{time-reversal} of an irreducible Markov chain $X$ on $\Omega$ with transition matrix $P$ and stationary distribution $\pi$ is the Markov chain $\hat{X}$ on $\Omega$ whose transition matrix $\hat{P}$ has entries
    \begin{align}\label{Eq:TimeReversalTransProb}
        \hat{P}_{yx} =  P_{xy} \frac{\pi(x)}{\pi(y)}
    \end{align}
    for each $x,y \in \Omega$.  
\end{definition}
A simple first observation about time reversal is that it preserves irreducibility.  That is, the original chain $X$ is irreducible if and only if its time reversal $\hat{X}$ also is.  Indeed, suppose $X$ is irreducible and pick any two states $x,y \in \Omega$.  Then there is a forward path $y \rightarrow x$ of positive probability.  Reversing this gives a positive-probability path $x \rightarrow y$ for $\hat{X}$.  For the other direction, exchange the roles of $X$ and $\hat{X}$.
% Since irreducible chains always have a stationary distribution (Theorem \ref{Thm:ReturnTimeLite}), we see that we can always reverse an irreducible chain to obtain a (potentially different) irreducible chain.

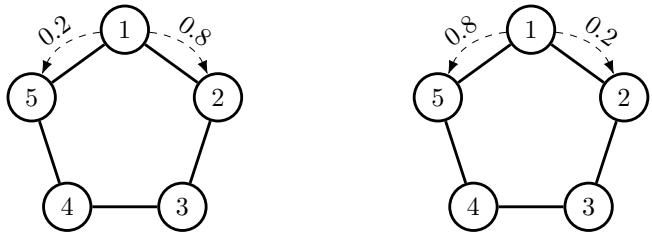
\begin{figure}[b]
    \centering
        
        \scalebox{0.9}{
        \begin{tikzpicture}
        \tikzset{
            node/.style={circle, draw=black, very thick, minimum size=7mm},
            arrow/.style={-{Latex[length=2mm]}, black},
            line/.style={black, very thick}}
        
        \node[node] (one) at (0,1) {1};
        \node[node] (two) at (1.376,0) {2};
        \node[node] (three) at (0.85065,-1.618) {3};
        \node[node] (four) at (-0.85065,-1.618) {4};
        \node[node] (five) at (-1.376,0) {5};
        
        \draw[line] (one) -- (two);
        \draw[line] (two) -- (three);
        \draw[line] (three) -- (four);
        \draw[line] (four) -- (five);
        \draw[line] (five) -- (one);
        
        \draw[arrow, dashed] (one) to[bend left] node[anchor=south, sloped] {0.8} (two);
        \draw[arrow, dashed] (one) to[bend right] node[anchor=south, sloped] {0.2} (five);
        
        \node[node] (oneb) at (6,1) {1};
        \node[node] (twob) at (7.376,0) {2};
        \node[node] (threeb) at (6.85065,-1.618) {3};
        \node[node] (fourb) at (5.14935,-1.618) {4};
        \node[node] (fiveb) at (4.624,0) {5};
        
        \draw[line] (oneb) -- (twob);
        \draw[line] (twob) -- (threeb);
        \draw[line] (threeb) -- (fourb);
        \draw[line] (fourb) -- (fiveb);
        \draw[line] (fiveb) -- (oneb);
        
        \draw[arrow, dashed] (oneb) to[bend left] node[anchor=south, sloped] {0.2} (twob);
        \draw[arrow, dashed] (oneb) to[bend right] node[anchor=south, sloped] {0.8} (fiveb);
        \end{tikzpicture}}
        
    \caption{\small The time reversal of a biased walk on the 5-cycle is another biased walk on the 5-cycle, but with the probabilities reversed.}
    \label{Fig:ReversibleMotivationRevisited}
\end{figure}

\begin{example}\label{Eg:PentagonReversableRevisited}
Let's revisit the 5-cycle.  For both the biased and unbiased walks the stationary distribution is uniform, and so \eqref{Eq:TimeReversalTransProb} becomes
\begin{align}\label{Eq:CycleTimeReversal}
    \hat{P}_{yx} = P_{xy}.
\end{align}
We can now make our introductory remarks regarding Figure \ref{Fig:ReversibleMotivation} more rigorous.  For the unbiased walk, \eqref{Eq:CycleTimeReversal} says
\begin{align}\label{Eq:PentagonReversible}
    \hat{P}(j,j+1) = P(j+1,j) = \frac{1}{2} = P(j,j+1),
\end{align}
so we cannot tell the difference between the original and the reversed movies.  For the biased walk with $P(j,j+1) = p = 1-P(j,j-1)$, however, we find
\begin{align}\label{Eq:CycleTimeReversal2}
    \hat{P}(j,j+1) = P(j+1,j) = 1-p = 1 - P(j,j+1),
\end{align}
and so the movies are different if $p\neq 1/2$.  In particular, for the example of $p=0.8$, \eqref{Eq:CycleTimeReversal2} says the probability that our time-reversed Markov chain $\hat{X}$ moves CW is 0.2, while the probability of moving CCW is 0.8.  See Figure \ref{Fig:ReversibleMotivationRevisited}.  So, as we expected, because the starting walk preferred CW steps, the time-reversed walk prefers to take CCW steps.  %%An observer obversing the movie can identity with high probability which walk they are obsreving.

Another helpful case for the intuition is $p=1$.  In this totally-biased walk we are guaranteed to move CW each step.  Equation \eqref{Eq:TimeReversalTransProb} for the time-reversal then says
\begin{align*}
    \hat{P}(j,j-1) = P(j-1,j) = 1,
\end{align*}
and so we move CCW with probability 1.  The movie is reversed in the obvious way.  See Figure \ref{Fig:UltraBiasedReversal}
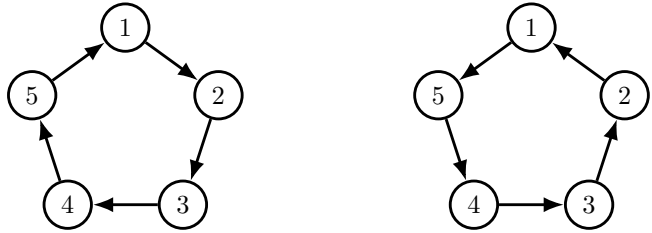
\begin{figure}
    \centering
    
    \scalebox{0.9}{
            \begin{tikzpicture}
        \tikzset{
            node/.style={circle, draw=black, very thick, minimum size=7mm},
            arrow/.style={-{Latex[length=3mm]}, black, very thick},
            line/.style={black, very thick}}
        
        \node[node] (one) at (0,1) {1};
        \node[node] (two) at (1.376,0) {2};
        \node[node] (three) at (0.85065,-1.618) {3};
        \node[node] (four) at (-0.85065,-1.618) {4};
        \node[node] (five) at (-1.376,0) {5};
        
        \draw[arrow] (one) -- (two);
        \draw[arrow] (two) -- (three);
        \draw[arrow] (three) -- (four);
        \draw[arrow] (four) -- (five);
        \draw[arrow] (five) -- (one);
        
        \node[node] (oneb) at (6,1) {1};
        \node[node] (twob) at (7.376,0) {2};
        \node[node] (threeb) at (6.85065,-1.618) {3};
        \node[node] (fourb) at (5.14935,-1.618) {4};
        \node[node] (fiveb) at (4.624,0) {5};
        
        \draw[arrow] (oneb) -- (fiveb);
        \draw[arrow] (twob) -- (oneb);
        \draw[arrow] (threeb) -- (twob);
        \draw[arrow] (fourb) -- (threeb);
        \draw[arrow] (fiveb) -- (fourb);
        \end{tikzpicture}
    }
    
    \caption{\small Reversing the totally-biased walk with $p=1$ gives the totally-biased walk with $p=0$, where $p$ is the probability of moving CW.}
    \label{Fig:UltraBiasedReversal}
\end{figure}

\end{example}

\subsection{Stationary distributions and reversibility; the detail balance equations}

If our original chain has stationary distribution $\pi$,  what can we say about the stationary distribution for the reversed chain?

\begin{theorem}\label{Thm:SamepiForReversed}
    If $(X_n)$ is an irreducible Markov chain with stationary distribution $\pi$, then $\pi$ is also stationary for the reversed chain $(\hat{X}_n)$.  Furthermore, writing $\hat{\nP}$ for probabilities under the transition matrix $\hat{P}$, 
    \begin{multline}\label{Eq:ReversedWalkProbability}
        \hat{\nP}_\pi(\hat{X}_0 = x_0, \hat{X}_1 = x_1, \ldots, \hat{X}_k = x_k)\\
        = \nP_\pi(X_0 = x_k, X_1 = x_{k-1}, \ldots, X_k = x_0).
    \end{multline}
\end{theorem}

Note that \eqref{Eq:ReversedWalkProbability} is just the mathematical formulation that we are reversing time: the path $x_0 \rightarrow x_1 \rightarrow \cdots \rightarrow x_k$ in the reversed chain arises from reversing the path $x_k \rightarrow x_{k-1} \rightarrow \cdots \rightarrow x_0$ in the original chain. Also, recall the fact that $\pi$ is stationary for both chains matches what happened with our biased walk on the 5-cycle: regardless of $p$, the uniform distribution is stationary, and so it is stationary for both the biased walk and its reversal.

\begin{proof}
    Let's first verify that $\pi$ is still stationary for $\hat{P}$.  Of course, this means
    \begin{align}\label{Eq:StationaryReversed}
        \pi \hat{P} = \pi,
    \end{align}
    which we can show component-wise, as usual. That is, we need to show
    \begin{align}\label{Eq:StationaryReversedComponent}
        \sum_{j=1}^n \pi(j)\hat{P}_{jk} = \pi(k)
    \end{align}
    for each $k$.  Recalling the definition \eqref{Eq:TimeReversalTransProb}, we see
    \begin{align*}
        \sum_{j=1}^n \pi(j)\hat{P}_{jk} &= \sum_{j=1}^n \pi(j)P_{kj} \frac{\pi(k)}{\pi(j)}\\
        &= \pi(k) \sum_{j=1}^n P_{kj} = \pi(k) \cdot 1,
    \end{align*}
    since we are summing the entire $k$th row of $P$.  Hence \eqref{Eq:StationaryReversedComponent} holds as does \eqref{Eq:StationaryReversed}; $\pi$ is indeed stationary for $\hat{P}$.
    
    We can use this to easily prove the second part, the formula \eqref{Eq:ReversedWalkProbability}.  We will only show
    \begin{align*}
        \hat{\nP}_\pi(\hat{X}_0 = x_0, \hat{X}_1 = x_1) = \nP_\pi(X_0 = x_1, X_1 = x_0),
    \end{align*}
    as the general argument is similar and follows from induction.  Indeed, we have
    \begin{align*}
        \hat{\nP}_\pi(\hat{X}_0 = x_0, \hat{X}_1 = x_1) &=  \hat{\nP}_\pi(\hat{X}_1 = x_1 \, | \, \hat{X}_0 = x_0) \hat{\nP}_\pi (\hat{X_0} = x_0)\\
        &= \hat{P}_{x_0,x_1} \pi(x_0)\\
        &= P_{x_1,x_0} \frac{\pi(x_1)}{\pi(x_0)} \pi(x_0)\\
        &= P_{x_1,x_0} \pi(x_1)\\
        &= \nP_{\pi}(X_1 = x_0 \, | \, X_0 = x_1) \nP_\pi (X_0 = x_1)\\
        &= \nP_\pi(X_0 = x_1, X_1=x_0). \qedhere
    \end{align*}
\end{proof}

So we know how to time-reverse a Markov chain (just calculate the new transition probabilities using \eqref{Eq:TimeReversalTransProb}), and we know that the stationary distibution is the same for both chains (Theorem \ref{Thm:SamepiForReversed}).  Still lingering, however, is the issue of when we get a genuinely different chain.  That is, when can we tell which direction the movie is being played?  When is $\hat{P} \neq P$? 

\begin{definition}\label{Def:Reversible}
    An irreducible Markov chain is \textbf{time reversible} (or simply \textbf{reversible}) with respect to its stationary distribution $\pi$ if 
    \begin{align}\label{Eq:Reversible}
        \hat{P} = P.
    \end{align}
\end{definition}
Note that $\pi$ is implicit in \eqref{Eq:Reversible}, since whenever we write $\hat{P}$ we are using \eqref{Eq:TimeReversalTransProb} and hence the stationary distribution $\pi$.  So it's important to keep in mind that reversibility is \emph{always with respect to a given stationary distribution $\pi$}.  In fact, reversibility can easily be broken if we do not start with $X_0$ distributed according to $\pi$, even if the chain is reversible in the sense of Definition \ref{Def:Reversible}.  For instance, we saw above in equation \eqref{Eq:PentagonReversible} that the simple walk on the 5-cycle is reversible with respect to its (uniform) stationary distribution.  What if we instead start with $X_0=1$, i.e. with the distribution $\mu = (1,0,0,0,0)$?  We use the definition \eqref{Eq:ReversedProbDef} to compute, for instance, that
    \begin{align*}
        \hat{P}(2,1) &= \nP_\mu(X_0 = 1 \, | \, X_1=2)\\
        &= \nP_\mu(X_1 =2 \, | \, X_0=1) \frac{\nP_\mu(X_0=1)}{\nP_\mu(X_1=2)} = \frac{1}{2}\cdot \frac{1}{1/2} = 1 \neq \frac{1}{2} = P(2,1),
    \end{align*}
    and so $\hat{P} \neq P$ when the chain is started with distribution $\mu$.  However, as we saw, $\hat{P} = P$ when we start with the stationary distribution $\pi$.

So, how can we tell if the chain is time reversible?  Suppose that we do have $P = \hat{P}$, and let's consider an individual entry $P(x,y)$.  From our assumption and \eqref{Eq:TimeReversalTransProb},
\begin{align*}
    P(x,y) = \hat{P}(x,y) = P(y,x)\frac{\pi(y)}{\pi(x)},
\end{align*}
and so we see $P = \hat{P}$ iff 
\begin{align}\label{Eq:DBE}
    \pi(x)P(x,y) = \pi(y)P(y,x)
\end{align}
for every $x,y \in \Omega$.  We will call \eqref{Eq:DBE} the \textbf{detail balance equations} or \textbf{DBEs} for short, and this will prove to be a very important equation.  So important, in fact, that you should memorize \eqref{Eq:DBE} to have it at your fingertips.

We have proven the second part of the following theorem.

\begin{theorem}\label{Thm:DBEConsequences}
    Let $X$ be an irreducible Markov chain on  $\Omega$ with transition matrix $P$, and let $\pi$ be a probability distribution on $\Omega$ such that the detail balance equations \eqref{Eq:DBE} hold for all $x,y \in \Omega$.  Then
    \begin{enumerate}[$(i)$]
        \item\label{Thm:DBEStationary} $\pi$ is a stationary distribution for $P$, and
        \item\label{Thm:DBEReverse} $P$ is reversible with respect to $\pi$.
    \end{enumerate}
\end{theorem}

This theorem highlights the power of the DBEs: if we verify that \eqref{Eq:DBE} holds, we know that both $\pi$ is stationary and that the chain is reversible with respect to $\pi$. 

\begin{proof}
    Suppose that $\pi$ is a probability distribution satisfying the DBE's; we start by showing $(\ref{Thm:DBEStationary})$.  Indeed, sum \eqref{Eq:DBE} over all $x\in\Omega$ to obtain
    \begin{align}\label{Eq:DBEStationary}
        \sum_{x \in \Omega} \pi(x) P(x,y) = \sum_{x \in \Omega} \pi(y) P(y,x).
    \end{align}
    Note that the left-hand side is the $y$-component of the vector-matrix product $\pi P$.  The right-hand side is
    \begin{align*}
        \pi(y) \sum_{x \in \Omega} P(y,x) = \pi(y)
    \end{align*}
    since we are summing over a fixed row in $P$.  Thus \eqref{Eq:DBEStationary} says
    \begin{align*}
        (\pi P)_y = \pi(y)
    \end{align*}
    for any $y$, showing $\pi P = \pi$, as claimed.

    For $(\ref{Thm:DBEReverse})$, we have by assumption that the DBEs holds, yielding
    \begin{align*}
        P(x,y) = \frac{\pi(y)}{\pi(x)}P(y,x) = \hat{P}(x,y)
    \end{align*}
    for all $x,y \in \Omega$.
\end{proof}

So far the only reversible chains we've seen are the simple walks on $n$-cycles.  But this is part of a much larger class of examples.

\begin{theorem}
    Simple symmetric random walks on graphs are reversible with respect to their stationary distribution \eqref{Eq:SRWStationaryDistribution}.
\end{theorem}

\begin{proof}
    We simply need to verify the DBEs \eqref{Eq:DBE}.  Indeed, pick any vertices $x$ and $y$.  If $y \not\sim x$, then both $P(x,y)$ and $P(y,x)$ are zero, and the DBEs holds.  If $y \sim x$, then
    \begin{align*}
        \pi(x)P(x,y) = \frac{\deg(x)}{2|E|} \cdot \frac{1}{\deg(x)} = \frac{1}{2|E|},
    \end{align*}
    while
    \begin{align*}
        \pi(y)P(y,x) = \frac{\deg(y)}{2|E|} \cdot \frac{1}{\deg(y)} =  \frac{1}{2|E|}
    \end{align*}
    as well.
\end{proof}

Another class of reversible Markov chains are those with \emph{symmetric} transition matrices.

\begin{definition}
    A transition matrix is \textbf{symmetric} if $P(x,y) = P(y,x)$ for all $x,y \in \Omega$.  In terms of linear algebra, this says that $P = P^T$, the transpose of the matrix $P$.
\end{definition}

\begin{theorem}\label{Thm:SymmetricReversible}
    If a Markov chain on $\Omega = \{1,2,\ldots, n\}$ has a symmetric transition matrix $P$, then the uniform distribution $\pi=(\frac{1}{n}, \frac{1}{n}, \ldots, \frac{1}{n})$ is stationary and the chain is reversible with respect to $\pi$.
\end{theorem}

\begin{exercise}
    Prove Theorem \ref{Thm:SymmetricReversible} in two different ways: first, directly verify that $\pi P = \pi$.  Secondly, show that the DBEs \eqref{Eq:DBE} hold and invoke Theorem \ref{Thm:DBEConsequences}. 
\end{exercise}

%%% Lecture 10 %%%%%%%%%%%%%%%%%%%%%%%%

\subsection{Two in-depth examples}

Time reversing a Markov chain can be conceptually intimidating.  In this section, we work through two in-depth examples to help bolster our understanding of how reversal works in practice.

\subsubsection{Example 1: A ``geometric'' walk on $\mathbb{N}$}

Let $\Omega$ be the infinite state space $\mathbb{N} = \{1,2,3,\ldots\}$ and fix $0<p<1$.  Consider the Markov chain with transition probabilities
\begin{align}\label{Eq:InDepthReversableProbs}
    P(j,j+1) = 1-p \qquad \text{ and } \qquad P(j,1) = p 
\end{align}
for all $j \in \Omega$.  See Figure \ref{Fig:ReversibleSeriousEg}. At each step, this Markov chain moves one step right with probability $1-p$, and jumps back to 1 with probability $p$.  So we can think of it as working in cycles where each cycle consists of walking up the integers via repeated ``failures,'' until it ``succeeds'' and resets at value 1.  Starting from 1, each of these cycles has random length $L$, where $L$ is a Geo$(p)$ random variable. 

Before we explicitly compute the time reversal, it is a good idea to first try to ``reverse the movie'' in our minds. This will give some intuition to guide the subsequent computations. 
\begin{exercise}\label{Ex:GuessReversal}
    Without performing any calculations, informally describe the behavior of the time-reversal of this Markov chain.
\end{exercise}

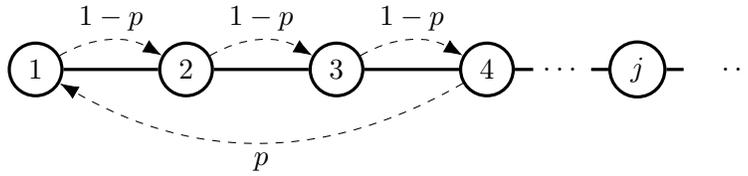
\begin{figure}
    \centering
    
    \scalebox{1}{
    \begin{tikzpicture}
    \tikzset{
        node/.style={circle, draw=black, very thick, minimum size=7mm},
        arrow/.style={-{Latex[length=2.5mm]}, black},
        line/.style={black, very thick}
        }
        
        \node[node] (one) at (0,0) {1};
        \node[node] (two) at (2,0) {2};
        \node[node] (three) at (4,0) {3};
        \node[node] (four) at (6,0) {4};
        \node[] (four5) at (6.75,0) {};
        \node[] (j-5) at (7.25,0) {};
        \node[node] (j) at (8,0) {$j$};
        \node[] (j+5) at (8.75,0) {};
        \node[] (k) at (10,0) {};

        \draw[line] (one) -- (two);
        \draw[line] (two) -- (three);
        \draw[line] (three) -- (four);
        \draw[line] (four) -- (four5);
        \path (four5) -- node[auto=false]{\ldots} (j-5);
        \draw[line] (j-5) -- (j);
        \draw[line] (j) -- (j+5);
        \path (j+5) -- node[auto=false]{\ldots} (k);
        % \draw[black, very thick, dashed] (four) -- (j);
        % \draw[black, very thick, dashed] (j) -- (10,0);
        
        \draw[arrow, dashed] (four) to[bend left] node[anchor=north] {$p$} (one);
        \draw[arrow, dashed] (one) to[bend left] node[anchor=south] {$1-p$} (two);
        \draw[arrow, dashed] (two) to[bend left] node[anchor=south] {$1-p$} (three);
        \draw[arrow, dashed] (three) to[bend left] node[anchor=south] {$1-p$} (four);
    \end{tikzpicture}
    }
    \caption{\small Is this Markov chain reversible?}
    \label{Fig:ReversibleSeriousEg}
\end{figure}

We would like to determine the time-reversal of this chain, but of course the time-reversal is with respect to the stationary distribution $\pi$.  So we first need to find $\pi$, which we will do via computing the expected return times $\nE_x(\tau_x^+)$.  We note that
\begin{align*}
    \nP_1(\tau_1^+ = k) &= \nP(1 \rightarrow 2 \rightarrow \cdots \rightarrow k \rightarrow 1) = (1-p)^{k-1}p,
\end{align*}
and so $\tau_1^+ | X_0=1$ is a $\text{Geo}(p)$ random variable.  Hence $\nE_1(\tau_1^+) = \frac{1}{p}$, and so $\pi(1) = p$ by \eqref{Eq:piFormula}.  To compute $\pi(j)$ for $j\neq 1$, we use our other formula \eqref{Eq:piFormula2} with $x=1$,
\begin{align*}
    \pi(j) = \frac{\nE_1 \big( \#\{ \text{visits of $j$ before $\tau_1^+$} \} \big)}{\nE_1(\tau_1^+)} = p \nE_1 \big( \#\{ \text{visits of $j$ before $\tau_1^+$}\}).
\end{align*}
Noting that the chain will visit $j$ at most once before returning to $1$, we find
\begin{align*}
    \nE_1 \big( \#\{ \text{visits of $j$ before $\tau_1^+$}\}) &= 0 \cdot \nP_1( \tau_1^+ < j) + 1 \cdot \nP_1( \tau_1^+ \geq j)\\
    &= \nP_1( \tau_1^+ \geq j)\\
    &= (1-p)^{j-1},
\end{align*}
and thus $\pi(j) = p(1-p)^{j-1}$.  As this formula also holds for $j=1$, we see $\pi$ has the pdf of a Geo$(p)$ random variable.  

We can then compute the time-reversal.  Recalling that
\begin{align*}
    \hat{P}(i,j) = P(j,i) \frac{\pi(j)}{\pi(i)},
\end{align*}
we have from \eqref{Eq:InDepthReversableProbs} that $\hat{P}(i,j) \neq 0$ iff $i = j+1$ or $i=1$. In the former case,
\begin{align}\label{Eq:InDepthReversedProb1}
    \hat{P}(j+1,j) &= P(j,j+1) \frac{\pi(j)}{\pi(j+1)} \notag\\
    &= (1-p) \frac{(1-p)^{j-1}p}{(1-p)^jp}=1.
\end{align}
In the latter case,
\begin{align}\label{Eq:InDepthReversedProb2}
    \hat{P}(1,j) &= P(j,1) \frac{\pi(j)}{\pi(1)} \notag\\
    &= p \frac{(1-p)^{j-1}p}{p}= (1-p)^{j-1}p.
\end{align}
Hence we obtain the following picture for the reversed chain: if we are at a vertex $j>1$, \eqref{Eq:InDepthReversedProb1} says that we decrease by 1 each step until we reach state 1, at which point \eqref{Eq:InDepthReversedProb2} says that we jump to a new integer $j$, where $j$ is a $\text{Geo}(p)$ random variable. See Figure \ref{Fig:ReversibleSeriousEgReversed}.  (Is this what you came up with in Exercise \ref{Ex:GuessReversal}?)  The reversed ``movie'' is thus quite different than the original, and this Markov chain is \emph{not} reversible.  
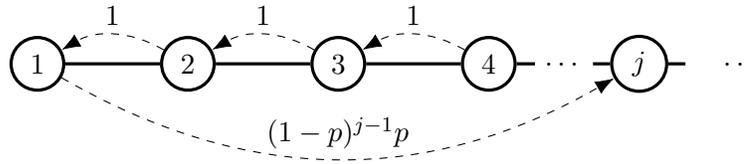
\begin{figure}
    \centering
    
    \begin{tikzpicture}
    \tikzset{
        node/.style={circle, draw=black, very thick, minimum size=7mm},
        arrow/.style={-{Latex[length=2.5mm]}, black},
        line/.style={black, very thick}
        }
        
        \node[node] (one) at (0,0) {1};
        \node[node] (two) at (2,0) {2};
        \node[node] (three) at (4,0) {3};
        \node[node] (four) at (6,0) {4};
        \node[] (four5) at (6.75,0) {};
        \node[] (j-5) at (7.25,0) {};
        \node[node] (j) at (8,0) {$j$};
        \node[] (j+5) at (8.75,0) {};
        \node[] (k) at (10,0) {};

        \draw[line] (one) -- (two);
        \draw[line] (two) -- (three);
        \draw[line] (three) -- (four);
        \draw[line] (four) -- (four5);
        \path (four5) -- node[auto=false]{\ldots} (j-5);
        \draw[line] (j-5) -- (j);
        \draw[line] (j) -- (j+5);
        \path (j+5) -- node[auto=false]{\ldots} (k);
        % \draw[black, very thick, dashed] (four) -- (j);
        % \draw[black, very thick, dashed] (j) -- (10,0);
        
        \draw[arrow, dashed] (one) to[bend right] node[anchor=south] {$(1-p)^{j-1}p$} (j);
        \draw[arrow, dashed] (two) to[bend right] node[anchor=south] {$1$} (one);
        \draw[arrow, dashed] (three) to[bend right] node[anchor=south] {$1$} (two);
        \draw[arrow, dashed] (four) to[bend right] node[anchor=south] {$1$} (three);
    \end{tikzpicture}
    
    \caption{\small The time reversal of the geometric walk.}
    \label{Fig:ReversibleSeriousEgReversed}
\end{figure}
Don't forget from Theorem \ref{Thm:SamepiForReversed}, though, that the stationary distribution for the reversed chain is still the same $\pi$ as for the forward chain, namely, $\text{Geo}(p)$.

\subsubsection{Example 2: An asymmetric exclusion on an $n$-cycle}

Consider $1 \leq k <n$ particles placed on the vertices of an $n$-cycle, where no two particles occupy the same vertex.  Order the vertices and assign each a 0 or 1 based on whether it is empty or has a particle, respectively.  The sample space $\Omega$ is then the collection of all elements of $\{0,1\}^n$ that have exactly $k$ ones, and an element $\omega \in \Omega$ is a vector of zeroes and ones.  For example, the starting position of two 6-cycles in Figure \ref{Fig:AsymmetricExclusion} is $\omega = (1,0,1,1,0,0)$.  Since we are choosing $k$ entries of an $n$-tuple to have ones, $\#\Omega = \binom{n}{k}$.

The chain proceeds by choosing a particle uniformly at random at each time and trying to move it one step CCW. It can only move if that position is unoccupied, and otherwise the move is suppressed and the particle remains in the same state. The 6-cycle on the left in Figure \ref{Fig:AsymmetricExclusion} shows a move from $(1,0,1,1,0,0)$ to $(0,0,1,1,0,1)$, while the chain on the right stays at $(1,0,1,1,0,0)$. 

In terms of the vector $\omega$, each step of the chain chooses one of the 1's uniformly at random.  If the entry immediately to the left of this 1 is a 0, the entries switch.  If it is another 1, nothing happens.  

This Markov chain is an \emph{asymmetric exclusion process} on the $n$-cycle. 

\begin{exercise}
    Without performing any calculations, informally describe the behavior of the time-reversal of this Markov chain.
\end{exercise}

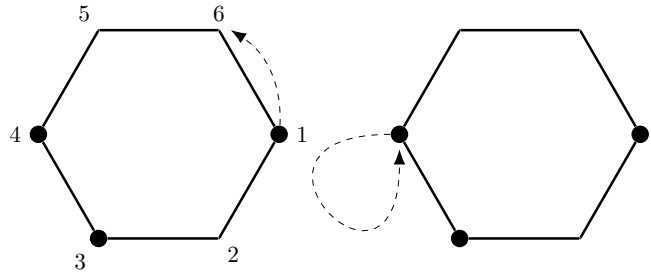
\begin{figure}
    \centering
    \scalebox{0.8}{\begin{tikzpicture}
    \tikzset{
        node/.style={circle, fill = black, thin, minimum size= 1pt, scale=0.8},
        line/.style={black, very thick},
        arrow/.style={-{Latex[length=2.5mm]}, black}
    }
    
    \node[node, scale=0.1, label=6] (one) at (1, 1.732) {};
    \node[node, label=0:1] (two) at (2,0) {};
    \node[node, scale=0.1, label=-60:2] (three) at (1,-1.732 ) {};
    \node[node, label=-120:3] (four) at (-1,-1.732) {};
    \node[node, label =-180:4] (five) at (-2,0) {};
    \node[node, scale=0.1, label=-240:5] (six) at (-1,1.732) {};
    
    \draw[line] (one) -- (two);
    \draw[line] (two) -- (three);
    \draw[line] (three) -- (four);
    \draw[line] (four) -- (five);
    \draw[line] (five) -- (six);
    \draw[line] (six) -- (one);
    \draw[arrow, dashed] (two) to[bend right] node[anchor=south] {} (1.2, 1.732);
    
    \node[node, scale=0.1] (oneb) at (7, 1.732) {};
    \node[node] (twob) at (8,0) {};
    \node[node, scale=0.1] (threeb) at (7,-1.732 ) {};
    \node[node] (fourb) at (5,-1.732) {};
    \node[node] (fiveb) at (4,0) {};
    \node[node, scale=0.1] (sixb) at (5,1.732) {};
    
    \draw[line] (oneb) -- (twob);
    \draw[line] (twob) -- (threeb);
    \draw[line] (threeb) -- (fourb);
    \draw[line] (fourb) -- (fiveb);
    \draw[line] (fiveb) -- (sixb);
    \draw[line] (sixb) -- (oneb);
     \draw[arrow, dashed] (fiveb) to[loop,out=180,in=270,distance = 3cm] node[anchor=south] {} (4,-0.25);
    \end{tikzpicture}}
    \caption{\small An asymmetric exclusion MC. A randomly selected particle tries to take one anti-clockwise step.  If the vertex is unoccupied, it moves there (left), while if occupied, it remains in the same position (right).}
    \label{Fig:AsymmetricExclusion}
\end{figure}
To formally determine the time-reversal, we first need to find the stationary distribution $\pi$.  We claim that this is uniform,
\begin{align}\label{Eq:AsymExclusionStationary}
    \pi(\omega) = \frac{1}{\binom{n}{k}}
\end{align}
for each $\omega \in \Omega$.  Indeed, by Exercise \ref{Ex:ColsAddToOne} $(a)$, it suffices to show that the column sums satisfy
\begin{align}\label{Eq:AsymExclusion1}
    \sum_{\omega' \sim \omega} P(\omega',\omega) =1
\end{align}
for any fixed $\omega \in \Omega$.  So we need to know what states $\omega'$ lead to $\omega$, and what each transition probability $P(\omega',\omega)$ is.  For example, in Figure \ref{Fig:AsymmetricExclusion}, the states that could lead to the given initial configuration $(1,0,1,1,0,0)$ are 
\begin{align*}
    (0,1,1,1,0,0), \; (1,0,1,0,1,0) \; \text{ and } \; (1,0,1,1,0,0),
\end{align*}
corresponding to a particle moving from 2 to 1, a particle moving from 5 to 4, and particle 4 being chosen and not moving, as on in the right in the figure.  Note that this is one $\omega'$ for each of the three particles.  Starting from any one of these $\omega'$, we have $P(\omega', \omega)=1/3$, since we have to pick exactly one of the particles in each case, and they are all equally likely.  Hence \eqref{Eq:AsymExclusion1} holds for the column corresponding to $\omega = (1,0,1,1,0,0)$. 

This logic readily extends.  For a general $\omega$, each of the $k$ ones corresponds to exactly one $\omega'$ such that $\omega' \sim \omega$: if the entry immediately to the right of the 1 is a 0, then swapping the 1 and 0 gives $\omega'$ (the shifted particle can move one step to yield $\omega$).  If the entry is another 1, then $\omega'=\omega$ (the adjacent particle is chosen and cannot move).  In either case, exactly one of the $k$ particles must be chosen, and since each has probability $1/k$, we again have \eqref{Eq:AsymExclusion1}.  We conclude the stationary distribution is uniform \eqref{Eq:AsymExclusionStationary}, as claimed.

For our reversal under $\pi$, we thus have
\begin{align}\label{Eq:AsymExclusion2}
    \hat{P}(\omega,\omega') = P(\omega',\omega) \frac{\pi(\omega')}{\pi(\omega)} = P(\omega',\omega).
\end{align}
What does this mean?  Consider two cases.  If $\omega'=\omega$, \eqref{Eq:AsymExclusion2} says that the probability of staying at state $\omega$ in the reversed chain is the same as in the forward chain.  If $\omega' \neq \omega$, then $P(\omega',\omega) \neq 0$ iff one particle in the $\omega'$ configuration can move one step CCW to obtain $\omega$.  When we switch the order in $\hat{P}(\omega,\omega')$, the particle is then moving clockwise.  Thus $\hat{P}(\omega,\omega') \neq 0$ iff one particle of $\omega$ can move one step \emph{clockwise} to obtain the state $\omega'$.

So time reversal in our exclusion process just switches the orientation of particle movement: we still uniformly choose a particle at each step, but now try to move clockwise instead of anti-clockwise, staying put if there is no open vertex.  This is obviously different than the original chain, and so we see the asymmetric exclusion is \emph{not} time-reversible.  

\begin{exercise}
    Use the DBEs \eqref{Eq:DBE} to give an alternate proof that this chain is not time-reversible.
\end{exercise}

%%%%%%%%%%%%%%%%%%%%%%%%%%%%%%%%%%%%%%%%
\vfill
\pagebreak
\section*{Problems for chapter \ref{Ch:Markov}}
\addcontentsline{toc}{section}{Problems for chapter 1}

\begin{problem}\label{Prob:MarkovPropertyNegate}
    Let $X$ be a Markov chain on a finite state space $\Omega$.  Then the Markov property \eqref{Eq:MarkovProperty} holds.  In this problem we show that we can replace the ``$=$'' signs in the conditioning event with ``$\neq$''.  Specifically, take $k=3$ and show that 
    \begin{align*}
        \nP(X_2=x_2 \, |\, X_0 \neq x_0, X_1 = x_1) = \nP(X_2=x_2 \, |\, X_1 = x_1).
    \end{align*}
    An extension of this idea was used in the proof of Theorem \ref{Thm:ReturnTimeLite}.
\end{problem}

\begin{problem}\label{Prob:B1.3} Consider the transition matrix 
\[
P=\begin{bmatrix} 1/3 & 1/3 & 1/3 \\ 1/2 & 0 & 1/2 \\ 1/2 & 1/2 & 0\end{bmatrix}
\]
\begin{enumerate}[$(a)$]
\item Draw a graphical representation for how this chain moves (see Figures \ref{Fig:Lect2Graph} and \ref{Fig:TrianglesInLecture6} for inspiration if needed).
\item What do the rows of $P$ sum to?  What do the columns of $P$ sum to?
\item Find the characteristic polynomial of $P$. 
\item Find the eigenvalues of $P$ and a basis for each eigenspace. What do you notice about the eigenvalues of $P$? 
\item Is $P$ diagonalizable? If so, diagonalize it as $QDQ^{-1}$ (that is, find all three of these matrices and verify that their product $QDQ^{-1}$ gives $P$). If not, explain why not.
\end{enumerate}
\end{problem}

\begin{problem}\label{Prob:B1.3b} Consider the transition matrix 
\[
P=\begin{bmatrix} 0 & 1 & 0 \\ 1/3 & 1/3 & 1/3 \\ 0 & 1/2 & 1/2\end{bmatrix}
\]
\begin{enumerate}[$(a)$]
\item Draw a graphical representation for how this chain moves (see Figures \ref{Fig:Lect2Graph} and \ref{Fig:TrianglesInLecture6} for inspiration if needed).
\item What do the rows of $P$ sum to?  What do the columns of $P$ sum to?
\item Find all the (right) eigenvalues for $P$. (Even though we care primarily about left-eigenvalues and vectors, here find the usual \emph{right} ones from linear algebra.)  What do you observe about the eigenvalues?
\item Show that the right-eigenvalues and left-eigenvalues for any square matrix are always the same.  (\emph{Hint}: think properties of the determinant.) 
\item Show that for our $P$, the right eigenspace for $\lambda=1$ and the left eigenspace for $\lambda=1$ are \emph{not} the same.
\item Find a stationary distribution $\pi$ for $P$.
\item Compute $P^{50}$.  What do you observe?
\end{enumerate}
\end{problem}

\begin{problem}\label{Prob:B1.4} Consider the following simple model for a Seattle weather forecast. The weather can be either `rain' or `no rain'. Suppose that the chance of rain tomorrow only depends on whether it rains today or not, and not on past weather conditions. If it rains today, the chance it will be rain tomorrow again is $\alpha \in (0,1)$; and if it does not rain today, the probability it will rain tomorrow is $\beta  \in (0,1)$. Starting from some day zero, let $X_0, X_1, X_2, \ldots$ be a sequence where 
\[
X_k=\begin{cases}
1, & \text{if it rains on day $k$}\\
0, & \text{otherwise}.
\end{cases}
\]
\begin{enumerate}[$(a)$]
\item Find the state space and the transition probability matrix of the Markov chain $\left(X_k,\; k=0,1,2,\ldots \right)$.
\item If it rains on day zero, find the probability it will rain on day $3$. 
\item If there is only $50\%$ chance that it rains on day zero, find the probability it will rain on day $3$. 
\item If it rains on day zero, what is the probability that it is going to rain nonstop from days 1 through 7? 
\end{enumerate}
\end{problem}

\begin{problem}\label{Prob:AbsorbingBarriers1} Consider the so-called random walk with \textit{absorbing} barriers on $\Omega=\{ 1,2,3,4,5\}$. This is also a Markov chain on $\Omega=\{ 1,2,3,4,5\}$ with the following transition probability matrix 
\[
P=\begin{bmatrix}
1 & 0 & 0 & 0 & 0\\
\frac{1}{2} & 0 & \frac{1}{2} & 0 & 0\\
0 & \frac{1}{2} & 0 & \frac{1}{2}& 0 \\
0 & 0 & \frac{1}{2} & 0 & \frac{1}{2} \\
0 & 0 & 0 & 0 & 1
\end{bmatrix}.
\]

\begin{enumerate}[$(a)$]
\item Draw a graphical diagram that shows how this chain moves. 
\item Is this chain irreducible? Why or why not?  
\item Find $\nE(X_5 \mid X_3=2, X_0=3)$. 
\item Find a row vector $\pi$ such that $\pi$ has every coordinate nonnegative and the sum of coordinates is one and $\pi P =\pi$. Is $\pi$ unique?
\item Use a computer to compute $P^{50}$, and round to the nearest hundredth.  What do you observe?  Explain what this intuitively means.
\end{enumerate}
\end{problem}

\begin{problem}\label{Prob:ReflectingBoundary} Consider the so-called random walk with \textit{reflecting boundary} on $\Omega=\{ 1,2,3,4,5\}$. This is a Markov chain with the following transition probability matrix 
\[
P=\begin{bmatrix}
0 & 1 & 0 & 0 & 0\\
\frac{1}{2} & 0 & \frac{1}{2} & 0 & 0\\
0 & \frac{1}{2} & 0 & \frac{1}{2}& 0 \\
0 & 0 & \frac{1}{2} & 0 & \frac{1}{2} \\
0 & 0 & 0 & 1 & 0
\end{bmatrix}.
\]

\begin{enumerate}[$(a)$]
\item Draw a graphical diagram that shows how this chain moves. 
\item Is this chain irreducible? Why or why not? 
\item If this chain is irreducible, what is its period?  
\item Find $\nE(X_2 \mid X_0=1)$. 
\item Find a stationary distribution for this Markov chain.
\item Estimate $\nP(X_{2200}=2 \,|\, X_{1200}=4)$ and $\nP(X_{2201}=2 \,|\, X_{1200}=4)$.
\end{enumerate}
\medskip
\end{problem}

\begin{problem}\label{Prob:B1.6} Consider the random walk on the graph in Figure \ref{Fig:B1.6}. 

\begin{figure}[ht]
    \centering
    \scalebox{0.85}{
        \begin{tikzpicture}
    \tikzset{
        node/.style={circle, draw=black, very thick, minimum size=7mm},
        line/.style={black, very thick}
    }
    
    \node[node] (four) at (0,0) {4};
    \node[node] (five) at (2,0) {5};
    \node[node] (two) at (2,-2) {2};
    \node[node] (three) at (0, -2) {3};
    \node[node] (one) at (4, -2) {1};
    \node[node] (six) at (-2, -2) {6};
    
    \draw[line] (four) -- (five);
    \draw[line] (four) -- (three);
    \draw[line] (five) -- (two);
    \draw[line] (three) -- (two);
    \draw[line] (one) -- (two);
    \draw[line] (one) -- (five);
    \draw[line] (six) -- (four);  
    
    \end{tikzpicture}}
    \caption{ }
    \label{Fig:B1.6}
\end{figure}

\begin{enumerate}[$(a)$]
\item Write down the transition matrix $P$ of this Markov chain.
\item Is the random walk irreducible? Why or why not? 
\item What is the period of this chain?  Is it aperiodic?
\item Suppose $X_0=6$. What is
\[
\nP(X_4=1, X_3\neq 1, X_2 \neq 1, X_1 \neq 1 \mid X_0=6)? 
\]
That is, what is the probability that, starting from $6$, you reach $1$ for the first time in $4$ time steps?
\item Use a computer to compute $P^{50}$.  What do you observe?
\end{enumerate}
\end{problem}

%Levinson 2

\begin{problem}\label{Prob:A1.3} Consider the $3\times 3$ matrix 
\[
A=\begin{pmatrix}
1/3 & 0 & 2/3\\
1/3 & 2/3 & 0\\
1/3 & 1/3 & 1/3
\end{pmatrix}
\]
\begin{enumerate}[$(a)$]
\item Draw a diagram for the Markov chain whose transition probabilities are given by $A$.
\item Find the characteristic equation for this matrix. 
\item Find all the eigenvalues of this matrix.  What do you notice about their values? (Even though we care primarily about left-eigenvalues and vectors, here find the \emph{right}-eigenvalues, as usual in linear algebra.  Can you show that the right- and left-eigenvalues are always the same?)
\item Find a solution to the linear system of equations $\pi A=\pi$ where $\pi=(\pi_1, \pi_2, \pi_3)$ is a vector in $\rr^3$ satisfying the constraint $\pi_1+\pi_2+\pi_3=1$. This is the stationary distribution for the Markov chain with transition matrix $A$.  
\item Use a computer to compute $A^{50}$.  What do you observe?
\end{enumerate}
\end{problem}

\begin{problem}\label{Prob:A1.5}  Modify the random walk from Problem \ref{Prob:ReflectingBoundary} by considering the so-called \textit{lazy} random walk with reflecting barriers on $\Omega=\{ 1,2,3,4,5\}$. This is a Markov chain with the same state space and a new transition matrix 
\[
\widetilde{P}= \frac{1}{2}\left( I + P\right),
\]
where $I$ represents the $5\times 5$ identity matrix. Answer parts $(a)$ through $(f)$ of Problem \ref{Prob:ReflectingBoundary} for this new Markov chain. 
\end{problem}
\medskip

\begin{problem}\label{Prob:A1.6} Consider a Markov chain with state space $\Omega=\{1,2,3,4,5\}$ and the following transition probability matrix 
\[
P=\begin{bmatrix}
\frac{1}{3} & \frac{2}{3} & 0 & 0 & 0\\
\frac{1}{3} & 0 & \frac{2}{3} & 0 & 0\\
\frac{1}{3} & 0 & 0 & \frac{2}{3}& 0 \\
\frac{1}{3} & 0 & 0 & 0 & \frac{2}{3} \\
1 & 0 & 0 & 0 & 0
\end{bmatrix}.
\]
That is, at every state $\{1,2,3,4\}$, the chain either goes up by one with probability $2/3$ or returns to 1  with probability $1/3$. Once it reaches $5$, it jumps to 1, and continues again. 

\begin{enumerate}[$(a)$]
\item Is this chain irreducible? Why or why not? 
\item If this chain is irreducible, what is its period?  
\item Find a stationary distribution $\pi$ for this chain.  
\item If you start from state $1$, find the exact distribution of the first return to $1$.
\item Compute $\nE(\tau_1^+ \, | \, X_0=1)$.
\item Use a computer to compute $P^{50}$.  What do you observe?  Explain what this intuitively means.
\end{enumerate}
\end{problem}

%%This concludes HW B1 and A1

%midterm 2 in 395
\begin{problem}\label{Prob:B2.1} Consider a Markov chain with state space $\Omega=\{0,1,2,3,4\}$ and the following transition probability matrix 
\[
P=\begin{bmatrix}
0 & \frac{1}{3} & \frac{1}{3} & \frac{1}{3} & 0\\
\frac{1}{3} & 0 & \frac{1}{3} & 0 & \frac{1}{3}\\
\frac{1}{2} & \frac{1}{2} & 0 & 0 & 0 \\
\frac{1}{2} & 0 & 0 & 0 & \frac{1}{2} \\
0 & \frac{1}{2} & 0 & \frac{1}{2} & 0
\end{bmatrix}.
\]
\begin{enumerate}[$(a)$]
\item Is this chain irreducible? Why or why not? 
\item If this chain is irreducible, what is its period?  
\item Find a stationary distribution $\pi$ for this chain.  
\item If you start from state $0$, find the expected hitting time to the states $\{1,2,4\}$.
\item If you start from state $0$, find the expected number of times you will visit $1$ before you hit the states $\{2,4\}$. Leave the answer as an indicated coordinate of a matrix inverse. That is, say something like ``the answer is the $(1,2)$ coordinate of $M^{-1}$'' where $M$ is written explicitly. You don't need to invert $M$. 
\item Find the expected return time to zero $\nE_0\left( \tau_0^+\right)$.
\item Use a computer to compute $P^{50}$.  What do you observe?  Explain what this intuitively means.
\end{enumerate}

\end{problem}

%HW5. Problem 9.

\begin{problem} A community of $N$ individuals is in the middle of a zombie outbreak. Fortunately, the community has managed to discover a serum that can cure zombies and prevents them from further infection. Suppose at the beginning of each day, every individual is in one of three possible conditions: non-infected, zombie, and cured. If during day $t$, a non-infected person becomes infected, he or she will become a zombie the next day $(t+1)$, but will get cured from the following day $(t+2)$ on-wards. Let $X_t$ and $Y_t$ denote the number of zombies and the number of non-infected persons on day $t$, respectively. During each day, the probability that a given non-infected person comes in contact with a given zombie is $0< p< 1$, independently for every zombie and every other non-infected person.

\begin{enumerate}[$(a)$]
\item If $X_t=i$, what is the probability that a given non-infected person will come in contact with a zombie during day $t$? 
\item Is the pair $\left( X_t, Y_t\right)$, $t=0,1,2,\ldots$, a Markov chain? If so, give an expression for its state space and transition probabilities. 
\item Suppose $X_0=1, Y_0=N-1$, find the distribution of $X_2$. Leave your answer in summation notation.  
\end{enumerate}

\nin(\emph{Remark:} This is the \emph{Reed-Frost model}, a simple epidemiological model for the spread of disease.  It is part of a larger class of \emph{SIR} models, which track the number of \emph{susceptible}, \emph{infected}, and \emph{recovered} individuals.  For an interactive app showing disease spread in this model, made by Jamie Forschmiedt, visit \url{https://forscj.shinyapps.io/ReedFrostModel/}.)
\end{problem}

%Levinson 2

\begin{problem} A math professor possesses $r$ umbrellas that he uses in going between his home and his office. If he is at his home at the beginning of the day and it is raining, then he will take an umbrella with him to his office, provided there is one at home to be taken. On his way back from his office, he will bring back an umbrella if it is raining and there is (at least) one umbrella at his office. If it is not raining, the professor does not use an umbrella. Assume that it rains at the beginning (or at the end) of each day with probability $1/2$, independently of the past. Let $X_{n}$ be the number of umbrellas at home at the beginning of the day $n=1,2,\ldots$. 

\begin{enumerate}[$(a)$]
\item Is $X_n$ a Markov chain? If so, find its state space and transition probabilities. 
\item Is this chain irreducible? Aperiodic ?  
\item Find a stationary distribution for this Markov chain for $r=3$. 
\item Suppose $r=3$. If the professor finds one day that there are no umbrellas left at home, what is the expected number of days after which he will find himself in a similar situation? 
\end{enumerate}
\end{problem}

\begin{problem} Consider the biased random walk $X$ on the $6$-cycle where at each step the random walker turns one step CCW with probability $\frac{1}{4}$ and turns onfe step CW with probability $\frac{3}{4}$.  
\begin{enumerate}[$(a)$]
\item Find the stationary distribution $\pi$ for this Markov chain. 
\item Describe the transition probabilities of the time-reversal of $X$ under $\pi$.
\item Suppose $X_0 \sim \pi$. What is the probability $\nP_\pi(X_{n-2}=1 \mid X_n=1)$? 
\end{enumerate}

\end{problem}

\begin{problem}\label{Prob:B2.6} Consider the complete bipartite graph. This is a graph with two kinds of vertices $n$ red vertices, $R_1, \ldots, R_n$ and $m$ blue vertices $B_1, \ldots, B_m$. There is an edge for every pair of red and blue vertices $\{R_i, B_j\}$, $1\le i \le n$ and $1\le j \le m$, and so a total of $mn$ edges. Consider the simple symmetric random walk on this graph. 

\begin{enumerate}[$(a)$]
\item What is the stationary distribution of this chain? 
\item Let $\tau_1$ denote the hitting time of vertex $R_1$. Starting from vertex $R_2$, what is the probability mass function and expectation of $\tau_1$? 
\item Let $\tau_1^+$ denote the return time to vertex $R_1$. Starting from vertex $R_1$, what is the  probability mass function and expectation of $\tau_1^+$?
\end{enumerate}
\end{problem}

%%%%% HW set 4

%midterm 2 in 395
\begin{problem}\label{Prob:A2.1} Let $K_n$ denote the complete graph on $n$ vertices. That is, the vertex set is $V=\{1,2,\ldots, n\}$ and any pair of possible edges $\{i,j\}$ for $i \neq j$ is present (in particular, note there are no loops). Consider the random walk on this graph. 
\begin{enumerate}[$(a)$]
\item What is the stationary distribution of this Markov chain? 
\item Let $\tau_1$ denote the hitting time of state $1$. What is the the distribution of $\tau_1$ if we start the chain at $k\neq 1$?  What is $\nE_k(\tau_1)$? 
\item Let $\tau$ denote the hitting time of the set $\{1,2\}$. Find $P_k(X_\tau=1)$ and $\nE_k(\tau)$ for $k=3,4,\ldots, n$. 
\item The \textbf{cover time} of a Markov chain is the time it takes the chain to visit every state. Find the expected cover time for this random walk, starting from any state. 
\end{enumerate}
\end{problem}

%HW5. Problem 9.

\begin{problem} Consider the following Markov chain on all permutations of the numbers $\{1,2,3,4\}$. There are $4!=24$ such permutations. For any given permutation, randomly choose two positions and switch the two numbers appearing in those positions. As an example, suppose the current state of the Markov chain is $1243$. Choose randomly a pair, say $1$ and $3$. After switching those two we obtain $3241$. Say, now we randomly pick $2$ and $1$. After switching, we obtain $3142$. And so on. 

\begin{enumerate}[(a)] 
\item Assume that this Markov chain is irreducible. Show that it is reversible with respect to the uniform distribution on the $24$ permutations.
\item Call rank of $4$ to be the position from the left where $4$ appears in the permutation. That is, if the permutation is $4321$, the rank of $4$ is $1$. If the permutation is $1432$, the rank of $4$ is $2$, and so on. Show that the rank of $4$ in the Markov chain of permutations in part (a) is itself a Markov chain with state space $\{1,2,3,4\}$. Find its transition probabilities and stationary distribution.  
\item Suppose we start from the permutation $1243$. Let $\tau$ be the stopping time when $4$ is either the leftmost number or the rightmost number in the permutation (say $4321$ or $4123$ or $1234$ etc.). That is, $\tau$ is the first time the rank of $4$ is either $1$ or $4$. 
What is $\nE_{1243}(\tau)$? Hint: use linear algebra. 
\item Starting at $1243$ again, what is the probability that at $\tau$, $4$ is in the rightmost position (i.e., the rank of $4$ is $4$)? Hint: use recursion.  
\end{enumerate}
\end{problem}

%Levinson 2

\begin{problem}\label{Prob:5CycleHittingTimes} Consider the random walk on the $5$-cycle. That is $\Omega=\{1,2,3,4, 5\}$ and the transition probability matrix is
\[
P=\begin{bmatrix}
0 & \frac{1}{2} & 0 & 0 & \frac{1}{2}\\
\frac{1}{2} & 0 & \frac{1}{2} & 0 & 0\\
0 & \frac{1}{2} & 0 & \frac{1}{2}& 0 \\
0 & 0 & \frac{1}{2} & 0 & \frac{1}{2} \\
\frac{1}{2} & 0 & 0 & \frac{1}{2} & 0
\end{bmatrix}.
\]

\begin{enumerate}[$(a)$]
\item Let $\tau_1$ denote the hitting time of the state $1$. Let $f_k=\nE_k(\tau_1)$. Note that $f_1=0$. Write down a recursive equation for $f_k, k=1,2,\ldots, 5$.
\item Solve the above recursion to get a formula for $f_k$. 
\item Now let $\tau$ denote the hitting time of $\{3,4\}$. Find $\nP_1(X_\tau=3)$ and $\nE_1(\tau)$ by using known formulas about simple symmetric random walks on the integers.
\end{enumerate}
\end{problem}

\begin{problem} Consider a Markov chain with state space $\Omega=\{1,2,3,4\}$ and the following transition matrix 
\[
P=\begin{bmatrix}
0 & \frac{1}{2} & \frac{1}{2} & 0 \\
\frac{1}{2} & 0 & \frac{1}{2} & 0 \\
\frac{1}{3} & \frac{1}{3} & 0 & \frac{1}{3}\\
0 & 0 & 1 & 0 
\end{bmatrix}.
\]
\begin{enumerate}[$(a)$]
\item Is this chain irreducible? Why or why not? 
\item Is this chain aperiodic? Why or why not? 
\item If the chain is irreducible find a stationary distribution for it. 
\item Find $\nE_4(\tau_{1})$, where $\tau_{1}$ is the first hitting time of $1$. 
\item Find the expected number of times the Markov chain will visit $3$ before $\tau_{12}$, starting from $4$. That is, compute
\[
\nE_4\left[ \sum_{i=0}^{\tau_{12}-1} 1\{X_i=3\} \right].
\]
\item Find the time reversal transition matrix $\hat{P}$ for this Markov chain. 
\item For the time-reversed Markov chain, find $\hat{\nE}_1(\tau_{4})$, where $\tau_{4}$ is the first hitting time of $4$.
\end{enumerate}

\end{problem}

\begin{problem}\label{Prob:A2.6} Consider $\Omega$ to be the set of all sequences of $+1$ or $-1$ of length $2N$ such that the sum of the coordinates is exactly zero. In other words, there are exactly $N$ many $+1$s and $N$ many $-1$s in each such sequence. If $N=2$, one such sequence is $u=(+1,+1, -1,-1)$, while another one is $v=(+1, -1, +1, -1)$.

\begin{enumerate}[$(a)$]
\item How many elements does $\Omega$ have? 
\item Consider a Markov chain $\left\{S_0, S_1, S_2, \ldots \right\}$ with state space $\Omega$ that proceeds as follows. At each step, pick a $+1$ at random, and independently pick a $-1$ at random and swap their positions. For example, $u$ can change to $v$ by switching the second and third coordinates. Is this Markov chain irreducible? 
\item What is the stationary distribution of this Markov chain?  
\item Let $X_n$ denote the number of $+1$s in the first $N$ coordinates of $S_n$. For example, this number for $u$ is $2$ and for $v$ is 1. Then $\{X_0, X_1, \ldots\}$ is also a Markov chain. Find its state space and transition probabilities.
\item Find the stationary distribution of this chain. 
\item Are both $X$ and $S$ time reversible? Why or why not? 
\end{enumerate}
\end{problem}

%b3.1
\begin{problem}\label{Prob:B3.1} (\emph{Random walk bridge})  Consider an urn with $N$-many $+1$'s and $N$-many $-1$'s. Randomly sample, without replacement, each of these $2N$ numbers one by one. This gives us a random sequence $w=(w_1, w_2, \ldots, w_{2N})$ of exactly $N$-many $+1$'s and $N$-many $-1$'s. 

Consider a stochastic process by declaring $S_0=0$ and then successively adding coordinates of $w$, i.e.,
\[
S_k= w_1 + w_2 + \ldots + w_k, \quad k=1,2,\cdots, 2N. 
\]
This process is called the \textbf{random walk bridge} because of the property in the first part below.
\begin{enumerate}[$(a)$]
\item Show that $S_{2N}=0$. 
\item Find the probability mass function and expectation of each $S_k$.
\item Show that the bridge has the following time-reversal symmetry: 
\[
\left( S_{2N}, S_{2N-1}, \ldots, S_1, S_0 \right)
\]
is again distributed as a random walk bridge. \emph{Hint:} Time-reverse the random sequence $w$. 
\end{enumerate}

\end{problem}

%Levinson 2 %B4.3
\begin{problem}\label{Prob:B4.3} Consider the so-called \textit{top to random} shuffling. Suppose you have a pack of $10$ cards labeled $\{0,1,2,\ldots, 9\}$ arranged in a permutation from left to right. There are $10!$ such arrangements. For example, arrangement $5038761249$ means the leftmost card is $5$, next comes card $0$, and so on, the rightmost two cards are $4$ and $9$. Consider the following Markov chain where at each step you pick the leftmost card, choose a number $I$ uniformly between $1$ and $10$, and insert the card at position $I$ from the left. For the example arrangement, you'd pick $5$ and pick a uniformly distributed number between $1$ and $10$, say $3$. Then insert $5$ as the third card from the left and get the new arrangement $0358761249$. The rank of card numbered $0$ is the position from the left where it appears. In the example shown, the rank of $0$ is initially $2$ (it appears as the second card from the left), and then becomes $1$ after card $5$ jumps. 
\begin{enumerate}[$(a)$] 
\item Show that the rank of card $0$ on $\Omega=\{1,2,\ldots, 10\}$ is a Markov chain by describing its transition probabilities. 
\item Find the stationary distribution for the Markov chain in $(a)$.
\item Describe the time-reversal of this Markov chain in $(a)$ by writing down the explicit transition probabilities.   
\end{enumerate}
\end{problem}

%Problem 18 takis %B4.4
\begin{problem}\label{Prob:B4.4} Consider the following Markov chain on $\Omega=\{1,2,\ldots, n\}$. If you are currently at $i$, for $1\le i\le n-1$, then you sample, uniformly at random, one of the numbers from $\{i+1, \ldots, n\}$ and jump there. If you are at $n$, you move down to $1$ with probability one. 
\begin{enumerate}[$(a)$]
\item Find $\nE_{1}(\tau_n)$, where $\tau_n$ is the hitting time of state $n$. 
\item Find the stationary distribution of this Markov chain. 
\item Starting from $X_0=1$, what is the probability that you will visit $i$ before hitting $n$ for the first time? That is, what is the probability that the number of visits of $i$ before $\tau_n$ is positive? Express as a formula in $2\le i\le n-1$. 
\end{enumerate}
\end{problem}

\begin{problem}\label{Prob:B3.2} Suppose the weather on any day depends on the weather conditions of the previous two days. More precisely, suppose the weather can only be sunny (S) or cloudy (C). If it was sunny today and yesterday, it will be sunny tomorrow with probability $0.8$. If it is sunny today and cloudy yesterday, it will be sunny tomorrow with probability $0.6$. If it is cloudy today but sunny yesterday, it will be sunny tomorrow with probability $0.4$, and if it cloudy for the last two days, it will be sunny tomorrow with probability $0.1$.  

Such a model can be turned to a Markov chain whose current state is the weather both today and yesterday. Consider $\Omega=\left\{ (S,C), (S,S), (C, S), (C,C) \right\}$, where $(S,C)$ means sunny yesterday but cloudy today.

\begin{enumerate}[$(a)$]
\item Find the transition probability matrix of this Markov chain?
\item What is the stationary distribution of this Markov chain? 
\end{enumerate}
\end{problem}

%Levinson 2
\begin{problem}\label{Prob:B3.3} Let $A$ and $B$ be two walkers independently doing the simple symmetric random walk on the $5$-cycle $\Omega=\{1,2,3,4,5\}$.  Consider the graph distance between $A$ and $B$, i.e., the length of the shortest path joining the two. For example, if $A$ is at $1$ and $B$ is at $5$, the graph distance is $1$ and not $4$. In particular, the graph distance can only take values in $\{0,1,2\}$.
\begin{enumerate}[$(a)$]
\item Let $X_t$ denote the graph distance between the two walkers at time $t$. Show that $X_t$ is a Markov chain by describing its transition probabilities. 
\item Suppose walker $A$ starts at $2$ while walker $B$ starts at $5$. How many steps will it take for them to be at the same vertex, on average? 
\item If $A$ and $B$ both start at $2$, how many times, on average, will they be one unit apart before they reach two units apart? 
\end{enumerate}
\end{problem}

\begin{problem}\label{Prob:1Sociologist}
    A sociologist is studying how socio-economic class changes from generation to generation in a certain country, and assumes that the class of the children depends only upon the class of the parents.  The sociologist uses a simple model of three classes $\{$lower, middle, higher$\}$ and comes up with the transition matrix
    \begin{align*}
        P = \begin{bmatrix}
            0.7 & 0.2 & 0.1 \\ 
            0.1 & 0.7 & 0.2\\
            0.05 & 0.15 & 0.8
        \end{bmatrix}
    \end{align*}
    For example, entry $P_{13} = 0.1$ means that the probability the children from a lower-class family enter the upper class is 0.1 (for simplicity we assume all the children are in the same class).  Round all numerical results to nearest thousandth, and feel free to use a computer.
    \begin{enumerate}[$(a)$]
        \item Is this chain irreducible?  Aperiodic?
        \item Find a stationary distribution $\pi$.
        \item Use a computer to compute $P^{50}$.  What do you observe?
        \item For a couple in the lower class, determine the expected number of generations for their descendants to enter the highest class.
        \item For middle-class parents, what is the expected number of generations before their descendants are in the middle class again?
        \item Is this chain reversible?  Explain what your conclusion intuitively means about the classes and the economy in this country.
        \item A couple comes to the sociologist and asks for the probability that their distant descendants, hundreds of years in the future, will be in the lower class.  What should the sociologist say?
    \end{enumerate}
\end{problem}

\begin{problem}
    (\emph{Random walk with geometric waiting times}) Start with a simple symmetric random walk $(X_n)_{n \geq 0}$ on the vertices of a graph $G = (V,E)$, and let $p:V \rightarrow [0,1]$ be a collection of probabilities.  That is, there is one probability $p_x$ assigned to each vertex $x$.  Now modify the chain as follows: if we arrive at a given vertex $x$ at step $k$, we sample an independent Geo$(p_x)$ random variable $T \in \{1,2,\ldots\}$, and we wait at vertex $x$ for an additional $T-1$ units of time before jumping to the next vertex.  That is, $X_k = X_{k+1} = \cdots = X_{k+T-1} =x$, and $X_{k+T}$ is one of the adjacent vertices $y \sim x$.  For example, if we arrive at $x$ at step $k$ and $T = 3$, then 
    \begin{align*}
        X_k = x, \quad X_{k+1} = x, \quad X_{k+2}=x, \quad X_{k+3}=y,
    \end{align*}
    where $y \sim x$ is a randomly-chosen neighboring vertex.  Is this new process $Y=(Y_k)_{k\geq 0}$ a Markov chain?  Explain.
\end{problem}

\vfill
\pagebreak

\chapter{Classical models}\label{Chapter:ClassicalModels}

\section{Random walks on $\mathbb{Z}$ and gambler's ruin}\label{Sec:RWZ}

An important class of Markov chains are the random walks on the integers $\mathbb{Z}$ or subsets $\{a,a+1,a+2,\ldots, a+n\}$ of the integers.  In this section we look at several such chains and establish some famous formulas describing their behavior.  Section \ref{Sec:GamblersRuin} looks at the probabilities for hitting either end of a subset $\{a,a+1,a+2,\ldots, a+n\} \subset \mathbb{Z}$, while Section \ref{Sec:GamblersRuinTime} looks at the average hitting times for the endpoints.  We conclude in Section \ref{Sec:SRWZBM} by looking at walks on all of $\mathbb{Z}$ and their large-scale behavior, where we observe the emergence of \emph{Brownian motion}, a fundamental stochastic process.

\subsection{Boundary hitting probabilities}\label{Sec:GamblersRuin}

Consider the simple walk on $\Omega = \{0,1,2,\ldots, n\}$ with \textbf{absorbing boundaries}, by which we mean that once the chain reaches either $0$ or $n$, it is permanently stuck there (recall this idea already appeared in Problem \ref{Prob:AbsorbingBarriers1}). We call $\{0,n\}$ the \textbf{boundary} of $\Omega$ and denote it by 
\begin{align}\label{Eq:Boundary}
    \partial \Omega = \{0,n\}.
\end{align}
The transition probabilities for the absorbing walk are
\begin{align}
    P(j,j-1) = P(j,j+1) &= \frac{1}{2}, \qquad 1 \leq j \leq n-1, \notag\\
    P(0,0)=P(n,n) &= 1. \label{Eq:AbsorbingBoundary}
\end{align}
Now, suppose we start with $1 \leq X_0 = k \leq n-1$.  Let $\tau = \tau_{\partial \Omega}$ be the hitting time to the boundary of $\Omega$.  Then either $X_\tau =0$ or $X_\tau = n$.  What is $\nP_k(X_\tau = n)$? It is clear that $\nP_0(X_\tau=n)=0$, since we cannot move away from zero if we begin there.  It is likewise clear that $\nP_n(X_\tau =n) =1$.  It also seems that we should have
\begin{align*}
    0 < \nP_1(X_\tau = n) < \nP_2(X_\tau = n) < \cdots < \nP_{n-1}(X_\tau=n) < 1, 
\end{align*}
since the closer we begin to $n$, the more likely it should be that we hit $\partial \Omega$ at $n$ rather than 0. 

Can we say anything more?  How much do the probabilities increase each step?  Is there an explicit formula?  The famous \emph{gambler's ruin} theorem answers these questions and says that the probabilities are, in fact, a \emph{linear} function in the starting position $k$.

\begin{theorem}[Gambler's ruin]\label{Thm:GamblersRuin}
    For the simple symmetric random walk with absorbing boundaries on $\Omega = \{0,1,2,\ldots, n\}$,
    \begin{align}
        \nP_k(X_\tau = n) &= \frac{k}{n} \label{Eq:GamblersRuinN}\\
        \nP_k(X_\tau = 0) &= 1 - \frac{k}{n}\label{Eq:GamblersRuin0}
    \end{align}
    for any $k \in \{0,1,2,\ldots, n\}$.
\end{theorem}

One way to think about \eqref{Eq:GamblersRuinN} is that the probability of terminating at the right endpoint is the ratio of the distance to the \emph{left} endpoint to the entire interval length, $\nP_k(X_\tau = n) = \frac{k-0}{n-0}$.  That is, we take the ratio of how far away the ``bad'' endpoint is to the length of the entire playing field.  This, of course, makes a lot of sense: the further away the bad point is, the more likely we are to end at the ``good'' boundary point.  The surprising thing about Theorem \ref{Thm:GamblersRuin} is that it says the probability is exactly this ratio, and not some complicated function of it.  

The proof exploits a very deep relationship between Markov chains and differential equations.  

\begin{proof}
    Set $p(k) := \nP_k(X_\tau = n)$.  We have that $p(0)=0$ and $p(n)=1$, and we need to show that $p$ is linear with slope $\frac{1}{n}$.  We first claim that $p$ satisfies the \emph{mean-value property}
    \begin{align}\label{Eq:pkMVP}
        p(k) = \frac{1}{2}p(k-1) + \frac{1}{2}p(k+1)
    \end{align}
    for each $1 \leq k \leq n-1$.  That is, we claim that each value of $p$ (outside of $\partial \Omega$) is just the average of its two neighbors.  We see this from conditioning on our first step.  For fixed $1 \leq k \leq n-1$, $\tau \geq 1$, and we have
    \begin{align}
        p(k) &= \nP(X_\tau =n \, | \, X_0 =k) \notag\\
        &= \nP(X_\tau =n, X_1 = k-1 \, | \, X_0 =k) + \nP(X_\tau =n, X_1 = k+1 \, | \, X_0 =k) \notag\\
        &= \nP(X_\tau =n \, | \, X_1 = k-1,  X_0 =k)\nP(X_1=k-1 \, | \, X_0 = k) \notag\\
        &\hspace{15mm} +\nP(X_\tau =n \, | \, X_1 = k+1,  X_0 =k)\nP(X_1=k+1 \, | \, X_0 = k) \notag\\
        &=  \nP(X_\tau =n \, | \, X_1 = k-1) \frac{1}{2} + \nP(X_\tau =n \, | \, X_1 = k+1)\frac{1}{2} \label{Eq:DeriveMVP1}\\
        &=  \nP(X_\tau =n \, | \, X_0 = k-1) \frac{1}{2} + \nP(X_\tau =n \, | \, X_0 = k+1)\frac{1}{2} \label{Eq:DeriveMVP2}\\
        &= \frac{1}{2}p(k-1) + \frac{1}{2}p(k+1),\notag
    \end{align}
    as claimed.  Here we have used the Markov property twice, first in \eqref{Eq:DeriveMVP1} to ignore the earlier information on $X_0$, and then again in \eqref{Eq:DeriveMVP2} to say that the chain ``starts over'' after the first step.  We remark that this simple idea of conditioning on the first step, called \textbf{the method of recursion}, is often very useful.
    
    Next, define the increment $\Delta_k := p(k) - p(k-1)$, $1 \leq k \leq n$.  If $p$ really is linear, as claimed, then the increment should be constant, independent of $k$.  We note subtracting $\frac{1}{2}p(k)$ and $\frac{1}{2}p(k-1)$ from both sides of \eqref{Eq:pkMVP} yields
    \begin{align*}
         \frac{1}{2}(p(k)-p(k-1)) = \frac{1}{2}(p(k+1)-p(k)),
    \end{align*}
    or $\Delta_k = \Delta_{k+1}$, showing that all adjacent $\Delta$'s are equal.  Hence all the $\Delta$'s are equal to $\Delta_1$, say.  Note that we thus have the telescopic sum
    \begin{align*}
        n\Delta_1 &= \Delta_1 + \Delta_2 + \Delta_3 + \cdots + \Delta_n\\
        &= p(1)-p(0) + p(2)-p(1) + p(3)- p(2) + \cdots + p(n)-p(n-1)\\
        &=p(n)-p(0) = 1-0,
    \end{align*}
    and so $\Delta_1= \frac{1}{n} = \Delta_k$ for each $k$.  Therefore, $p$ is linear, as claimed, with slope $\frac{1}{n}$, and
    \begin{align*}
        p(k) &= p(0) + \Delta_1 + \Delta_2 + \cdots + \Delta_k = \frac{k}{n},
    \end{align*}
    yielding \eqref{Eq:GamblersRuinN}.  The other formula \eqref{Eq:GamblersRuin0} immediately follows, as 
    \begin{equation*}
        \nP_k(X_\tau = 0) = 1 - \nP_k(X_\tau = n).\qedhere
    \end{equation*}
\end{proof}

Sometimes we want to consider different behaviors once we reach $\partial \Omega$, and it is clear that Theorem \ref{Thm:GamblersRuin} will still hold, since the walk does not visit the boundary until the hitting time.  In particular, \eqref{Eq:GamblersRuinN} and \eqref{Eq:GamblersRuin0} hold for the simple symmetric random walk on $\Omega=\{0,1,\ldots, n\}$ with \textbf{reflecting boundary}, which has transition probabilities
\begin{align*}
    P(0,1) = 1, \quad P(n,n-1) = 1,
\end{align*}
as in Figure \ref{Fig:ReflectingBoundary}.  The difference, of course, is that the walk with the reflected boundaries continues going after hitting one of the endpoints. 

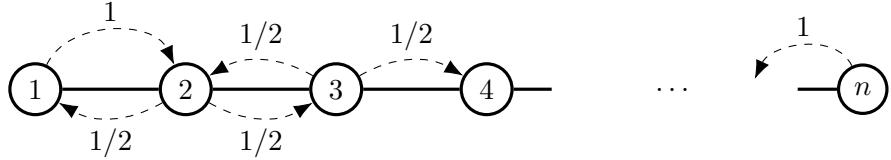
\begin{figure}
    \centering
        \begin{tikzpicture}
    \tikzset{
        node/.style={circle, draw=black, very thick, minimum size=3mm},
        arrow/.style={-{Latex[length=2.5mm]}, black},
        line/.style={black, very thick}
        }
        
        \node[node] (one) at (0,0) {1};
        \node[node] (two) at (2,0) {2};
        \node[node] (three) at (4,0) {3};
        \node[node] (four) at (6,0) {4};
        \node[] (four5) at (7,0) {};
        \node[] (second) at (10,0) {};     
        \node[] (othersecond) at (9.5,0) {};  
        \node[node] (n) at (11,0) {$n$};

        \draw[line] (one) -- (two);
        \draw[line] (two) -- (three);
        \draw[line] (three) -- (four);
        \draw[line] (four) -- (four5);
        \draw[line] (second) -- (n);
        \path (four5) -- node[auto=false]{\ldots} (second);
        
        \draw[arrow, dashed] (one) to[in=115, out=65] node[anchor=south] {$1$} (two);
        \draw[arrow, dashed] (two) to[bend right] node[below] {$1/2$} (three);
        \draw[arrow, dashed] (two) to[bend left] node[below] {$1/2$} (one);
        \draw[arrow, dashed] (three) to[bend right] node[above] {$1/2$} (two);
        \draw[arrow, dashed] (three) to[bend left] node[above] {$1/2$} (four);
        \draw[arrow, dashed] (n) to[out=115, in=65] node[anchor=south] {$1$} (othersecond);
    \end{tikzpicture}
    \caption{\small The simple walk on $\{0,1,\ldots,n\}$ with reflecting boundary.}
    \label{Fig:ReflectingBoundary}
\end{figure}

We can also eliminate the boundary altogether, and consider the simple symmetric random walk $(X_n)$ on the infinite state space $\mathbb{Z}$, we still have
\begin{align*}
    \nP_k(X_\tau = n) = \frac{k}{n}
\end{align*}
for $\tau=\tau_{\{0,n\}}$, and $0 \leq k \leq n$.

Thinking geometrically, it is clear that the principle behind Theorem \ref{Thm:GamblersRuin} should continue to hold even if we choose a different portion of the integers $\{a,a+1, \ldots, b\}$ to do our walk on, and hence have boundary of $\{a,b\}$ instead of $\{0,1\}$.  We ask the reader to use Theorem \ref{Thm:GamblersRuin} to prove the following general version of the hitting probability formulas.

\begin{theorem}[Gambler's ruin for a general interval]\label{Thm:GeneralGamblersRuin}
    Let $(X_n)$ be a simple symmetric random walk on $\mathbb{Z}$, and let $\tau = \tau_{\{a,b\}}$ be the hitting time to integers $a<b$.  Then for any integer $a \leq k \leq b$,
    \begin{align}\label{Eq:GamblersRuinGeneral}
        \nP_k(X_\tau = b) = \frac{k-a}{b-a} \qquad \text{and} \qquad \nP_k(X_\tau = a) = \frac{b-k}{b-a}
    \end{align}
\end{theorem}

\begin{exercise}
    Prove \eqref{Eq:GamblersRuinGeneral} by using Theorem \ref{Thm:GamblersRuin} and a change of variables.
\end{exercise}

We will see another elegant way to show \eqref{Eq:GamblersRuinGeneral} in Chapter \ref{Ch:Martingales} once we have the tools of \emph{martingales} and the \emph{Optional Sampling Theorem}.  See Problem \ref{Prob:5Opt1}.

%%%%% Lecture 11 %%%%%%%%%%%%%%%%%%%%%%%%%%%

\subsection{Expected hitting times}\label{Sec:GamblersRuinTime}

We completely understand the boundary hitting \emph{probabilities} via Theorem \ref{Thm:GeneralGamblersRuin}: the probability we hit the right boundary is a linear function of the distance we are from the left boundary, with boundary values of 0 and 1.  What about the \emph{time} it takes, on average, to hit the boundary?  How many steps will the random walk take?  We again begin with the simple setting $\Omega = \{0,1,\ldots, n\}$, but now set $\tau := \tau_{\{0,n\}} = \tau_{\partial \Omega}$.  We can use the same ideas as in the proof of Theorem \ref{Thm:GamblersRuin} to obtain an explicit formula.

\begin{exercise}
    \begin{enumerate}[$(i)$]
        \item What are $\nE_0(\tau)$ and $\nE_n(\tau)$?
        \item Where should the maximum of $\nE_k(\tau)$ occur?  Explain why the function $k \mapsto \nE_k(\tau)$ should be symmetric about its maximum.
        \item Sketch your guess of the shape of the graph of the function $k \mapsto \nE_k(\tau)$.
    \end{enumerate} 
\end{exercise}
\noindent We now formalize what you intuitively worked out in this exercise.

\begin{theorem}\label{Thm:GamblersRuinTimes}
    For the simple symmetric random walk on $\Omega = \{0,1,\ldots, n\}$, the hitting time $\tau := \tau_{\partial \Omega}$ satisfies
    \begin{align}\label{Eq:GamblersRuinTime}
        \nE_k(\tau) = k(n-k), \qquad 0 \leq k \leq n.
    \end{align}
\end{theorem}
Note that the function $k\mapsto k(n-k)$ in \eqref{Eq:GamblersRuinTime} is maximized at the integer(s) closest to $k=n/2$ and doesn't change when $k$ is replaced with $n-k$.  To give a concrete example, if we are doing the walk on $\{0,1,\ldots, 11\}$, we see that the average hitting time would be largest at both $k=5$ and $k=6$, the integers surrounding the midpoint.  Also, starting at 2, for instance, will have the same average hitting time as starting at 9, as will starting at 3 and 8.  The symmetries in the hitting times reflect the symmetries of the interval and the random walk.

\begin{proof}
    The proof will be conceptually similiar to the proof of Theorem \ref{Thm:GamblersRuin}.  Write $f_k:=\nE_k(\tau)$, and we first claim that we have the averaging property
    \begin{align}\label{Eq:GRHTAveraging}
        f_k = 1 + \frac{1}{2}f_{k-1} + \frac{1}{2}f_{k+1}
    \end{align}
    for each $1 \leq k \leq n-1$ with boundary condition $f_0=f_n=0$ (compare \eqref{Eq:pkMVP}).  Note that \eqref{Eq:GRHTAveraging} is intuitively clear: we take one step ($+1$) and with equal probability arrive at either $k-1$ or $k+1$. The chain starts over, and so we add the average number of steps from that new vertex to hit the boundary.
    
    To make this rigorous, we can again use the method of recursion, combined with the tower property of conditional expectation:
    \begin{align}
        f_k &= \nE_k(\tau) \notag \\
        &= \nE_k(\nE_k(\tau\, | \, X_1)) \notag \\
        &= \nE_k(\tau\, | \, X_1=k-1)\nP_k(X_1=k-1) + \nE_k(\tau\, | \, X_1=k+1)\nP_k(X_1=k+1) \notag \\
        &= \nE_k(1+ \tau-1\, | \, X_1=k-1)\frac{1}{2} + \nE_k(1+ \tau-1\, | \, X_1=k+1)\frac{1}{2} \notag\\
        &= \big(1+\nE_k(\tau-1\, | \, X_1=k-1)  \big)\frac{1}{2} + \big(1+\nE_k(\tau-1\, | \, X_1=k+1) \big)\frac{1}{2} \notag\\
        &= 1+ \nE_k(\tau-1\, | \, X_1=k-1)\frac{1}{2} +    \nE_k(\tau-1\, | \, X_1=k+1)\frac{1}{2} \notag\\
        &= 1 + \nE_{k-1}(\tau)\frac{1}{2} + \nE_{k+1}(\tau)\frac{1}{2} \label{Eq:GRHT1}\\
        &= 1 + \frac{1}{2}f_{k-1} + \frac{1}{2}f_{k+1},\notag
\end{align}
where in \eqref{Eq:GRHT1} we have used the fact that the number of remaining steps until $X_1$ hits the boundary is $\tau-1$ (we have taken one step), and hence
\begin{align*}
    \nE_k(\tau-1\, | \, X_1=k-1) &= \nE(\tau \, | \, X_0=k-1) = \nE_{k-1}(\tau) \quad \text{and}\\
    \nE_k(\tau-1\, | \, X_1=k+1) &= \nE(\tau \, | \, X_0=k+1) = \nE_{k+1}(\tau).
\end{align*}

Thus \eqref{Eq:GRHTAveraging} holds, and we observe that it is equivalent to saying
\begin{align}
    \frac{1}{2}(f_k -f_{k-1}) &= 1 + \frac{1}{2}(f_{k+1}-f_k) \notag\\
    \Leftrightarrow \quad \Delta_k &= 2 + \Delta_{k+1}, \qquad 1 \leq k \leq n-1,\label{Eq:GRHTDeltaRecursion}
\end{align}
where, as in the proof of Theorem \ref{Thm:GamblersRuin}, we have defined $\Delta_k := f_k-f_{k-1}$.  Thus $\Delta_2 = \Delta_1-2$, 
\begin{align*}
    \Delta_3 &= \Delta_2-2 = \Delta_1 - 2 \cdot 2,\\
    \Delta_4 &= \Delta_3-2 = \Delta_1 - 2 \cdot 3,
\end{align*}
and, in general,
\begin{align}\label{Eq:GRHTDeltaRecursion2}
    \Delta_k = \Delta_1 - 2(k-1), \qquad 2 \leq k \leq n.
\end{align}
We can solve for $\Delta_1$, and thus all the increments by this formula, by summing \eqref{Eq:GRHTDeltaRecursion} over all $k$:
\begin{align*}
    \sum_{k=1}^{n-1} \Delta_k &= 2(n-1) + \sum_{k=1}^{n-1}\Delta_{k+1}\\
    \Leftrightarrow \quad f_{n-1}-f_0 &= 2(n-1) + f_n-f_1
\end{align*}
by the telescopic nature of the sum of the $\Delta_j$'s.  Since $f_0 = f_n =0$, we thus have $f_{n-1} + f_1 = 2(n-1)$, and so by symmetry,
\begin{align*}
    f_{n-1} = f_1 = n-1.
\end{align*}
We conclude that $$f_1 = f_1-f_0 = \Delta_1 = n-1,$$ and so \eqref{Eq:GRHTDeltaRecursion2} becomes
\begin{align}\label{Eq:GRHTDeltaRecursion3}
    \Delta_k = n-1 - 2(k-1), \qquad 2 \leq k \leq n.
\end{align}
We thus see that the increment $\frac{f_k-f_{k-1}}{1}$ or ``discrete derivative'' of $f_k$ decreases linearly, and so we expect $f_k$ itself to be quadratic (a function with linear derivative is, after all, quadratic).  Indeed, we do the ``discrete integral'' or sum of the ``derivative identity'' \eqref{Eq:GRHTDeltaRecursion3} to obtain
\begin{align}
f_\ell = f_\ell-f_0 = \sum_{k=1}^\ell \Delta_k &= \sum_{k=1}^\ell\big(n-1 - 2  (k-1)\big) \notag\\
&= \ell(n-1) - 2 \sum_{k=1}^\ell (k-1) \notag\\
&= \ell(n-1) - 2 \sum_{m=0}^{\ell-1} m \notag\\
&= \ell(n-1) - (\ell-1)\ell = \ell(n-\ell),\label{Eq:GRHTDone}
\end{align}
where we have used the summation formula
\begin{align*}
    \sum_{m=0}^{\ell-1} m = \sum_{m=1}^{\ell-1}m = \frac{(\ell-1)\ell}{2}.
\end{align*}
Equation \eqref{Eq:GRHTDone} says $\nE_\ell(\tau) = \ell(n-\ell)$, as claimed.
\end{proof}

Just as we extended our hitting probabilities to general intervals of integers in Theorem \ref{Thm:GeneralGamblersRuin}, so we can do so with Theorem \ref{Thm:GamblersRuinTimes}.

\begin{theorem}[Gambler's ruin hitting times for a general interval]\label{Thm:GeneralGamblersRuinTimes}
    Let $(X_n)$ be a simple symmetric random walk on $\mathbb{Z}$, and let $\tau = \tau_{\{a,b\}}$ be the hitting time to integers $a<b$.  Then for any integer $a \leq k \leq b$,
    \begin{align}\label{Eq:GamblersRuinGeneralTimes}
        \nE_k(\tau) = (k-a)(b-k)
    \end{align}
\end{theorem}

\begin{exercise}
    Prove \eqref{Eq:GamblersRuinGeneralTimes} by using Theorem \ref{Thm:GamblersRuinTimes} and a change of variables.
\end{exercise}
\begin{exercise}\label{Ex:SRWTimeForDistanceN}
    On average, how many units of time does it take for a simple symmetric random walk on $\mathbb{Z}$ to reach a distance of $n$ units away from its starting position?    
\end{exercise}

As with the hitting probability, you will revisit the average hitting time in Problem \ref{Prob:5Opt2} once we have \emph{martingales} and the \emph{Optional Sampling Theorem} to work with.

\subsection{The simple symmetric random walk on $\mathbb{Z}$ and Brownian motion}\label{Sec:SRWZBM}

It is also important to consider the case where there are no boundary conditions and our walk is on all of $\mathbb{Z}$.  Let $X_n$ be such a random walk starting at $X_0=0$.  We can think of this as flipping a fair coin at each step and moving up one unit if we get heads, and down one unit if we get tails.  Plotting $X_n$ as a function of the step $n$ and linearly interpolating between the points (that is, drawing a straight line between them) gives a random piece-wise linear function. Thirty steps of one sampling are shown in Figure \ref{Fig:SRW30}.  We call this random function $X_t$, where now $t$ no longer needs to be an integer.

\begin{figure}[ht]
    \centering
    \includegraphics[scale=0.5]{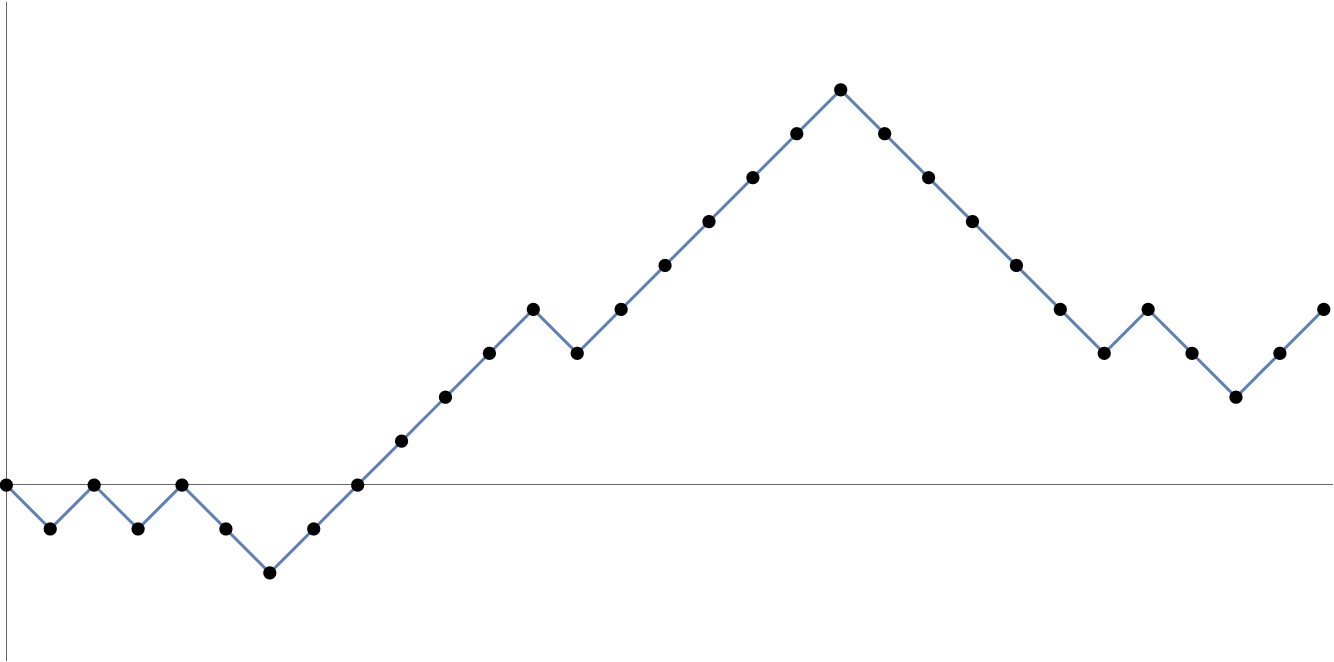}
    \caption{\small Thirty steps of a simple random walk $(X_n)$ on $\mathbb{Z}$, with linear interpolation between the values.}
    \label{Fig:SRW30}
\end{figure}

An interesting question is what happens with this picture as we ``zoom out'' and display more and more steps.  What if we plot up to $n=300$ or $n=3000$, instead of $n=30$?  Since the trajectory $X_t$ is random, some thought is actually needed to know how to ``zoom out'' in the correct manner.  One approach would be to consider $X_{nt}$ as $n \rightarrow \infty$, which compresses more and more steps of $X_n$ into the same amount of time.  At $t=1$, for instance, $X_{n \cdot 1}= X_n$, and thus in one time unit we have already gone $n$ steps in our random walk.  

In terms of the function's graph, as in Figure \ref{Fig:SRW30}, replacing $X_t$ with $X_{nt}$ scales the $x$-axis (time) by a factor of $n$ but leaves the $y$-axis (space) unchanged.  As $n$ gets larger and larger, though, the walk tends to get further and further away from 0, and so if we do not also re-scale space we will not see the entire picture in one window.
\begin{figure}
    \centering
    \includegraphics[scale=0.7]{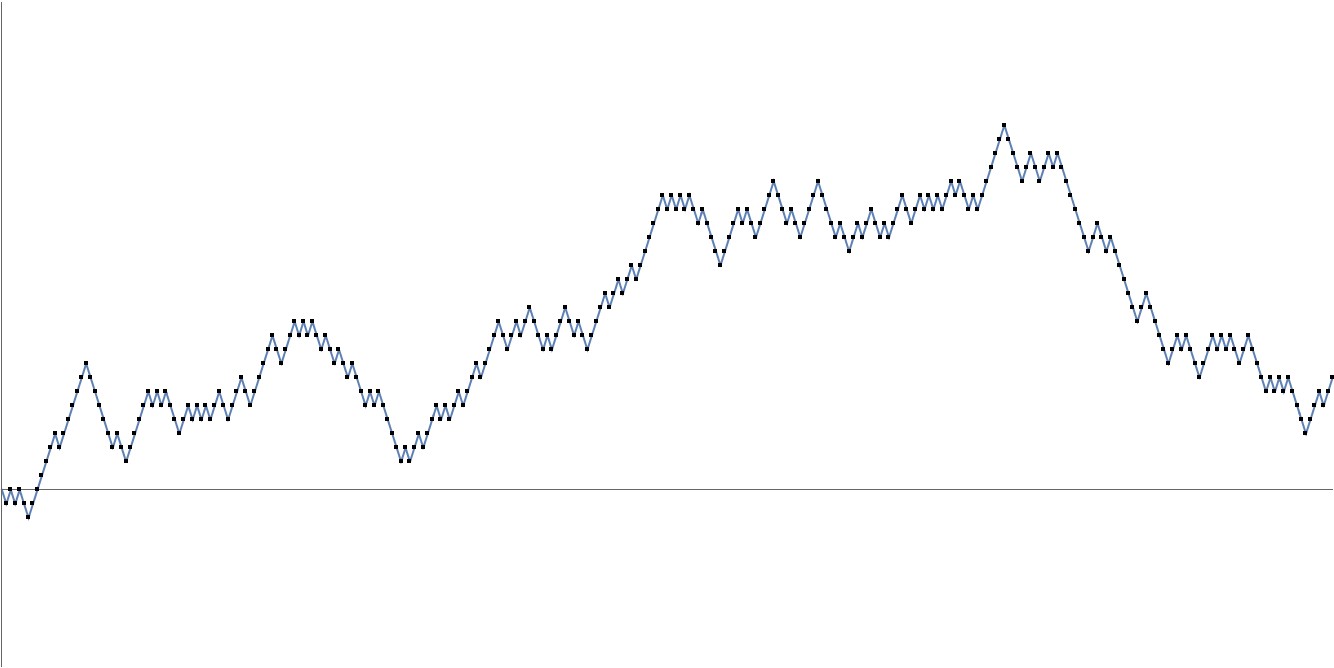}
    \vspace{5mm}
    
    \includegraphics[scale=0.7]{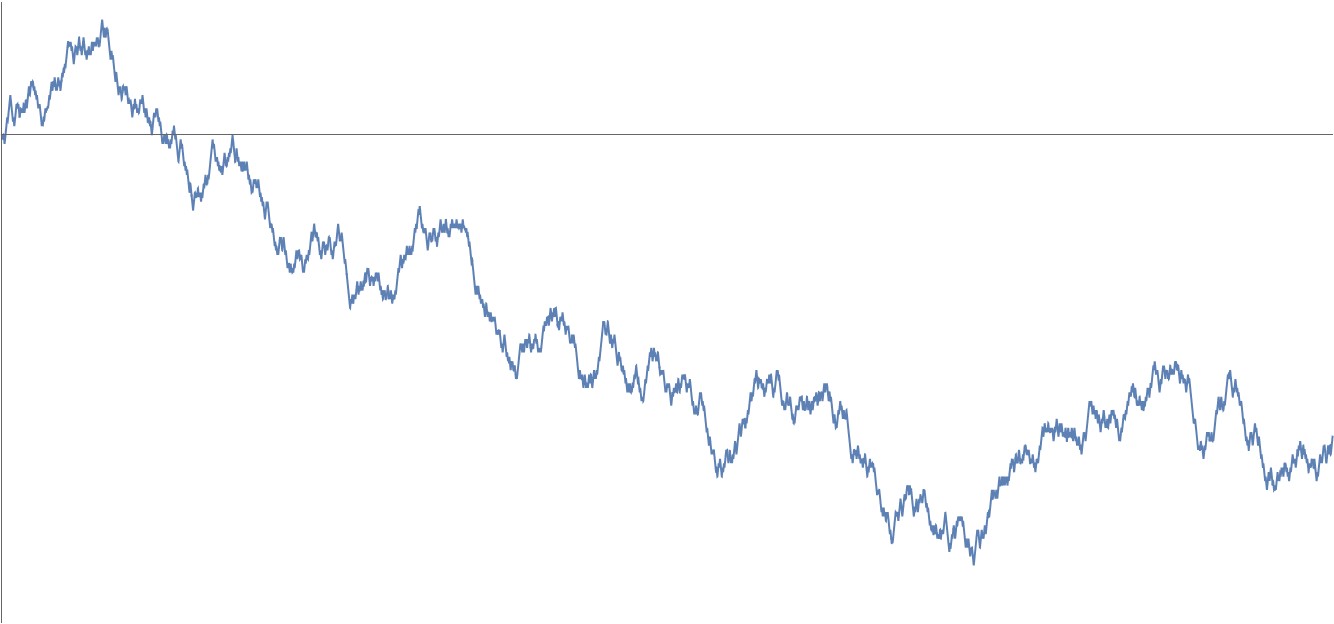}
    \caption{\small The same sampling of the random walk in Figure \ref{Fig:SRW30}, (appropriately) zoomed out to show 300 and 3000 steps, respectively.}
    \label{Fig:SWR300Plus}
\end{figure}

Suppose we scale space by a factor of $n^{-\beta}$, and so consider $\frac{1}{n^\beta} X_{nt}$.  Let us fix $t=1$ and see if we can determine which value of $\beta$ would be reasonable.  The mean of the walk is 
\begin{align}\label{Eq:WalkMeanZero}
    \nE\Big( \frac{1}{n^\beta} X_{n} \Big) = \frac{1}{n^\beta} \nE(X_n) = \frac{1}{n^\beta} \nE\Big(\sum_{j=1}^n Y_j \Big) = \frac{1}{n^\beta} \sum_{j=1}^n \nE( Y_j ) = 0,
\end{align}
where the $Y_j$ are the individual iid $\pm 1$ steps of the random walk, which have zero average.  So the mean is impervious to the choice of $\beta$.  The same is not true for the variance, however, since
\begin{align}\label{Eq:WalkVariance}
    \Var\Big( \frac{1}{n^\beta} X_{n} \Big) = \frac{1}{n^{2\beta}} \Var\Big(\sum_{j=1}^n Y_j \Big) = \frac{1}{n^{2\beta}} \sum_{j=1}^n \Var(Y_j) = \frac{1}{n^{2\beta-1}}.
\end{align}
Here we have used the independence of the individual walk steps and the fact that $\Var(Y_j)=1$.  Thus if we re-scale space by $1/n^\beta$ for any $\beta > 1/2$, such as by $1/n$, then in the $n\rightarrow \infty$ limit the variance at $t=1$ is zero, which implies by \eqref{Eq:WalkMeanZero} that the limiting value is identically zero.  This is compressing space too much and one can show the entire picture collapses to the zero function in the limit.  If we scale by a power $\beta<1/2$, on the other hand, then by \eqref{Eq:WalkVariance} the variance at $t=1$ blows up, and so we see wilder and wilder oscillations in the $t=1$ value as $n$ increases.  

This strongly suggests that the right scaling of space is with $\beta=1/2$.  Then the variance at $t=1$ is always 1, and furthermore the Central Limit Theorem says
\begin{align*}
    \frac{1}{\sqrt{n}}X_n = \frac{1}{\sqrt{n}}\sum_{j=1}^n Y_n \Rightarrow Z,
\end{align*}
a $N(0,1)$ random variable.  We can see the non-trivial distributional limit for $\beta=1/2$ emerging in the two graphs in Figure \ref{Fig:SWR300Plus}.  These plot the linear interpolation of $(\frac{1}{\sqrt{n}} X_{nt})_{0 \leq t \leq 30}$ with $n=10$ and $n=100$ for the same sampling of the random walk as in Figure \ref{Fig:SRW30}.

The jagged, fractal-like random function we get in the distributional limit of $(\frac{1}{\sqrt{n}}X_{nt})_{t \geq 0}$ is of fundamental importance in probability theory, physics, biology and economics, and is called \emph{Brownian motion} $(B_t)_{t \geq 0}$.  Brownian motion is, in fact, a \emph{continuous-time Markov process}.

One technique for analyzing $(B_t)_{t \geq 0}$ is to start with the piece-wise linear functions $\frac{1}{\sqrt{n}}X_{nt}$.  One can study their properties and then attempt to take a distributional limit as $n\rightarrow \infty$ to obtain information about the Brownian motion path.  This boils down to understanding the simple random walk $(X_n)$, our original Markov chain on $\mathbb{Z}$ (without boundary conditions). And a natural starting place for understanding $X_n$ is to determine its distribution for all $n$.

\begin{lemma}\label{Lemma:SRWpdf}
    For $n \in \mathbb{N}$, the pmf of $X_n$ is 
    \begin{align}\label{Eq:SRWonZpdf}
        \nP(X_n = k) = \binom{n}{\frac{n+k}{2}} \frac{1}{2^n} \quad \text{for} \quad k = -n, -n+2, \ldots, n-2, n,
    \end{align}
    while $\nP(X_n=k)=0$ for all other $k$.
\end{lemma}
Contrary to its appearance, this formula is actually quite intuitive.  The condition on $k$ says that the parity of our position $X_n$ must match the parity of the step number $n$.  For instance, if $n=5$, then $X_n$ is odd and $X_n \in \{-5,-3,-1,1,3,5\}$.  If we take $\frac{n+k}{2}$ steps of $+1$, then the number of $-1$ steps is 
\begin{align*}
    n - \frac{n+k}{2} = \frac{n-k}{2},
\end{align*}
and our ending position is $1 \cdot \frac{n+k}{2} - 1 \cdot \frac{n-k}{2} = k$.  Hence the binomial coefficient counts the number of ways to pick out the correct number of $+1$ steps from the $n$ total so that we reach position $k$.  And, of course, the probability for any fixed selection of $+1$'s and $-1$'s in the iid steps is $(1/2)^n$.  

Put another way, each of the $2^n$ random walk paths of length $n$ are equally likely, and the binomial coefficient weights by how many end at height $k$.

This sketches one proof of \eqref{Eq:SRWonZpdf}; we proceed with a slightly more formal argument.
\begin{proof}
    We may write $X_n = \sum_{j=1}^n Y_j$ for iid coin flips $Y_1, \ldots, Y_n$, $\nP(Y_j=1) = \nP(Y_j=-1) = 1/2$.  Let 
    \begin{align*}
        H_n = \sum_{j=1}^n Y_j \I_{\{Y_j=1\}} \quad \text{and} \quad T_n = \sum_{j=1}^n -Y_j \I_{\{Y_j=-1\}}
    \end{align*}
    be the number of heads and tails through $n$ steps, respectively.  Since $T_n = n-H_n$, we have
    \begin{align*}
        X_n &= H_n-T_n = 2H_n -n.
    \end{align*}
    Noting that $H_n \sim \Bin(n,1/2)$, we thus see
    \begin{align*}
        \nP(X_n = k) &= \nP(2H_n-n = k)\\
        &= \nP\big(H_n = \frac{n+k}{2}\big) = \binom{n}{\frac{n+k}{2}} \frac{1}{2^n}
    \end{align*}
    when $\frac{n+k}{2} \in \{0,1,\ldots, n\}$, which is equivalent to $k \in \{-n, -n+2, -n+4, \ldots, n\}$.
\end{proof}
\begin{exercise}
    Use Stirling's formula to show that 
    \begin{align*}
        \lim_{n \rightarrow \infty} \sqrt{\pi n} \, \nP(X_{2n}=0) =1,
    \end{align*}
    and thus that $\nP(X_{2n}=0) \approx 1/\sqrt{\pi n}$ when $n$ is large.
\end{exercise}
We close this section by saying further examination of Brownian motion is beyond the scope of our text, but we hope this very brief introduction has whetted your appetite to study it further in the future.

\subsection{Birth and death chains}

The simple random walk on $\mathbb{Z}$ is fundamentally important, but other variations occur in theory and practice.  One class of such models are the \emph{birth and death chains}, which model population dynamics.  Consider a population between 0 and $n$ that increases or decreases by at most one during each unit of time.  We can view this as a Markov chain on the state space $\Omega = \{0, 1,2,\ldots, n\}$ of size $n+1$ with transition probabilities
\begin{align}\label{BDChainNotation}
	p_k := p(k,k+1), \quad q_k := p(k,k-1), \quad r_k := p(k,k)
\end{align}
for $k \in \Omega$, where $0\leq p_k, q_k, r_k \leq 1$ are parameters chosen such that $p_k + q_k + r_k =1$, with $q_0 = 0$ and $p_n = 0$ (so that our population size $X_n$ always stays in $\Omega$).  Here the $p_k, q_k$ and $r_k$ can all vary with $k$.  We call such a Markov chain a \textbf{birth and death chain}. Note that the chain is irreducible iff
\begin{enumerate}[$(i)$]
	\item $p_k, q_k >0$ for all $1 \leq k \leq n-1$, and 
	\item $p_0 > 0$ and $q_n >0$,
\end{enumerate}
since it is only under these conditions that we can move between any two states.

A natural class of examples of birth and death chains is provided by the so-called queuing model. Imagine customers standing in a queue to get serviced at a service window (see Figure \ref{fig:queue}).  Let $X_k$ denote the number of customers in the queue at time $k$ for $k=0,1,2,\ldots$. Assume that the queue can hold at most $n$ customers. Then each $X_k \in \{0,1,2,\ldots, n\}$. 

Fix a number $p\in (0,1)$ and let $q=1-p$. Imagine that at time point $k-1$ the queue is neither empty nor full. That is, assuming that $1\le X_{k-1} \le n-1$, two things can happen at time $k$: (1) a new customer gets added with probability $p$, thereby $X_{k}=X_{k-1}+1$; or, (2) the first customer gets served with probability $q$ and leaves, making $X_k = X_{k-1}-1$. If $X_{k-1}=n$, i.e., the queue is full, no more new customers may be added to the queue, but the first customer still gets served with probability $q$. Thus $X_k=n-1$ with probability $q$ and $X_k=n$ with probability $p$. Finally, if the queue is currently empty, i.e., $X_{k-1}=0$, then a new customer arrives with probability $p$. That is $X_k=1$ with probability $p$, otherwise $X_k=0$. The Markov chain $(X_k)_{k\ge 0}$ is clearly a birth and death chain where each $p_i=p$ and each $q_i=q$.  

\begin{figure}[h!]
\begin{tikzpicture}[>=Stealth, thick]

% Title

% Baseline (queue)
\draw (0,0) -- (8.5,0);

% Customers (ticks + heads)
\foreach \x in {2,4,6,8}{
  \draw (\x,0) -- (\x,0.55);      % body/tick
  \draw (\x,0.85) circle (0.25);  % head
}

% Service window
\draw (8.5,-0.9) -- (8.5,0.9);
\draw (8.5,-0.9) rectangle (10.5,0.9);
\node at (9.5,0.35) {SERVICE};
\node at (9.5,-0.35) {WINDOW};

% Arrow to first customer (near window)
\draw[->] (8,-1.1) -- (8,-0.1);
\node[align=center, anchor=north] at (8,-1.2)
  {Customer getting served};

% Arrow to join point (left)
\draw[->] (0.7,-1.1) -- (0.7,-0.1);
\node[align=center, anchor=north] at (0.7,-1.2)
  {New customer joins here};

\end{tikzpicture}
\caption{Customers waiting in a queue get served at the service window one after another. New customers join at the end of the queue. Customers may decide to leave the queue anytime. Here the queue length is four.}
\label{fig:queue}
\end{figure}
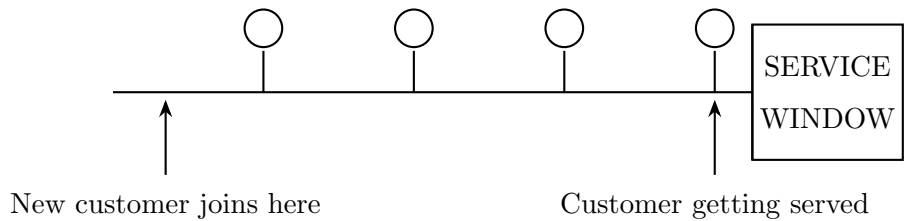

What is the stationary distribution of these chains?  Are they reversible?  The next theorem answers both of these questions in the affirmative.

\begin{theorem}
	Every irreducible birth and death chain is time reversible under a unique and explicit stationary distribution.
\end{theorem}

\begin{proof}
	We immediately know from Theorem \ref{Thm:ReturnTime} that an irreducible birth and death chain has a unique stationary distribution.  Furthermore, recall from Theorem \ref{Thm:DBEConsequences} that if we can build a probability distribution $\pi$ which satisfies the detail balance equations \eqref{Eq:DBE}, then $\pi$ is stationary and the chain is automatically time reversible under $\pi$.  So, it suffices to find $\pi = (\pi_0, \ldots, \pi_n)$ with $\pi_k >0$ and $\sum_{k=0}^n \pi_k = 1$ satisfying the detail balance equations 
\begin{align*}
	\pi_k p(k,j) = \pi_j p(j,k)
\end{align*}
for all $j,k$.  Since $p(k,j) = 0$ except when $j=k-1,k,k+1$, there are only three cases to check over all possible values of $k$:
\begin{enumerate}[$(i)$]
	\item $\pi_k p(k,k+1) = \pi_{k+1} p(k+1,k)$,
	\item $\pi_k p(k,k-1) = \pi_{k-1} p(k-1,k)$, and
	\item $\pi_k p(k,k) = \pi_k p(k,k)$.
\end{enumerate}
Case $(iii)$ is trivially true, and $(ii)$ is encompassed by $(i)$ through replacing $k$ with $k-1$.  So, recalling our notation   \eqref{BDChainNotation}, we just need build $\pi$ for which
\begin{align}\label{Id:BDChainPf}
	\pi_k p_k = \pi_{k+1} q_{k+1}, \qquad k = 0,1, \ldots, n-1.
\end{align}
To that end, define ``weights'' $w_k$, $k=0,1, \ldots, n$, via $w_0 := 1$ and, for $1 \leq k \leq n$,
\begin{align}
	w_k :&= \frac{p_0 p_1 \cdots p_{k-1}}{q_1 q_2 \cdots q_k} \notag\\
	&= \frac{\nP(X_1=1, X_2=2, \ldots, X_k=k | X_0 = 0)}{\nP(X_1 = k-1, X_2 = k-2, \ldots, X_k = 0 | X_0 = k)} \label{BDChainPfMemory}
\end{align}
and set 
\begin{align}\label{Eq:BDChainStationaryDist}
	\pi_k = \frac{w_k}{\sum_{j=0}^n w_j}, \quad k = 0,1, \ldots, n.
\end{align}
Although at first glance the formula appears unwieldy, it is actually routine to check the following, and you should do so to get your hands dirty.
\begin{exercise}
	Verify that $\pi$ is a probability distribution on $\Omega$ and that $\pi$ satisfies \eqref{Id:BDChainPf}.
\end{exercise}
\noindent We conclude from Theorem \ref{Thm:DBEConsequences} that $\pi$ is stationary and that the birth and death chain is reversible under $\pi$.
\end{proof}

Note that the expression \eqref{BDChainPfMemory} can serve as a helpful mnemonic for the $w_k$ formula: the numerator is the probability of starting at 0 and continually increasing until you reach $k$, and the denominator is the probability of the return journey.

\begin{example}
	The simplest case is when $p_k = p$ for all $0 \leq k \leq n-1$ and $q_k =q$ for $1 \leq k \leq n$, where $p+q =1$ (and so $r_k=0$ for $k=1,2, \ldots n-1$, but $r_0 = 1-p$ and $r_n = 1-q$).  That is, when all the $p$'s and $q$'s are the same. This is the queue length example covered above Figure \ref{fig:queue}.  
    
    What is our stationary distribution \eqref{Eq:BDChainStationaryDist} here?  We find
	\begin{align*}
		w_0 = 1 = \left(\frac{p}{q}\right)^0, \quad w_k = \frac{p_0 p_1 \cdots p_{k-1}}{q_1 q_2 \cdots q_k} = \left(\frac{p}{q}\right)^k,
	\end{align*}
	and thus
	\begin{align*}
		\pi_k = \frac{(p/q)^k}{\sum_{j=0}^n (p/q)^j}, \qquad k=0,1, \ldots, n.
	\end{align*}

Notice that when $p>1/2$, $p/q > 1$, and the mass of the stationary distribution $\pi_k$ is geometrically large for large values of $k$. That is, the queue tends to be near full in stationarity. This makes intuitive sense since the probability of new customers joining the queue is higher than customers getting serviced, which will fill the queue. 

On the other hand, when $p < 1/2$, $p < q$, and, hence, $\pi_k$ is geometrically small for large values of $k$. Thus, customers get serviced faster than new customers join and the queue tends to be mostly empty.  

The case when $p=q$ is critical since there the stationary distribution $\pi$ is uniform over $\{0,1,2,\ldots, n\}$. 
\end{example}

\begin{exercise}
	For $p=q$, the birth and death chain becomes a random walk. Identify the boundary condition. Show that $\pi$ is uniform.
\end{exercise}

% Lecture 15 %%%%%%%%%%%%%%%%%%%%%%%%%%%%%

\section{The Ehrenfest urn}\label{Sec:EhrenfestUrn}

%Lecture 12

The Ehrenfest urn is a Markov chain modelling the diffusion of particles through a porous membrane.  It was developed by the physicist Paul Ehrenfest (1880-1933) and his wife Tatyana (1876-1964) as a way of explaining the second law of thermodynamics.\footnote{See 
\url{en.wikipedia.org/wiki/Ehrenfest_model}, for instance.}  It is an intrinsically interesting model, but also important for us because it is a Markov chain which is not a simple random walk on a graph, but rather a \emph{projection} of such a walk.  We explain precisely what this means below.

The model starts with two urns having a total of $N$ identical balls.  One step in the chain consists of picking one of the $N$ balls uniformly at random and switching its urn. So if the selected ball is in the first urn, move it to the second, while if in the second urn, move it to the first.  Tracking the number of balls $X_n$ in the first urn after $n$ steps gives a Markov chain on $\Omega = \{0,1, \ldots, N\}$, the \textbf{Ehrenfest urn model}.  The following two exercises will help you get some feel for this chain.

\begin{exercise}
	Show that the transition probabilities for the Ehrenfest urn chain $(X_n)$ are
	\begin{align*}
		p(k,k+1) &= \nP_k(\text{adding a ball from urn 2 to urn 1}) = 1 - \frac{k}{N},\\
		p(k,k-1) &= \nP_k(\text{adding a ball from urn 1 to urn 2}) = \frac{k}{N},
	\end{align*}
 for $0 \leq k \leq N-1$ and $1 \leq k \leq N$, respectively.
\end{exercise}

\begin{exercise}\label{Ex:EhrenfestTransProb}
Show the following:
	\begin{enumerate}[$(a)$]
        \item $\nE(X_{n+1} \,|\, X_n=k) = 1 + (1-2/N)k$.
        \item Therefore, $\nE(X_{n+1} \,|\, X_n) < X_n$ when $X_n > N/2$, and  $\nE(X_{n+1} \,|\, X_n) > X_n$ when $X_n < N/2$.  Thus, the chain on average moves towards equal numbers of balls in each urn.
	\end{enumerate}    
\end{exercise}

A moment's thought shows that the Ehrenfest urn is irreducible, and it therefore has a unique stationary distribution by Theorem \ref{Thm:ReturnTime}.  Can we explicitly compute $\pi$?  The question is simplified by broadening our viewpoint to see the chain as a \emph{projection} of the simple random walk on the $n$-dimensional hypercube.  
% The Ehrenfest urn is a random walk on the integers $\Omega$, but it is not a simple random walk.  However, it will be the projection of a simple random walk.

% We proceed to define the hypercube in $\mathbb{R}^N$, what we mean by ``projecting'' the random walk on this hypercube, and we conclude by finding the stationary distribution of the Ehrenfest urn.

\subsection{The Ehrenfest urn as a projection}

Let 
\begin{align*}
	\Omega' = \{0,1\}^N = \big\{\, \{ s_1, s_2, \ldots, s_N\} \; | \; s_j \in \{0,1\} \text{ for all } j\,\}
\end{align*}
be the set of all sequences of 0's and 1's of length $N$.  This is simply the collection of vertices of the (hyper)cube in $\mathbb{R}^N$, and as such, it has a natural graph structure: say that two vertices $u$ and $v$ are adjacent,
\begin{align*}
	u = (u_1, u_2, \ldots, u_N) \sim (v_1, v_2, \ldots, v_N) =v,
\end{align*} 
if you can go from $u$ to $v$ by switching exactly one coordinate from 0 to 1 or from 1 to 0.  See Figure \ref{Fig:hypercube} for the case $N=3$.  More formally, $u \sim v$ if 
\begin{align*}
	\Norm{u-v}_1 := \sum_{j=1}^N |u_j - v_j| =1.
\end{align*}

\begin{figure}
    \centering
        \scalebox{0.7}{
    \begin{tikzpicture}
    \tikzset{
        node/.style={circle, fill=black, very thick, minimum size=1mm},
        arrow/.style={->, black, very thick},
        line/.style={black, very thick},
        dashl/.style={black, very thick, dashed}
     }
     
     \node[node] (one) at (0,0) {};
     \node (000) at (0,-0.5) {(0, 0, 0)};
     \node[node] (two) at (4,0) {};
     \node (100) at (4,-0.5) {(1, 0, 0)};
     \node[node] (three) at (6,2.667) {};
     \node (110) at (7,2.667) {(1, 1, 0)};
     \node[node] (four) at (2,2.667) {};
     \node (010) at (2.4,2.2) {(0,1,0)};
     \node[node] (five) at (0,4) {};
     \node (001) at (-0.9,4) {(0,0,1)};
     \node[node] (six) at (4,4) {};
     \node (101) at (4.9,4) {(1,0,1)};
     \node[node] (seven) at (6,6.667) {};
     \node (111) at (7,6.667) {(1, 1, 1)};
     \node[node] (eight) at (2,6.667) {};
     \node (011) at (1.1,6.667) {(0,1,1)};
     
     \draw[line] (one) -- (two);
     \draw[line] (one) -- (five);
     \draw[dashl] (one) -- (four);
     \draw[line] (two) -- (three);
     \draw[line] (two) -- (six);
     \draw[line] (three) -- (seven);
     \draw[dashl] (three) -- (four);
     \draw[dashl] (four) -- (eight);
     \draw[line] (five) -- (six);
     \draw[line] (five) -- (eight);
     \draw[line] (six) -- (seven);
     \draw[line] (seven) -- (eight);
    \end{tikzpicture}
    }
    \caption{\small The hypercube $\{0,1\}^3$ in $\mathbb{R}^3$}
    \label{Fig:hypercube}
\end{figure}
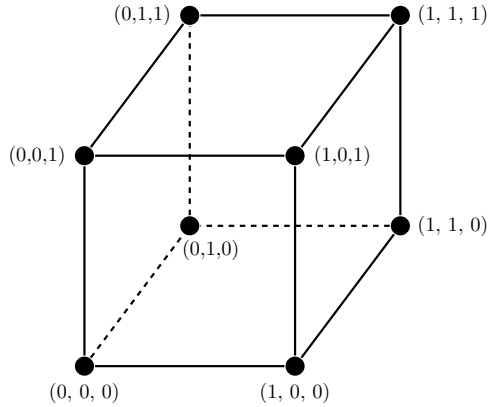

\begin{exercise}\label{Ex:Hypercube}
		What are the degrees of the vertices of the hypercube in $\mathbb{R}^3$?  How many edges are there?  Generalize these two questions to the hypercube in $\mathbb{R}^N$.
\end{exercise}

Consider the simple symmetric random walk $X'=(X_n')_{n\geq 0}$ on the $N$-dimensional hypercube.  By the edge structure we saw above, each step in this walk selects one of the coordinates uniformly at random and switches its value.  Note that the hypercube is connected and is also a  \emph{regular graph} of degree $N$, which is to say every vertex $v$ has the same degree $N$.  By Theorem \ref{Thm:SRWpi} and Exercise \ref{Ex:Hypercube}, the stationary distribution for $X'$ is therefore
\begin{align}\label{Eq:HypercubeStationary}
	\pi'(u) = \frac{\deg(u)}{2|E|} = \frac{N}{2^{N}N} = \frac{1}{2^{N}},
\end{align}
which is simply the uniform distribution on $\Omega'$.  Of course we expect as much from the symmetry of the graph.

% How does this all relate to the Ehrenfest urn?  Give each ball a label $1,2, \ldots, N$ and run the urn while keeping track of which balls are in which urn.  We could write this as a vector $v = (v_1, \ldots, v_N)$ of $0$'s and $1$'s, with entry $v_j=1$ if the $j$th ball is in the first urn, while $v_j=0$ if ball $j$ is in the second urn.  For example, if $N=3$ and balls 1 and 3 were in the first urn and 2 was in the second, our vector would be $(1,0,1)$.  Then in each step of the Ehrenfest urn, exactly one uniformly-chosen entry of the vector is changed, which precisely corresponds to one step on the simple symmetric walk on the hypercube.  So we see that running the Ehrenfest urn chain while keeping track of the additional information of the urn holding each ball recovers the walk on the hypercube.

% This gives a one-to-one correspondence between vertices on the hypercube and lists of which balls are in each of the two urns.

So how does this all relate to the Ehrenfest urn?  Label the balls from $1$ to $N$, and given our current location $v = (v_1, \ldots, v_N)$ on the hypercube, place ball $j$ into the first urn if $v_j=1$, and into the second urn if $v_j=0$.  So, for example, if $N=3$ and we are at $v=(1,0,1)$ in the random walk, put balls 1 and 3 into the first urn and ball 2 into the second urn.  This gives a one-to-one correspondence between vertices on the hypercube and lists of which balls are in which urn.  Furthermore, one step of the hypercube walk exactly gives one step of the Ehrenfest urn: a moment's thought shows that for a step on random walk we choose one coordinate $v_j$ uniformly at random and change its value, which corresponds to choosing one ball uniformly at random and switching its urn.  So while running the random walk on the hypercube we are simultaneously running the urn process, but keeping track of additional information (the urn for each ball).  

To recover the actual Ehrenfest urn chain $(X_n)$, we forget the labelling and simply count the balls  in the first urn.  This  ``projects down'' from the $N$-dimensional list $(v_1, \ldots, v_N)$ to give a single number, and is formally accomplished through the function $F: \Omega' \rightarrow \Omega = \{0,1, \ldots, N\}$ given by 
\begin{align*}
	F(v) = \#\{\, \text{coords of $v$ which are 1} \, \} = \sum_{j=1}^N v_j.
\end{align*}
So if $(X_1', X_2',\ldots)$ is the walk on the hypercube, then $(F(X_1'), F(X_2'), \ldots)$ is the Ehrenfest urn.  You are probably already convinced of this, but we can rigorously verify it by looking at the transition probabilities of the $F(X_j')$.  If $F(X_j')=k$, then our current coordinate $v$ on the hypercube has $k$ ones and $N-k$ zeros, and so to go to $k+1$ we must move to one of the $N-k$ adjacent vertices with an additional entry of 1, yielding
\begin{align*}
    \nP( F(X_{j+1}')=k+1 | F(X_j')=k)= \frac{N-k}{N}.
\end{align*}
Of course does not matter \emph{which} vertex $v$ with $k$ ones we are currently at; the above computation is the same for all of them by the symmetry of the hypercube.

Similarly, to go down to $k-1$ from $k$ we must change a 1 coordinate to a 0 coordinate through moving to one of the $k$-such adjacent vertices, yielding
\begin{align*}
	\nP( F(X_{j+1}')=k-1 \, |\, F(X_j')=k) &= \frac{k}{N}.
\end{align*}
By Exercise  \ref{Ex:EhrenfestTransProb} these are the same transition probabilities as the urn chain, showing the projection $(F(X_1'), F(X_2'), \ldots)$ is indeed the Ehrenfest urn, as claimed.

\subsection{Projecting to get the stationary distribution}

The projection $F$ also gives us a convenient way to obtain the stationary distribution $\pi$ for the Ehrenfest urn.  By \eqref{Eq:HypercubeStationary} the uniform distribution $\pi'$ is stationary for the hypercube walk, and is therefore also stationary for the Ehrenfest urn when also keeping track of the location of each of the balls. To get $\pi$ for the actual Ehrenfest urn, we need to ``project'' or \emph{push forward} $\pi'$ from $\Omega'$ to $\Omega$ via $F$.  This means that we want the measure $\pi$ on $\Omega$ defined by
\begin{align*}
    \pi(k) = \pi'\big( F^{-1}(k) \big) &= \pi'\big(\{\, v \in \Omega' \;:\; F(v)=k  \,\} \big)\\
    &= \pi'\big(\{\, v \in \Omega' \;:\; \sum_{j=1}^N v_j =k  \,\} \big).
\end{align*}
Since $\pi'$ is uniform on the $2^N$ vertices of the hypercube,
\begin{align*}
   \pi(k) = \pi'\big(\{\, v \in \Omega' \;:\; \sum_{j=1}^N v_j =k  \,\} \big) &= \frac{1}{2^N} \cdot \#\{\, v \in \Omega' \; : \; \sum_{j=1}^N v_j =k\,\}\\
   &= \frac{1}{2^N} \binom{N}{k}.
\end{align*}
That is, $\pi$ for the Ehrenfest urn is the Bin$(N,1/2)$ distribution! Since we don't have any formal results about pushing forward a stationary distribution $\pi'$ under a projection $F$, however, let's prove this by verifying the DBEs \eqref{Eq:DBE}.  On the one hand,
\begin{align}\label{Eq:EhrenfestDBE1}
	\pi(k)p(k,k+1) = \frac{1}{2^N} \binom{N}{k} \frac{N-k}{N} = \frac{(N-1)!}{k!(N-k-1)!2^N},
\end{align}
while on the other hand,
\begin{align}\label{Eq:EhrenfestDBE2}
	\pi(k+1)p(k+1,k) =\frac{1}{2^N} \binom{N}{k+1} \frac{k+1}{N} = \frac{(N-1)!}{k!(N-k-1)!2^N}
\end{align}
as well.  Thus, by Theorem \ref{Thm:DBEConsequences}, $\pi$ is stationary for the Ehrenfest urn and the chain is also time-reversible under it.\footnote{Note again the power of the DBEs: by simply verifying the equality of \eqref{Eq:EhrenfestDBE1} and \eqref{Eq:EhrenfestDBE2}, we show both stationarity and reversibility.}  See Figure \ref{Fig:Ehrenfest} for a plot of $\pi$ with $N=20$.  

The binomial nature of $\pi$ gives us a nice interpretation of the urn.  We may think of each ball as being independent of the others and having equal probability $1/2$ of being in either urn.  Hence $\pi(k)$ is the probability we have $k$ ``successes'' when looking to see which balls lie in the first urn.  Equivalently, we choose each ball's urn based on a fair coin flip, and total the heads for $N$ tosses.

There is also another intuitive way of seeing that $\pi$ is the stationary distribution for $(X_n)$.  We will see below in Chapter \ref{Ch:Asymptotics} that, for irreducible and aperiodic chains, the chain behaves likes its stationary distribution after many steps (see Theorem \ref{Thm:ConvToStationary} for the precise statement).  So for the walk $(X_n')$ on the hypercube, $X_n'$ is approximately uniformly distributed on $\{0,1\}^N$ when $n$ is large.  In particular, each coordinate $v_j$ on the list $(v_1, \ldots, v_N)$ is equally likely to be either 0 or 1, and so when adding up the 1 coordinates we obtain a Bin$(1/2,N)$ distribution.

\begin{figure}
\centering
\includegraphics[scale=0.32]{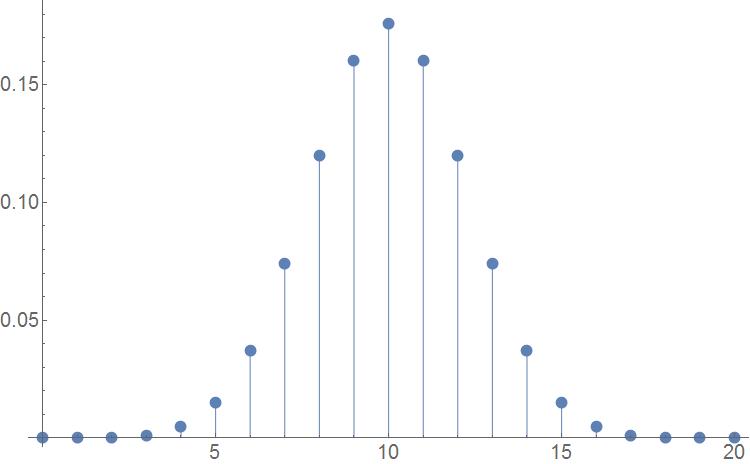}
\caption{{\small The stationary distribution $\pi$ for the Ehrenfest urn with $N=20$ particles.  We see $\pi$ is symmetric and that there is a strong preference for a balance of particles between the two containers.  This is not surprising given what we saw in Exercise \ref{Ex:EhrenfestTransProb}.}}
\label{Fig:Ehrenfest}
\end{figure}

\section{Bernoulli-Laplace diffusion}

A related urn model, the \emph{Bernoulli-Laplace diffusion}, starts with two urns again, but now with $N$ red and $N$ blue balls.  The evolution is identical to the Ehrenfest model: take a ball uniformly at random and switch its urn.  But now keep track of the number $X_n$ of \emph{red} balls in the first urn after $n$ steps, instead of the total number of balls.  Then $(X_n, n=0,1,\ldots)$ is a Markov chain on $\{0, 1, \ldots, N\}$. 

\begin{exercise}
\begin{enumerate}[$(a)$]
		\item Compute the transition probabilites $p(k,k-1), p(k,k)$ and $p(k,k+1)$ for Bernoulli-Laplace diffusion chain.
		\item Show that the hypergeometric distribution 
		\begin{align*}
			\pi(k) = \frac{\binom{N}{k}^2}{\binom{2N}{N}}, \quad k=0,1,\ldots, N
		\end{align*}
		is stationary for the chain by verifying the DBEs.
\end{enumerate}
\end{exercise}

%%Lecture 13%%%%%%%%%%%%%%%%%%%%%%%%%%%%%%%%%%%%%%%%%

\section{The P\'{o}lya urn}\label{Sec:PolyaUrn}

\subsection{Negative feedback vs positive feedback}
% This is obvious at the extreme ends: if all the balls are in the same urn, then the next step has to move a ball to the other urn.  Similarly,   
The preference in the Ehrenfest urn is to return the system to having approximately the same number of balls in both urns, as is evident from the transition probabilities of the chain: If one urn has more balls than the other, it is more likely that that urn loses a ball rather than gains one in the next step. This is a type of \emph{negative feedback}, where the system pushes against unbalances.  \emph{Positive feedback} models have the opposite behavior: growth in one direction encourages further growth in the same way.  These are the models where ``the rich get richer.''

\subsection{The P\'{o}lya urn}

The \textbf{P\'{o}lya urn model} is one such scheme.  Here's how it works: start with one black ball and one red ball together in a single urn.  At each step, pull out a ball uniformly at random, and then return it to this same urn with another ball of the same color.  Let $X_n$ be the number of black balls after $n$ steps.  Then $(X_n, n=0,1,\ldots)$ is a Markov chain on $\Omega = \mathbb{N}$.  

The transition probabilities are easy enough: in each step $X_n$ can either stay constant (if we pull out a red ball) or increase by one (if we pull out a black ball).  After $n$ steps there are a total of $n+2$ balls, and so we observe
\begin{multline*}
	\nP(X_{n+1}=k+1 \,|\, X_n = k)\\
	= \nP(\text{choose black in step }n+1 \, | \, X_n=k) = \frac{k}{n+2},
\end{multline*}
and hence $\nP(X_{n+1} = k \, | \, X_n = k) = 1 - \frac{k}{n+2}$.  

Notice there is something striking here: the transition probabilities \emph{depend on $n$}.  This is in contrast to all of the other Markov chains we have seen so far, and we therefore say the P\'olya urn is \textbf{time-inhomogeneous}.  In particular, for each $n$ we need a different transition matrix $P_n$ to describe the transition probabilities.

\subsection{Long-term behavior of the urn (or, life is like a P\'olya urn)}

How does the P\'olya urn behave in the long run?  The chain $(X_n)$ is a random non-decreasing sequence, and so we expect that it may tend to $\infty$ as $n$ grows.  The \emph{proportion} $\frac{X_n}{n+2}$ of black balls in the urn, however, is always between $0$ and $1$, and may therefore be a natural lense through which to view the chain's long-term behavior.  Take a look at Figure \ref{Fig:PoylaUrn}, which plots this proportion for several samples of the urn over 1000 steps.  

At least two things immediately jump out.  First, it appears that $\frac{X_n}{n+2}$ is limiting each time to a constant.  That is, each time we ``play'' the P\'olya urn game, the ratio of black balls after many steps basically becomes fixed.  But this fixed ratio $R$ is itself a \emph{random variable}, so that when we restart and play the game again, we end up with a different limiting ratio, or a different sampling of $R$.  By construction $R \in [0,1]$, and our main goal in this section is to prove the surprising result that $R$ is \emph{uniformly distributed} between 0 and 1.

A second observation from Figure \ref{Fig:PoylaUrn} is that the value of $R$ on a given run of the urn becomes apparent rather quickly: we only need 100 steps or so before there is very little further variation in this ratio.  That is, the first handful of steps in the P\'olya urn basically determine the long-term behavior of $\frac{X_n}{n+2}$.  If we may wax philosophical for a moment, perhaps this is not too unlike life itself, where often our early decisions chart the course of our lives for many years to come.  Life, it turns out, is like a P\'olya urn!

% We see that, perhaps surprisingly, this proportion appears to tend to a constant.  What is the distribution of this limiting constant?  In Theorems \ref{Thm:PolyaUrnpmf} and \ref{Thm:PolyaAsymptotics} we will show that it is nothing other than a Unif$(0,1)$ random variable. 

The uniformity of the limiting ratio $R$ becomes more plausible from the following theorem, which says that the count $X_n$ of black balls is always uniform.
\begin{theorem}\label{Thm:PolyaUrnpmf}
For any $n \geq 0$,
	\begin{align}\label{Eq:PolyaUrnXn}
		\nP(X_n = k) = \frac{1}{n+1}, \qquad k = 1,2, \ldots, n+1.
	\end{align}
	That is, $X_n$ is a $\Unif(\{1,2, \ldots, n+1\})$ random variable.
\end{theorem}
\noindent At first glance this is surprising, since the positive-feedback nature of the urn does not seem consistent with a uniform distribution.  Are the rich not getting richer?  But this theorem says we are equally likely to have any possible count of black balls after any step of the game.  If we draw any vertical line on Figure \ref{Fig:PoylaUrn}, the theorem says the intersection of the ratio $\frac{X_n}{n+2}$ with the line will be uniformly-distributed among its possible values.
\begin{proof}
	Let $B_n = X_n-1$ be the number of black balls added to the urn after $n$ steps.  We wish to calculate $P(B_n = k)$, and we observe that there are many ways to add $k$ black balls in $n$ steps, such as
	\begin{multline*}
		\underbrace{B,B, \ldots, B}_k, \underbrace{R,R,\ldots, R}_{n-k} \quad \text{or} \quad B,R, \underbrace{B,B, \ldots, B}_{k-1}, \underbrace{R,R,\ldots, R}_{n-k-1}\\
		\text{or} \quad \underbrace{R,R, \ldots, R}_{n-k}, \underbrace{B,B, \ldots, B}_k.
	\end{multline*}
    In total there are $\binom{n}{k}$ ways to add $k$ balls in $n$ steps, and the surprising thing is that they all have the same probability.  For example, for immediately adding $k$ black balls, we have
	\begin{align*}
		\nP(B,B, \ldots, B, R, R, \ldots, R) &= \underbrace{\frac{1}{2} \cdot \frac{2}{3} \, \cdots \frac{k}{k+1}}_{k\text{ black}} \cdot \underbrace{\frac{1}{k+2} \cdot \frac{2}{k+3} \, \cdots \, \frac{n-k}{n+1}}_{n-k\text{ red}}\\
		&= \frac{k!(n-k)!}{(n+1)!}
	\end{align*}
	\begin{exercise}
		Show that we also have
		\begin{align*}
		\nP(B,R, \underbrace{B,B \ldots, B}_{k-1}, \underbrace{R, R, \ldots, R}_{n-k-1}) &= \frac{k!(n-k)!}{(n+1)!} \quad \text{and} \\
		\nP(R, R, \ldots, R, B,B, \ldots, B) &= \frac{k!(n-k)!}{(n+1)!}.
		\end{align*}
	\end{exercise}
	What is happening is that the denominator is always the same, since in each step we always add one ball, and the numerator is also always the same, since we always have the same factors appearing, albeit in different orders.  Thus for $k=0,1, \ldots, n$,
	\begin{align*}
		\nP(B_n = k) = \binom{n}{k} \frac{k!(n-k)!}{(n+1)!} = \frac{n!}{k!(n-k)!} \cdot \frac{k!(n-k)!}{(n+1)!} = \frac{1}{n+1},
	\end{align*}
	which is independent of $k$.  We therefore have, for $k=1,2, \ldots, n+1$, that
	\begin{align*}
		\nP(X_n = k) = \nP(B_n = k-1) = \frac{1}{n+1},
	\end{align*}
	as claimed.
\end{proof}

We can now prove that the limiting proportion $R$ of black balls is uniform in $[0,1]$.

\begin{figure}
\centering
\includegraphics[scale=0.75]{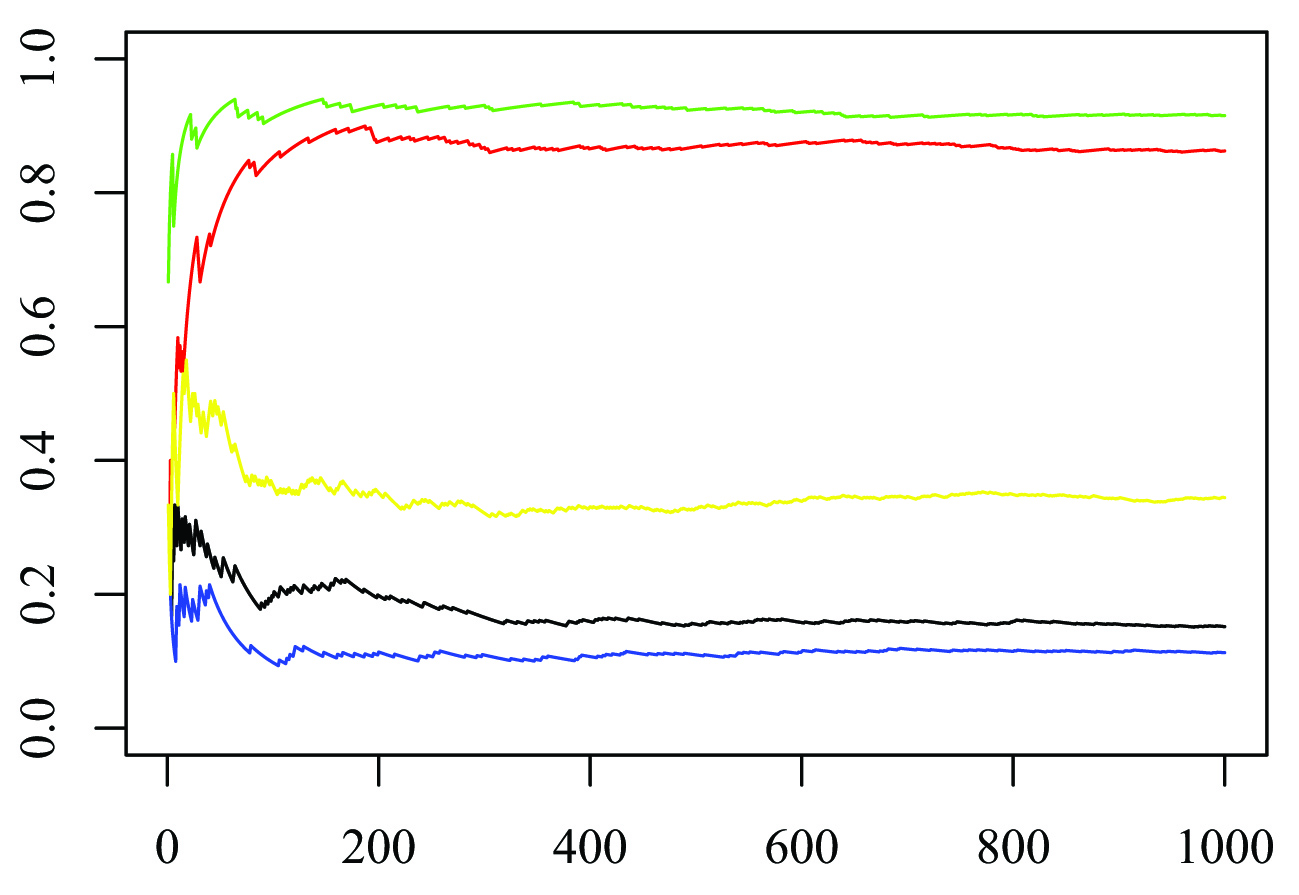}
\caption{{\small Five samples of 1000 steps of the P\'{o}lya urn, beginning with one black and one red ball.  The $y$-axis is the proportion of black balls $\frac{X_n}{n+2}$.  We numerically see that this proportion eventually stabilizes, limiting in a random proportion $P\in (0,1)$.  Theorem \ref{Thm:PolyaAsymptotics} gives the distribution of $P$.  Simulations by Celeste Zeng.  For more simulations, the R code behind them and further analysis, visit her github page at \url{https://celeste-zeng.github.io/Polya_Urn/}.}}
\label{Fig:PoylaUrn}
\end{figure}

\begin{theorem}\label{Thm:PolyaAsymptotics}
	The proportion $X_n/(n+2)$ of black balls converges in distribution to a $\Unif(0,1)$ random variable.  That is, for every $x \in [0,1]$,
	\begin{align}\label{Lim:Polya}
		\lim_{n \rightarrow \infty} \nP\left( \frac{X_n}{n+2} \leq x \right) = x.
	\end{align}
\end{theorem}

\begin{proof}
	We have by \eqref{Eq:PolyaUrnXn} that
	\begin{align}\label{Eq:PolyaUrnLargeN}
		\nP\left( \frac{X_n}{n+2} \leq x \right) = \nP(X_n \leq (n+2)x) = \frac{\lfloor (n+2)x \rfloor}{n+1},
	\end{align}
	where $\lfloor (n+2)x \rfloor$ is the ``floor'' of $(n+2)x$, which is the greatest integer less than or equal to $(n+2)x$.  For large values of $n$, we would like to replace the floor with $(n+2)x$ and say that it makes very little difference.  Indeed, observe that
	\begin{align*}
		\frac{\lfloor (n+2)x \rfloor}{(n+2)x} &= \frac{(n+2)x - \big((n+2)x - \lfloor (n+2)x \rfloor\big)}{(n+2)x}\\
		&= 1 - \frac{(n+2)x - \lfloor (n+2)x \rfloor}{(n+2)x} \rightarrow 1
	\end{align*}
	as $n\rightarrow \infty$ since $0 \leq |(n+2)x - \lfloor (n+2)x \rfloor|\leq 1$.  Hence sending $n\rightarrow \infty$ in \eqref{Eq:PolyaUrnLargeN} is the same as
	\begin{equation*}
		\lim_{n \rightarrow \infty} \frac{(n+2)x}{n+1} = x. \qedhere
	\end{equation*}
\end{proof}

You will work with a more general P\'{o}lya urn scheme, starting with $a$ black balls and $b$ red balls, in Problem \ref{Prob:B3.5}.

%Lecture 14%%%%%%%%%%%%%%%%%%%%%%%%%%%%%

\subsection{The P\'{o}lya urn and Bayesian statistics}

P\'{o}lya urns arise naturally in \emph{Bayesian statistics},\footnote{See the Wikipedia page for a helpful introduction to the topic.} in the context of the following simple problem: if we are flipping a coin with unknown probability $p$ of heads, how do we estimate $p$?  If $Y_n$ is the count of heads after $n$ tosses, then a reasonable guess is $\widehat{p} := Y_n/n$, the proportion of heads.\footnote{Indeed, in the statistics terminology this is the \emph{maximum likelihood estimator}, or the value of $p$ which has highest probability of being correct, given $Y_n$ is binomial.}

But what if $p$ itself is random?  In Bayesian statistics one assumes that $p$ has its own specific distribution, the \emph{prior belief}.  If we assume that $p$ is a $\Unif(0,1)$ random variable, then the sequence of the number of heads $(Y_1, Y_2, \ldots)$ turns out to be the \emph{same Markov chain} as the count of added black balls $(B_1, B_2, \ldots)$ for the P\'{o}lya urn.  This is another compelling and beautiful result about the P\'olya urn.

\begin{theorem}\label{Thm:RandomCoinFlips}
	Let $U$ be a $\Unif(0,1)$ random variable.  Given that $U=p$, let $I_n$, $n=1,2,\ldots$, be iid Bernoulli random variables with probability of success $p$.  Let $Y_n := \sum_{k=1}^n I_k$ be the number of successes after $n$ trials. Then the process $(1+Y_n)_{n \geq 1}$ is a P\'{o}lya urn starting with one black and one red ball.
\end{theorem}
The precise meaning of the conclusion is that $(1+Y_n)_{n \geq 1}$ is a Markov chain on $\mathbb{N}$ with the same transition probabilities as the P\'olya urn.  

This result is a consequence of De Finetti's Theorem,\footnote{A statement and proof may be found in \cite{Durrett}, for example.} which is sometimes called the ``fundamental theorem of Bayesian statistics.''  We will tackle the proof from a more elementary angle, though.
% , and will not invoke the De Finetti machinery.

\begin{proof}
	Let $(X_n)_{n \geq 1}$ be the P\'{o}lya urn.  Note that
	\begin{align*}
		\nP(X_{n+1} = k+2 \, | \, X_n = k+1) = \frac{k+1}{n+2},
	\end{align*}  
	and thus we want to show
	\begin{align}\label{Eq:BayesianPolyaGoal}
		\nP(Y_{n+1} = k+1 \, | \, Y_n = k) &= \nP(1 + Y_{n+1} = k+2 \, | \, 1 +Y_n = k+1) \notag\\
		&= \frac{k+1}{n+2}.
	\end{align}
	In fact, this is all we need to show: since $Y_{n+1} - Y_n \in \{0,1\}$, this implies
	\begin{align*}
		\nP(Y_{n+1} = k \, | \, Y_n = k) = 1 - \frac{k+1}{n+2} = \nP(X_{n+1} = k+1 \, | \, X_n = {k+1}),
	\end{align*}
	completing the proof.
	
	How do we obtain \eqref{Eq:BayesianPolyaGoal}?  Start with the definition
	\begin{align}\label{Eq:BayesianPolya2}
		\nP(Y_{n+1} = k+1 \, | \, Y_n = k) = \frac{\nP(Y_{n+1}=k+1, Y_n = k)}{\nP(Y_n=k)},
	\end{align}
	and note that $\nP(Y_n=k)$ is obvious if we are given $U$:
	\begin{align*}
		\nP(Y_n = k \, | \, U=p) = \binom{n}{k}p^k (1-p)^{n-k},
	\end{align*}
	since $Y_n$ is $\Bin(n,p)$ given $U=p$.  The tower property of conditional expectations thus gives
	\begin{align*}
		\nP(Y_n = k) = \nE(\nP(Y_n = k \, | \, U)) = \int_0^1 \binom{n}{k}p^k (1-p)^{n-k} dp,
	\end{align*}
	where the expectation is taken with respect to $U$, and we are integrating against the density $f_U(p) \equiv 1$ of $U$.  We similarly have
	\begin{align*}
		\nP(Y_{n+1}=k+1, Y_n = k) = \int_0^1 \binom{n}{k}p^{k+1} (1-p)^{n-k} dp.
	\end{align*}
	Combining \eqref{Eq:BayesianPolyaGoal} and \eqref{Eq:BayesianPolya2}, we thus need to show
	\begin{align}\label{Eq:BayesianPolya3}
		\frac{\int_0^1 \binom{n}{k}p^{k+1} (1-p)^{n-k} dp}{\int_0^1 \binom{n}{k}p^k (1-p)^{n-k} dp} = \frac{k+1}{n+2}.
	\end{align}
	The pressing question, then, is how to evaluate integrals like $\int_0^1 \binom{n}{k}p^k (1-p)^{n-k} dp$.
	Two clever ideas come to the rescue: first, introduce a parameter $\lambda >0$, and then note that
	\begin{align*}
		\sum_{k=0}^n \binom{n}{k} \lambda^k p^k (1-p)^{n-k} = (\lambda p + 1-p)^n
	\end{align*}
	by the binomial theorem.  Now we have something that we can easily integrate on the right-hand side:
	\begin{align*}
	\sum_{k=0}^n \lambda^k \int_0^1 \binom{n}{k}  p^k (1-p)^{n-k}dp &= \int_0^1 (\lambda p + 1-p)^n dp\\ 
	&= \frac{1}{n+1} \frac{\lambda^{n+1} -1}{\lambda -1} = \frac{1}{n+1} \sum_{k=0}^n \lambda^k.
	\end{align*}
	Since $\lambda$ is arbitrary, the coefficients of $\lambda^k$ on both sides must agree:
	\begin{align}\label{Eq:BayesianPolyaIntFormula}
		\int_0^1 \binom{n}{k}  p^k (1-p)^{n-k}dp = \frac{1}{n+1}
	\end{align}
	for all $k=0,1,\ldots, n$.  Hence we have the denominator of \eqref{Eq:BayesianPolya3}.  But also the numerator, if we note
	\begin{align*}
		\int_0^1 \binom{n}{k}p^{k+1} (1-p)^{n-k} dp &= \frac{\binom{n}{k}}{\binom{n+1}{k+1}} \int_0^1 \binom{n+1}{k+1} p^{k+1} (1-p)^{n+1-(k+1)} dp\\
		&= \frac{\binom{n}{k}}{\binom{n+1}{k+1}}\cdot \frac{1}{n+2}
	\end{align*}
	by \eqref{Eq:BayesianPolyaIntFormula} with $n$ replaced by $n+1$ and $k$ by $k+1$.  Therefore,
	\begin{align*}
		\frac{\int_0^1 \binom{n}{k}p^{k+1} (1-p)^{n-k} dp}{\int_0^1 \binom{n}{k}p^k (1-p)^{n-k} dp} &= \frac{\frac{\binom{n}{k}}{\binom{n+1}{k+1}}\cdot \frac{1}{n+2}}{ \frac{1}{n+1}}\\
		&= \frac{n!}{k!(n-k)!} \cdot \frac{(k+1)!(n-k)!}{(n+1)!} \cdot \frac{n+1}{n+2}\\
		&= \frac{k+1}{n+2},
	\end{align*}
	which was exactly our goal.
\end{proof}
We close by re-interpreting the uniform limit of $\frac{X_n}{n+2}$ in Theorem \ref{Thm:PolyaAsymptotics} in light of this result.  We have, by Theorem \ref{Thm:RandomCoinFlips}, that 
\begin{align}\label{Eq:PolyaLimitCoinFlips}
    \lim_{n \rightarrow \infty} \frac{X_n}{n+2} = \lim_{n \rightarrow \infty}\frac{1 + Y_n}{n+2} = \lim_{n \rightarrow \infty}\frac{Y_n}{n} = \lim_{n \rightarrow \infty}\frac{1}{n} \sum_{k=1}^n I_k,
\end{align}
which is a Unif$(0,1)$ random variable by Theorem \ref{Thm:PolyaAsymptotics}.  Given the right-hand side of \eqref{Eq:PolyaLimitCoinFlips}, we can thus interpret Theorem \ref{Thm:PolyaAsymptotics} as a type of law of large numbers.  When $p$ is fixed, the Law of Large Numbers says that the right-hand side of \eqref{Eq:PolyaLimitCoinFlips} converges with probability 1 to the number $p$ (add reference?).  Theorem \ref{Thm:PolyaAsymptotics} says that when $p$ is Unif$(0,1)$, the limit still recovers the correct probability distribution.

%%%%%%%%%%%%%%%%%%%%%%%%%%%%%%%%%%%%%%%%
\vfill
\pagebreak
\section*{Problems for chapter \ref{Chapter:ClassicalModels}}
\addcontentsline{toc}{section}{Problems for chapter \ref{Chapter:ClassicalModels}}

\begin{problem}\label{Prob:2Hitting}
Consider the random walk on $\{1,2,3,4,5\}$ with the following transition matrix:
\[
\begin{bmatrix}
\frac{1}{2} & \frac{1}{2} & 0 & 0 & 0\\
\frac{1}{2} & 0 & \frac{1}{2} & 0 & 0 \\
0 & \frac{1}{2} & 0 & \frac{1}{2} & 0 \\
0 & 0 & \frac{1}{2} & 0 & \frac{1}{2} \\
0 & 0 & 0 & \frac{1}{2} & \frac{1}{2}
\end{bmatrix}.
\]  
\begin{enumerate}[$(a)$]
\item Let $f_k=\nE_k(\tau_4)$ be the expected hitting time of $4$ starting at $k$. Find $f_k$ for $k=1,2,3,4,5$. 
\item Let $\tau$ be the first hitting time of $\{1,5\}$. Compute $\nE_3(\tau)$.
\end{enumerate}  
\end{problem}

\begin{problem}\label{Prob:A2Urn}  Consider an urn with balls of three colors: red, blue, and green. The total number of balls is $N$. Run the following Markov chain on the state space 
\begin{align*}
    \Omega=\{ (i,j,k) \in \mathbb{Z}^3:\; i\ge 0, j\ge 0, k\ge 0,\; i+j+k=N\}.
\end{align*} 
Pick a ball at random. If it is red, replace it with a blue or a green ball with equal probability. Similarly, if it is blue, replace it with a red or a green with equal probability. Finally, if it is green, replace it with a red or a blue with equal probability. Find a reversible stationary distribution for this chain.  (\emph{Hint:} Think of a standard named distribution that is a good guess for $\pi$. Then verify the DBEs to check for reversibility under $\pi$.) 

\end{problem}

%Takis -page 17

\begin{problem}\label{Prob:A4.1} Consider the Gambler's ruin problem for the biased random walk on the line. That is, consider the chain $(X_0, X_1, \ldots)$ on the integers $\mathbb{Z}$ where the transition probabilities are $P(j, j+1)=p$ and $P(j, j-1)=1-p$, for some $0< p < 1$, $p \neq 1/2$. Fix integers $0 < n$ and let $\tau = \tau_{\{0,n\}}$ be the hitting time of $\{0, n\}$. \begin{enumerate}[$(a)$]
	\item Show that 
\[
\nP_k(X_\tau=n) = \frac{(q/p)^k - 1}{(q/p)^n - 1}, \quad \text{where}\; q=1-p.
\]
\item Show that
\begin{align*}
	\lim_{p \rightarrow 1/2} \frac{(q/p)^k - 1}{(q/p)^n - 1} = \frac{k}{n},
\end{align*}
and explain the significance of this limit in this context.
\end{enumerate}
\end{problem}

%Takis page 72

\begin{problem}\label{Prob:A4.5}  Consider the star graph that has one center vertex $0$ that is connected to $N$ arms of lengths $l_1, l_2, \ldots, l_N$, for positive integers $l_i$'s. See Figure \ref{Fig:StarGraphEx} for an example with $N=3$, $l_1=3, l_2=4$ and $l_3=1$. Suppose you start a random walk on this graph starting at the center $0$. What is the probability that it will reach the end of the first arm before it reaches the end of any other arm? Write it as a formula in $l_1, \ldots, l_N$.  Your formula should give $\frac{4}{19}$
 for the graph in Figure \ref{Fig:StarGraphEx}.
 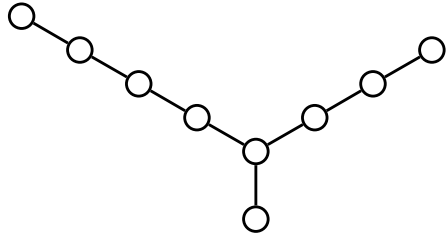
\begin{figure}
    \centering
    
    \scalebox{0.9}{
    \begin{tikzpicture}
    \tikzset{
        node/.style={circle, draw=black, very thick},
        arrow/.style={-{Latex[length=2mm]}, black},
        line/.style={black, very thick}
        }
    
    \node[node] (zero) at (0,0) {};
    
    \node[node] (one) at (0,-1) {};

    \node[node] (two) at (0.87,0.5) {};
    \node[node] (three) at (1.73,1) {};
    \node[node] (four) at (2.6,1.5) {};
    
    \node[node] (five) at (-0.87,0.5) {};
    \node[node] (six) at (-1.73,1) {};
    \node[node] (seven) at (-2.6,1.5) {};
    \node[node] (eight) at (-3.46,2) {};
     
    \draw[line] (zero) -- (one);
    \draw[line] (zero) -- (two);
    \draw[line] (zero) -- (five);
    \draw[line] (two) -- (three);
    \draw[line] (three) -- (four);
    \draw[line] (five) -- (six);
    \draw[line] (six) -- (seven);
    \draw[line] (seven) -- (eight);

    \end{tikzpicture}}
    
    \caption{\small A star-like graph considered in problem \ref{Prob:A4.5}.}
    \label{Fig:StarGraphEx}
\end{figure}

\end{problem}

\begin{problem}\label{Prob:B3.4} Consider the random walk on the $n$-cycle. Suppose the random walk starts at $1$. We are interested in the expectation of the cover time $\tau_n$ which is the first time when the walk has visited every vertex of the cycle.  

\begin{enumerate}[$(a)$]
\item Let $\tau_{n-1}$ be the first time when all but exactly one of the vertices of the $n$-cycle has been visited. That is, $X_{\tau_{n-1}}$ visits the last but one new vertex to be visited. Use Gambler's Ruin to show that 
\[
\nE\left( \tau_n - \tau_{n-1}\right)= n-1.
\] 
\emph{Hint:} Look at the set of all already visited vertices at $\tau_{n-1}$. Suppose $X_{\tau_{n-1}}=5$ and vertex $6$ is the last vertex to be yet visited. Then the random walk can either go to $6$ from $5$, or go all the way in the reverse direction and get to $6$ from $7$. 
 
\item Repeat the argument above to show that if $\tau_k$ is the stopping time which is the first time $k$ vertices in the cycle have been visited, then 
\[
\nE(\tau_k - \tau_{k-1})=k-1.
\]
\item Show that the expected cover time is given by the formula
\[
\nE(\tau_n)= \binom{n}{2}.
\]
\end{enumerate}
\end{problem}

\begin{problem}\label{Prob:Lec28.2025}
Consider again the random walk on the $n$-cycle ($n\ge 2$) and suppose that the random walk starts at $1$. Let $L$ be the last new vertex that the walker visits. Prove that $L$ is uniformly distributed over $\{2, \ldots, n\}$, i.e., $\nP(L=k)=\frac{1}{n-1}$, for $k\in \{2,3,\ldots, n\}$. Say $n=5$ and the random walker visits vertices in the following sequence $1,2,1,5,4,5,4,3$. Then $L=3$ is the last new vertex visited. (This is a surprising result since you'd expect vertices further away from $1$ to be more likely to be visited last.) 

\emph{Hint:} Consider the event $\{L=k\}$. This can happen in two disjoint ways. Either the walker visits $k-1$ before $k+1$, and then visits $k+1$ before $k$. Or, the walker visits $k+1$ before $k-1$ and then visits $k-1$ before $k$. Use Gambler's ruin to compute the two probabilities. 
\end{problem}

\begin{problem}\label{Prob:B3.5} Fix $a,b \in \mathbb{N}$ and consider a P\'olya urn starting with $a$ black balls and $b$ red balls. As usual, for each turn you pick a ball at random and then return it to the urn along with a ball with the same color. Let $X_n$ denote the number of black balls in the urn after the $n$th turn is completed. 
\begin{enumerate}[$(a)$]
\item Show that, for all $n=1,2,3,\ldots$, 
\[
\nE\left( \frac{X_n}{n+a+b} \mid X_{n-1} \right)= \frac{X_{n-1}}{n+a+b-1}.
\]
\item Use the above (or otherwise) to compute $\nE(X_n)$ for every $n$.  Your answer should only depend on $a$, $b$ and $n$.
\item Use the above (or otherwise) to show the amazing formula
\[
\nP(\text{$n$th pick is black})=\frac{a}{a+b}, \quad \text{for all $n\ge 1$}.
\]  
\item Show that, similarly, 
\[
\nP(\text{$n$th pick is black} \mid \text{first two picks are black})=\frac{a+2}{a+b+2}, \quad \text{for all $n\ge 3$}.
\]  
\end{enumerate}
\end{problem}

\begin{problem}\label{Prob:B3.6} Imagine the queue in front of an ice cream shop. Suppose there are currently $k$ people standing in the queue. In the next time period, the first person in the queue gets his/her ice cream and leaves with probability $1/2$. Independently, another person joins the queue with probability $1/2$. If the queue is currently empty, there is $1/2$ probability that someone will join the queue in the next time period. Suppose the shop can only accommodate at most a queue of $N$ people after which no more new customers are allowed to join the queue until someone leaves. 

\begin{enumerate}[$(a)$] 
\item Let $X_n$ be the number of people standing in the queue at time $n$. This is a Birth-and-Death chain. Find its state space and transition probabilities. Is it irreducible? Aperiodic? 
\item What is the probability that starting with a $k$ person queue, the shop will become full before it completely empties?    
\item Find the stationary distribution for this chain.  
\item Suppose $N=\infty$, i.e., there is no upper bound on the length of the queue. Consider a lazy random walk $\{ S_k,\; k=0,1,2,3,\ldots\}$ on the integers with transition probabilities
\[
p(j,j+1)=\frac{1}{4}, \; p(j,j-1)=\frac{1}{4}, \; p(j,j)=\frac{1}{2}.
\] 
Show that the absolute value process $\left\{ \abs{S_k},\; k=0,1,2,3, \ldots \right\}$ is the same Birth-and-Death chain as $X$ by find its transition probabilities. 
\item Find the expected time it will take from an empty queue ($0$ people) to become full ($N$ people) for the first time.  
\end{enumerate}
\end{problem}

\begin{problem}\label{Prob:A3.2} Consider a sequence of coin flips with $\nP(H)=p$, for some $0 < p < 1$. Define a Markov chain with state space $\{ TT,TH,HT,HH\}$ which shows the outcome of the previous toss and the current one. For example, $X_n=TH$ means that the $(n-1)$th toss was $T$ and the $n$th toss was $H$. 
\begin{enumerate}[$(a)$]
\item Start flipping the coin.  How many tosses, on average, does it take to get two heads in a row?  Use Markov chain methods.
\item As we flip, how many times do we expect to see $HT$ before getting two heads in a row? 
\end{enumerate}
\end{problem}

\vfill
\pagebreak

%%%%%%%%%%%%%%%%%%%%%%%%%%%%%%%%%%%%%%%%%%

\chapter{Asymptotic behavior of Markov chains}\label{Ch:Asymptotics}

\section{Asymptotics of Markov chains}\label{Sec:AsymptoticsIntro}

Many problems in Chapter 1 asked you to compute high powers of a transition matrix $P$.  The idea was to help you develop some intuition for the long-term behavior of the chain.  This is a crucial point of analysis: after a large number of steps $n$, what is the probability distribution of $X_n$?  You may have noticed that often, although not always, high powers $P^n$ of $P$ became close to the matrix where every row is the stationary distribution $\pi$.  This striking fact suggests that the distribution of $X_n$ limits to nothing other than the chain's stationary distribution, irrespective of the starting state.

In this chapter we make this observation rigorous, and provide assumptions that ensure this occurs.  The conditions are simple and familiar: the chain must be irreducible and aperiodic.   Our first main theorem in this chapter, Theorem \ref{Thm:ConvToStationary}, says that such chains converge to their stationary distribution.  More than that, it says that the convergence occurs \emph{exponentially fast}.  This is why the limiting behavior is often apparent after a relatively small number of steps.

What does it mean, however, for the distribution of $X_n$ to \emph{converge} to a given distribution?  In order to make sense of this we need some way of describing how far two probability distributions are from each other.  To this end we use the \emph{total variance distance}, which we introduce and study in \S\ref{Sec:TotalVariation}.  

Prior to that, however, we consider in \S\ref{Sec:Lazy} how to handle the absence of one of the above conditions.  Namely, what if our chain is \emph{not} aperiodic? Recall the example of the 6-cycle, which has period 2.  We saw in Example \ref{Eg:Hexagon2} that the 50th and 51st powers of the transition matrix are
\begin{align*}
P^{50} &\approx        \left[
\begin{array}{cccccc}
 0.333 & 0 & 0.333 & 0 & 0.333 & 0 \\
 0 & 0.333 & 0 & 0.333 & 0 & 0.333 \\
 0.333 & 0 & 0.333 & 0 & 0.333 & 0 \\
 0 & 0.333 & 0 & 0.333 & 0 & 0.333 \\
 0.333 & 0 & 0.333 & 0 & 0.333 & 0 \\
 0 & 0.333 & 0 & 0.333 & 0 & 0.333 \\
\end{array}
\right]
\end{align*}
and
\begin{align*}
P^{51} &\approx \left[
\begin{array}{cccccc}
 0 & 0.333 & 0 & 0.333 & 0 & 0.333 \\
 0.333 & 0 & 0.333 & 0 & 0.333 & 0 \\
 0 & 0.333 & 0 & 0.333 & 0 & 0.333 \\
 0.333 & 0 & 0.333 & 0 & 0.333 & 0 \\
 0 & 0.333 & 0 & 0.333 & 0 & 0.333 \\
 0.333 & 0 & 0.333 & 0 & 0.333 & 0 \\
\end{array}
\right].
\end{align*}
There is no hope for a limiting distribution here over \emph{all} steps, because on each even step the chain can only be at half of the vertices, while on each odd step it can only be on the other half.  So, just like for the sequence of real numbers $\{(-1)^n\}$,  there is no limit.  

We have a simple work-around, though, which you have also already encountered in Problem \ref{Prob:A1.5}.  Making the chain \emph{lazy} turns it into an aperiodic chain, while crucially not altering the stationary distribution.  Hence by our upcoming  convergence theorem, the lazy version of the walk will still converge to stationarity.

\subsection{Lazy Markov chains}\label{Sec:Lazy}

% Our long-term behavior theorems will be for \emph{aperiodic} Markov chains.  The point of this section is that this is no great loss of generality: given any chain, we can make it aperiodic by a process we could call ``lazyfication''.

\begin{definition}
	For a Markov chain $X$ with $n\times n$ transition matrix $P$, the \textbf{lazy version of $X$} is the Markov chain $\tilde{X}$ with transition matrix
	\begin{align}\label{Eq:LazyTransMatrix}
		\tilde{P} := \frac{1}{2}(I_n+ P),
	\end{align}
	where $I_n$ is the $n\times n$ identity matrix.
	\end{definition}
\begin{example}
    Let's think about the 6-cycle again, whose transition matrix is
\begin{align*}
   P= \begin{bmatrix}
        0 & 1/2 & 0 & 0 & 0 & 1/2\\
        1/2 & 0 & 1/2 & 0 & 0 & 0\\
        0 & 1/2 & 0 & 1/2 & 0 & 0\\
        0 & 0 & 1/2 & 0 & 1/2 & 0\\
        0 & 0 & 0 & 1/2 & 0 & 1/2\\
        1/2 & 0 & 0 & 0 & 1/2 & 0
    \end{bmatrix},
\end{align*}
and hence whose lazy version has transition matrix
\begin{align*}
   \tilde{P}= \begin{bmatrix}
        1/2 & 1/4 & 0 & 0 & 0 & 1/4\\
        1/4 & 1/2 & 1/4 & 0 & 0 & 0\\
        0 & 1/4 & 1/2 & 1/4 & 0 & 0\\
        0 & 0 & 1/4 & 1/2 & 1/4 & 0\\
        0 & 0 & 0 & 1/4 & 1/2 & 1/4\\
        1/4 & 0 & 0 & 0 & 1/4 & 1/2
    \end{bmatrix}.
\end{align*}
What has happened?  You can see the effect of adding the identity matrix is to make each diagonal entry $\tilde{P}_{jj}>0$.  That is, now we have the possibility of remaining at the same state for a unit of time.  In other words, loops are added in the graph.  We could say that we now have to flip a coin at the start of each step.  Heads, we stay at the same vertex.  Tails, we move according to the original transition probabilities.  
\end{example}
The lazy walk has the following properties, the most important of which are $(iii)$ and $(iv)$: our lazy walk is \emph{always} aperiodic, which is obvious from the loops, and furthermore we have not disrupted the stationary distribution.
\begin{lemma}
    The lazy random walk defined by \eqref{Eq:LazyTransMatrix} satisfies the following.
    \begin{enumerate}[$(i)$]
        \item $\tilde{P}$ is a transition matrix.
        \item If the original chain $X$ is irreducible, so is the lazy chain $\tilde{X}$.
        \item $\tilde{X}$ is aperiodic.
        \item If $\pi$ is stationary for $P$, $\pi$ is also stationary for $\tilde{P}$.
    \end{enumerate}
\end{lemma}
\begin{proof}
For $(i)$, all the entries of $\tilde{P}$ are obviously non-negative.  So we just need the rows to sum to one.  Setting $\mathbf{1} := (1,1, \ldots, 1)$, the row vector of all ones, we have
\begin{align*}
	\tilde{P} \mathbf{1}^T = (\frac{1}{2}I + \frac{1}{2} P)\mathbf{1}^T &= \frac{1}{2} \mathbf{1}^T + \frac{1}{2} P \mathbf{1}^T = \frac{1}{2} \mathbf{1}^T + \frac{1}{2} \mathbf{1}^T = \mathbf{1}^T,
\end{align*}
as needed.  

For $(ii)$, we have added loops to the chain, which does not disconnect it. 

To see $(iii)$, observe that we can always return to a given state in one step.
\end{proof}
\noindent It is a good (and simple) exercise for you to prove part  $(iv)$.  
\begin{exercise}\label{Ex:LazyStationary}
	Suppose $P$ has a stationary distribution $\pi$.  Show that $\pi$ is also stationary for $\tilde{P}$.
\end{exercise}

How do hitting times change under $\tilde{P}$?  If you have to flip a coin and get heads to be allowed to move each step, how would that affect the number of steps you need to reach $A \subset \Omega$?  The following theorem answers this, but we encourage you to guess for yourself before reading it.  What do you intuitively think?

\begin{theorem}
	Let $X$ be an irreducible chain on $\Omega$, $A \subset \Omega$, and let $\tau_A$ be the hitting time of $A$.  Let $f_k := \nE_k(\tau_A)$ be the expected hitting time of $A$ under $P$ starting from $k$, and $\tilde{f}_k := \tilde{\nE}_k(\tau_A)$ the corresponding expectation under the lazy chain $\tilde{P}$.  Then
	\begin{align*}
		\tilde{f}_k = 2f_k \quad \text{ for all } k.
	\end{align*}
\end{theorem}
For simplicity we assume that $X$ is irreducible, although it is possible to make a more general statement. Our proof will use the following elementary fact.
\begin{exercise}
    Let $M$ be an invertible $n \times n$ matrix, and $c\neq 0$.  Show that $cM$ is invertible and that $(cM)^{-1} = c^{-1}M^{-1}$.
\end{exercise}
\begin{proof}
    Recall Theorem \ref{Thm:ExpectedHittingTimes} explicitly gives $f_k$ by the row-sum formula \eqref{Eq:ExpectedHittingTime}.  To find the times $\tilde{f}_k$, we simply need the matrix $\tilde{M}$ corresponding to $\tilde{P}$, which, recalling \eqref{Eq:ReturnTimeMMatrix}, is 
    \begin{align*}
        \tilde{M} = \Big(I - \frac{1}{2}(I+P)\Big)^{-1} = \Big(\frac{1}{2}(I - P)\Big)^{-1} = 2(I-P)^{-1} = 2M.
    \end{align*}
Thus the formula \eqref{Eq:ExpectedHittingTime} for $\tilde{\nE}_k(\tau_A)$ yields twice that for  $\nE_k(\tau_A)$.
\end{proof}

\subsection{Total variation distance of probability distributions}\label{Sec:TotalVariation}

We wish to say that the distribution
\begin{align*}
	P^n_{xk} := \nP(X_n = k \, | \, X_0 =x), \qquad k \in \{1,2,\ldots, N\}
\end{align*}
of the location of our chain after $n$ steps converges to some fixed probability distribution as $n \rightarrow \infty$.  However, we currently lack a precise meaning for the \emph{convergence} of a probability distribution.  We address this by introducing the \emph{total variation norm}, which gives us a way to measure the ``distance'' between two probability distributions.  Once we have this, we can say that a sequence of probability distributions $(\mu_0, \mu_1, \ldots )_{n \geq 0}$ converges to a distribution $q$ if the total variation between the $\mu_n$ and $q$ goes to zero.

\begin{definition}
	Let $\Omega = \{1,2, \ldots, n\}$ be a sample space with probability mass functions $p$ and $q$ on $\Omega$.  That is, $p$ and $q$ are functions from $\Omega$ to $[0,1]$ with $1 = \sum_{k=1}^np_k = \sum_{k=1}^n q_k$.  The \textbf{total variation (TV) distance} between $p$ and $q$ is
	\begin{align}\label{Def:TVNorm}
		\Norm{p-q}_{TV} := \frac{1}{2} \sum_{k=1}^n |p_k - q_k|.
	\end{align}
\end{definition}
\noindent Here $p_k = p(k)$ and $q_k = q(k)$ are, of course, the values $p$ and $q$ give to vertex $k \in \Omega$. We will interchangeably use ``TV-distance'' and ``TV-norm'' for $\|p-q\|_{TV}$.

% \footnote{To be precise, $\| \cdot \|_{TV}$ defines a \emph{norm} on $\mathbb{R}^n$, and the distance function $d_{TV}(p,q) := \|p-q\|_{TV}$ is the associated \emph{metric} on $\mathbb{R}^n$.  Since our focus is not abstract metric space theory, we will not be too worried about using this terminology without absolute precision.}

The total variance distance satisfies some nice properties.  In particular, part $(iv)$ of the following lemma gives us helpful intuition: the TV-distance is the biggest discrepancy we can see when comparing the probability of the same event under both distributions.

\begin{lemma}\label{Lemma:TVNorm}
	Let $p,q$ and $r$ be probability distributions on $\Omega$.  The total variation norm has the following properties:
	\begin{enumerate}[$(i)$]
		\item It is a \emph{metric}.  That is, $\Norm{p-q}_{TV} \geq 0$, $\Norm{p-q}_{TV} = \Norm{q-p}_{TV}$, and 
		\begin{align*}
			\Norm{p-q}_{TV} \leq \Norm{p-q}_{TV} + \Norm{q-r}_{TV}.
		\end{align*}
		\item\label{Lemma:TVNormOneBound} $\Norm{p-q}_{TV} \leq 1$.
        \item 
        \begin{align}\label{Eq:TVBset}
        \|p-q\|_{TV} = \sum_{p_k\geq q_k}(p_k-q_k).
        \end{align}
		\item $|p_k - q_k| \leq 2 \Norm{p-q}_{TV}$ for all $k$.
		\item \label{Lemma:TVNormEventBound}\begin{align}\label{Eq:TVMax}
		    \Norm{p-q}_{TV} = \max_{A \in 2^\Omega}\,\big(\nP_p(A)-\nP_q(A)\big),
		\end{align}
		where $\nP_p(A)$ is the probability of the set $A$ under the distribution $p$, $\nP_q(A)$ its probability under $q$, and $2^\Omega$ the power set of $\Omega$.
	\end{enumerate}
\end{lemma}
\noindent Note that the notation in \eqref{Eq:TVBset} means that the sum is over all elements $k$ such that $p_k \geq q_k$.  In \eqref{Eq:TVMax}, observe that $\max_{A \in 2^\Omega}\,\big(\nP_p(A)-\nP_q(A)\big) = \max_{A \in 2^\Omega}\,\big|\nP_p(A)-\nP_q(A)\big|$, since 
\begin{align*}
    \nP_p(A^c)-\nP_q(A^c) = \big(1 - \nP_p(A) \big) - \big( 1- \nP_q(A)\big) = \nP_q(A) - \nP_p(A).
\end{align*}

\begin{proof}
		For $(i)$, the first two properties are clear from the definition \eqref{Def:TVNorm}.  The third property follows from the fact that $|p_k - r_k| \leq |p_k - q_k| + |q_k - r_k|$, the familiar \emph{triangle inequality} for absolute values.
		
		For $(ii)$, set 
		\begin{align}\label{Eq:TVPfB}
		    B := \{ k \in \Omega \; : \; p_k \geq q_k\}.
		\end{align}
		Then $B^c = \{ k \in \Omega \; : \; p_k < q_k\}$, and we observe
		\begin{align}
		    \Norm{p-q}_{TV} &= \frac{1}{2} \sum_{k \in B} |p_k - q_k| + \frac{1}{2} \sum_{k \in B^c} |p_k - q_k| \notag\\
			&= \frac{1}{2} \sum_{k \in B} (p_k - q_k) - \frac{1}{2} \sum_{k \in B^c} (p_k - q_k) \label{Eq:TVPfStart}\\
			&= \frac{1}{2} (\nP_p(B) - \nP_q(B)) - \frac{1}{2} (\nP_p(B^c) - \nP_q(B^c)) \notag\\
			&= \frac{1}{2} (\nP_p(B) - \nP_q(B)) - \frac{1}{2} (1 - \nP_p(B) - 1 + \nP_q(B))\notag\\
			&= \nP_p(B) - \nP_q(B) \label{Eq:TVSetB}\\
			&\leq \nP_p(B) \leq 1. \notag
		\end{align}
Thus $(ii)$ holds, and we have also shown $(iii)$ in \eqref{Eq:TVSetB}.
  
		For $(iv)$, we note 
		\begin{equation*}
			|p_k - q_k| \leq \sum_{j=1}^n |p_j - q_j| \leq 2 \Norm{p-q}_{TV}. 
		\end{equation*}
		
Lastly, for $(v)$, we observe that by \eqref{Eq:TVSetB} the set $B$ in \eqref{Eq:TVPfB} satisfies
  \begin{align}\label{Eq:TVBOneDirection}
    \|p-q\|_{TV}=\nP_p(B) - \nP_q(B) \leq \max_{A \in 2^\Omega}\,\big(\nP_p(A)-\nP_q(A)\big).
  \end{align}
We next claim that  
\begin{align}\label{Ineq:TVBNeed}
    \nP_p(A)-\nP_q(A) \leq \nP_p(B)-\nP_q(B)
\end{align}
for any $A \in 2^\Omega$.  To see this, write 
\begin{align}
    \nP_p(A)&-\nP_q(A) \notag\\ 
    &= \nP_p(A\cap B) -\nP_q(A\cap B) + \nP_p(A \cap B^c) - \nP_q(A \cap B^c), \label{Ineq:TVProp4}
\end{align}
and then first observe that
\begin{align*}
    \nP_p(A\cap B) -\nP_q(A\cap B) \leq \nP_p(B) - \nP_q(B),
\end{align*}
as adding more elements of $B$ only increases the sum, by definition of $B$.  Also by definition of $B$,
\begin{align*}
    \nP_p(A \cap B^c) - \nP_q(A \cap B^c) \leq 0,
\end{align*}
and so applying these two observations to \eqref{Ineq:TVProp4} yields
\begin{align*}
    \nP_p(A)-\nP_q(A) \leq \nP_p(B)-\nP_q(B) + 0,
\end{align*}
as claimed.  Since $A \in 2^\Omega$ was arbitrary, this inequality also holds in the max,
\begin{align*}
    \max_{A \in 2^\Omega}\,\big(\nP_p(A)-\nP_q(A)\big) \leq \nP_p(B)-\nP_q(B),
\end{align*}
which yields equality in \eqref{Eq:TVBOneDirection} and thus completes the argument.
\end{proof}

Let's look at an explicit computation to see how the TV norm operates in practice.

\begin{example}\label{Eg:TVComputation}
    How close to being uniformly distributed is a walk on the 5-cycle after three steps?  That is, suppose we start a simple symmetric random walk on the 5-cycle at vertex 1.  What is the TV-norm between the resulting distribution $p$ after three steps from the uniform distribution
    \begin{align*}
        q := \Big( \frac{1}{5}, \frac{1}{5}, \frac{1}{5}, \frac{1}{5}, \frac{1}{5} \Big)?
    \end{align*}
    To answer this, we recall the transition matrix $P$ for the 5-cycle is 
    \begin{align*}
       P= \begin{bmatrix}
 0 & \frac{1}{2} & 0 & 0 & \frac{1}{2} \\
 \frac{1}{2} & 0 & \frac{1}{2} & 0 & 0 \\
 0 & \frac{1}{2} & 0 & \frac{1}{2} & 0 \\
 0 & 0 & \frac{1}{2} & 0 & \frac{1}{2} \\
 \frac{1}{2} & 0 & 0 & \frac{1}{2} & 0 \\
\end{bmatrix},
    \end{align*}
    and so the three-step transition matrix is 
    \begin{align*}
        P^3 = \begin{bmatrix}
 0 & \frac{3}{8} & \frac{1}{8} & \frac{1}{8} & \frac{3}{8} \\
 \frac{3}{8} & 0 & \frac{3}{8} & \frac{1}{8} & \frac{1}{8} \\
 \frac{1}{8} & \frac{3}{8} & 0 & \frac{3}{8} & \frac{1}{8} \\
 \frac{1}{8} & \frac{1}{8} & \frac{3}{8} & 0 & \frac{3}{8} \\
 \frac{3}{8} & \frac{1}{8} & \frac{1}{8} & \frac{3}{8} & 0 \\
\end{bmatrix},
    \end{align*}
    which shows that the distribution $p_3$ of $X_3$, given that $X_0=1$, is the first row of $P^3$,
    \begin{align*}
        p_3 = \Big( 0, \frac{3}{8}, \frac{1}{8}, \frac{1}{8}, \frac{3}{8}  \Big).
    \end{align*}
    Using the definition \eqref{Def:TVNorm}, we have that the TV-norm to the uniform distribution is thus
    \begin{align*}
        \|p_3-q\|_{TV} = \frac{1}{2}\Big( \Big|0 - \frac{1}{5}\Big| + \Big|\frac{3}{8}- \frac{1}{5}\Big| + \Big| \frac{1}{8}- \frac{1}{5}\Big| + \Big| \frac{1}{8}- \frac{1}{5}\Big| + \Big| \frac{3}{8}- \frac{1}{5} \Big|  \Big) = \frac{7}{20}.
    \end{align*}
    Note that we have $\|p_3-q\|_{TV} \leq 1$, as must be the case from Lemma \ref{Lemma:TVNorm}($\ref{Lemma:TVNormOneBound}$).  By part $(\ref{Lemma:TVNormEventBound})$ of the lemma, we also know that the difference between the probability that $X_3$ is in any collection  $B$ of vertices and the probability that a uniformly-chosen vertex $Y$ is in the same collection $B$ is at most $7/20=0.35$.  So, for instance, for the set of vertices $B=\{1,3,5\}$, we see
    \begin{align*}
        |\nP(X_3 \in B \, | \, X_0=1)-\nP(Y \in B)| &\leq \max_{A \subset V}|\nP(X_3 \in A \, | \, X_0=1)-\nP(Y \in A)|\\
        &= \|p_3-q\|_{TV}=0.35,
    \end{align*}
    where the first equality is by \eqref{Eq:TVMax}.  Equivalently,
    \begin{align*}
        \nP(Y \in B) - 0.35 \leq \nP(X_3 \in B \, | \, X_0=1) \leq \nP(Y \in B) + 0.35.
    \end{align*}
    So beyond giving us a way to measure distance between two probability distributions, we also see part of the power of the TV-norm is that it gives an upper bound on the difference in probability of the same event under the two distributions.
    
    We will see below in Theorem \ref{Thm:ConvToStationary} that as $n \rightarrow \infty$, the distribution $p_n$ of $X_n$ after $n$ steps converges to the uniform distribution $q$ on the vertices.
\end{example}

\subsubsection{Excursus}
We pause to make a couple of asides that we will not subsequently use but are important nevertheless.

\paragraph{Other norms on finite probability distributions.} There are other natural ways that we could assign a distance between $p$ and $q$.  If we think of $p = (p_1, \ldots, p_N)$ and $q=(q_1, \ldots, q_N)$ as vectors in $\mathbb{R}^N$, we could take the Euclidean distance
\begin{align}\label{Eq:TwoNorm}
    \Norm{p-q}_2 = \Big( \sum_{k=1}^N (p_k-q_k)^2 \Big)^{1/2},
\end{align}
for example, which is how far these vectors are from each other in $N$-dimensional space.  The TV-distance \eqref{Def:TVNorm} is not the same and usually gives a different number.  While the Euclidean norm is more geometrically natural, the TV-norm is probabilistically more suitable, as we get properties like Lemma \ref{Lemma:TVNorm} $(iv)$, which do not hold for the Euclidean norm.  

We could also replace the ``2'' in \eqref{Eq:TwoNorm} with any power $r \geq 1$ to obtain $\|p-q\|_r$, with $r=1$ corresponding to the TV-distance up to the factor of $1/2$.  All of these give distinct measurements of distance between $p$ and $q$, but there is a sense in which they are all ``comparable'' to the TV-distance.  More precisely, one can prove that for any $r\geq 1$ there exists a constant $C = C(r,N)$ such that 
\begin{align}\label{Ineq:EquivNorms}
    \frac{1}{C}\|p-q\|_r \leq \|p-q\|_{TV} \leq C\|p-q\|_r 
\end{align}
for any $p,q \in \mathbb{R}^N$.  In particular, this inequality shows that $\|p_n-q\|_{TV} \rightarrow 0$ as $n \rightarrow \infty$ if and only if $\|p_n-q\|_r \rightarrow 0$.  So with respect to convergence, we cannot tell the difference between the TV norm and any of the other $r$-norms.

We, however, are concerned with the \emph{rate} of convergence, which is to say, with how many steps it takes for distribution of the $n$th step of the chain to be close to the stationary distribution, as $N$ grows.  For example, we may want to know how long it takes for the simple random walk on the $N$-cycle to be within a small amount $\epsilon$ of stationarity, in terms of $N$.  Does the required number of steps grow like $N^2$?  Like $N \log(N)$?  Like $e^N$?  From this point of view, the different norm choices do all not behave the same, as the constant $C(r,N)$ in \eqref{Ineq:EquivNorms} can grow very quickly as $N$ increases.  So the selection of a specific norm \emph{does} matter when one wants to concretely estimate the needed number of steps, and we continue to exclusively use $\| \cdot \|_{TV}$ in what follows.  

\paragraph{The method of couplings.} It is hard to exactly compute the total variation distance between two probability distributions. For example, let $p$ denote the binomial $\Bin(n, 1/2)$ distribution and let $q$ be $\Bin(n,1/3)$. It is hard to exactly compute $\norm{p-q}_{TV}$ exactly even for such commonplace distributions. However, one may easily produce a good upper bound on $\norm{p-q}_{TV}$ using the very useful \textit{method of couplings}.  

A \textit{coupling} of two probability distributions $p$ and $q$ refers to any construction of two random variables $X$ and $Y$ on a common probability space such that $X$ has distribution $p$ and $Y$ has distribution $q$. What is implicit is the joint distribution of $(X,Y)$, the correct choice of which plays a crucial role. For example, when $p$ and $q$ are two binomials as above, a coupling can be a joint distribution of two independent random variables $X$ and $Y$, where $X\sim p$ and $Y\sim q$. However, this simple coupling is rarely useful. Here is a nontrivial coupling of the two binomials. Let $U_1, U_2, \ldots, U_n$ be iid random variables with distribution $\Unif(0,1)$. Let $X:=\sum_{i=1}^n \I\{U_i \le 1/2\}$ and $Y:= \sum_{i=1}^n \I\{U_i \le 1/3\}$. That is, $X$ counts how many of the $U_i$s are to the left of $1/2$ and $Y$ counts how many are to the left of $1/3$. Clearly, $X\sim \Bin(n,1/2)$ and $Y \sim \Bin(n,1/3)$, and $(X,Y)$ is a coupling of $p$ and $q$. Note that, $X$ and $Y$ are not independent since $X \ge Y$ always. That is precisely the point of this coupling as we shall see. 

The useful result is, for any coupling of $p$ and $q$, 
\begin{equation}\label{eq:tvbndcoupling}
\norm{p-q}_{TV} \le \nP(X \neq Y). 
\end{equation}
To see this, start from \eqref{Eq:TVMax}. For any event $A$, $\nP_p(A)-\nP_q(A)=\nP(X \in A) - \nP(Y \in A)$. Here we are using that $X \sim p$ and $Y \sim q$. By a union bound, 
\[
\begin{split}
\nP(X \in A) &\le \nP(X \in A, Y \in A) + \nP(X \in A, Y \notin A) \\
&\le \nP(Y \in A) + \nP(X \in A, Y \notin A).
\end{split}
\]
Thus, by rearranging terms,  
\[
\nP_p(A)-\nP_q(A)\le \nP(\{X \in A\}\cap\{ Y \notin A\}) \le \nP(X \neq Y). 
\]
Thus, $\norm{p-q}_{TV}=\max_A\left[\nP_p(A)-\nP_q(A) \right]\le \nP(X \neq Y)$. This proves \eqref{eq:tvbndcoupling}.

In fact, a much more powerful theorem is that there always exists some coupling such that \eqref{eq:tvbndcoupling} is an equality. That is, 
$\norm{p-q}_{TV} \le \inf\nP(X \neq Y)$, where the infimum is over all possible couplings of $p$ and $q$. This is a special case of the so-called Kantorovich duality. 

Let's see how \eqref{eq:tvbndcoupling} can be used to bound the TV distance between $p=\Bin(n,1/2)$ and $q=\Bin(n,1/3)$. Use the coupling described in the paragraph above \eqref{eq:tvbndcoupling}. From the construction it is clear that $X=Y$ if and only if none of the $U_i$s land in the interval $(1/3, 1/2]$. This happens with probability $(1 - (1/2-1/3))^n=(5/6)^n$. Thus
\[
P(X\neq Y)= 1- P(X=Y)= 1- \left( \frac{5}{6}\right)^n. 
\]
Thus $\norm{p-q}_{TV} \le 1- \left( \frac{5}{6}\right)^n$.

\subsection{The first key convergence theorem: exponential convergence to $\pi$}
For an irreducible and aperiodic chain with transition matrix $P$, consider
\begin{align*}
	P^n(x,k) = \nP(X_n = k \, | \, X_0 =x),
\end{align*}
the $(x,k)$-entry of the matrix $P^n$.  For a fixed $n$ and a fixed $x$, this is a pmf $p$ on $\Omega$, which we will write as 
\begin{align*}
	P^n(x, \cdot).
\end{align*}
The ``dot'' notation means we can plug in any element of $\Omega$ and
get the probability we are there after $n$ steps, starting from $x$.  Note that if $\Omega = \{1,2, \ldots, N\}$, then collecting these probabilities gives the pmf
\begin{align*}
	P^n(x, \cdot ) = (P^n(x,1), P^n(x,2), \ldots, P^n(x,N)) \in \mathbb{R}^N.
\end{align*}
For instance, we saw in Example \ref{Eg:TVComputation} that for the simple walk on a 5-cycle started from the first vertex,
\begin{align*}
    P^3(x, \cdot) = p = \Big( 0, \frac{3}{8}, \frac{1}{8}, \frac{1}{8}, \frac{3}{8}  \Big).
\end{align*}
Now that we have the total variation norm we can talk about the limit $P^n(x, \cdot)$ as $n \rightarrow \infty$, because we have a way to talk about how far probability distributions are from each other.  The following theorem says that the limiting distribution is the stationary distribution $\pi$.  Recall from Theorem \ref{Thm:ReturnTime} that irreducible chains have a unique stationary distribution.

\begin{theorem}\label{Thm:ConvToStationary}
    Let $X$ be an irreducible and aperiodic Markov chain with stationary distribution $\pi$.  Then there exist constants $C>0$ and $0 < \alpha < 1$ such that, for any initial state $x \in \Omega$,
	\begin{align}\label{Ineq:ConvToStationary1}
		\Norm{P^n(x, \cdot) - \pi }_{TV} \leq C \alpha^n.
	\end{align}
    Furthermore, as the right-hand side of \eqref{Ineq:ConvToStationary1} is independent of $x$, we have
	\begin{align}\label{Ineq:ConvToStationary2}
		\max_{x \in \Omega} \Norm{P^n(x, \cdot) - \pi }_{TV} \leq C \alpha^n.
	\end{align}
\end{theorem}
For instance, if $\alpha = 1/2$ and $x=2$ then this says
\begin{align*}
	0 \leq \Norm{P^n(2, \cdot) - \pi }_{TV} \leq \frac{C}{2^n} \rightarrow 0,
\end{align*}
which is exponentially-fast convergence.  So, Theorem \ref{Thm:ConvToStationary} is saying that \emph{irreducible and aperiodic Markov chains converge exponentially fast to their stationary distributions, regardless of where they are started}.  While we elect to omit the proof of this theorem, we proceed to study and digest its contents, as it is one of the crucial results for our text. The reader is encouraged to see \cite[\S4.3]{Peres} for a proof.

\subsubsection{Why the hypotheses in Theorem \ref{Thm:ConvToStationary}?}

We first remark that we should not be overly surprised by the hypotheses of irreducibility and aperiodicity in the theorem statement.  

First, why irreducibility?  The issue here is uniqueness of $\pi$.  We saw in Theorem \ref{Thm:ReturnTime} that irreducible chains have a unique stationary distribution $\pi$, and so there is a natural candidate for the limiting distribution of $X_n$ in such chains, namely, $\pi$.  But $\pi$ is not necessarily unique for non-irreducible chains, and so we don't necessarily have a unique candidate for the limit.  For example, consider the trivial chain on $\Omega=\{1,2\}$ with transition matrix $P=I_2$, the $2 \times 2$ identity matrix.  Here both $\pi_1 = (1,0)$ and $\pi_2=(0,1)$ are stationary for $P$,
\begin{align*}
    \pi_jP = \pi_j, \qquad j=1,2,
\end{align*}
and what $X_n$ limits to depends on the starting point.  The point of Theorem \ref{Thm:ConvToStationary}, however, is that we always see the \emph{same} distribution $\pi$ in the limit, irrespective of the starting position.  So it is natural to assume irreducibility.

We also saw above in the introduction \S\ref{Sec:AsymptoticsIntro} to this chapter that  aperiodicity is also a vital assumption, as otherwise the distribution of $X_n$ may oscillate between disjoint collections of states.  We proceed to study one more such example.

\begin{example}\label{Eg:PentHex}

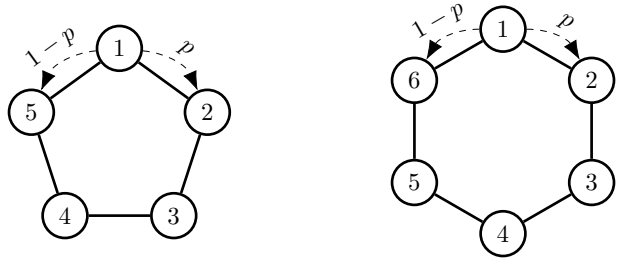
\begin{figure}
    \centering
    \scalebox{0.85}{
     \begin{tikzpicture}
\tikzset{
        node/.style={circle, draw=black, very thick, minimum size=7mm},
        arrow/.style={-{Latex[length=3mm]}, black},
        line/.style={black, very thick}
        }
    \node[node] (one) at (0,1) {1};
    \node[node] (two) at (1.376,0) {2};
    \node[node] (three) at (0.85065,-1.618) {3};
    \node[node] (four) at (-0.85065,-1.618) {4};
    \node[node] (five) at (-1.376,0) {5};
    
    \draw[line] (one) -- (two);
    \draw[line] (two) -- (three);
    \draw[line] (three) -- (four);
    \draw[line] (four) -- (five);
    \draw[line] (five) -- (one);
    
    \draw[arrow, dashed] (one) to[bend left] node[anchor=south, sloped] {$p$} (two);
    \draw[arrow, dashed] (one) to[bend right] node[anchor=south, sloped] {$1-p$} (five);
    
    \node[node] (one) at (6,1.3) {1};
    \node[node] (two) at (7.3856, 0.5) {2};
    \node[node] (three) at (7.3856,-1.1) {3};
    \node[node] (four) at (6,-1.9) {4};
    \node[node] (five) at (4.6144,-1.1) {5};
    \node[node] (six) at (4.6144,0.5) {6};
    
    \draw[line] (one) -- (two);
    \draw[line] (two) -- (three);
    \draw[line] (three) -- (four);
    \draw[line] (four) -- (five);
    \draw[line] (five) -- (six);
    \draw[line] (six) -- (one);
    \draw[arrow, dashed] (one) to[bend left, sloped] node[anchor=south] {$p$} (two);
    \draw[arrow, dashed] (one) to[bend right, sloped] node[anchor=south] {$1-p$} (six);
    \end{tikzpicture}}
    \caption{\small A random walk on a 5-cycle and a 6-cycle.  Even though the transition probabilities at any vertex are the same, only one has a distribution which converges to stationarity.}
    \label{Fig:PentHexWalks}
\end{figure}
	Consider a biased walk $(X_n)$ on the vertices of a 5-cycle $\Omega_1 = \{1, 2, \ldots, 5\}$ and another biased walk $(Y_n)$ on the vertices of a 6-cycle $\Omega_2 = \{1,2, \ldots, 6\}$, as in Figure \ref{Fig:PentHexWalks}.  The stationary distributions for these walks are both uniform,
	\begin{align*}
		\pi_1 = (1/5, 1/5, \ldots, 1/5) \quad \text{ and } \quad \pi_2 = (1/6, 1/6, \ldots, 1/6).
	\end{align*}
	Since the $(Y_n)$ chain has period 2, for a given $n$ we always have that $Y_n \in \{1,3,5\}$ or $Y_n \in \{2,4,6\}$, and so there is no possibility that $\Norm{P^n_Y(x, \cdot) - \pi_2}_{TV} \rightarrow 0$.  We thus again see the importance of the assumption of aperiodicity.
	
	On the other hand, the $(X_n)$ chain fulfills the assumptions of Theorem \ref{Thm:ConvToStationary}, and therefore its distribution converges exponentially fast to $\pi_1$.  For instance,
	\begin{align}\label{Ineq:FiveCycleConverges}
		\Norm{P^n_X(5, \cdot) - \pi_1 }_{TV} \leq C \alpha^n.
	\end{align}
	We considered the unbiased example $p=1/2$ in Example \ref{Eg:TVComputation}, and saw that $\|P_X^3(1,\cdot) - \pi_1\|_{TV} = 0.35$ (recall the biased and unbiased walks have the same stationary distribution).  The inequality \eqref{Ineq:FiveCycleConverges} says that this quickly gets \emph{much} smaller as $n$ increases.  For example, using a computer we find
	\begin{align*}
	    P_X^{30}(1,\cdot) \approx (0.2007,0.1994,0.2002,0.2002,0.1994  ),
	\end{align*}
	and thus using \eqref{Def:TVNorm} we compute
	\begin{align*}
	    \|P_X^{30}(1,\cdot) - \pi_1\|_{TV} \approx 0.0011.
	\end{align*}
	That is, our position $X_{30}$ after just 30 steps is extremely close to uniformly-random among the five vertices.
\end{example}

Even though the periodic 6-cycle example shows Theorem \ref{Thm:ConvToStationary} cannot possibly hold for chains with period $\geq 2$, we recall from Section \ref{Sec:Lazy} that this is not insurmountable.  We may simply form the corresponding lazy chain $\tilde{P} = \frac{1}{2} I + \frac{1}{2}P$ which is aperiodic and has the same stationary distribution $\pi$ as $P$ (recall Exercise \ref{Ex:LazyStationary}). Theorem \ref{Thm:ConvToStationary} then applies to $\tilde{P}$ and says
	\begin{align*}
		\max_{x \in \Omega} \Vert \tilde{P}^n_X(x, \cdot) - \pi \Vert_{TV} \leq C \alpha^n.
	\end{align*}
	
\begin{exercise}
    \begin{enumerate}[$(a)$]
        \item Form the lazy walk from the walk on the 6-cycle in Figure \ref{Fig:PentHexWalks}.  Update the diagram with the new transition probabilities.
        \item For $p=0.3$, write the corresponding transition matrix $\tilde{P}$ and use a computer to find $\tilde{P}^{25}, 
        \tilde{P}^{40}$ and $\tilde{P}^{55}$.  What do you observe?  How does this contrast with $P^{25}, P^{40}$ and $P^{55}$?
    \end{enumerate}
\end{exercise}
	
While Theorem \ref{Thm:ConvToStationary} gives us exponential converge to $\pi$, it does not say anything about the precise rate, or what $\alpha$ is in \eqref{Ineq:ConvToStationary1}.  We will take up that question when we discuss \emph{mixing times} in Section \ref{Sec:MixingTimes} below; see Theorem \ref{Thm:ExponentialRate}.

%7-8-24
%%%%%% Lecture 16 %%%%%%%%%%%%%%%%%%%%%%%%%%%%%
\subsubsection{Two perspectives on exponential convergence}

Now that we understand the hypothesis of Theorem \ref{Thm:ConvToStationary}, let's proceed to consider it from both a linear-algebraic and a probabilistic point of view, thinking in terms of \emph{functions} $f$ on our state space $\Omega = \{1,2,\ldots, N\}$.  This will also provide a natural bridge to our next convergence theorem.  

Up to this point we have not thought very much about functions $f$ on $\Omega$, but note that there is an exact correspondence between such functions and vectors $\mathbf{f} \in \mathbb{R}^N$.  Indeed, a function $f: \Omega \rightarrow \nR$ just assigns a number $f(i)$ to each vertex $i \in \Omega$, and so we can list these as a vector $\mathbf{f}=(f(1), f(2), \ldots, f(N)) \in \mathbb{R}^N$.  Conversely, a row vector $\mathbf{f}$ in $\mathbb{R}^N$ defines a function $f:\Omega \rightarrow \mathbb{R}$ by considering the value at state $i \in \Omega$ to be the $i$th component $\mathbf{f}_i$, $f(i):=\mathbf{f}_i$.  So functions $f:\Omega \rightarrow \mathbb{R}$ and vectors $\mathbf{f} \in \mathbb{R}^N$ are same.

Now let's consider Theorem \ref{Thm:ConvToStationary} in terms of linear algebra and the product $P^n \mathbf{f}^T$, which is the $N \times N$ matrix $P^n$ with the column vector $\mathbf{f}^T$ (recall that by default our vectors are row vectors, like $\pi$). Theorem \ref{Thm:ConvToStationary} says that, in particular, \emph{each row} of $P^n$ converges to $\pi$,
\begin{align}\label{Limit:PowersOfP}
	P^n = \begin{bmatrix}
			P^n(1, \cdot) \\
			P^n(2, \cdot)\\
			\vdots\\
			P^n(n, \cdot)
		\end{bmatrix} \underset{n \rightarrow \infty}{\longrightarrow}
		\begin{bmatrix}
			\pi \\
			\pi \\
			\vdots \\
			\pi
		\end{bmatrix} = \mathbf{1}^T \pi.
\end{align}
Returning to the 5-cycle, for instance, we see that
\begin{align*}
    P^{30} \approx \begin{bmatrix}
 0.2007 & 0.1994 & 0.2002 & 0.2002 & 0.1994 \\
 0.1994 & 0.2007 & 0.1994 & 0.2002 & 0.2002 \\
 0.2002 & 0.1994 & 0.2007 & 0.1994 & 0.2002 \\
 0.2002 & 0.2002 & 0.1994 & 0.2007 & 0.1994 \\
 0.1994 & 0.2002 & 0.2002 & 0.1994 & 0.2007
\end{bmatrix},
\end{align*}
and each row is manifestly close to $\pi = (0.2,0.2,0.2,0.2,0.2)$.  Given \eqref{Limit:PowersOfP}, we have the convergence
\begin{align*}
	P^n\mathbf{f}^T \rightarrow \mathbf{1}^T \pi \mathbf{f}^T.
\end{align*}
So, after many steps, multiplying $P^n$ by $\mathbf{f}^T$ is very close to multiplying $\mathbf{f}^T$ by the matrix where every row is $\pi$.  Looking at this row by row, we have 
\begin{align}\label{Limit:FuncConvergence}
    \sum_{j=1}^N P^n_{xj} f_j \rightarrow \sum_{j=1}^N \pi_j f_j.
\end{align}
for each $x \in \Omega$.

Let's secondly re-cast this in terms of probability.  Given a function $f: \Omega \rightarrow \mathbb{R}$, suppose we wish to find $\mathbb{E}_x(f(X_n))$.  That is, we start from $x$ and run the chain for a large number $n$ of steps, and then evaluate $f(X_n)$.  What is this value, on average? By definition,
\begin{align*}
	\nE_x(f(X_n)) &= \sum_{j=1}^N f(j) \, \nP_x(X_n=j) = \sum_{j=1}^N P^n_{xj}f(j)  \underset{n \rightarrow \infty}{\longrightarrow} \sum_{j=1}^N \pi_j f(j)
\end{align*}
as we noted above in \eqref{Limit:FuncConvergence}.  However, we also have that
\begin{align}\label{Eq:FunctionAverage}
    \sum_{j=1}^N \pi_j f(j) = \nE_\pi(f(X_0)),
\end{align}
the expectation of $f(X_0)$ when $X_0$ is distributed on $\Omega$ according to $\pi$.  We could also simply write this as $\nE_\pi(f)$.  Hence we see Theorem \ref{Thm:ConvToStationary} implies
\begin{align}\label{Eq:LimitOfAverages}
	\lim_{n \rightarrow \infty} \nE_x(f(X_n)) = \nE_\pi(f).
\end{align}
In words, the average of a function evaluated after many steps is basically its average against the stationary distribution.  This is intuitively clear, because the theorem says the distribution of $X_n$ is very close to $\pi$ after many steps.

Seeing things in this light leads us to our second main convergence theorem.  What if, instead of $\nE_x(f(X_n))$, we averaged over all the values $f(X_1), f(X_2), \ldots, f(X_n)$ up to the current state?  In other words, can we say anything about
\begin{align}\label{TempAvg}
	\lim_{n \rightarrow \infty}\frac{1}{n} \sum_{j=0}^{n-1} f(X_j)?
\end{align}

Notice the difference between these two limits.  In \eqref{Eq:LimitOfAverages}, we run the chain for many steps, and then take an expectation of $f(X_n)$.  This averaging yields a \emph{deterministic} number (why?), and so the limit in \eqref{Eq:LimitOfAverages} is the limit of a deterministic sequence.  In contrast, the expression in \eqref{TempAvg} is \emph{random}, and the sum up to a fixed $n$ varies each time we sample the Markov chain.  Our next main theorem, however, states that both these limits agree.

\subsection{The second key convergence theorem: ergodicity}

\begin{theorem}[Ergodic Theorem]\label{Thm:Ergodic}
	Let $X$ be an irreducible Markov chain on $\Omega$ and $f: \Omega \rightarrow \mathbb{R}$ a function.  Then, with probability 1,
	\begin{align}\label{Limit:Ergodic}
		\lim_{n \rightarrow \infty} \frac{1}{n} \sum_{j=0}^{n-1} f(X_j) = \nE_\pi(f),
	\end{align}
	where $\nE_\pi(f) = \nE_\pi(f(X_0))$ is as in \eqref{Eq:FunctionAverage}.
\end{theorem}
\noindent Let's start by making some remarks and parsing what the theorem is saying.
\begin{enumerate}[$(a)$]
    \item First, as just observed, the sum $\frac{1}{n} \sum_{j=0}^{n-1} f(X_j)$ is random; it changes each time we observe the first $n$ steps of the chain.  So the theorem is saying that, with probability 1, this random sequence converges to the fixed, deterministic number $\nE_\pi(f)$.
	\item We can also observe that the average on the left in \eqref{Limit:Ergodic} is \emph{temporal}, since it is the average over the first $n$ steps in the chain.  The average on the right, however, is \emph{spatial}, since it is the average, with respect to $\pi$, of the function over all the states in $\Omega$.  So the theorem says that the temporal average converges to the $\pi$-spatial average,
	\begin{align*}
		\lim_{n \rightarrow \infty} \underbrace{\frac{1}{n} \sum_{j=0}^{n-1} f(X_j)}_{ \textit{temporal average}} =  \underbrace{\sum_{j=1}^N \pi_j f(j).}_{\textit{spatial average}}
	\end{align*}
	\item There is also an analogy with the Law of Large Numbers, which states that an iid sequence $Y_0, Y_1, \ldots$ with $\nE(Y_j) = \mu$ satisfies
	\begin{align}\label{Limit:LLN}
		\lim_{n \rightarrow \infty} \frac{1}{n} \sum_{j=0}^{n-1} Y_j = \mu
	\end{align}
	with probability 1.  The similarity with \eqref{Limit:Ergodic} is striking; on both left-hand sides we have a random temporal average, and on both right-hand sides we have a deterministic mean.  We can therefore regard the Ergodic Theorem as a generalized law of large numbers for irreducible Markov chains.  In fact, one of Andrei Markov's (1856–1922) key motivations for studying the stochastic processes that are now named after him was to find a situation where a version of the Law of Large Numbers held but where the random variables were not independent.\footnote{See \cite{History} for more on the history of Markov and Markov chains.}  
	
		\item It is lastly helpful to compare and contrast this theorem with our first key convergence result, Theorem \ref{Thm:ConvToStationary}.  Note the hypothesis are different: while the latter requires the chain to be both irreducible and aperiodic, Theorem \ref{Thm:Ergodic} only requires irreducibility.  So the walk on the 6-cycle, for instance, satisfies this new averaging limit \eqref{Limit:Ergodic}, even though it does not satisfy the exponential convergence of \eqref{Ineq:ConvToStationary1}.

  \quad The conclusions of the two theorems are also different.  While Theorem \ref{Thm:ConvToStationary} gives a \emph{rate} of convergence, nothing of the sort is provided in Theorem \ref{Thm:Ergodic}.  All we know is that \emph{eventually} the temporal average is close to the spatial average; the theorem says nothing about precisely how long this takes.
\end{enumerate}

\begin{example}
	One of the things we can do with the ergodic theorem is compute the long-term proportion of time that we spend on a given vertex or collection of vertices. For instance, for the simple walk $X_j$ on the 6-cycle $\Omega = \{1,2,\ldots, 6\}$, what is the long-term proportion of time we spend on vertex 2?  
\begin{exercise}
    Intuitively reason out what the answer should be before reading further.
\end{exercise}
\noindent Let $f(x) = \I_2(x)$, the function that gives a zero if we don't feed it the 2nd vertex, but gives 1 when we do.  This is the \textbf{indicator function} for state 2.  In vector form  $\mathbf{f} = (f_j)_{j=1}^6 = (0,1,0,0,0,0)$. Theorem \ref{Thm:Ergodic} says
	\begin{align*}
		\lim_{n \rightarrow \infty} \frac{1}{n} \sum_{j=0}^{n-1} \I_2(X_j) &= \lim_{n \rightarrow \infty} \frac{\#(\text{visits to vertex } 2 \text{ through time } n-1)}{n}\\
		&= \lim_{n \rightarrow \infty} (\text{proportion of time $X_j$ spends at vertex 2})\\
		&= \nE_\pi f = \sum_{k=1}^6 \I_2( \pi_k) = \pi_2 = \frac{1}{6}.
	\end{align*}
	This is, of course, exactly what we would expect: in the long run, the simple walk should spend an equal proportion of time at any vertex.  The Ergodic Theorem gives us a way to rigorously show this.
\end{example}

\begin{exercise}
Suppose instead we have a biased walk $X_j$ on the 6-cycle, where the probabilities of CW and CCW steps are 0.9 and 0.1, respectively.  What is the long-term proportion of time that $X_j$ spends on vertex 2? 
\end{exercise}

\begin{exercise}
	Consider a simple symmetric random walk $Y_j$ on the vertices of the 13-cycle $\Omega := \{1,2,\ldots, 13\}$.  What is the long-term proportion of time that $X_j$ spends on vertices 5, 6 and 13?
\end{exercise}

\subsection{Proof of the Ergodic Theorem}

We need to show \eqref{Limit:Ergodic} holds for all functions $f: \Omega \rightarrow \mathbb{R}$, and we claim that it suffices to show it for the indicator functions 
\begin{align*}
	\I_i(x) = \begin{cases}
							1 & x=i,\\
							0 & x \neq i,
					\end{cases}, \qquad i=1,2, \ldots, N
\end{align*}
Indeed, we can write any function $f: \Omega \rightarrow \nR$ as 
\begin{align}\label{Eq:ErgodicSimp}
	f(x) = \sum_{i=1}^N f(i) \I_i(x),
\end{align}
as is easy to see by evaluating both sides at any $j \in \Omega$.    So, suppose for the moment that we have proved
\begin{align}\label{Limit:ErgodicIndicator}
	\lim_{n \rightarrow \infty} \frac{1}{n} \sum_{j=0}^{n-1} \I_i(X_j) = \nE_\pi(\I_i) = \pi_i
\end{align}
for each $i = 1,2, \ldots, N$.  Then, using the expression \eqref{Eq:ErgodicSimp}, we see
\begin{align*}
	\lim_{n \rightarrow \infty} \frac{1}{n} \sum_{j=0}^{n-1} f(X_j) &= \lim_{N \rightarrow \infty} \frac{1}{n} \sum_{j=0}^{n-1} \sum_{i=1}^N f(i) \I_i(X_j)\\
	&= \sum_{i=1}^N f(i) \Big( \lim_{n \rightarrow \infty} \frac{1}{n} \sum_{j=0}^{n-1}\I_i(X_j) \Big)\\
	&=  \sum_{i=1}^N f(i) \pi_i = \nE_\pi(f),
\end{align*}
where the second-to-last equality is our assumption \eqref{Limit:ErgodicIndicator}.  So, all we need to prove is \eqref{Limit:ErgodicIndicator}.

As an aside, note that we can re-cast everything we are doing in terms of linear algebra: the indicators $\I_i$ correspond to the standard basis vectors \begin{align*}
    \mathbf{e}_1 = (1,0,0,\ldots, 0), \;\mathbf{e}_2 = (0,1,0, \ldots, 0), \;\ldots, \;\mathbf{e}_N = (0,0, \ldots, 0,1),
\end{align*}
and \eqref{Eq:ErgodicSimp} writes the vector $\mathbf{f}$ as a linear combination of the $\mathbf{e}_j$'s.  What we just argued is that it suffices to prove the theorem for the basis vectors. 

We proceed to show \eqref{Limit:ErgodicIndicator} holds for any fixed $i=1,2, \ldots, N$.  Our main tool will be the Law of Large Numbers \eqref{Limit:LLN}, using a sequence of \emph{return times} to a given vertex.  Indeed, given $X_0 = x$, set $T_0  := 0$ and
\begin{align*}
	T_1  &:= \text{the first return time to $x$} \\
	&= \min \{\, t>0 \; : \; X_t = x \,\},\\
	T_2  &:= \text{the second return time to $x$}\\
	&= \min \{\, t>T_1  \; : \; X_t = x \,\},\\
	&\vdots\\
	T_k  &:= \text{the $k$th return time to $x$}\\
	& = \min \{\, t> T_{k-1}  \; : \; X_t = x \,\}.
\end{align*}
By the Markov property, the chain restarts each time we return to $x$, and so in particular,
\begin{align*}
	T_k  - T_{k-1} , \qquad k=1,2,\ldots
\end{align*}
is an iid sequence of random variables, each with the same distribution as $T_1$, and satisfying
\begin{align*}
    \nE_x(T_k - T_{k-1}) = \nE_x(T_1-T_0) = \nE_x(\tau_x^+) < \infty
\end{align*}
by Lemma \ref{Lemma:StationaryPositive}.  The Law of Large Numbers \eqref{Limit:LLN} therefore tells us 
\begin{align}\label{Lim:LLNTelescope}
    \lim_{n \rightarrow \infty} \frac{1}{n} \sum_{k=1}^n (T_k  - T_{k-1} ) = \nE_x(T_1 ).
\end{align}
The above is a telescopic sum, however, yielding
\begin{align*}
	\sum_{k=1}^n (T_k  - T_{k-1} ) = T_n  - T_0  = T_n ,
\end{align*}
and so \eqref{Lim:LLNTelescope} also says
\begin{align}\label{Lim:ErgodicPfTau}
	\nE_x(T_1 ) =\lim_{n \rightarrow \infty} \frac{1}{n} \sum_{k=1}^n (T_k  - T_{k-1} ) = \lim_{n \rightarrow \infty} \frac{T_n }{n}.
	\end{align}
Now back to our indicator $\I_i$.  Let $S_k$ be the number of visits to $i$ between the $(k-1)$th and $k$th return of the chain to $x$,
\begin{align}\label{Eq:ErgodicPfS_k}
	S_k := \sum_{j= T_{k-1} }^{T_k -1} \I_i(X_j).
\end{align}
As with the hitting times, the Markov property implies that $S_1, S_2, \ldots$ is an i.i.d. sequence, and so invoking the Law of Large Numbers again yields
\begin{align}\label{Lim:ErgodicPf1}
	\lim_{n\rightarrow \infty} \frac{1}{n} \sum_{k=1}^n S_k &= \nE(S_1) = \nE_x(\# \text{ visits to $i$ before returning to $x$}) .
\end{align}
But recalling our formula \eqref{Eq:piFormula2} from the proof of Theorem \ref{Thm:ReturnTimeLite},
\begin{align*}
	\pi_i = \frac{\nE_x(\# \text{ visits to $i$ before returning to $x$})}{\nE_x(T_1 )},
\end{align*}
\eqref{Lim:ErgodicPf1} says
\begin{align}\label{Lim:ErgodicPfSjLimit}
	\lim_{n\rightarrow \infty} \frac{1}{n} \sum_{j=1}^n S_j = \pi_i \nE_x(T_1 ).
\end{align}
The clever idea is to combine all of this as follows.  We have by \eqref{Eq:ErgodicPfS_k} that
\begin{align*}
	\sum_{j=0}^{T_n  -1}\I_i(X_j) = \sum_{j=1}^n S_j,
\end{align*}
since we are counting all the visits to state $i$ up to the $n$th return to $x$.  Hence
\begin{align*}
	\frac{1}{T_n }\sum_{j=0}^{T_n  -1}\I_i(X_j) = \frac{1}{T_n }\sum_{j=1}^n S_j = \frac{\frac{1}{n}\sum_{j=1}^n S_j}{\frac{T_n }{n}} \underset{n \rightarrow \infty}{\longrightarrow} \frac{\pi_i \nE_x(T_1 )}{\nE_x(T_1 )} = \pi_i
\end{align*}
by \eqref{Lim:ErgodicPfTau} and \eqref{Lim:ErgodicPfSjLimit}.  Hence we have
\begin{align}\label{Lim:AlmostErgodicIndicator}
    \lim_{n \rightarrow \infty}\frac{1}{T_n }\sum_{j=0}^{T_n  -1}\I_i(X_j) = \pi_i,
\end{align}
which is almost our desired limit \eqref{Limit:ErgodicIndicator}.  To finish the argument, we simply note that for any $T_n  \leq m \leq T_{n+1} $,
\begin{align*}
    \sum_{j=0}^{T_n  -1}\I_i(X_j) \leq \sum_{j=0}^{m -1}\I_i(X_j) \leq \sum_{j=0}^{T_{n+1}  -1}\I_i(X_j)
\end{align*}
since the number of visits is non-decreasing, and therefore
\begin{align}\label{Ineq:Sandwich}
    \frac{1}{T_{n+1} }\sum_{j=0}^{T_n  -1}\I_i(X_j) \leq \frac{1}{m}\sum_{j=0}^{m -1}\I_i(X_j) \leq \frac{1}{T_n }\sum_{j=0}^{T_{n+1}  -1}\I_i(X_j).
\end{align}
Note carefully the indices.  We would like to replace $1/T_{n+1} $ on the left with $1/T_n $, and similarly replace $1/T_n $ on the right with $1/T_{n+1} $, and then invoke \eqref{Lim:AlmostErgodicIndicator}.  For the left, we observe
\begin{align*}
    \frac{1}{T_{n+1} } = \frac{T_n }{T_{n+1} } \cdot \frac{1}{T_n } \overset{d}{=} \frac{T_n }{T_{n}  + T_1 } \cdot \frac{1}{T_n } = \frac{1}{1+ \frac{T_1 }{T_n }} \cdot \frac{1}{T_n } \underset{n \rightarrow \infty}{\longrightarrow} 1 \cdot  \frac{1}{T_n }
\end{align*}
since $T_n  \rightarrow \infty$ and $T_1 $ is fixed, and where ``$\overset{d}{=}$'' means equal in distribution.  Thus, with probability 1, $1/T_{n+1} $ is very close to $1/T_n $ for large $n$, and similarly for $1/T_n $ and $1/T_{n+1} $ for the right-hand side.  Our inequality \eqref{Ineq:Sandwich} therefore sandwiches $\frac{1}{m}\sum_{j=0}^{m -1}\I_i(X_j)$ between quantities that are very close to $\pi_i$ by \eqref{Lim:AlmostErgodicIndicator}, which yields the desired limit \eqref{Limit:ErgodicIndicator}. \qed

\section{Mixing times}\label{Sec:MixingTimes}

One of the primary uses of Markov chains is to simulate a random sample from a given distribution $\pi$, a topic we study in depth in the next chapter.  We can now begin to see how this can work, though, via Theorem \ref{Thm:ConvToStationary}: if we run an irreducible and aperiodic chain for a large number of steps $n$, \eqref{Ineq:ConvToStationary1} says that 
\begin{align}\label{Ineq:MixingIntro}
	\frac{1}{2} \sum_{j=1}^N |\nP(X_n=j) - \pi_j)| < C \alpha^n.
\end{align}
In other words, $X_n$ is very close to $\pi$ in distribution, and so $X_n$ is itself an approximate sample of $\pi$.  

In practice it is important to quantify how close we are to $\pi$.  For instance, suppose we want to ensure that 
\begin{align}\label{Ineq:MixingIntro2}
	\sum_{j=1}^n |\nP(X_N=j) - \pi_j)| < \frac{1}{100}.
\end{align}
How many steps $n$ must we take?  In this section we will see that for reversible chains, we can answer this in terms of the eigenvalues and eigenvectors of the transition matrix $P$.

We will prove in the following sections that the exponential rate of convergence $\alpha$ in \eqref{Ineq:MixingIntro} is controlled by the second-largest eigenvalue of $P$ (see Theorem \ref{Thm:ExponentialRate} below), and we will give an upper bound for how many steps we need to reach a given threshold of error, as in \eqref{Ineq:MixingIntro2}.

In \eqref{Ineq:MixingIntro} and \eqref{Ineq:MixingIntro2} we did not specify a starting distribution, as in $\nP_x(X_n=j)$.  The point is that we want a bound that is irrespective of the starting distribution.

%%%%%%%%%% Lecture 17 %%%%%%%%%%%%%%%%%%%%%%%%

\subsection{Definition of mixing time}

We begin by giving a name to first time when we achieve an error threshold as in \eqref{Ineq:MixingIntro2}. Let $X_n$, $n=0,1,\ldots$, be an irreducible and aperiodic chain with transition matrix $P$, and taking values in $\Omega = \{1,2, \ldots, N\}$, as usual.  Theorem \ref{Thm:ConvToStationary} tells us that
\begin{align*}
	\max_{x \in \Omega} \Norm{P^n(x, \cdot) - \pi }_{TV} \underset{n \rightarrow \infty}{\longrightarrow} 0.
\end{align*}	
\begin{definition}[Mixing time]
	For $0 < \epsilon <1$, the \textbf{$\epsilon$-mixing time} $t_{\text{mix}}(\epsilon)$ is the minimal time $n \in \mathbb{N}$ such that
	\begin{align*}
		\max_{x \in \Omega} \Norm{P^n(x, \cdot) - \pi }_{TV} \leq \epsilon.
	\end{align*}
\end{definition}
\noindent So the mixing time $t_{\text{mix}}(\epsilon)$ tells us the earliest time $n$ when we are guaranteed to be at most $\epsilon$-far away from the stationary distribution in total variation, no matter what state $x$ we begin at.  Usually we will not be able to exactly compute $t_{\text{mix}}(\epsilon)$, but upper bounds will suffice in practice.  We will need linear algebra techniques to obtain upper bounds, and so we first turn our attention to developing (or reviewing, depending on your background) some tools for dealing with powers of matrices.

\subsection{The spectral decomposition}
Recall that a square matrix $M \in \mathbb{R}^{N \times N}$ is orthogonally diagonalizable if we have the \textbf{spectral decomposition}
\begin{align}\label{Eq:SpectralDecomp}
	M = U^T \Lambda U
\end{align} 
for some matrix $U^T$ with orthonormal columns and a diagonal matrix 
\begin{align*}
	\Lambda = \text{diag}(\lambda_1, \lambda_2, \ldots, \lambda_N) =  \begin{bmatrix}
		\lambda_1 & 0 & \cdots & 0 \\
		0 & \lambda_2 & \cdots & 0\\
		\vdots & & \ddots & \vdots \\
		0 & 0 & \cdots & \lambda_N 
	\end{bmatrix}.
\end{align*}
In particular, $M^T = U^T\Lambda U= M$.  Comment on iff

The $\lambda_j$'s, which may not all be distinct, are the eigenvalues of $M$, and the columns $u_j^T$ of $U^T$ are the corresponding right eigenvectors,
\begin{align}\label{Eq:SpectralDecomp3}
	Mu_j^T = \lambda_j u_j^T
\end{align}
for all $j$.\footnote{Recall our convention is that vectors are row vectors, and hence a column vector is written via the transpose.}  Note that this is exactly what you get if you multiply both sides of \eqref{Eq:SpectralDecomp} by $U^T$ and compare the columns. 

Let's comment on a possible point of confusion: note from \eqref{Eq:SpectralDecomp3} that the spectral decomposition gives us the  \emph{right}-eigenvalues and \emph{right}-eigenvectors of $M$, while up to this point we have primarily been interested in solving $\pi P = \pi$ for the \emph{left}-eigenvector $\pi$ of $P$ with \emph{left}-eigenvalue 1.  This distinction does not matter for the eigenvalues: the characteristic polynomial of a matrix $M$ and its transpose $M^T$ are the same, which implies the left- and right-eigenvalues of $M$ are identical.  However, the left- and right-\emph{eigenvectors} for a fixed eigenvalue $\lambda$ are generally not identical for non-symmetric matrices.  See Exercise \ref{Ex:LinearAlgebra1} below. 

If we multiply out the right-hand side of \eqref{Eq:SpectralDecomp}, we find
\begin{align}\label{Eq:SpectralDecomp2}
	 M = \sum_{j=1}^N \lambda_j u_j^T u_j,
\end{align}
which implies
\begin{align}\label{Eq:SpectralPowers}
	M^k = \sum_{j=1}^N \lambda_j^k u_j^T u_j
\end{align}
by orthogonality and the fact that the $u_j$ are unit vectors.  It is a good idea to convince yourself of these last two assertions.

\begin{exercise}\label{Ex:LinearAlgebra1}
    Show that \eqref{Eq:SpectralDecomp2} follows from \eqref{Eq:SpectralDecomp}, and derive \eqref{Eq:SpectralPowers} from \eqref{Eq:SpectralDecomp2} using induction on $k$.
\end{exercise}

\begin{exercise}\label{Ex:LinearAlgebra2}
\begin{enumerate}[$(a)$]
    \item Show that the left- and right- eigenvalues of a matrix are identical.  That is, for any given $\lambda \in \mathbb{R}$, there is some non-zero row vector $v$ such that $v P = \lambda v$ if and only if there is a non-zero column vector $w^T$ such that $Pw^T = \lambda w^T$.  (\emph{Hint}: use the characteristic polynomial.)
    \item The stationary distribution $\pi$ for the simple walk on the graph in Figure \ref{Fig:FirstGraph1} is given in Example \ref{Example:FirstGraphStationaryDist}.  We know $\pi$ is a left-eigenvector for $P$ with $\lambda=1$.  Show that $\pi^T$ is \emph{not} a right-eigenvector for $P$ for any $\lambda$. 
\end{enumerate}
\end{exercise}

\subsection{The Perron-Frobenius Theorem and convergence rate for symmetric $P$}\label{Sec:PF}

If we want to understand the high powers of a symmetric matrix $P$, it is clear from \eqref{Eq:SpectralPowers} that we need to understand the eigenvalues $\lambda$ of $P$.  We already know that 1 is an eigenvalue.  This is clear from the existence of the stationary distribution $\pi$ satisfying $\pi P = \pi$, and also from the equation $P \mathbf{1}^T = \mathbf{1}^T$, where $\mathbf{1} = (1,1,\ldots, 1)$ is the vectors of one's in $\mathbb{R}^N$.  

Let's think further about a specific example.  If $P$ is symmetric, then the columns of $P$ sum to 1 and $\pi = (\frac{1}{N}, \frac{1}{N}, \ldots, \frac{1}{N}) = \frac{1}{N}\mathbf{1}$ by Exercise \ref{Ex:ColsAddToOne}$(a)$.  Hence $\pi^T = \frac{1}{N} \mathbf{1}^T$ is also a right-eigenvector of $P$.  Noting that $\| \pi\|_2 = \frac{1}{\sqrt{N}}$ (reference excurus), we see the unit-vector version of $\pi$ is $u_1 = \sqrt{N}\pi$, and \eqref{Eq:SpectralPowers} then says
\begin{align}\label{Eq:SpectralSymmetric}
	P^n = N \pi^T \pi + \sum_{j=2}^N \lambda_j^n u_j^T u_j.
\end{align}
If all the remaining eigenvalues are smaller in absolute value than 1,
\begin{align}\label{Eq:PEigenvalues1}
	-1 < \lambda <1 \quad \text{ for } j =2,3, \ldots, N,
\end{align}
then $\lambda_j^n \rightarrow 0$, and taking the limit in \eqref{Eq:SpectralSymmetric} yields
\begin{align}
	\lim_{n \rightarrow \infty} P^n &= N \pi^T \pi + 0 \notag\\
	&= \begin{bmatrix}
		\frac{1}{N} & \frac{1}{N} & \cdots & \frac{1}{N} \\
		\frac{1}{N} & \frac{1}{N} & \cdots & \frac{1}{N}\\
		\vdots & & \ddots & \vdots \\
		\frac{1}{N} & \frac{1}{N} & \cdots & \frac{1}{N}
	\end{bmatrix}. \label{Eq:PSymmetricLimit}
\end{align}

This is thus the simplest situation: a one-dimensional eigenspace for $\lambda=1$, and all other eigenvalues satisfying \eqref{Eq:PEigenvalues1}.  The next theorem tells us this is always the case when our chain is irreducible and aperiodic.
\begin{theorem}[Perron-Frobenius]\label{Thm:PF}
	Let $P$ be a transition matrix.
	\begin{enumerate}[$(i)$]
		\item If $\lambda$ is an eigenvalue of $P$, then $|\lambda| \leq 1$.
		\item If $P$ is irreducible, then the eigenspace corresponding to $\lambda_1 =1$ has dimension 1.
		\item If $P$ is irreducible and aperiodic, then $-1$ is not an eigenvalue of $P$.
	\end{enumerate}
\end{theorem}

\begin{proof}
	Parts $(i)$ and $(ii)$ are a version of the \emph{Perron-Frobenius Theorem},  a result in linear algebra that we will not prove here.  We will argue for $(iii)$, however.(Add reference, see Wikipedia)  Suppose $P$ is irreducible and aperiodic, and $Pv^T = -v^T$ for some column vector $v^T$.  We need to show that $v = \mathbf{0} = (0,0,\ldots, 0)$.  Note that $Pv^T = -v^T$ implies $P^nv^T = (-1)^nv^T$.  However, by Theorem \ref{Thm:ConvToStationary}, we know that \emph{each row} of $P^n$ converges to $\pi$ (recall the discussion around \eqref{Limit:PowersOfP}), and so 
\begin{align}\label{Lim:MatrixLimitExists}
	P^n v^T \underset{n \rightarrow \infty}{\longrightarrow} \begin{bmatrix}
			\pi \\
			\pi \\
			\vdots \\
			\pi
		\end{bmatrix} v^T = \begin{bmatrix}
			\nE_\pi(v) \\
			\nE_\pi(v) \\
			\vdots \\
			\nE_\pi(v)
		\end{bmatrix}.
\end{align}
So, if we take the limit only using even $n$, we have
\begin{align*}
	\begin{bmatrix}
			\nE_\pi(v) \\
			\nE_\pi(v) \\
			\vdots \\
			\nE_\pi(v)
		\end{bmatrix} = \lim_{\substack{n \rightarrow \infty\\ n \text{ even}}}P^n v^T =  \lim_{\substack{n \rightarrow \infty\\ n \text{ even}}} v^T = v^T,
\end{align*}
while if we go along odd $n$, we find
\begin{align*}
		\begin{bmatrix}
			\nE_\pi(v) \\
			\nE_\pi(v) \\
			\vdots \\
			\nE_\pi(v)
		\end{bmatrix} = \lim_{\substack{n \rightarrow \infty\\ n \text{ odd}}}P^n v^T =  \lim_{\substack{n \rightarrow \infty\\ n \text{ odd}}} (-1)^nv^T = -v^T.
\end{align*}
However, both of these limits are the same by \eqref{Lim:MatrixLimitExists}, and so $v = -v$ implying $v =\mb{0}$, as claimed. We conclude that $P$ does not have an eigenvalue of $-1$.
\end{proof}
% Thus, for aperiodic, irreducible and symmetric $P$, we have that the eigenspace for $\lambda_1=1$ is one-dimensional, and all other eigenvalues satisfy $|\lambda_j|<1$.  
Order the eigenvalues of $P$ as a  decreasing sequence $\lambda_1 \geq \lambda_2 \geq \cdots \geq \lambda_n$ and repeat them according to multiplicity (so if the eigenspace of $\lambda_1$ has dimension 2, then we would have $\lambda_1=\lambda_2 > \lambda_3$).  The Perron-Frobenius theorem then says 
\begin{align*}
	\lambda_1 = \underbrace{1 > \lambda_2}_{\text{if irreducible}} \geq \lambda_3 \geq \cdots \geq \underbrace{\lambda_{n} > -1}_{\text{if irreducible \& aperiodic}}
\end{align*}
\noindent In particular, powers $P^n$ of the matrix satisfy the limit \eqref{Eq:PSymmetricLimit} when $P$ is symmetric, irreducible and aperiodic, an observation we collect into the following theorem.
\begin{theorem}\label{Thm:ExponentialRateSymP}
	Let $P$ be a symmetric transition matrix for an irreducible and aperiodic chain, with eigenvalues $\lambda_1=1 > \lambda_2\geq \lambda_3\geq \cdots \geq \lambda_N >-1$.  Then the exponential rate of convergence $\alpha$ in Theorem \ref{Thm:ConvToStationary} is given by
	\begin{align}\label{MaxLambda2}
		\max_{j \geq 2} |\lambda_j|.
	\end{align}
\end{theorem}
\noindent  Note that even though $\lambda_2 \geq \lambda_j$ for $j \geq 2$, note that $\max_{j \geq 2} |\lambda_j|$ is not necessarily $|\lambda_2|$ since the $\lambda_j$ can be negative.
\begin{proof}
	We have seen that the spectral theorem gives
	\begin{align*}
		P^n &= N \pi^T \pi + \sum_{j=2}^N \lambda_j^n u_j^T u_j.
	\end{align*}
	The first matrix $N\pi^T\pi$ is the constant $N\times N$ matrix with entries $1/N$.  By Theorem \ref{Thm:PF}, $|\lambda_j|<1$ for $j =2,3, \ldots, N$, and hence 
	\begin{align*}
		\sum_{j=2}^N \lambda_j^n u_j^T u_j \underset{n \rightarrow \infty}{\longrightarrow}0
	\end{align*}
	exponentially fast as $n \rightarrow \infty$, with rate controlled by \eqref{MaxLambda2}.  Each row of what is left, the matrix $N\pi^T\pi$, is the stationary distribution.
\end{proof}

The Perron-Frobenius theorem also allows us to prove uniqueness of the stationary distribution $\pi$ for irreducible Markov chains.  Recall that in Theorem \ref{Thm:ReturnTimeLite} we proved the \emph{existence} of $\pi$ for any irreducible chain, along with the explicit formula
\begin{align}\label{Eq:StationaryFormulaAgain}
    \pi(x) = \frac{1}{ \nE(\tau_x  \, | \, X_0 = x)}.
\end{align}
What remains to finish the proof of Theorem \ref{Thm:ReturnTime} is to show that \emph{any} stationary distribution satisfies the formula \eqref{Eq:StationaryFormulaAgain}.

\begin{proof}[Proof of Theorem \ref{Thm:ReturnTime}]
    Let $\rho$ be any stationary distribution.  Then $\rho P = \rho$ and so $\rho$ is a left-eigenvector of $P$ corresponding to $\lambda_1=1$.  Note that the left- and right-eigenspaces for $\lambda_1$ have identical dimension; this follows from the argument in Exercise \ref{Ex:LinearAlgebra2}$(a)$, as the characteristic polynomials for $P$ and $P^T$ are identical. Hence Theorem \ref{Thm:PF} $(ii)$ says that the left-eigenspace for $\lambda_1$ is spanned by the $\pi$ given by \eqref{Eq:StationaryFormulaAgain}.  In particular, there exists $c \in \mathbb{R}$ such that $\rho = c\pi$, and so summing the entries $\rho_j$ yields
    \begin{align*}
        \sum_{j=1}^N \rho_j = c \sum_{j=1}^N \pi_j = c.
    \end{align*}
    Since $\rho$ is a probability distribution, $c=1$ and so $\rho = \pi$.
\end{proof}

\begin{exercise}
	Consider the simple symmetric random walk $(X_n)$ on the 7-cycle $\Omega = \{1,2,3,4,5,6,7\}$.  Write out the transition probability matrix $P$, and use a computer to take large powers of $P$.  Numerically verify that
	\begin{align*}
		\lim_{n \rightarrow \infty} P^n = 
		\begin{bmatrix}
		1/7 & 1/7 & 1/7 & 1/7 & 1/7 &1/7 &1/7 \\
		1/7 & 1/7 & 1/7 & 1/7 & 1/7 &1/7 &1/7\\
		1/7 & 1/7 & 1/7 & 1/7 & 1/7 &1/7 &1/7\\
		1/7 & 1/7 & 1/7 & 1/7 & 1/7 &1/7 &1/7\\
		1/7 & 1/7 & 1/7 & 1/7 & 1/7 &1/7 &1/7\\
		1/7 & 1/7 & 1/7 & 1/7 & 1/7 &1/7 &1/7\\
		1/7 & 1/7 & 1/7 & 1/7 & 1/7 &1/7 &1/7
	\end{bmatrix}.
	\end{align*}
	What is the exponential rate of convergence?
\end{exercise}

\subsection{Convergence rate for reversible $P$}

Having \emph{symmetric} transition probabilities $P$ in Theorem \ref{Thm:ExponentialRateSymP} is a very restrictive condition.  Can we say anything about the convergence rate for aperiodic, irreducible and reversible $P$ that are not symmetric? The answer is yes, via a heavy dose of linear algebra (as we are dealing with high powers of the transition matrix $P$, there is no other route).  We will find that the rate of convergence in Theorem \ref{Thm:ExponentialRateSymP} also holds in this case.

Suppose $P$ is reversible with respect to $\pi$,
\begin{align}\label{Eq:DBEMix}
	\pi_x P_{xy} = \pi_y P_{yx}, \qquad x,y \in \Omega.
\end{align}
If $\pi$ is not uniform, then $P$ is not symmetric, and we can't immediately apply the spectral decomposition as above.  However, if we build the right symmetric matrix out of $P$, using reversibility, we can still gain useful information.  

To that end, we define diagonal matrices
\begin{align*}
	D_\pi := \text{diag}(\pi_1, \pi_2, \ldots, \pi_n) =  \begin{bmatrix}
		\pi_1 & 0 & \cdots & 0 \\
		0 & \pi_2 & \cdots & 0\\
		\vdots & & \ddots & \vdots \\
		0 & 0 & \cdots & \pi_n 
	\end{bmatrix}
\end{align*}
and $D_\pi^p := \text{diag}(\pi_1^p, \pi_2^p, \ldots, \pi_n^p)$ for some power $p$, and consider
\begin{align*}
	A := D_\pi^{1/2} P D_\pi^{-1/2}.
\end{align*}
Noting that the $(x,y)-$component of $A$ is 
\begin{align}\label{Eq:Axy}
	A_{xy} = \frac{\sqrt{\pi_x}}{\sqrt{\pi_y}}P_{xy},
\end{align}
we claim that $A$ is symmetric, $A_{xy} = A_{yx}$. 

\begin{exercise}
Convince yourself of \eqref{Eq:Axy}, and use \eqref{Eq:DBEMix} and \eqref{Eq:Axy} to show that $A_{xy} = A_{yx}$.
\end{exercise}
Symmetric matrices have the spectral decomposition \eqref{Eq:SpectralDecomp2}, and so
\begin{align}\label{Eq:SpectralA}
	A = D_\pi^{1/2} P D_\pi^{-1/2} = \sum_{j=1}^N \lambda_j u_j^T u_j
\end{align}
for some eigenvalues $\lambda_j$ and corresponding orthonormal eigenvectors $u_j^T$.  Note that when you multiply out $ D_\pi^{1/2} P D_\pi^{-1/2}$ with itself $k$ times, the inner $D_\pi^{\pm 1/2}$ matrices cancel, and hence
\begin{align*}
	A^k = D_\pi^{1/2} P^k D_\pi^{-1/2},
\end{align*}
and so we see by \eqref{Eq:SpectralPowers} that
\begin{align*}
	A^k = D_\pi^{1/2} P^k D_\pi^{-1/2} = \sum_{j=1}^N \lambda_j^k u_j^T u_j.
\end{align*}
We interested in $P^kD_\pi^{-1}$, and the above shows
\begin{align}
	P^k D_\pi^{-1} = D_\pi^{-1/2} A^k D_\pi^{-1/2} &=  \sum_{j=1}^N \lambda_j^k D_\pi^{-1/2} u_j^T u_j D_\pi^{-1/2} \notag\\
	&= \sum_{j=1}^N \lambda_j^k (D_\pi^{-1/2})^T u_j^T u_j D_\pi^{-1/2}\notag \\
	&= \sum_{j=1}^N \lambda_j^k (u_j D_\pi^{-1/2})^T u_j D_\pi^{-1/2} \notag \\
	&= \sum_{j=1}^N \lambda_j^k v_j^T v_j \label{Eq:Pk}
\end{align}
for $v_j := u_j D_\pi^{-1/2}$. Note that the $v_j^T$ are eigenvalues of $P$ with eigenvalue $\lambda_j$, since
\begin{align*}
	Pv_j^T = P D_\pi^{-1/2} u_j^T &= D_\pi^{-1/2}AD_\pi^{1/2} D_\pi^{-1/2} u_j^T\\
	 &= D_\pi^{-1/2}A u_j^T = \lambda_j D_\pi^{-1/2} u_j^T = \lambda_j v_j^T.
\end{align*}
In particular, by Theorem \ref{Thm:PF} $(ii)$, the eigenvector $v_1^T$ for $\lambda_1 =1$ must be a multiple of the constant vector $\mathbf{1}$ ($P$ is a transition matrix, so $\mathbf{1}^T$ is an eigenvector, and by Theorem \ref{Thm:PF}, the eigenspace is one-dimensional).  So $v = (c,c,\ldots, c)$ for some $c$.

We claim, however, that $c=\pm 1$, i.e. $v= \pm \mathbf{1}$.  Indeed, we have
\begin{align*}
	v_1 = (c,c, \ldots, c) = u_1D_\pi^{-1/2} = \left(\frac{u_{11}}{\sqrt{\pi_1}}, \frac{u_{12}}{\sqrt{\pi_2}}, \ldots, \frac{u_{1N}}{\sqrt{\pi_N}}\right),  
\end{align*}
and so $u_{1j} = c\sqrt{\pi_j}$ for each $j$.  However, $u_1$ is a unit vector, as that's what the spectral decomposition gave us for $A$ in \eqref{Eq:SpectralA}.  Therefore
\begin{align*}
	1 = \Norm{u_1}^2 = \sum_{j=1}^N u_{1j}^2 = c^2 \sum_{j=1}^N \pi_j = c^2 \cdot 1,
\end{align*}
and so $c = \pm 1$, as claimed.

Now suppose that $P$ is irreducible and aperiodic.  Picking out the $(x,y)-$entry of \eqref{Eq:Pk}, we have
\begin{align}
	(P^k D_\pi^{-1})_{xy} = \frac{P^k_{xy}}{\pi_y} = \sum_{j=1}^N \lambda_j^k v_{j,x} v_{j,y} &= \lambda_1 v_{1,x}v_{1,y} + \sum_{j=2}^N \lambda_j^k v_{j,x} v_{j,y} \notag \\
	&= 1 \cdot 1 + \sum_{j=2}^N \lambda_j^k v_{j,x} v_{j,y}. \label{Eqn:ExponentialRateEntry}
\end{align}
By Theorem \ref{Thm:PF} $(i)$ and $(iii)$, we know $|\lambda_j|<1$ for $j=2,3,\ldots, N$, and thus taking limits in the above yields
\begin{align}\label{Eqn:PPowersEntryLimit}
	\lim_{k \rightarrow \infty} \frac{P^k_{xy}}{\pi_y} = 1.
\end{align}
We have recovered our convergence result, Theorem \ref{Thm:ConvToStationary}.  Note that \eqref{Eqn:PPowersEntryLimit} is independent of $x$, as we would expect from \eqref{Ineq:ConvToStationary2}.  Furthermore, we see from \eqref{Eqn:ExponentialRateEntry} that the exponential rate of convergence is again controlled by the largest of $\lambda_2, \lambda_3, \ldots, \lambda_N$.  In other words, we have proved the following extension of Theorem \ref{Thm:ExponentialRateSymP} to more general $P$.

\begin{theorem}\label{Thm:ExponentialRate}
	Let $P$ be the transition matrix of an irreducible, aperiodic and reversible chain with eigenvalues $1, \lambda_2, \lambda_3, \ldots, \lambda_N$.  Then the exponential rate of convergence $\alpha$ in Theorem \ref{Thm:ConvToStationary} is given by
	\begin{align*}
		\max_{j \geq 2} |\lambda_j|.
	\end{align*}
\end{theorem}

\subsection{The relaxation time}\label{Sec:relax}

Motivated by Theorems \ref{Thm:ExponentialRate} (and its more restrictive counterpart \ref{Thm:ExponentialRateSymP}), we set
\begin{align*}
	\lambda^* := \max_{j \geq 2} |\lambda_j|,
\end{align*}
the largest of the non-identity eigenvalues of $P$.  We define the \emph{absolute spectral gap} as $\gamma^* := 1 - \lambda^*$.

\begin{definition}
	The \textbf{relaxation time} $t_{\text{rel}}$ is the reciprocal of the absolute spectral gap,
	\begin{align*}
		t_{\text{rel}} = \frac{1}{\gamma^*} = \frac{1}{1- \lambda^*} = \frac{1}{1-\max_{j \geq 2} |\lambda_j| }
	\end{align*}
\end{definition}
The name comes from the following theorem, which should seem plausible in light of the exponential decay we've seen in \eqref{Eqn:ExponentialRateEntry}.  (We won't give a proof, however.)  

Recall that the \emph{mixing time} $t_{\text{mix}}(\epsilon)$ is the minimal time $t$ we need to guarantee that our chain is $\epsilon$-close to stationarity, regardless of where we start.  In other words, it's the minimal $t$ such that
	\begin{align*}
		\max_{x \in \Omega} \Norm{P^t(x, \cdot) - \pi }_{TV} < \epsilon.
	\end{align*}
\begin{theorem}\label{Thm:MixingTimeBounds}
	Let $P$ be irreducible, aperiodic and reversible with respect to $\pi$.  Let $\pi_m$ be the smallest of the entries of $\pi$,
	\begin{align*}
		\pi_m := \min_{j =1,2, \ldots, N} \pi_j.
	\end{align*} 
	Then, for any $0 < \epsilon <1$, 
	\begin{align}\label{Ineq:MixingTimeBounds}
		(\tr - 1)\log \left( \frac{1}{2\epsilon} \right) \leq \tm(\epsilon) \leq \log \left(  \frac{1}{\epsilon \pi_m}\right) \tr.
	\end{align}
\end{theorem}
Note what this theorem says practically: if we know the eigenvalues of $P$ (and we can figure these out on a computer), and hence the relaxation time $\tr$, we have bounds \eqref{Ineq:MixingTimeBounds} on the number of steps we need to run the chain to ``mix'' $\epsilon$-close to stationarity.  In other words, we have control on  how many steps we have to take in order to gain any level of precision in sampling our stationary distribution.  Note that the logarithm in \eqref{Ineq:MixingTimeBounds} is base $e$, or the natural logarithm.    

The proof, unfortunately, is beyond the scope of our text; see \cite[Ch.4]{Peres}.

%%%%%%%%%%%%% Lecture 18

\section{Two examples of mixing times}

Let's work out the bounds in Theorem \ref{Thm:MixingTimeBounds} in a couple of concrete examples.

\subsection{Random walk on an $n$-cycle}

\begin{figure}
\centering
\includegraphics[scale=0.32]{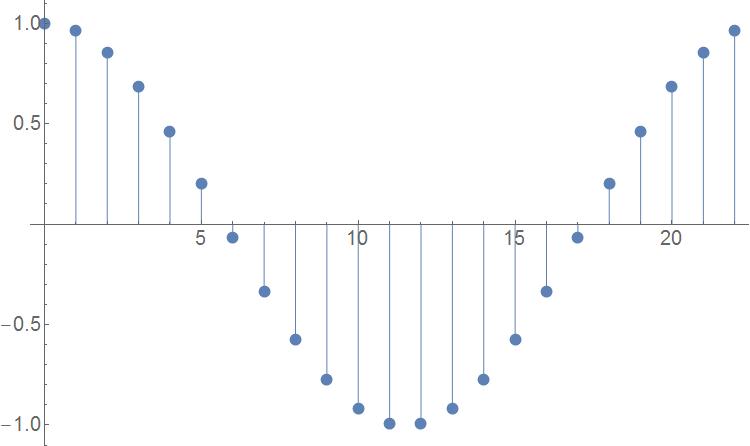}
\caption{{\small The eigenvalues $\lambda_j$ for the transition matrix $P$ of the simple walk on the 23-cycle.}}
\label{Fig:23CycleEigen}
\end{figure}

Consider a random walk on the $n$-cycle $\Omega = \{0,1,\ldots, n-1\}$.  The transition matrix has eigenvalues $\lambda_0, \lambda_1, \ldots, \lambda_{n-1}$, where we claim
\begin{align}\label{Eq:ncycleEigenvalues}
	\lambda_j = \cos\Big( \frac{2\pi j}{n} \Big), \qquad j = 0,1, \ldots, n-1.
\end{align}
See Figure \ref{Fig:23CycleEigen} the case of the $n=23$.  One way we can prove that these are the eigenvalues is to explicitly give the eigenvectors.  We claim that each $\lambda_j$ has an eigenvector $f_j^T \in \mathbb{R}^n$ that satisfies the formula
\begin{align}\label{Eq:ncycleEigenvector}
	f_j(k) = \cos\Big( \frac{2\pi j}{n} k\Big), \qquad k = 0,1,\ldots,n-1.
\end{align}
In particular, when $j=0$ we recover the eigenvalue $\lambda_0 =1$ and the corresponding (right) eigenvector is $f_0^T = \mathbf{1}^T$.
\begin{exercise}
	Prove \eqref{Eq:ncycleEigenvector} by showing that, for each $j$, the $k$th entry of $Pf_j^T$ is the same as the $k$th entry of $\lambda_j f_j^T$, $k=0,1,\ldots, n-1$.  Hint: use \eqref{Eq:ncycleEigenvalues}, the $k$th row of $P$ (which only has two non-zero entries) and the trig formula 
	\begin{align*}
		\cos(a)\cos(b) = \frac{1}{2}(cos(a+b) + \cos(a-b)).
	\end{align*}
\end{exercise}

Does the $n$-cycle have an eigenvalue of $-1$?  From \eqref{Eq:ncycleEigenvalues}, we would need 
\begin{align*}
	\frac{2\pi j}{n} = \pi \quad \Leftrightarrow \quad j = \frac{n}{2},
\end{align*}
which is possible for integer $j$ if and only if $n$ is even.  Note that this consistent with the Perron-Frobenius Theorem, Theorem \ref{Thm:PF}$(iii)$: aperiodic chains do not have $-1$ as an eigenvalue, and for the $n$-cycle aperiodicity is equivalent to $n$ being odd.  Since our mixing time bounds \eqref{Ineq:MixingTimeBounds} are for aperiodic chains, we restrict our attention to odd $n$.

We can read off the second-largest eigenvalue from our formula \eqref{Eq:ncycleEigenvalues}.  Indeed, note that the argument of cosine is between $0$ and $2\pi$, and so the second-largest value is when we are the second-closest to 0 or $2\pi$, yielding
\begin{align*}
	\lambda^* = \max_{j \neq 0} |\lambda_j| = \lambda_1 = \cos\Big( \frac{2\pi}{n} \Big) = \lambda_{n-1}=\cos\Big( \frac{2\pi(n-1)}{n} \Big),
\end{align*}
and so the spectral gap is $\gamma^* = 1 - \cos( 2\pi/n)$ and the relaxation time is
\begin{align}\label{Eq:nCycleRelax}
    \tr = \frac{1}{1-\cos(2\pi/n)}.    
\end{align}
Since $\cos(x) = 1 - \frac{1}{2}x^2 + O(x^4)$ for $x$ close to 0, we find
\begin{align*}
	\gamma^* \approx 1 - \left( 1- \frac{1}{2} \cdot\frac{4\pi^2}{n^2}\right) = \frac{2\pi^2}{n^2}
\end{align*}
for large $n$, and thus $\tr = \frac{1}{\gamma^*} \approx \frac{n^2}{2\pi^2}$.  The stationary distribution is uniform, $\pi = (1/n, 1/n, \ldots, 1/n)$, and so we have the upper bound
\begin{align}\label{Ineq:nCycleRelaxUpper}
	\tm(\epsilon) \leq \log \left(  \frac{1}{\epsilon \pi_m}\right) \tr \approx \log\left( \frac{n}{\epsilon} \right) \frac{n^2}{2\pi^2}.
\end{align}
If we run the chain for this many steps, we are guaranteed to be past the $\epsilon$-mixing point to stationarity.  The lower bound is similarly
\begin{align*}
    (\tr-1)\log\Big(\frac{1}{2\epsilon} \Big) \approx \left( \frac{n^2}{2\pi^2}-1\right)\log\Big(\frac{1}{2\epsilon}\Big)
\end{align*}
If we only run for fewer than this many steps, we know that we have not reached $\epsilon$-mixing.

For both the upper and lower bounds, we do not have to use the Taylor series expansion for cosine; we can simply use the exact formula \eqref{Eq:nCycleRelax} for the relaxation time instead.  The point of using the Taylor series is to see how the bounds grow as $n$ increases, which becomes explicit in \eqref{Ineq:nCycleRelaxUpper} but is harder to immediately see if we use \eqref{Eq:nCycleRelax}.

\begin{exercise}
	You are running a simple symmetric random walk $(X_n)$ on the vertices of a 23-gon.  Start the chain from any given vertex.  Approximately how many steps $N$ should you run to ensure that the distribution of $X_N$ is within $1/1000$ of the uniform distribution on the vertices of the 23-gon, in total variation norm?  What is the minimum possible number of steps we could use?
\end{exercise}

\begin{exercise}
	Answer the same questions for the 5-cycle.
\end{exercise}

\subsection{Random walk on the hypercube}

Consider the random walk on the vertices of the hypercube $\{-1,1\}^N$ in $\mathbb{R}^N$.  If $N=1$, this is a walk on the two vertices of a line segment, if $N=2$, on the four vertices of a square, if $N=3$, on the eight vertices of the cube.  See Figure \ref{Fig:hypercube} for the case of $N=3$ for the hypercube $\{0,1\}^3$.  (We will use $\{-1,1\}^N$ instead because our formulas work out particularly nicely.)
 
Are these walks aperiodic?  You were asked to compute the number of edges for the hypercube in $\nR^N$ in Exercise \ref{Ex:Hypercube}. This is always an even number, and it is not too hard to see that the walk always has period 2, just like the walk on the $n$-cycle when $n$ is even.  For example, when $N=2$ we have the 4-cycle.  We therefore consider the lazy chain with transition matrix $\tilde{P} = \frac{1}{2}(I + P)$ to get an aperiodic walk.  

What happens to our eigenvalues?  Note that if $\lambda$ is an eigenvalue for $P$ with eigenvector $v$, then 
\begin{align}\label{Eq:LazyLambda}
	\tilde{\lambda} = \frac{1}{2}(1+\lambda)
\end{align}
is an eigenvalue for $\tilde{P}$ with eigenvector $v$.  Indeed, 
\begin{align*}
	\tilde{P}v = \frac{1}{2}\Big(v + Pv\Big) = \frac{1}{2}\Big(v + \lambda v\Big)= \Big(\frac{1}{2}+ \frac{1}{2} \lambda \Big) v,
\end{align*}
as claimed.  But if $-1 \leq \lambda \leq 1$, as is the case for transition matrices, then $0 \leq \frac{1}{2}(1 + \lambda) \leq 1$, and so all the eigenvalues of $\tilde{P}$ are positive.  In particular, there is no -1 eigenvalue, which is again as we expect from Theorem \ref{Thm:PF} $(iii)$ once we consider the lazy walk.

We start the eigenvalues and eigenvectors for $P$, before switching to the lazy walk via \eqref{Eq:LazyLambda}.  The vertices of our hypercube are all vectors $\omega = (\omega_1, \omega_2, \ldots, \omega_N)$, where each $\omega_j = \pm 1$.  Each vertex $\omega$ has a corresponding row and column on $P$, and we can describe the eigenvectors $v$ in terms of their components $v(\omega)$ for each of the $2^N$ vertices $\omega$.  

We can parametrize the eigenvalues the eigenvectors in terms of the $2^N$ subsets $J$ of $\{1,2,\ldots, N\}$.  For each $J \subset \{1,2,\ldots, N\}$, there is an eigenvalue
\begin{align*}
	\lambda_J = 1 -\frac{2 \cdot \#(J)}{N},
\end{align*}
with corresponding eigenvector $v_J$, which has the $2^N$ components
\begin{align}\label{Eq:HypercubeEigenvectors}
	v_J(\omega) = \prod_{j \in J} \omega_j
\end{align}
for each vertex $\omega \in \{-1,1\}^N$.  

An example will clarify, and so let's return to $N=3$.  Here the vertices are
\begin{align*}
	\omega_1 &= (1,1,1),\\
	\omega_2 &= (-1,1,1),\\
	\omega_3 &= (1,-1,1),\\
		&\vdots\\
	\omega_8 &= (-1,-1,-1).
\end{align*}	
Our statement about the eigenvalues and eigenvectors says that if we pick $J = \{1,3\} \subset \{1,2,3\}$, for instance, we get a corresponding eigenvalue
\begin{align*}
	\lambda_{\{1,3\}} = 1- \frac{2 \cdot 2}{3} = - \frac{1}{3},
\end{align*}
with a formula \eqref{Eq:HypercubeEigenvectors} for the eight entries of $v_{\{1,3\}}$,
\begin{align*}
	v_{\{1,3\}}(\omega_1) &= \omega_{11} \cdot \omega_{13} = 1 \cdot 1 = 1,\\
	v_{\{1,3\}}(\omega_2) &=\omega_{21} \cdot \omega_{23} = -1 \cdot 1 = -1,\\
		&\vdots\\
	v_{\{1,3\}}(\omega_8) &= \omega_{81} \cdot \omega_{83} = (-1)(-1) = 1.
\end{align*}	

Moving to the lazy walk $\tilde{P}$ for aperiodicity, we have eigenvalues
\begin{align*}
	\tilde{\lambda}_J = \frac{1}{2}(1+\lambda_J) = 1 - \frac{\#(J)}{N} \in [0,1]
\end{align*}
for each subset $J \subset \{1,2,\ldots, N\}$.  The largest is therefore $\lambda = 1$ for $J = \emptyset$, and the second largest is $\lambda^* = 1 - \frac{1}{N}$ for any singleton set $J$.  Thus the spectral gap is $\gamma^* = 1 - (1- \frac{1}{N}) = \frac{1}{N}$, and our relaxation time is $\tr = \frac{1}{\gamma^*} = N$.  It is not too hard to see that the stationary distribution is again uniform, $\pi = (\frac{1}{2^N}, \ldots, \frac{1}{2^N})$, and so our estimate \eqref{Ineq:MixingTimeBounds} says an upper bound for mixing is
\begin{align}
	\tm(\epsilon) \leq \log \left(  \frac{1}{\epsilon \pi_m}\right) \tr &= \log \left(  \frac{2^N}{\epsilon}\right) N = N^2 \log\left( \frac{2}{\sqrt[N]{\epsilon}} \right). \label{Eq:HypercubeMix}
\end{align}
For instance, if we are in $\mathbb{R}^3$ and want to guarantee that our walk is within a millionth of the stationary distribution $(1/8, 1/8, \ldots, 1/8)$ in TV-norm, we need to run for at most
\begin{align*}
	3^2 \log(2 \cdot \sqrt[3]{10^6}) <48
\end{align*}
steps. Pretty fast, right?  That's an example of the exponential convergence \eqref{Ineq:ConvToStationary1} to stationarity.

When $\epsilon \leq 1$ (as we want for small deviations from stationarity), 
\begin{equation}\label{eq:tmixhypercube}
    \tm(\epsilon) \leq N^2 \log\left( \frac{2}{\sqrt[N]{\epsilon}} \right) \leq N^2 \log\left( \frac{2}{\epsilon} \right),
\end{equation}
and so we see for fixed $\epsilon$ that our upper bound grows like $N^2$.  This is slower than for the $N$-cycle, where our upper bound \eqref{Ineq:nCycleRelaxUpper} grows like $\log(N)N^2$.  In other words, we need fewer steps to mix in a high-dimensional hypercube than in a large $N$-cycle.

Even \eqref{eq:tmixhypercube} itself is off by an order of magnitude. It is known that, in fact, $\tm(\epsilon) \approx \frac{1}{4}N\log N + N \log\left(\frac{1}{\epsilon} \right)$, which grows like $N\log N$, for fixed $\epsilon$. The culprit is that the inequality \eqref{Ineq:MixingTimeBounds} itself is not sharp. Getting sharp mixing time estimates for Markov chains is an active topic of modern research. See the textbook \cite{Peres}.

\begin{exercise}
    Give an intuitive explanation for why the random walk on the hypercube $\{-1,1\}^N$ mixes faster than the the random walk on the $N$-cycle. 
\end{exercise}

\vfill
\pagebreak
\section*{Problems for chapter \ref{Ch:Asymptotics}}
\addcontentsline{toc}{section}{Problems for chapter \ref{Ch:Asymptotics}}

%Takis page 22

\begin{problem}\label{Prob:TVID}
    Show that the TV-distance satisfies
    \begin{align*}
        1 - \|p-q\|_{TV} = \sum_{k=1}^n p_k \wedge q_k.
    \end{align*}
\end{problem}

\begin{problem}\label{Prob:A4.4} Toss a fair coin repeatedly. Let $S_n$ be the number of heads after $n$ tosses. Show that there is a limiting value for the proportion of times that $S_n$ is divisible by $7$, and compute the value for this limit. (\emph{Hint}: Find an appropriate Markov chain that models this phenomenon.)
\end{problem}

\begin{problem}\label{Prob:B4.1} Consider the complete graph $K_n$ on $n$ vertices with loops. That is, $\Omega=\{1,2,\ldots, n\}$ and every edge $\{i,j\}$ is present, including $\{i,i\}$. A random walker starts from a vertex and randomly jumps to any other vertex with equal probability. The transition matrix of the chain is  
\[
P= \frac{1}{n}\mathbf{1}^T \mathbf{1},
\]  
where $\mathbf{1}$ is the row vector of all ones. 
\begin{enumerate}[$(a)$]
\item Show that $P$ has exactly one eigenvalue that is one, and all other eigenvalues are zeroes. 
\item Find the absolute spectral gap and the relaxation time for this chain. Estimate an upper bound on the mixing time $t_{\textbf{mix}}(\epsilon)$ using the relaxation time.  
\item Show that $t_{\textbf{mix}}(\epsilon) =1$ for any $ 0 \leq \epsilon < 1 - 1/n$.  
\end{enumerate}  
\end{problem}

%HW5. Problem 9.
\begin{problem}\label{Prob:B4.2} (Random walks on edges) Consider a connected graph $G=(V,E)$. Instead of the simple symmetric random walk on vertices of this graph, we consider a walk on the \emph{edges}. For an edge $e=\{x,y\}$ between two vertices $x$ and $y$, think of two \textit{directed edges}, $x\rightarrow y$ and $y \rightarrow x$, depending on whether you go from $x$ to $y$ or the other direction. Let $\Omega$ be the set of all directed edges $x\rightarrow y$ and $y\rightarrow x$, where $\{x,y\}$ is an edge in the graph. Run a Markov chain with the following rule: if you are currently at an edge $x\rightarrow y$, then pick uniformly at random one of the neighbors of $y$ (including $x$). Call this neighbor $z$ and jump to the edge $y\rightarrow z$. 
\begin{enumerate}[$(a)$]
\item Is this new RW irreducible? 
\item In the long run, what fraction of time do you see the new random walk visiting a directed edge $x\rightarrow y$? Use ergodicity of the simple symmetric RW on the vertices of the graph to answer this question.
\item Use part $(b)$ to identify the stationary distribution of the new RW on the directed edges. 
\item Is the new RW on the directed edges reversible with respect to the stationary distribution? Why or why not?
\item Compare and contrast your results with the simple walk on the vertices of the graph.
\end{enumerate}
\end{problem}

\begin{problem}\label{Prob:A3.1} Consider $\Omega=\{0,1\}^n$, the hypercube in dimension $n$. Consider the following Markov chain $\{ X_0, X_1, X_2, \ldots \}$ on $\Omega$. Suppose the current state is $w=(w_1, w_2, \ldots, w_n)$. In the next step turn $w$ to the new vector $(w_0, w_1, w_2, \ldots, w_{n-1})$, where $w_0$ is uniformly chosen to be $0$ or $1$. That is, drop $w_n$, shift all coordinates by one step to the right and add a random $0$ or $1$ at the beginning.

\begin{enumerate}[$(a)$] 
\item Is this chain irreducible? Why or why not?   
\item Find its stationary distribution $\pi$. 
\item Suppose $X_0=\mathbf{0}=(0,0,\ldots, 0)$, the vector of all zeros. Find the \emph{exact} total variation distance between $P^k(\mathbf{0}, \cdot)$ and $\pi$ as a function of $k$.
\item Find the smallest $k$ when $\max_x\norm{P^k(x, \cdot)- \pi}_{\mathrm{TV}}\le 1/2$. 
\item Find the smallest $k$ when $\max_x\norm{P^k(x, \cdot)- \pi}_{\mathrm{TV}}=0$.
\item Describe its time-reversed chain by computing its transition probabilities. 
\end{enumerate}
\end{problem}

%Levinson 2

\begin{problem}\label{Prob:A3.3}  For parameters $0 < p,q < 1$ consider the following stochastic matrix
\[
P=\begin{bmatrix}
1-p & p \\
q & 1-q
\end{bmatrix}
\]
Consider a Markov chain with state space $\Omega=\{1,2\}$ and transition matrix $P$. 
The stationary distribution for this chain is $\pi=(q/(p+q), p/(p+q))$.

\begin{enumerate}[$(a)$]
\item Show that this chain is reversible. 
\item Find the eigenvalues of $P$.
\item Find the relaxation time $t_{\mathrm{rel}}$ of this chain and an upper bound on the mixing time $t_{\mathrm{mix}}(\epsilon)$ for $\epsilon=\frac{1}{10}$ using the relaxation time.
\end{enumerate}
\end{problem}

\begin{problem}\label{Prob:A3.4} Consider again the Markov chain from Problem \ref{Prob:A3.3}. Let $\mu_n$ denote $\nP_1(X_n=1)$ and define 
\[
\Delta_n = \mu_n - \frac{q}{p+q}.
\]
\begin{enumerate}[$(a)$]
\item Show that the following recursion holds: 
\[
\Delta_{n+1} = (1-p-q) \Delta_n.
\]
\item Compute the total variation distance 
\[
\norm{P^n(1, \cdot) - \pi}_{\mathrm{TV}}.
\]
\item Compute the mixing time $t_{\mathrm{mix}}(\epsilon)$ for $\epsilon=\frac{1}{10}$ and compare with the upper bound obtained in Problem \ref{Prob:A3.3}.
\end{enumerate}
\end{problem}

\begin{problem}\label{Prob:hw42025}
Consider the star graph where a central vertex $1$ is connected to vertices $2,\ldots,n$, for some $n\ge 3$. Consider the \textbf{lazy} random walk on this graph. Let $\tilde{P}$ denote the transition matrix of this lazy random walk. 

\begin{enumerate}[(a)]
\item Show that the eigenvalues of $\tilde{P}$ are $1$ and $0$ with multiplicity one and $1/2$ with multiplicity $n-2$.  
\item Find the absolute spectral gap and the relaxation time for this chain. Estimate an upper bound on the mixing time $t_{\textbf{mix}}(\epsilon)$ using the relaxation time.  
\end{enumerate}
\end{problem}

\begin{problem}\label{Prob:final2025}
Let $\Omega=\{1,2,\ldots, n\}$. Consider a Markov chain on the state space $\Omega$ with the following transition probability matrix 
\[
P= \frac{1}{2} I +  \frac{1}{2n}\mathbf{1}^T \mathbf{1},
\]  
where $I$ is the $n \times n$ identity matrix and $\mathbf{1}$ is the row vector of all ones. 
\begin{enumerate}[(a)]
\item Can you describe the transition behavior of the Markov chain in words?
\item Show that $P$ has exactly one eigenvalue that is one, and all other eigenvalues are $1/2$. 
\item Find the absolute spectral gap and the relaxation time for this chain. Estimate an upper bound on the mixing time $t_{\textbf{mix}}(\epsilon)$ using the relaxation time.  
\item Argue that $t_{\textbf{mix}}(\epsilon) \le \log_2(1/\epsilon)$ for any $0< \epsilon < 1$.  
\end{enumerate}  
\end{problem}

\begin{problem}\label{Prob:A3.5}  Consider a random walk $X_t$, $t=0, 1,2,3,\ldots$, on the hypercube $\Omega=\{0, 1\}^N$. For $w$ in the hypercube let $\abs{w}:=\sum_{i=1}^N w_i$ denote the number of ones in $w$. 
\begin{enumerate}[(a)]
\item In the long run, what fraction of time does the random walker have exactly $k$ many ones, for $k=0, 1,2,\ldots, N$? 
\item Let $s_t= \abs{X_t}$, and let $\bar{s}_t$ be the empirical average $t^{-1} \sum_{i=0}^{t-1} s_i$.  Compute
\begin{align*}
    \lim_{t \rightarrow \infty} \bar{s}_t.
\end{align*}
\item Evaluate the limit of the empirical variance,
\[
\lim_{t\rightarrow \infty} \frac{1}{t}\sum_{i=0}^{t-1} \left( s_i - \bar{s}_t\right)^2. 
\] 
\end{enumerate}
\end{problem}

\begin{problem}\label{Prob:A3.6} Consider the Ehrenfest urn model with $N\geq 2$ balls in two urns $A$ and $B$. Start with $a$ balls in urn $A$. At each step, pick a ball at random and switch its urn. Let $X_n$ be the number of balls in urn $A$ after $n$ steps, and let $Y_n=N-2X_n$ be the difference of the number of balls in urn $B$ and urn $A$. 

\begin{enumerate}[$(a)$]
\item Show that 
\[
\nE(Y_{n+1} \mid Y_n)= Y_n\left( 1 - \frac{2}{N}\right). 
\]
\item Compute $\nE(Y_n)$ for each $n$ as a function of $X_0=a$ and show that $\lim_{n\rightarrow\infty} \nE(Y_n)=0$. 
\item Let $\tau_N$ be the first time that all $N$ balls are in urn $A$ (i.e., first hitting time of $N$ for $X_n$). Find $\nE_{N-1}(\tau_N)$. (Hint: There is a reason we are not asking for $\nE_k(\tau_N)$ for all $k$.)
\item Start with $X_0=N-2$. Before $\tau_N$, there will be a \textit{last time} that urn $A$ will have $(N-2)$ balls. That is, consider the (random) time 
\[
L_{N-2}= \max\left\{ 0\le t \le \tau_N: X_t=N-2 \right\}.
\]
Find $\nE_{N-2}(L_{N-2})$.  (Hint: compute $\nE_{N-2}(\tau_N)$). 

\item The stationary distribution of the Ehrenfest model is Bin$(N, 1/2)$. Given a number $0<p<1$, $p\neq 1/2$, modify the Ehrenfest model suitably so that the new Markov chain has a stationary distribution Bin$(N,p)$.
\end{enumerate}
\end{problem}

%%%%%%%%%%%%%%%%%%%% Lecture 19

\chapter{Monte Carlo Methods}\label{Ch:MC}

%% Lecture 19 starts here

\section{An introduction to sampling algorithms}\label{sec:samplingintro}

%%Lecture 19

Very often working statisticians will need to \textit{simulate} random variables on their computers. Our computers can only generate a random number uniformly distributed in the unit interval $(0,1)$. In reality, computers do slightly worse since they can only generate pseudo-random numbers that will pass statistical tests for randomness but are actually deterministic. It is a fascinating philosophical discussion to distinguish between true randomness in nature, such as those arising from quantum mechanics, and those generated by pseudo-random generating algorithms, such as the linear congruential generator (look it up!), but in practice this difference is mostly inconsequential (as they say, ``If you can't tell, does it matter?''). But one thing is certain, that humans left to their own devices are terrible at trying to simulate a random sequence of zeroes and ones in their mind. 

Let's assume that you have a \textit{random number generator}, i.e. a computer that generates an i.i.d. sequence of perfect Uni$(0,1)$ random variables. The problem of simulations is to use this sequence to create other random variables and stochastic processes that are required in your model. The objective is to see the output your model produces when the simulated random variables are fed to it. You can then, say, compare the results with real data to verify the efficacy of your model, or perhaps, predict a future outcome. 

Thus, simulation is an enormous big business and a number of standard algorithms are frequently used. The simplest example is to simulate a coin toss with a probability of H given by $0<p<1$. Here is a solution. Generate a Uni$(0,1)$ distributed random number $U$ using your generator. If $U<p$, declare $X=1$ (or $X=H$), otherwise, if $U\ge p$, declare $X=0$ (or $X=T$). It is easy to see that $X$ is a Bernoulli$(p)$ random variable. In other words, I have simulated the outcome of a coin toss with probability $p$ of Heads. 

More generally, the easiest case is the simulation of one-dimensional random variables with explicitly invertible cumulative distribution functions.  

Let's start with a basic observation. Suppose $X$ is a random variable with a cumulative distribution function (cdf) $F$. Recall that the cdf is a function from $\rr$ to $[0,1]$ given by 
\[
F(t)=P\left(X\le t \right).  
\]
Assume that $F$ is continuous, i.e., $X$ is a continuous random variable. Since $X$ is a real-valued, and $F$ acts on real numbers, one gets a new transformed random variable $U=F(X)$. 

I want to warn you that the following statement is nonsense: ``Since, $F(t)=P\left(X\le t \right)$, then $F(X)=P(X\le X)=1$''. Instead, let's see what we mean by an example. Consider $X$ to be an Exp$(1)$ random variable. Then 
\[
F(t)= \begin{cases}
0,\; \text{if $t\le 0$}, \\
1- e^{-t}, \; \text{if $t > 0$}.
\end{cases}
\]
In this case $U=1-e^{-X}$, since $X$ is always positive. Note that $F$ is strictly increasing on $(0, \infty)$, where it can be inverted. In fact, verify that, for $u>0$, $F^{-1}(u)=-\log(1-u)$.

\begin{theorem}\label{Thm:inversecdf} Assume that $F$ is continuous and strictly increasing at all $t$ such that $0< F(t) < 1$. Hence the inverse function $F^{-1}$ is well-defined. Then, the random variable $U=F(X)$ is distributed as Uni$(0,1)$. Conversely, if $U$ is a Uni$(0,1)$ random variable, then $X=F^{-1}(U)$ has cdf $F$.  
\end{theorem}

\begin{proof}
    Let us compute the cdf of $U=F(X)$. Since $F$ has a range between zero and one, clearly $0\le U\le 1$. Pick an $0<a<1$. Since $F$ is invertible with an inverse $F^{-1}$,  
    \[
    P(U\le a)=P\left( F(X) \le a \right)=P\left( X \le F^{-1}(a)\right)=F\left( F^{-1}(a) \right)=a.
    \]
Thus, the cdf of $U$ coincides with that of Uni$(0,1)$, and, hence, by the uniqueness of cdf, $U$ must be uniformly distributed over $(0,1)$.  

Conversely, let $U$ be a Uni$(0,1)$ random variable. Define $X=F^{-1}(U)$. Then 
\[
P(X\le t)= P\left( F^{-1}(U) \le t \right)= P(U\le F(t))= F(t). 
\]
The final equality is due to the cdf of $U$ and since $0\le F(t)\le 1$.
\end{proof}

The assumption in the previous theorem can be relaxed a bit by simply defining an inverse when $F$ is continuous but perhaps not strictly increasing. We will not prove the following statement (but you are welcome to try). Theorem \ref{Thm:inversecdf} continues to hold for all continuous $F$ if we define $F^{-1}(t)=\inf\left\{ x: \; F(x) \ge t \right\}$, the left-continuous generalized inverse of $F$. 

Theorem \ref{Thm:inversecdf} gives us our first algorithm to simulate. Suppose we wish to simulate a random variable $X$ with a cdf $F$ whose inverse $F^{-1}$ is explicitly computable. Then, to get a sample of $X$, simply ask your computer to output a Uni$(0,1)$ random number $U$, and declare $X=F^{-1}(U)$. This is called the inverse cdf method of sampling.

\begin{exercise}
    Describe how you can use the inverse cdf method to sample any Exp$(\lambda)$ random variable for $\lambda>0$. Write a program to execute it. 
\end{exercise}

Unfortunately, the inverse cdf method does not work for some standard distributions. One prominent example is the standard normal whose cdf (and also the inverse cdf) is not explicit. Before we move on to more generally applicable methods of sampling, let us generalize the inverse cdf method to sample discrete random variables. 

Suppose $X$ is a discrete random variable taking values in $\NN$ with a probability mass function $\pi$. That is, $P(X=i)=\pi_i$, $i=1,2,\ldots$, with $\pi_i\ge 0$ and $\sum_{i=1}^\infty \pi_i=1$. Here's a method to simulate a sample of $X$: divide up the closed unit interval $[0,1]$ in successive pieces of sub-intervals of length $\pi_i$, as in Figure \ref{Fig:sampleddiscrete}. Generate a uniformly distributed random number $U$ in $(0,1)$ and note in which sub-interval $U$ falls. Call $X$ to be the number of sub-intervals to the left of $U$, including the one in which it falls. That is, if $U\le \pi_1$, declare $X=1$. If $\pi_1 < U \le \pi_1 + \pi_2$, declare $X=2$. Generally, if 
\begin{equation}\label{eq:stickbreaking}
\sum_{i=1}^{k-1} \pi_i < U \le \sum_{i=1}^{k} \pi_i,
\end{equation}
declare $X=k$. Clearly $X$ takes values in natural numbers. 

We claim that $X$ has pmf $\pi$. To see this, note from \eqref{eq:stickbreaking}, for any $k=1,2,\ldots$,
\[
P(X=k)=P\left(\sum_{i=1}^{k-1} \pi_i < U \le \sum_{i=1}^{k} \pi_i \right)= \sum_{i=1}^{k} \pi_i - \sum_{i=1}^{k-1} \pi_i= \pi_k. 
\]
This completes the proof.

%% Needed figure from Lecture19
    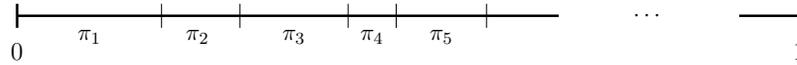
\begin{figure}
        \centering
            \scalebox{0.8}{
    \begin{tikzpicture}
    \tikzset{
        node/.style={circle, draw=black, very thick, minimum size=4mm},
        line/.style={black, very thick}
    }
    \draw[line] (0,0) -- (9,0);
    \draw[line] (12,0) -- (13,0);
    \draw[line] (0,-0.2) -- (0,0.2);
    \draw[line] (13,-0.2) -- (13,0.2);
    
    \draw (2.4,-0.2) -- (2.4,0.2);
    \draw (3.7,-0.2) -- (3.7,0.2);
    \draw (5.5,-0.2) -- (5.5,0.2);
    \draw (6.3,-0.2) -- (6.3,0.2);
    \draw (7.8,-0.2) -- (7.8,0.2);
    
    \node (dots1) at (10.5,0) {$\cdots$};
    \node (pi1) at (1.2,-0.35) {$\pi_1$};
    \node (pi2) at (3,-0.35) {$\pi_2$};
    \node (pi3) at (4.6,-0.35) {$\pi_3$};
    \node (pi4) at (5.9,-0.35) {$\pi_4$};
    \node (pi5) at (7.05,-0.35) {$\pi_5$};
    
    \node (zero) at (0,-0.6) {\large $0$};
    \node (one) at (13,-0.6) {\large $1$};
    
        % \node[node] (one) at (0,1.5) {1};
    
        % \draw[line] (four) -- (six);
    \end{tikzpicture}}
        \caption{\small Sampling from a discrete distribution $\pi$.}
        \label{Fig:sampleddiscrete}
    \end{figure}

\subsection{Rejection sampling} Rejection sampling (sometimes called acceptance-rejection sampling) enables you sample from a probability distribution $\pi$, assuming that you can generate a sample from another probability distribution $p$. Both $p$ and $\pi$ could be either probability mass functions or probability density functions (pdf), as long as the following assumption holds.  

\begin{assumption}\label{Asmp:boundedratio}
$\pi(x)=0$ if $p(x)=0$ and for some $M>0$ the ratio $\pi(x)/p(x)\le M$, if $p(x)>0$.  
\end{assumption}

Note that, one can take $M= \max_{x: p(x)>0}\frac{\pi(x)}{p(x)}$ if this quantity is finite. But you don't have to. Any number larger than this maximum will also do as $M$. The point is to find an $M$ such that the ratio $0<\pi(x)/Mp(x)<1$, whenever the ratio is well-defined. There is also an implicit assumption here: for any $x$ one can compute the quantity $\pi(x)/p(x)$ explicitly.   

The algorithm for rejection sampling to get a sample from $\pi$ under Assumption \ref{Asmp:boundedratio} goes like this.
\smallskip

\noindent\textbf{Step 1.} Generate a sample $Y$ from $p$.
\\
\textbf{Step 2.} Generate an independent Uni$(0,1)$ random variable $U$. 
\\
\textbf{Step 3.} If $U \le \frac{\pi(Y)}{Mp(Y)}$, declare $X=Y$. Stop and return the output.
Otherwise, discard the sample $Y$. Go back to Step 1. 
Continue until you stop, i.e., an output $X$ is generated. $X$ is clearly a random variable.

\begin{theorem}
The distribution of the final output $X$ is $\pi$. 
\end{theorem}

\begin{proof}
Suppose $Y$ is distributed according to $p$. For simplicity, we are going to assume that $p$ (and $\pi$) is a pmf. Let $I$ be the indicator random variable $I=1\{ U\le \pi(Y)/Mp(Y) \}$. Then, note that $X$ is only outputted when $I=1$. Thus 
\[
P(X=x)= P\left( Y=x \mid I=1 \right).
\]
(By the way, ask yourself, why is $P(X=x)\neq P\left( Y=x, I=1\right)$).  

Thus
\[
\begin{split}
P(X=x)&= \frac{P(Y=x, I=1)}{P(I=1)}= \frac{P(Y=x) P(I=1 \mid Y=x)}{P(I=1)}\\
&=\frac{p(x)P\left( U \le \pi(x)/Mp(x) \right)}{P(I=1)}\\
&= \frac{p(x) \pi(x)/ Mp(x)}{P(I=1)}= \frac{\pi(x)}{M P(I=1)}.
\end{split}
\]
At this point, we will be done if we show that $MP(I=1)=1$. But this has to be so. Just add over $x$ on both sides to get 
\[
1= \sum_x P(X=x)= \sum_{x} \frac{\pi(x)}{M P(I=1)}= \frac{\sum_x \pi(x)}{MP(I=1)}  = \frac{1}{M P(I=1)}.
\]
Thus $MP(I=1)=1$, and we are done.
\end{proof}

Note that the probability of \textit{acceptance}, i.e., $P(I=1)$, is exactly $\frac{1}{M}$. Since this quantity must be at most one, $M\ge 1$ always.  
$M=1$ if and only if $p=\pi$, when the algorithm is not needed. So $M>1$ in all nontrivial cases. 
The larger the value of $M$ is, the more frequently the algorithm will reject. That is, the run time of the algorithm grows linearly with $M$. In fact, convince yourself that the number of times the algorithm runs through Steps 1--3 before finally accepting a sample is Geo$(1/M)$.      

\begin{example}
Let's see an example. Let $\pi$ be the density 
\[
\pi(x)=30x^2(1-x)^2,\quad 0< x< 1.
\]
This is the beta$(3,3)$ probability density. Take $p$ to be the Uni$(0,1)$ density. Then, clearly $p(x)=0$ if $x$ is outside $(0,1)$, and, for such an $x$, $\pi(x)=0$ as well. When $x\in (0,1)$, 
\[
\frac{\pi(x)}{p(x)}= \frac{30x^2(1-x)^2}{1} = 30 (x(1-x))^2\le 30 \left(\frac{1}{4} \right)^2= \frac{30}{16}=\frac{15}{8}.
\]
The only inequality above follows from the bound $x(1-x)\le \frac{1}{4}$, for all $0\le x\le 1$ (why?). Thus, we can take $M=15/8$. 

Thus the rejection sampling algorithm to generate a sample from beta$(3,3)$ follows these steps. 

\noindent\textbf{Step 1.} Generate a Uni$(0,1)$ random variable $Y$.\\
\textbf{Step 2.} Generate an independent Uni$(0,1)$ random variable $U$. \\
\textbf{Step 3.} If 
\[
U\le \frac{30 Y^2(1-Y)^2}{15/8}= 16Y^2(1-Y)^2,
\]
output $X=Y$ as your desired sample and stop. Otherwise, discard $Y$ and $U$, and go back to Step 1. 

When the algorithm stops, your output $X$ is a sample from beta$(3,3)$. 
\end{example}

\begin{example}
For a second example, let's consider a geometric probability problem. Suppose I wish to sample uniformly from the region $D \subseteq [0,1]^2$. See Figure \ref{Fig:regionD}. 

\begin{figure}
    \centering
    \scalebox{0.85}{\begin{tikzpicture}
    \draw (-2,2) -- (2,2) -- (2,-2) -- (-2,-2) -- (-2,2);
    \node (one) [label=above:{\large$1$}] at (-2,2) {};
    \node (zero) [label=below:{\large$0$}] at (-2,-2) {};
    \node (oneb) [label=below:{\large$1$}] at (2,-2) {};
    \node (D) [label=below:{\LARGE$D$}] at (0,0) {};
    \draw  plot[smooth, tension=.7] coordinates {(-0.3,0.6) (0,0.9) (0.3,1.2)
    (0.6,1.5) (.9,1.5) (1.2,1.2) (1.2,0.3) (1.2,-1.2) (0,-1.8) (-1.2,-1.2) (-0.3,0.6)};
    \end{tikzpicture}}
    
    \caption{\small How can we sample a point uniformly from the region $D$?}
    \label{Fig:regionD}
\end{figure}
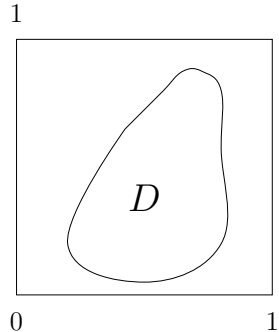
%% Insert figure from Lecture19

So, $\pi$ is the uniform density over $D$, i.e., $\pi(x)= 1/\mathrm{Area}(D)$, for $x\in D$, and zero outside. What is an easy density to sample from? A natural choice is Uni$\left([0,1]^2 \right)$, since this is just sampling a pair of independent Uni$(0,1)$ random variables $Y_1, Y_2$ and representing them as a vector $Y=(Y_1, Y_2)$. 

Thus $p(y)=1$, for $y=(y_1, y_2)\in [0,1]^2$. Hence, 
\[
\frac{\pi(y)}{p(y)} \le \frac{1}{\mathrm{Area}(D)}. 
\]
Hence, we can take $M=1/\mathrm{Area}(D)$, assuming that we can compute it. Thus
\[
\frac{\pi(y)}{M p(y)}= \begin{cases}
\frac{1/\mathrm{Area}(D)}{1/\mathrm{Area}(D)}=1, \quad \text{if $y\in D$},\\
\frac{0}{1/\mathrm{Area}(D)}=0, \quad \text{otherwise}.
\end{cases}
\]
Note that this ratio is either zero or one. So, if $Y$ is sampled from $p$ and $U$ is independently sampled from Uni$(0,1)$, then $U\le \pi(Y)/Mp(Y)$ if and only if $Y\in D$.

Hence, the rejection sampling algorithm is quite simple here. Generate a vector $Y=(Y_1, Y_2)$, where $Y_1, Y_2$ are independent Uni$(0,1)$ random variables. If $Y$ is in $D$, accept the sample and declare $X=Y$. Otherwise generate a fresh sample of $Y$. Repeat until you get a sample that lies in $D$. 

As you can guess from here, the smaller $\mathrm{Area}(D)$ is, the longer it takes to generate a sample from $\pi$. This becomes a serious problem in high-dimensions. In the next section we will see how Markov chains combined with the rejection algorithm can lead to much more efficient sampling algorithms that work even in high-dimensions.
\end{example}

\begin{exercise}
    Describe how you will sample from the following density using rejection sampling
    \[
    \pi(x)= \frac{e^{-x}}{1-e^{-1}},\quad 0\le x\le 1,
    \]
    and zero otherwise. Give at least two choices for the density $p$ that you can use to simulate $\pi$. 
\end{exercise}

\begin{exercise}
Given a random number generator, describe any method to sample Geo$(p)$ distribution for any $0<p<1$. Explain why the rejection sampling method would fail if we try to sample $\pi=\text{Poi}(1)$, using $p=\text{Geo}(1/2)$.  
\end{exercise}

\section{Markov Chain Monte Carlo}

%% Lecture 20

Consider the set $\Omega$ of all permutations of the set $\{1,2,\ldots, n\}$. How many permutations are there? $n!$, a number so large that it grows faster than all exponentials. There is a classic estimate, called Stirling's approximation, that shows $n!\approx \sqrt{2\pi}n^{n+1/2} e^{-n}$. Thus, $\Omega$ might be finite but it is gigantic even for moderate values of $n$, say $n=15$. The problem is to simulate a uniformly distributed random permutation from this set. I will describe two solutions to this problem, one not using Markov chains, and another using a Markov chain. 
\medskip

\noindent\textbf{Solution 1.} Generate $n$ i.i.d. Uni$(0,1)$ random variables $U_1, U_2, \ldots, U_n$. Sort them from the smallest to the largest values. That is, let $U_{(1)}=\min_{1\le i\le n} U_i$, denote the smallest of the values, $U_{(2)}$ be the second smallest, and so on, until $U_{(n)}=\max_{1\le i \le n} U_i$ is the largest value. For example, if our $U_i$'s are, to three decimal places, 
\begin{align*}
    U_1 = 0.520, \quad U_2 = 0.512, \quad U_3 = 0.765, \quad U_4=0.428,
\end{align*}
then
\begin{align*}
    U_{(1)} = 0.428, \quad U_{(2)} = 0.512, \quad U_{(3)} = 0.520, \quad U_{(4)}= 0.765.
\end{align*}
This sorting generates a permutation $\sigma$: the value of $\sigma_1$ is the \textit{rank} of $U_1$, that is, the index $i$ such that $U_{(i)}=U_{1}$ (why is $i$ unique?). Similarly, $\sigma_2$ is the rank of $U_2$, the index $i$ such that $U_{(i)}=U_{2}$, and so on. In our example, our permutation $\sigma = (\sigma_1, \sigma_2, \sigma_3, \sigma_4)$ is
\begin{align*}
    \sigma = (3,2,4,1).
\end{align*}
In general, the vector $\sigma=(\sigma_1, \sigma_2, \ldots, \sigma_n)$ is a random element of $\Omega$, and it follows from symmetry that the distribution of $\sigma$ must be uniform over all permutations. 

The complexity of this algorithm is therefore determined by the complexity of the sorting algorithm used. Common sorting algorithms typically have a time complexity of the order of $n\log n$. This is, of course, a much much smaller number than the size of $\Omega$ itself, which is $n!$. Our next algorithm is a Markov chain based algorithm which also achieves similar time-complexity.  

\medskip

\noindent\textbf{Solution 2.} The goal is to create a reversible Markov chain on $\Omega$ that moves by ``small steps''. There are many such candidate Markov chains which are often visualized best by card shuffling. We will talk about a particularly nice one called the \textit{random transposition} chain. 

Imagine a deck of cards marked $1$ through $n$ on top of each other. The goal is to shuffle these cards. Although many of you have shuffled cards in real life, in mathematics this has a precise meaning. The goal of a card shuffling technique is to arrange the cards randomly such that every arrangement of the $n$ cards is equally likely to appear from top to bottom. Clearly every arrangement from top to bottom of the $n$ cards corresponds to a permutation and vice versa. %See Figure \ref{Fig:cards}. 

%% Lecture 20

Now, given a deck of cards, here is one method of shuffling. Pick two random numbers $I$ and $J$ independently from $\{1,2,\ldots, n\}$. Suppose $I=i, J=j$. Pick the $i$th card from the top and swap its position with the $j$th card from the top. There is a trivial case, if $I=J$, when the deck does not change at all. Otherwise the deck changes by a ``small step'' due to the swapping of the two cards. Repeat this procedure, by sampling a fresh pair of $I, J$ independently. 

It is clear that this creates a Markov chain on $\Omega$, given by the arrangement of the cards. In fact, if $I=i, J=j$, you hop from one permutation to another by multiplying with the transposition $(i,j)$. There is a result in combinatorics which says that every permutation can be obtained as a product of transpositions. It follows that this Markov chain is irreducible. Since $P(I=J)=1/n>0$, the chain is aperiodic as well.

Notice that, if you go from one arrangement $\sigma$ to another arrangement $\sigma'\neq \sigma$ by swapping two cards, you can also go from $\sigma'$ to $\sigma$ by swapping the same pair of cards again (with the same probability $1/n^2$). Thus, the Markov chain is reversible and its unique stationary distribution is the uniform distribution over $\Omega$. \emph{Et volia!}    

Note that each step of the Markov chain simply requires generating two (discrete) uniform random variables. Convince yourself that this can be done easily by a random number generator. Thus, equipped with a random number generator, one can run this Markov chain for as many steps as one wishes. How many steps? The following is a deep result (see \cite{diaconis1988group}) that is proved by using an area of mathematics called \textit{representation theory}.

\begin{theorem}
For $c>0$, let $k=\frac{1}{2}n\log n + cn$. Then, 
\[
\max_{\sigma \in \Omega} \norm{P^k(\sigma, \cdot) - \mathrm{Uni}}_{TV} \le ae^{-2c}, 
\]
for some universal constant $a$ that does not depend on $n$ or $c$. 
\end{theorem}

In other words, $t_{\text{mix}}(\epsilon)$ is about $\frac{1}{2}n\log n$ (plus a multiple of $n$ depending on $\epsilon$, which is of a smaller order). Thus, this random procedure will give you an approximately uniformly distributed permutation at about the same order of efficiency as the sorting algorithm. 

This is our first example of Markov Chain Monte Carlo or MCMC for short, a workhorse of modern data science. MCMC is a procedure where one uses a Markov chain to generate a sample from a specified distribution.  

Let's start with a situation very similar to the above algorithm. Suppose $\Omega$ is a large finite set, the operative word being large, and we wish to sample an element of $\Omega$ uniformly at random. Create a connected graph with vertex set $\Omega$ and a relatively small number of edges per vertex. If possible, make the graph regular. Then the lazy random walk on this graph is reversible with respect to the uniform distribution. Run this random walk until it mixes, and the terminal value will be an approximate sample from the uniform distribution over $\Omega$. If you want multiple samples, run the entire chain multiple times, independently.

In the example of random transpositions, $\Omega$ is the set of permutations. The graph is the following. Two permutations $\sigma$ and $\sigma'$ are neighbors, if one can be obtained by multiplying the other by a transposition $(i,j)$. This is a regular graph where every vertex has degree exactly $n(n-1)/2$. You can verify that the random transposition chain is similar to a lazy random walk on this graph. 

How do I create a Markov chain to simulate from a distribution that is not uniform? There are as many choices as your creativity. But there are some choices that are standard and popular for various reasons. The next section describes the grand-daddy of such Markov chains that runs by combining a random walk with rejection sampling at each step. This is the famous Metropolis algorithm.

\section{The Metropolis-Hastings algorithm}\label{Sec:Metropolis}

Suppose $\Omega$ is a finite state space. Assume that we can run an irreducible Markov chain on $\Omega$ with a symmetric transition matrix $q(x,y)=q(y,x)$. The stationary distribution is uniform over $\Omega$, and hence, after running this chain for a large number of steps, the terminal value is approximately a sample from the uniform distribution. But what if we are interested in sampling from another probability distribution $\pi$ on $\Omega$ that is not uniform? This question led to one of the top five impactful algorithms of the 20th century, the Metropolis-Hastings algorithm. For brevity, we will refer to this algorithm as simply Metropolis.

The Metropolis algorithm creates a reversible Markov chain $X$ on $\Omega$ with distribution $\pi$. Suppose currently $X_t=x$. Let us describe the probability $P(X_{t+1}=y \mid X_t=x)$. 

\noindent\textbf{Step 1.} Generate a step from the symmetric transition matrix $q$ (which is called the \textit{base chain}). That is, create an auxiliary random variable $Y_{t+1}$ such that $$P\left( Y_{t+1}=y \mid X_t=x \right)=q(x,y).$$

\noindent\textbf{Step 2.} Suppose $Y_{t+1}=y$. If $\pi(y) \ge \pi(x)$, \textit{accept} this step and declare $X_{t+1}=y$. 

Otherwise, $\pi(y) < \pi(x)$, in which case you accept this step with probability $\pi(y)/\pi(x)$. That is to say, generate a coin toss with probability of $H$ given by $\pi(y)/\pi(x)$. If the coin lands Heads, declare $X_{t+1}=y$, otherwise, if the coin lands Tails, \textit{reject} the step by keeping $X_{t+1}=X_t=x$.

\begin{theorem}\label{Thm:MetroSymm}
$X_t,\; t=0,1,2,\ldots$ is a reversible Markov chain with stationary distribution $\pi$.
\end{theorem}

\begin{proof}
Assume, for simplicity, that $q(x,x)=0$ for all $x$. Suppose for every distinct pair $(x,y)\in \Omega \times \Omega$, I have a probability $a(x,y)$ (between zero and one) to accept the move $Y_{t+1}=y$ from $X_t=x$. That is, 
\[
P\left( Y_{t+1}=y \mid X_t=x\right)= q(x,y), \; P\left(X_{t+1}=y \mid X_t=x, Y_{t+1}=y \right)= a(x,y), 
\]
while
\[
P\left(X_{t+1}=x \mid X_t=x, Y_{t+1}=y \right)= 1- a(x,y).
\]
Thus the transition probabilities for the chain $X$ are 
\[
\begin{split}
    p(x,y)= q(x,y)a(x,y), \; y\neq x,\quad  p(x,x)= 1 - \sum_{y} q(x,y) a(x,y). 
\end{split}
\]

What kind of choices for the function $a(\cdot, \cdot)$ will lead to $p$ satisfying the DBEs for the distribution $\pi$, 
\begin{equation}\label{eq:dbemetro}
\pi(x) p(x,y)= \pi(y) p(y,x), 
\end{equation}
for all $x\neq y$?
Suppose \eqref{eq:dbemetro} holds, then, for all $x\neq y$, $\pi(x)q(x,y)a(x,y)=\pi(y)q(y,x)a(y,x)$. Thus, necessarily
\[
\pi(x) a(x,y) = \pi(y) a(y,x).
\]
Since $0\le a(x,y)\le 1$, we must have $\pi(x)a(x,y)\le \pi(x)$. On the other hand, from the last display, $\pi(x) a(x,y)=\pi(y) a(y,x)\le \pi(y)$, since $0\le a(y,x)\le 1$. Therefore we conclude that
\[
\begin{split}
\pi(x) a(x,y) &\le \min\left(\pi(x), \pi(y) \right),\; \text{or, equivalently,}\\
a(x,y) &\le \min\left(1, \frac{\pi(y)}{\pi(x)} \right).
\end{split}
\]

It is in our interest to maximize $a(x,y)$ while maintaining \eqref{eq:dbemetro} in order to minimize run time. Thus, the largest possible value of $a(x,y)$ we can take is $a(x,y)=\min\left(1, \pi(y)/\pi(x) \right)$. In other words, 
\eq\label{eq:whatisa}
a(x,y)=\begin{cases}
1, & \text{if $\pi(y) \ge \pi(x)$},\\
\pi(y)/\pi(x), & \text{if $\pi(y) < \pi(x)$}.
\end{cases}
\en
Verify directly that this formula for $a(x,y)$ satisfies \eqref{eq:dbemetro}: 
\[
\begin{split}
\pi(x)a(x,y)&= \pi(x) \min\left(1, \pi(y)/\pi(x) \right)= \min(\pi(x), \pi(y))\\
&= \pi(y) \min\left( \pi(x)/\pi(y), 1 \right) = \pi(y) a(y,x).
\end{split}
\]
This completes the proof. 
\end{proof}

In Theorem \ref{Thm:MetroSymm}, the base chain is symmetric and we wish to simulate from a non-uniform distribution $\pi$. But, what about the case when it is easy to simulate from a non-uniform distribution, but we wish to simulate from $\pi$ equal to the uniform distribution on $\Omega$. Here is a such an example: suppose we are given a graph that is not regular, i.e., vertices do not have a constant degree. It is easy to run a random walk on this graph, but the stationary distribution will not be uniform. Recall that the stationary distribution at a vertex $v$ is in fact $\deg(v)/2\abs{E}$, where $\deg(v)$ is the degree of the vertex $v$ and $\abs{E}$ is the total number of edges in the graph. How can I modify this random walk to obtain a uniform pick from the vertex set? The generalized version of the Metropolis algorithm allows you to do just that.    

\textbf{Metropolis algorithm for a general base chain.} Suppose $q(x,y)$ is the transition probability from $x$ to $y$ for an arbitrary irreducible Markov chain on the finite state space $\Omega$. Let $\pi$ be any probability distribution on $\Omega$. The Metropolis chain is given by the following steps. Suppose, currently, $X_t=x$. 

\noindent\textbf{Step 1.} Generate a step from $q$, i.e., let $Y_{t+1}$ be given by
\[
P\left(Y_{t+1}=y \mid X_t=x \right)= q(x,y). 
\]

\noindent\textbf{Step 2.} Suppose $Y_{t+1}=y$. Accept $X_{t+1}=y$ with probability 
\begin{equation}\label{eq:whatisa2}
a(x,y)=\min\left(1, \frac{\pi(y) q(y,x)}{\pi(x) q(x,y)} \right).
\end{equation}
Otherwise, reject the sample and declare $X_{t+1}=X_t=x$.

\begin{theorem}
$(X_t,\; t=0,1,2,\ldots)$ is a Markov chain that is reversible with stationary distribution $\pi$. 
\end{theorem}

Since the proof is exactly the same as before and we omit it (fill in the details, based on the previous proof, if you need to convince yourself). Let us instead go back to our problem of sampling uniformly from a non-regular graph. There, $q$ is the transition probability of the simple random walk. Thus 
\[
q(x,y)= \frac{1}{\deg(x)}1\{y\sim x\}. 
\]
Take $\pi$ to be the uniform distribution. Thus $\pi(x)=\frac{1}{\abs{\Omega}}$, where $\abs{\Omega}$ is the size of $\Omega$, the vertex set. Thus, 
\[
a(x,y)= \min\left(1, \frac{1/\abs{\Omega}}{1/\abs{\Omega}} \frac{1/\deg(y)}{1/\deg(x)} \right)1\{y \sim x\}= \min\left( 1, \frac{\deg(x)}{\deg(y)}\right)1\{y \sim x\}.
\]
Thus, the Metropolis algorithm will run in the following manner. Suppose the chain is currently at vertex $x$. Choose one of the neighbors of $x$ uniformly at random. Say, the chosen neighbor is $y$. If $\deg(y)\le \deg(x)$, the chain jumps to $y$. Otherwise, if $\deg(y) > \deg(x)$, the chain jumps to $y$ with probability $\deg(x)/\deg(y)$ and stays at $x$ with probability $1- \deg(x)/\deg(y)$.

Thus the chain tends to move towards lower degree vertices to counterbalance the tendency of the simple random walk to move towards higher degree vertices. This changes the stationary distribution of the simple random walk, which gives more mass to vertices of higher degree, to the uniform distribution over all vertices. That's it! That's the secret!

\section{Sampling from the Gibbs distribution}

%Lecture 22

The acceptance probabilities appearing in \eqref{eq:whatisa} and \eqref{eq:whatisa2} only depend on the ratio $\pi(y)/\pi(x)$. This simple observation has an enormous advantage in practice. Let us demonstrate via an example. 

Suppose we want to generate a random graph (a social network, say) between $N$ individuals. Here the vertices are given by the set $V=\{1,2,\ldots, N\}$ and there is an edge $\{i,j\}$ if individuals $i$ and $j$ are friends of each other. How many possible such graphs are there? There are $N(N-1)/2$ many possible pairs of friends. Each such pair could be friends or could not be friends. This tells us there are $2^{N(N-1)/2}$ many possible graphs with vertices $V$. Let $\Omega$ denote the set of all such possible graphs. Thus $\abs{\Omega}=2^{N(N-1)/2}$, growing super-exponentially with $N$. Any respectable social network should have at least thousands of users which gives you an idea of how large $\Omega$ can be.   

A natural probability distribution on $\Omega$ is the uniform distribution. There is a nice way to generate a sample from this uniform distribution, called the Erd\H{o}s-R\'enyi random graph model. For each possible pair $\{i,j\}$, toss a fair coin. If the coin is H, draw the edge, and if the coin turns T, don't draw the edge. I will leave it to you to convince yourself that the resulting random graph is distributed uniformly on $\Omega$. The uniform distribution is also the stationary distribution of the Markov chain on $\Omega$ where at each step you choose a random pair $\{i,j\}$. If there is an edge $\{i,j\}$, remove that edge, while if there is no such edge, draw that edge. As you can see this creates a symmetric transition matrix and the resulting Markov chain has a uniform stationary distribution. In fact, the Markov chain is nothing but the random walk on the hypercube in dimension ${N(N-1)/2}$.    

Real social networks are not uniformly distributed. If A is a friend of B and B is a friend of C, then it is quite likely that $A$ is a friend of $C$. This dependence is not captured in the Erd\H{o}s-R\'enyi model where the three edges $\{A,B\}, \{B, C\}$ and $\{C, A\}$ are independent. How would one go about modeling such random graphs? 

One commonly proposed solution is the following. Given a graph $G=(V, E)$, count the number of triangles in the graph. That is, let \[
\Delta(G)=\#\left\{ \{i,j,k\}:\; i \sim j, j\sim k, k\sim i\right\}. \]
For a constant $\beta$, create a probability distribution on $\Omega$ given by the formula 
\[
\pi_\beta(G)=\frac{e^{\beta \Delta(G)}}{\sum_{G' \in \Omega} e^{\beta G'}}=\frac{e^{\beta \Delta(G)}}{Z_\beta}, 
\]
where $Z_\beta=\sum_{G' \in \Omega} e^{\beta G'}$, where the sum is taken over all possible graphs in $\Omega$. $Z_\beta$ is called the normalizing constant and is not explicitly computable.

Clearly $\pi_\beta$ is a probability mass function on $\Omega$. How does it behave? If $\beta >0$, the larger $\Delta(G)$ is, the bigger is the weight $e^{\beta \Delta(G)}$. Thus, for $\beta >0$, $\pi_\beta$ favors graphs with more triangles. In fact, the larger $\beta>0$ is, graphs with more triangles become more likely under $\pi_\beta$. 

On the other hand, if $\beta <0$, the weight $e^{\beta \Delta(G)}$ gets smaller as $\Delta(G)$ gets larger. Thus, $\pi_\beta$ favors graphs with fewer triangles. When $\beta=0$, $\pi_\beta$ is simply the uniform distribution on $\Omega$. 

Models such as $\pi_\beta$ are called Gibbs distributions and are commonly used to incorporate dependency structures in data. Originally, such models come from a field called statistical physics where $1/\beta$ is called the temperature. The primary difficulty in trying to sample from $\pi_\beta$ is the non-computability of the normalizing constant $Z_\beta$. 

Fortunately, Metropolis does not require $Z_\beta$. Take the base chain to be the symmetric chain on $\Omega$ described above. You can turn this chain aperiodic by choosing $I$ and $J$ i.i.d. uniformly at random from $\{1,2,\ldots, N\}$. If $I=J$, stay where you are. Otherwise if $I=i\neq J=j$, then erase the edge $\{i,j\}$ if it exists, or add it if it does not. Thus $G$ can jump to $G'$ if and only if these two graphs differ by at most one edge, and $q(G, G')= q(G', G)= 2/N^2$.

In order to modify this chain according to Metropolis, we compute $a(G, G')$ from \eqref{eq:whatisa}, Notice 
\[
\frac{\pi(G')}{\pi(G)}= \frac{e^{\beta \Delta(G')}/ Z_\beta}{e^{\beta \Delta(G)}/ Z_\beta}= e^{\beta(\Delta(G') - \Delta(G))}. 
\]
This ratio does not depend on the normalizing constant at all!

Recall we are interested in the case of $\beta >0$ (favors more triangles). Thus, $\pi(G') \ge \pi(G)$ if and only if $\Delta(G') \ge \Delta(G)$. Thus, if you add an edge (thereby increasing the number of triangles), jump from $G$ to $G'$. If you remove an edge (thus possibly decreasing the number of triangles), jump from $G$ to $G'$ with probability $e^{\beta(\Delta(G') - \Delta(G))}$. This Markov chain run for a large number of steps will be approximately distributed according to $\pi_\beta$. 

\section{Gibbs sampling.}

So far from Section \ref{sec:samplingintro} we learned that simulating one-dimensional probability distributions are often easy. What if we can simulate a multidimensional probability distribution $\pi(x_1, \ldots, x_n)$ by reducing the problem to one dimension? This is the idea behind Gibbs sampling which is also called Glauber dynamics in certain contexts. I want to warn about the nomenclature: Gibbs sampling is a Markov chain Monte Carlo algorithm while the Gibbs distribution introduced in the last section is a probability distribution. They are both named after the American physicist Josiah Willard Gibbs (1839--1903), one of the founders of the field called \textit{statistical mechanics}. 

We describe the Gibbs sampling algorithm below for a joint pmf, but a similar algorithm works for a joint pdf as well. So, let $\pi$ be a pmf in $\rr^n$. Given a vector $x=(x_1, \ldots, x_n)\in \rr^n$, let
\[
x^{i}=\left(x_1, \ldots, x_{i-1}, x_{i+1}, \ldots, x_n \right), \quad i=1,2,\ldots, n. 
\]
Thus $x^i \in \rr^{n-1}$ is simply the vector $x$ whose $i$th coordinate is dropped. Consider now the conditional pmf, under $\pi$, of $X_i$ at $x_i$, given the rest of the values $X_1=x_1, \ldots, X_{i-1}=x_{i-1}, X_{i+1}=x_{i+1}, \ldots, X_n=x_n$, 
\[
\pi^i_{x^i}(y):=P\left( X_i=y \mid X^i = x^i \right), \quad i=1,2,\ldots, n. 
\]
These are $n$ one-dimensional conditional distributions. Assume that it is possible to sample from each of these, no matter the value of the conditioned variables. We will combine these distributions to run an MCMC algorithm.

\medskip

\nin\tbf{Gibbs sampling algorithm.} Start from an arbitrary vector $X(0)$. Suppose $X(t)=x=(x_1, \ldots, x_n)$. Pick an independent random variable $I$ uniformly distributed among $\{1,2,\ldots, n\}$. Suppose $I=i$. \textit{Update} the $i$th coordinate by drawing a sample from the conditional distribution $\pi^i_{x^i}$. Let the value of this fresh sample be $y$. Then, define $X(t+1)$ by  
\[
X_i(t+1)=y, \; \text{and}\; X_j(t+1)= X_j(t), \; \text{for all $j\neq i$}.
\]
Repeat. 

\medskip

\begin{theorem}
The Markov chain $X$ is reversible with respect to the stationary distribution $\pi$. 
\end{theorem}

\begin{proof}
Let's check Detail Balance Equations (DBE). We only consider one case, $I=1$. The rest are similar. Suppose $x$ and $x'$ are two vectors that differ only in the first coordinate. That is $x=(z,x_2, \ldots, x_n)$ and $x'=(y, x_2, \ldots, x_n)$. Then $X(t)=x, X(t+1)=x'$ is only possible if $I=1$. Thus 
\[
\pi(x)p(x, x')=\pi(x)P(I=1)\pi^1_{x^1}(y)=\pi(x)\frac{1}{n}\pi^1_{x^1}(y). 
\]
But, by definition of conditional pmf, $\pi(x)= P(X^1=x^1)\pi^1_{x^1}(z)$. Thus 
\[
\pi(x)p(x, x')=\frac{1}{n} P(X^1=x^1)\pi^1_{x^1}(z)\pi^1_{x^1}(y)= \pi(x')p(x', x),
\]
where the last equality follows by symmetry. This verifies the DBEs and we are done! Convince yourself that this chain is irreducible and aperiodic.
\end{proof}

Let us do an example. 

\begin{example}
Fix a natural number $N$. Let $S_N$ denote the set of natural numbers $(m,k)$ such that $m+k\le N$. That is 
\[
S_N = \left\{(m,k):\; m\ge 1, \; k\ge 1,\; m+k \le N \right\}.
\]
Let us use Gibbs sampling to pick a uniformly distributed sample $(X,Y)$ from $S_N$. 

In order to do this we have to identify the two conditional distributions. Clearly, if $X=m$, then $Y$ can be any of $\{1,2\ldots, N-m\}$ with equal probability. Thus the conditional distribution of $Y$, given $X=m$, is Uni$\{1,\ldots, N-m\}$, i.e., 
\[
P\left(Y=k\mid X=m \right)=\frac{1}{N-m}, \quad k=1,2, \ldots, N-m.
\]
Exactly in the same way, the conditional distribution of $X$, given $Y=k$, is Uni$\{1, \ldots, N-k\}$. Hence, Gibbs sampling proceeds by picking a random integer $I$ uniformly from the pair $\{0,1\}$. If $I=0$, update the current value of $X$ by sampling a uniform natural number between $1$ and the $N-k$, where $k$ is the current value of $Y$. If $I=1$,update the current value of $Y$ by sampling a uniform natural number between $1$ and the $N-m$, where $m$ is the current value of $X$. Repeat several times to get an approximate sample from the uniform distribution over $S_N$.  

Notice here that the algorithm can be made slightly more efficient. If $I=0$ twice in a row, we are updating the value of $X$ twice successively while keeping the current value of $Y$. This means that the first update is thrown away and the only the second update is retained. This is a waste of computational resources. Thus, it is more efficient to not choose $I$ at random at all but rather to alternate between $0$ and $1$. That is, $I=0$ in the first step, followed by $I=1$, followed by $I=0$, followed by $I=1$, and so on. This is a common variant of the Gibbs sampling algorithm.  
\end{example}

\begin{example}
For our second example let us assume that Gibbs sampling works for continuous densities (it does). Sample from the joint density 
\[
f(p,q,r)= \frac{1}{W} pr,\quad 0<p<q<1,\quad 0<r<1,
\]
where $W$ is the normalizing constant. $W$ is unknown and I urge you to keep it that way and not compute it. It is unimportant for Gibbs sampling.  

Let $(X,Y,Z)$ be random variables with joint density $f$. We need to compute the three conditional densities of $X$, given $Y=q,Z=r$, $Y$, given $X=p, Z=r$, and $Z$, given $X=p, Y=q$. 

Let us start from the conditional density of $Y$ at $q$ given the other two. Note that the joint density $f$ does not involve $q$ explicitly in the formula. Thus 
\[
f_{Y}(q \mid X=p, Z=r)=c, \; p< q < 1,
\]
where $c$ is a constant that does not depend on $q$. This is only possible if this conditional density is uniform over the interval $[p,1]$. Since uniform densities are explicit, $c$ must be $1/(1-p)$. Thus  
\[
f_{Y}(q \mid X=p, Z=r)=\frac{1}{1-p}, \; p< q < 1,
\]

Similarly, consider $f_Z(r\mid X=p, Y=q)$. The joint density $f$ as a function of $r$ is simply a constant times $r$. Thus 
\[
f_Z(r\mid X=p, Y=q)= cr,\; 0 < r< 1,
\]
for some normalizing constant $c$. It is easy to see by integrating the above that $c= 2$. Thus 
\[
f_Z(r\mid X=p, Y=q)= 2r,\; 0 < r< 1,
\]
which is a beta$(2,1)$ density and does not depend on $p$ and $q$. Thus $Z$ is independent of $(X,Y)$.

Finally
\[
f_X(p \mid Y=q, Z=r)= c p, \quad 0 < p < q.
\]
By integrating both sides we must get one. Thus 
\[
c\int_0^q p dp = 1,
\]
or that $c=2/q^2$. Note that, although both $f_X(\cdot \mid Y=q, Z=r)$ and $f_Z(\cdot \mid X=p, Y=q)$ are linear functions of the argument, they have very different supports where they are positives. This is why $Z$ is independent from $X,Y$ while $X$ is dependent on $Y$. Nevertheless, with these three explicit densities at hand one can run Gibbs sampling as usual. 
A little bit of efficiency can be gained by separately generating $Z$, which is independent, while simulating only $(X,Y)$ via Gibbs sampling.

\end{example}

\section{Stochastic optimization}

%Lecture 24

Suppose $G=(V, E)$ is a finite graph where, without loss of generality, we take the set of vertices $V=[n]$. Suppose there is a function $f:V\rightarrow \rr$. We wish to find the minimum of the function $\min_{v \in V}f(v)$ and the minimizing vertex $v^*$ such that $f(v^*)= \min_{v\in V} f(v)$. You might think what is the big problem here. After all I have a bunch of numbers and I can simply go through them one by one until I find the minimum. The problem is that $n$ could be a very large number, let's say in the order of millions. Evaluating the function at every vertex and sorting them is computationally expensive. One might also consider a greedy ``gradient descent'' algorithm. Start from any vertex $v_0$ and evaluate $f(v_0)$. Now look at the neighbors of $v_0$. Among those neighbors find the vertex $v_1$ that minimizes $f$.      

\begin{figure}  
\begin{center}
\begin{tikzpicture}[scale=0.9]
\tikzset{
        node/.style={circle, draw=black, very thick, minimum size=7mm},
        arrow/.style={->, black, very thick},
        line/.style={black, very thick}
        }
    \node[node] (one) [label=above:{f(1)=101}] at (1.5,1.5) {1};
    \node[node] (two) [label=left:{f(2)=25}] at (0,0) {2};
    \node[node] (three) [label=below:{f(3)=4}] at (1.5,-1.5) {3};
    \node[node] (four) [label=below:{f(4)=2}] at (3,0) {4};
    \node[node] (five) [label=below:{f(5)=10}] at (4.5,-1.5) {5};
    \node[node] (six) [label=above:{f(6)=33}] at (4.5,1.5) {6};
    \node[node] (seven) [label=right:{f(7)=1}] at (6,1.5) {7};
    \node[node] (eight) [label=right:{f(8)=30}] at (6,-1.5) {8};
    
    \draw[line] (one) -- (two);
    \draw[line] (one) -- (four);
    \draw[line] (two) -- (three);
    \draw[line] (two) -- (four);
    \draw[line] (three) -- (four);
    \draw[line] (four) -- (five);
    \draw[line] (four) -- (six);
    \draw[line] (five) -- (six);
    \draw[line] (five) -- (eight);
    \draw[line] (six) -- (seven);
\end{tikzpicture}
\end{center} 
\caption{\small Optimization of a function over a graph.}
\label{Fig:optgraph}
\end{figure}
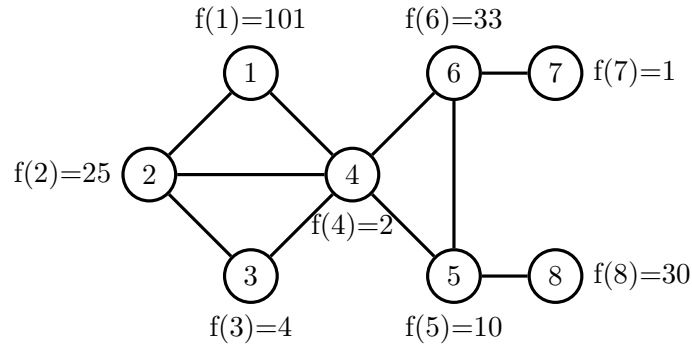

This doesn't quite work due to local minimum. Take the example of the graph shown in Figure \ref{Fig:optgraph}. Suppose $v_0$ is vertex $8$. It has only one neighbor $5$ with a lower $f$-value, $f(5)=10 < f(8)=30$. So you jump to $v_1=5$. $v_1$ has two neighbors $4$ and $6$. The minimum value of $f$ among them is achieved at $4$ and $f(4)=2< f(5)=10$. Thus $v_2=2$. Now we have nowhere to go since all the neighbors of $v_2$ have a higher value of $f$ making $v_2$ a local minimum of the function $f$. However, the absolute minimum of $f$ is at vertex $7$ which our gradient descent algorithm cannot reach. To avoid getting trapped in local minimums we will add randomness by creating a probability distribution on the vertices.

For a parameter $\lambda\ge 0$, consider the so-called \textit{Gibbs probability distribution at temperature $T=1/\lambda$}:
\[
\pi_\lambda(i)= \frac{1}{Z_\lambda} e^{-\lambda f(i)}, \quad i \in V=[n].
\]
Here $Z_\lambda$ is called the normalizing constant given by 
\[
Z_\lambda = \sum_{j=1}^n e^{-\lambda f(j)}.
\]
Note that this is the unique choice of $Z_\lambda$ that turns $\pi_\lambda$ to be a probability mass function. 

The idea comes from statistical physics (hence the concept of temperature). When $\lambda=0$, i.e., the temperature $T=\infty$ (very hot), 
\[
Z_0=\sum_{j=1}^n e^{-0 \cdot f(i)}= n, \quad \pi_0(i) = \frac{1}{n}, 
\]
making $\pi_0$ the uniform distribution on $[n]$. On the other hand, when the temperature gets very cold (i.e., $\lambda\rightarrow \infty$ or $T\rightarrow 0$), $\pi_\lambda$ puts almost its entire mass on the minimizer of $f$. To see what I mean, assume that $f$ has a unique minimizer, say vertex $1$. That is $f(1) < f(v)$ for every vertex $v\neq 1$. Then, rewrite $\pi_\lambda$ as follows 
\[
\pi_\lambda(i) = \frac{1}{Z_\lambda} e^{-\lambda f(i)}= \frac{e^{-\lambda f(i)}}{\sum_{j=1}^n e^{-\lambda f(j)}}= \frac{e^{-\lambda (f(i) - f(1))}}{\sum_{j=1}^n e^{-\lambda (f(j) - f(1))}}. 
\]
Consider the denominator $\sum_{j=1}^n e^{-\lambda (f(j) - f(1))}$. Either $j=1$, in which case the corresponding term $e^{-\lambda (f(1) - f(1))}=e^{0}=1$ or $j > 1$, in which case $f(j) - f(1) > 0$, and hence $\lim_{\lambda \rightarrow \infty} e^{-\lambda(f(j) - f(1))}=0$. Thus, summing up over all $j$, we get 
\[
\lim_{\lambda \rightarrow \infty} \sum_{j=1}^n e^{-\lambda (f(j) - f(1))}=1. 
\]
Applying the same logic to the numerator of $\pi_\lambda$, we get 
\[
\lim_{\lambda\rightarrow \infty} \pi_\lambda(i) = \begin{cases}
1, \quad \text{if $i=1$},\\
0, \quad \text{if $i > 1$}. 
\end{cases}
\]
In other words the probability distribution $\pi_\lambda$ converges to the trivial distribution that puts mass one at the minimizer of $f$. Notice that the same proof shows that if $f$ has more than one minimizer, then $\pi_\lambda$, as $\lambda \rightarrow \infty$, converges to the uniform distribution on the set of minimizers of $f$. 

The idea of stochastic optimization of $f$ is that instead of greedily minimizing $f$, one looks to sample from the probability distribution $\pi_\lambda$ for a large value of $\lambda$. Will such a sample always give us the minimizer? No, not unless $\lambda=\infty$. But with high probability it will sample a vertex whose $f$ values is close to the minimum. So the next question is how do we sample from $\pi_\lambda$ and one possible answer is via the Metropolis algorithm. 

\nin\tbf{Hill Climb algorithm for stochastic optimization.} Fix a base chain, say simple symmetric random walk on the graph. Thus the transition probability for the base chain is given by 
\[
q(x,y)= \frac{1}{\deg(x)}, \; \text{if $y$ is a neighbor of $x$, and zero otherwise.}
\]
Recall that this is a reversible Markov chain but not symmetric unless the graph is regular. 

Now run Metropolis with this base chain. Suppose, currently $X_t=x$, pick $Y_{t+1}$ uniformly from one of the neighbors of $x$. Then accept this choice as $X_{t+1}$ with probability 
\[
\begin{split}
\min\left(\frac{\pi_\lambda(y) q(y,x)}{\pi_\lambda(x) q(x,y)}, 1 \right)&= \min\left( \frac{e^{-\lambda f(y)}/ \deg(y)}{e^{-\lambda f(x)}/\deg(x)}, 1 \right)\\
&= \min\left( \frac{\deg(x)}{\deg(y)}e^{-\lambda(f(y) - f(x))} ,1\right).    
\end{split}
\]
Otherwise, remain where you are by declaring $X_{t+1}=x$. 

In the special case of regular graphs, $\deg(x)=\deg(y)$ and the algorithm becomes simpler. Accept $y$ uniformly chosen among neighbors of $x$ with probability $\min\left(e^{-\lambda(f(y) - f(x))}, 1 \right)$. Thus, if $f(y) < f(x)$, you'd always accept $y$ and hence decrease $f$ by doing so. But if $f(y) > f(x)$, you may still accept $x$ with probability $e^{-\lambda(f(y) - f(x))}$. This step lets you escape the dreaded local minima which cannot be avoided in a gradient descent. 

%\medskip

\nin\tbf{Simulated Annealing.} But what if I insist on getting the actual minimum instead of a sample from a near-minimum with high probability? A variation of the hill climb algorithm called simulated annealing lets you do that. The idea is that you grow $\lambda=\lambda(t)$ \textit{slowly} with time $t$, typically of the order of $\log(t)$. There is a theorem that, under suitable assumptions, the Markov chain converges to the minimizer of $f$ but there are plenty of other ad hoc choices that appear to work in practice. See the paper by Dimitris Bertsimas and John Tsitsiklis titled Simulated Annealing that appeared in the journal \textit{Statistical Science}, volume 8, number 1, pages 10-15, 1993.

\pagebreak
\section*{Problems for chapter \ref{Ch:MC}}
\addcontentsline{toc}{section}{Problems for chapter \ref{Ch:MC}}

\begin{problem}\label{Prob:BadBase}  Suppose you are trying to sample from some distribution $\pi$ on a sample space $\Omega$ via the Metropolis-Hastings algorithm.
\begin{enumerate}[$(a)$]
\item Define your base chain to have transition probabilities $q(x,y) = \pi(y)$.  What are your acceptance probabilities and the transition probabilities for the metropolis chain with this base chain?
\item What is the mixing time for metropolis chain with this choice of $q$? 
\item Explain why this is not the approach we take in \S\ref{Sec:Metropolis}, and why this is not used in practice.
\end{enumerate}
\end{problem}

%B4.5
\begin{problem}\label{Prob:B4.5} We wish to find a reversible Markov chain on the hypercube $\Omega=\{0,1\}^N$ with a stationary distribution that favors vertices with more ones than zeroes in the following sense. For $w$ in $\Omega$, let $\abs{w}=\sum_{i=1}^N w_i$ denote the number of ones in $w=(w_1, \ldots, w_N)$. Define the probability distribution
\[
\pi(w)= \frac{\abs{w}^2}{Z}, \quad \text{where} \quad Z=\sum_{v\in \Omega} \abs{v}^2
\] 
is the normalizing constant. 
\begin{enumerate}[$(a)$]
    \item Describe a reversible Markov chain with stationary distribution $\pi$ by writing its transition probabilities.  Simplify your expressions for the transition probability $p(w_1,w_2)$ so that they are only in terms of $|w_1|$.
    \item Explicitly verify that your new chain satisfies the DBEs for $\pi$.
\end{enumerate} 
\end{problem}

%B4.6
\begin{problem}\label{Prob:B4.6}  For each problem clearly describe the conditional distribution of each coordinate given the others. Then describe the procedure for running Gibbs sampling to sample from the joint distribution. Assume that Gibbs sampling works for continuous densities as well as discrete distributions. \emph{Guess} the conditional from the structure of the joint distribution. \emph{Avoid} doing integration as much as possible. Use your knowledge of the all the named one dimensional distributions/ densities.  
\begin{enumerate}[$(a)$]
\item Sample from the joint density 
\[
f(x,y,z)= \frac{1}{W}, \quad 0 < x < y < z < 1, 
\]
and zero elsewhere. Here $W$ is the normalizing constant. 
\item Sample from the mixed joint pmf/pdf: 
\[
p(n,t)= \frac{1}{Z} (1-p)^{n-1} n e^{-nt}, \quad t >0, \; n=1,2,\ldots,
\]
where $Z$ is the normalizing constant. 
\end{enumerate}
\end{problem}

%A4.2
\begin{problem}\label{Prob:A4.2} For a number $s>1$, consider the Riemann zeta function $\zeta(s)= \sum_{n=1}^\infty \frac{1}{n^s}$. This number is finite and therefore results in a probability distribution on integers
\[
\pi(n) = \frac{1}{n^s \zeta(s)}, \quad n\ge 1.
\]
Develop a MCMC scheme based on Metropolis algorithm to sample a random integer according to the following probability distributions.
\begin{enumerate}[$(a)$] 
\item The conditional distribution of $\pi(\cdot \mid X\le N)$. That is,
\[
P(X=n)= \frac{\pi(n)}{\sum_{i=1}^N \pi(i)},\quad n=1,2,\ldots, N.
\]
Explicitly give the transition probabilities $p(j,k)$ for your new chain.
\item The entire distribution $\pi$ on $\NN$.  Explicitly give the transition probabilities $p(j,k)$ for your new chain and verify that the chain satisfies the DBEs for $\pi$.
\end{enumerate}
\end{problem}

%A4.3
\begin{problem}\label{Prob:A4.3}  Consider the joint density
\[
f(x,y,z)=\frac{1}{W} e^{-xyz-x-2y-3z}, \quad x>0,\; y>0, \; z>0,
\]
where $W$ is the normalizing constant. 
\begin{enumerate}[$(a)$]
\item Clearly identify each of the three conditional densities (i) $X$, given $Y,Z$; (ii) $Y$, given $X,Z$; and (iii) $Z$, given $X,Y$, in terms of standard named distribution (such as normal, exponential, geometric, Poisson etc.). (\emph{Hint}: Do not try to integrate.)

\item Describe how you will use Gibbs sampling to generate a sample from the joint density 
\end{enumerate}
\end{problem}

%A4.6
\begin{problem}\label{Prob:A4.6} For each problem clearly describe the conditional distribution of each coordinate given the others. Then describe the procedure for running Gibbs sampling to sample from the joint distribution. Assume that Gibbs sampling works for continuous densities as well as discrete distributions. \textbf{Guess} the conditional from the structure of the joint distribution. \textbf{Avoid} doing integration as much as possible. Use your knowledge of the all the named one dimensional distributions/ densities.  
\begin{enumerate}[$(a)$]
\item Sample from the mixed joint pmf/pdf: 
\[
f(p,n)= p(1-p)^{n-1}, \quad 0< p < 1, \; n=1,2,\ldots.
\]
You will need one integration to describe the density of $p$ given $N=n$. 

\item Sample from the joint density:
\[
f(p,q,r)= \frac{1}{Z} pqr, \quad 0< p,q,r < 1.
\] 
You will need one easy integration. What is special about this joint density?  Why do you actually not need Gibb's algorithm to sample from this?
\end{enumerate}
\end{problem}

%A5B5.2
\begin{problem}\label{Prob:AB5.2} Consider the following joint distributions of random variables. Describe how you will sample from them using \textit{any method} of sampling. 
\begin{enumerate}[$(a)$]
\item $(X_1, X_2, \ldots, X_n)$ is Mult$\left( N, 1/n, \ldots, 1/n\right)$.  Explain why using Gibb's sampling with only updating coordinate at a time will not work.  
\item $\left( X, Y\right)$ is a standard bivariate normal with means $0$, variances $1$, and correlation $\rho$. 
\end{enumerate}
\end{problem}

%A5B5.4
\begin{problem}\label{Prob:AB5.4} Consider a connected graph $G=(V,E)$ with $V=\{1,2,\ldots,n\}$. For $-\infty< \alpha < \infty$, consider a probability distribution 
\[
\pi^\alpha(x) = \frac{1}{Z_\alpha}\left( \mathrm{deg}(x) \right)^\alpha,
\]
where $Z_\alpha$ is the normalizing constant. Describe a reversible Markov chain with stationary distribution $\pi^\alpha$. Describe in words how $\pi^\alpha$ is qualitatively different for different values of $\alpha$.
\end{problem}

%A5B5.5
\begin{problem}\label{Prob:AB5.5} Let $\Omega=\{1,2,\ldots, n\} \times \{1,2,\ldots, n\}$ be the $n\times n$ grid. For $\beta \ge 0$, consider the probability distribution
\[
\pi^\beta(i,j)= \frac{1}{Z}\exp\left( - \abs{i-j}^\beta \right) ,\quad \text{for $(i,j)\in \Omega$}. 
\]
Describe a reversible Markov chain with stationary distribution $\pi^\beta$. Describe in words how $\pi^\beta$ is different from the uniform distribution on the grid. 
\end{problem}

%A5B5.6
\begin{problem}\label{Prob:AB5.6} Let $\Omega=\{1,2,\ldots,n\}$. Define a probability distribution 
\[
\pi(i)=\frac{1}{Z} i(n+1-i), \quad 1\le i \le n,
\]
where $Z$ is the normalizing constant. Describe any method to get a sample from this probability distribution.
\end{problem}

\chapter{Martingales and harmonic functions}\label{Ch:Martingales}

%Lecture 25

\section{Martingales: intuition, definition and first examples}

%Sometimes the most effective way to analyze a mathematical object is to consider \emph{functions} of it, instead of only the object itself.  This point of view is common in areas of mathematics such as modern algebra, for instance, where functions (``morphisms'') of algebraic objects often shed significant light on the structure of the original objects themselves.  The ``First Isomorphism Theorem'' from your introductory course in abstract algebra, for example, is a nice instance of this principle.  

In this chapter we take up the analysis of certain special functions of a Markov chain $(X_n)_{n \geq 0}$.  We define a new stochastic process $(Y_n)_{n \geq 0}$ as a function of $(X_n)_{n \geq 0}$, and we will see that when $(Y_n)$ has a certain averaging property it yields rich information about the original chain itself.  We will call these new processes \emph{martingales}.    While martingales can seem abstract in the beginning, the underlying intuition is simple and the large number of concrete examples in our exposition will make them accessible.  Martingales are an important stepping stone to more advanced topics in probability, and they are utilized in a number of deep proofs in the modern theory.  While much of this will fall beyond the scope of our text, we hope you will nevertheless begin to appreciate the value of martingales in this chapter. The intellectual effort you may need to expend to internalize these new concepts is well worth the effort.

\subsection{Adapted processes and martingales}

Let $(X_n)_{n \geq 0}$ be a Markov chain.  The new process $(Y_n)_{n \geq 0}$ that we will build from $(X_n)$ will be such that, for each $n$, $Y_n$ is a function of the chain history $X_0, X_1, \ldots, X_n$ through the first $n$ steps as well as the step index $n$. That is, 
\begin{align}\label{Eq:AdaptedProcess}
    Y_n = f(n,X_0,X_1, \ldots, X_n) = f_n(X_0, X_1, \ldots, X_n).
\end{align}
Since $X_1, \ldots, X_n$ are random, $Y_n$ is a new random variable, and $(Y_n)_{n\geq 0}$ is thus a new stochastic process built from our original $(X_n)$. Note that the function $f_n$ defining $Y_n$ can change with each $n$.  We call a stochastic process $(Y_n)_{n \geq 0}$ satisfying \eqref{Eq:AdaptedProcess} for all $n\geq 0$ \textbf{adapted} to $(X_n)$.  This means that if we know the values of $X_0, X_1, \ldots, X_n$, we know the value of $Y_n$, as is clear from the formula \eqref{Eq:AdaptedProcess} (we assume the functions $f_n$ are given).  

\begin{example}
    Suppose $(X_n)_{n \geq 0}$ is the simple symmetric random walk on $\mathbb{Z}$.  Examples of adapted processes on $(X_n)$ include $Y_n = X^2_n$ and, say, $Y_n = X_n - X_{n-1} + X_{n-2} + \ldots + (-1)^n X_0$.
\end{example}

While all the functions $(Y_n)$ of chains $(X_n)$ we will consider are adapted processes, we are particularly interested in a subclass of adapted processes which possess a special averaging property.

\begin{definition}\label{Def:Martingale}
    A process $Y = (Y_n)_{n \geq 0}$ adapted to the Markov chain $(X_n)_{n \geq 0}$ is a \textbf{martingale} if
    \begin{align}\label{Eq:Martingale}
        \nE(Y_{n} \, | \, X_0, X_1, \ldots, X_{n-1}) = Y_{n-1}, \quad \text{ for each $n\geq 1$.}
    \end{align}
\end{definition}

\noindent So, given all the information up to step $n-1$, we expect the next value $Y_{n}$ of our martingale to be identical to the current value $Y_{n-1}$. The key intuition here is that martingales are thus ``fair games.''  If $(Y_n)$ were gambling winnings, for instance, this says that our expected net change in each stage is zero; on average we neither win nor lose anything on each turn.  

\subsection{Examples of martingales}\label{SubSec:MartingaleEg}
Let's carefully walk through four examples to get a better sense of what Definition \ref{Def:Martingale} is all about.

\begin{example}\label{Eg:FirstMartingale}
    Consider a simple symmetric random walk $(X_n)_{n \geq 0}$ on $\mathbb{Z}$ starting from $X_0=0$.  We claim 
    \begin{align}\label{Eq:MartEg1}
        Y_n = X_n
    \end{align}
    is itself a martingale.  Here our functions $f_n$ in \eqref{Eq:AdaptedProcess} are very simple,
    \begin{align*}
        f_n(X_0,X_1, \ldots, X_n) = f_n(X_n) = X_n,
    \end{align*}
    which is the \emph{projection} of the random vector $(X_0,X_1,\ldots, X_n)$ to its last coordinate.  While this function is not terribly exciting, it is instructive to see that it does satisfy the desired averaging property \eqref{Eq:Martingale}.  Indeed, we have
    % This is clearly adapted; if we know the values of $X_0, X_1, \ldots, X_n$, we obviously know $X_n$.  The averaging property \eqref{Eq:Martingale} is also satisfied:
    \begin{align*}
        \nE(Y_{n+1} \, | \, X_0, X_1, \ldots, X_n) &= \nE(X_{n+1} \, | \, X_0, X_1, \ldots, X_n)\\
        &= \nE(X_{n+1} \, | \, X_n)\\
        &= \frac{1}{2}(X_n+1) + \frac{1}{2}(X_n-1)\\
        &= X_n = Y_n,
    \end{align*}
    where we have used the Markov property in the second line.  We conclude that $(X_n)$ itself is a martingale.  This should not be surprising: we can think of $(X_n)$ as a game where you are flipping a fair coin and win a dollar for tossing heads and lose a dollar for tails, which is fair.  
\end{example}

    Other martingales, however, can be less obvious.

\begin{example}\label{Eg:SecondMartingale}
    Let $(X_n)_{n \geq 0}$ be the simple symmetric random walk on $\mathbb{Z}$ again, and consider the process $(Y_n)_{n \geq 0}$ defined by
    \begin{align}\label{Eq:MartEg2}
        Y_n = X_n^2-n.
    \end{align}
    In terms of coin flipping, in this game your total earnings after $n$ turns is the square of the difference $X_n$ of heads and tails flipped, minus the number of flips. 
    
    Our process $(Y_n)$ is certainly adapted: the function $f(n,X_0,\ldots, X_n)$ is given by the right-hand side of \eqref{Eq:MartEg2}.  It is not immediately obvious, however, if this is a martingale.  Let's see what we can do with the expectation \eqref{Eq:Martingale} in the definition of a martingale:
    \begin{align*}
         \nE(Y_{n+1} \, | \, X_0, X_1, \ldots, X_n) &= \nE(X_{n+1}^2-(n+1) \, | \, X_0, X_1, \ldots, X_n)\\
        &= \nE(X_{n+1}^2-(n+1) \, | \, X_n)\\
        &= \nE(X_{n+1}^2 \, | \, X_n) -n-1\\
        &= \frac{1}{2}(X_n+1)^2 + \frac{1}{2}(X_n-1)^2 -n-1\\
        &= X_n^2-n = Y_n,
    \end{align*}
    where  we have again used the Markov property in the second line.  Thus we do have a fair game and $(Y_n)$ is indeed a martingale.  
    
    %This stochastic process is called the \emph{quadratic martingale} for the simple symmetric random walk.  
\end{example}
    Both the martingales in Examples \ref{Eg:FirstMartingale} and \ref{Eg:SecondMartingale} have nice applications: you will use them in Problems \ref{Prob:5Opt1} and \ref{Prob:5Opt2}, respectively, to re-derive the gambler's ruin hitting probability and time formulas from  \S\ref{Sec:GamblersRuin} and \S\ref{Sec:GamblersRuinTime}.  That is, these martingales give us a completely different tool to arrive at the same formulas.  Problem \ref{Prob:ExpMartingale} explores two other martingales for the simple random walk, in particular, the important \emph{exponential martingale}.

\begin{example}\label{Eg:PolyaMartingale}
    Let's re-visit the P\'{o}lya urn of \S \ref{Sec:PolyaUrn} with $a$ black balls and $b$ red balls.  Recall that in each step of the chain, you pick a uniformly-random ball from the urn and then return it to the urn with one additional ball of the same color.  Let $X_n$ be the number of black balls after $n$ steps.  In Problem \ref{Prob:B3.5} we saw that
    \begin{align*}
        \nE\Big(\, \frac{X_{n+1}}{n+1+a+b} \, \Big| \, X_{n} \Big) = \frac{X_{n}}{n + a + b}.
    \end{align*}
    Hence if 
    \begin{align}\label{Eq:PolyaMartingale}
        Y_n = \frac{X_n}{n+a+b},    
    \end{align}
    we have $\nE(Y_{n+1} | X_n) = Y_n$, showing $(Y_n)$ is a martingale.  Note that $Y_n$ is just the proportion of blacks balls after $n$ steps, and thus we see that this proportion is a ``fair game'' of the P\'{o}lya urn.
\end{example}
\begin{exercise}
    Show that the proportion of red balls is also a martingale.
\end{exercise}

\begin{example}\label{Eg:HarmonicMartingale1}
    Consider the simple symmetric random walk $(X_n)_{n \geq 0}$ on the graph $G = (V,E)$ of Figure \ref{Fig:HarmonicEg}.  Suppose we start with 
    \begin{align}\label{Eq:HarmonicMartingale1Start}
        X_0 = a
    \end{align}
    and we run the walk until the first time $\tau$ that we reach a gray vertex, $\tau = \min\{n \; : \; X_n \in \{0,1\}\}$.  What is $\nP(X_\tau = 1)$?  This is a ``gambler's ruin'' flavor of question, but now no longer for the simple walk on $\mathbb{Z}$.  Martingales, it turns out, give us a slick method to answer this question.
\begin{figure}
    \centering
    \scalebox{0.75}{
        \begin{tikzpicture}
    \tikzset{
        node/.style={circle, draw=black, very thick, minimum size=7mm},
        bnode/.style={circle, draw=black, fill=gray!30, very thick, minimum size=7mm},
        line/.style={black, very thick}
    }
    
    \node[node] (a) at (0,0) {$a$};
    \node[node] (b) at (2,0) {$b$};
    \node[node] (c) at (4,0) {$c$};
    \node[node] (d) at (6, 0) {$d$};
    \node[node] (e) at (8, 0) {$e$};
    \node[bnode] (zero) at (2, 1.5) {0};
    \node[bnode] (one) at (6, -1.5) {1};
    
    \draw[line] (a) -- (b);
    \draw[line] (b) -- (c);
    \draw[line] (c) -- (d);
    \draw[line] (d) -- (e);
    \draw[line] (b) -- (zero);
    \draw[line] (d) -- (one);
    
    \end{tikzpicture}}
    \caption{\small A graph $G$ with seven vertices.  Defining a martingale for the simple symmetric random walk on $G$ can give information about the hitting probabilities of the shaded vertices.}
    \label{Fig:HarmonicEg}
\end{figure}
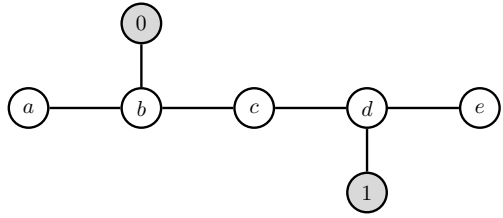
  
  We start by defining a function $h:V \rightarrow \mathbb{R}$ on the vertices as in Figure \ref{Fig:HarmonicEg2}.  That is, $h(a)=h(b) = 0.25$, $h(c)=0.5$, $h(d)=h(e)=0.75$, and $h(0)=0$, $h(1)=1$. We claim that the new process defined by the composition 
\begin{align}\label{Def:GraphHarmonicMart}
    Y_n := h(X_n)
\end{align}
is a martingale while $X_n$ is in the white vertices $\{a,b,c,d,e\}$.  From \eqref{Def:GraphHarmonicMart} it is clear that $Y_n$ is adapted (once we know $X_n$ we know the value of $Y_n$), and we just need to check the averaging property.  Suppose, for instance, that $X_n=b$.  Then
\begin{align}
        \nE(Y_{n+1} \, | \, X_0, X_1, \ldots, X_n) &= \nE(h(X_{n+1}) \, | \, X_n=b)\notag\\
        &= \frac{1}{3}(h(a) + h(0) + h(c)) \label{Eq:MartingaleHarmonicAvg}\\
        &= \frac{1}{3}(0.25 + 0 + 0.5) = 0.25 = h(b) = h(X_n) = Y_n.\notag
\end{align}
Hence the averaging property holds here, and we can similarly check that it holds at all the vertices $a,b,c,d$ and $e$ (as you should verify).  Thus $(Y_n)_{n \geq 0}$ is a martingale so long as $n < \tau$.\footnote{Note that this last restriction is necessary: the averaging property does not hold at vertices 0 and 1.  For example, if $X_n=0$, then $X_{n+1}=b$, and so 
\begin{align*}
    \nE(Y_{n+1}|X_n=0) = h(b) =0.25 \neq 0 = h(0) = Y_n.
\end{align*}
So we only have a martingale so long as the walk remains in the non-shaded vertices.}

The power of martingales becomes evident when we observe the following: taking expectations in the equation $\nE(Y_{n+1} \, | \, X_0, X_1, \ldots, X_n) = \nE(Y_n)$ yields 
\begin{align*}
    \nE(Y_{n+1}) = \nE(\nE(Y_{n+1} \, | \, X_0, X_1, \ldots, X_n)) = \nE(\nE(Y_n)) = \nE(Y_n).
\end{align*}
Since we similarly have $\nE(Y_{n} \, | \, X_0, X_1, \ldots, X_{n-1}) = \nE(Y_{n-1})$, we again take expectations and see $\nE(Y_{n}) = \nE(Y_{n-1})$, and hence that $\nE(Y_{n+1}) = \nE(Y_n) = \nE(Y_{n-1})$.  Repeating this over and over shows that the expectations are constant (on average, the value of our martingale is always the same!), and so, in particular, 
\begin{align}\label{Eq:HarmonicWalkEg1}
    \nE(Y_{n+1}) = \nE(Y_0) = h(X_0) = h(a) = 0.25
\end{align}
by \eqref{Eq:HarmonicMartingale1Start}, whenever $X_n \in \{a,b,c,d,e\}$.  But eventually $X_{n+1}$ will be either vertex 0 or 1 for the first time (i.e. $\tau=n+1$), and so we have
\begin{align*}
    \nE(Y_{n+1}) = \nE(h(X_{\tau})) &= h(0)\nP(X_\tau =0\,|\, X_0=a) + h(1)\nP(X_\tau =1\,|\, X_0=a)\\
    &= 0 + 1 \cdot \nP(X_\tau =1\,|\, X_0=a),
\end{align*}
by definition of the random variable $h(X_\tau)$. However, since $\nE(Y_{n+1}) = 0.25$ by \eqref{Eq:HarmonicWalkEg1}, we have answered our question: $\nP(X_\tau =1 \, | \, X_0=a) = 0.25 = h(a)$.  This martingale enables us to immediately compute the hitting probability of vertex 1 before 0, starting from $a$; it is simply the initial value $Y_0$.\footnote{While we know $\nE(Y_n) = \nE(Y_0)$ for any fixed, \emph{deterministic} time $n$, $\tau$ is a \emph{random} time, and so there is a slight issue in immediately concluding $\nE(Y_0) = \nE(Y_\tau)$.  This is indeed the case, though, and we will make this argument rigorous below in \S\ref{Sec:Optional}.}

Similarly, we can repeat this argument for starting at other vertices to see that $h(v) = \nP(X_\tau = 1 \, | \, X_0=v)$ for all $v \in \{a,b,c,d,e\}$.  So a martingale solves the gambler's ruin problem for our new graph!  Our new martingale $h(X_n)$ shows that our hitting probabilities are the values of $h$ in Figure \ref{Fig:HarmonicEg2} as our starting position $X_0$ varies.  (Note these values pass the sanity check that the probabilities for $a$ and $b$ must be equal, as well as those for $d$ and $e$ (why?).  Furthermore, by symmetry, we should have $1/2$ probability if we start at $c$.)

But where on earth did this function $h$ come from? The answer is that we had to build a function that satisfies the averaging property in \eqref{Eq:MartingaleHarmonicAvg} at each of $\{a,b,c,d,e\}$.  That leads to system of five linear equations with a unique solution, which then yields the martingale $Y_n = h(X_n)$.  Functions with this averaging property are called \textbf{harmonic}, and we will study them in much greater depth in \S\ref{Sec:Harmonic} below.
    
The graph in this example is relatively simple, but this martingale idea extends to \emph{any} connected graph, and thus gives a \emph{universal} approach for the generalized gambler's ruin hitting probability problem.  We thus obtain an exciting generalization of the work we did in \S\ref{Sec:GamblersRuin}.

\begin{figure}[]
    \centering
    \scalebox{0.75}{
        \begin{tikzpicture}
    \tikzset{
        node/.style={circle, draw=black, very thick, minimum size=7mm},
        bnode/.style={circle, draw=black, fill=gray!30, very thick, minimum size=7mm},
        inode/.style={minimum size=7mm},
        line/.style={black, very thick}
    }
    
    \node[node] (a) at (0,0) {$a$};
    \node[node] (b) at (2,0) {$b$};
    \node[node] (c) at (4,0) {$c$};
    \node[node] (d) at (6, 0) {$d$};
    \node[node] (e) at (8, 0) {$e$};
    
    \node[inode, label=-90:0.25] (al) at (0,0) {};
    \node[inode, label=-90:0.25] (bl) at (2,0) {};
    \node[inode, label=-90:0.5] (cl) at (4,0) {};
    \node[inode, label=90:0.75] (dl) at (6, 0) {};
    \node[inode, label=90:0.75] (el) at (8, 0) {};
    
    \node[bnode, label=90:0] (zero) at (2, 1.5) {0};
    \node[bnode, label=-90:1] (one) at (6, -1.5) {1};
    
    \draw[line] (a) -- (b);
    \draw[line] (b) -- (c);
    \draw[line] (c) -- (d);
    \draw[line] (d) -- (e);
    \draw[line] (b) -- (zero);
    \draw[line] (d) -- (one);
    
    \end{tikzpicture}}
    \caption{\small Defining a function $h$ on the vertices that yields a martingale $Y_n = h(X_n)$.  We conclude from the martingale property that the probability we hit vertex 1 before 0 starting from vertex $x$ is precisely $h(x)$.}
    \label{Fig:HarmonicEg2}
\end{figure}
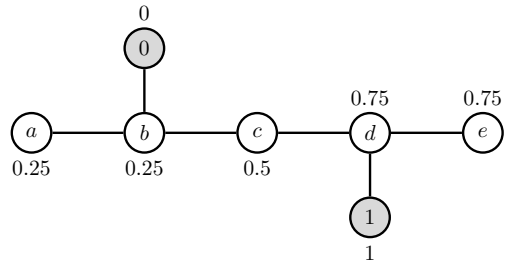
\end{example}

\subsection{Martingale expectations}

The argument in Example \ref{Eg:HarmonicMartingale1} which yielded $\nE(Y_{n}) = \nE(Y_0)$ for all $n$ works for any martingale, as it only uses the averaging property \eqref{Eq:Martingale}.  We formalize this as a useful theorem statement.  
\begin{theorem}\label{Thm:MartingaleAverage}
    If $(Y_n)_{n \geq 0}$ is a martingale, adapted to $(X_n)_{n\geq 0}$, then for all $n \geq 0$,
    \begin{align}\label{Eq:MartingaleAverage}
        \nE(Y_n) = \nE(Y_0),
    \end{align}
    and furthermore
\begin{align}\label{Eq:ConditionalIncrementZero}
    \nE(Y_{n+1}-Y_n \, | \, X_0, \ldots, X_n) = 0.
\end{align}
\end{theorem}
Property \eqref{Eq:MartingaleAverage} is commonly referred to as martingales having a constant expectation, while property \eqref{Eq:ConditionalIncrementZero} is referred to as martingale increments having zero conditional expectations.

\begin{proof}
    The argument for \eqref{Eq:MartingaleAverage} is as above in Example \ref{Eg:HarmonicMartingale1}: take expectations in \eqref{Eq:Martingale} and use induction.  To see \eqref{Eq:ConditionalIncrementZero}, observe that
\begin{align*}
    \nE(Y_{n+1}-Y_n \, | \, X_0, \ldots, X_n) &= \nE(Y_{n+1}\, | \, X_0, \ldots, X_n) - \nE(Y_n \, | \, X_0, \ldots, X_n)\\
    &= Y_n - \nE(Y_n \, | \, X_0, \ldots, X_n)
\end{align*}
since $(Y_n)$ is a martingale.  In the remaining expectation, we are given $X_0, \ldots, X_n$, and since $Y_n = f_n(X_0,\ldots, X_n)$ is adapted, this means we know the value of $Y_n$ itself.  Hence  $\nE(Y_n \, | \, X_0, \ldots, X_n) = Y_n$, and the above difference is zero.
\end{proof}
\begin{exercise}
    Show that martingale increments are uncorrelated. That is, if $(Y_n)$ is a martingale and $n_1 < n_2 $, then
    \begin{align*}
        \nE\big( (Y_{n_2+1}-Y_{n_2})(Y_{n_1+1}-Y_{n_1})  \big) = 0.
    \end{align*}
\end{exercise}
Let's close by seeing what Theorem \ref{Thm:MartingaleAverage} says about our example martingales from Section \ref{SubSec:MartingaleEg}.
\begin{itemize}
    \item In Example \ref{Eg:FirstMartingale} we saw that the simple random walk $(X_n)$ on $\mathbb{Z}$ starting from $0$ is itself a martingale, and thus it has constant expectation
    \begin{align*}
        \nE(X_n) = \nE(X_0) = 0
    \end{align*}
    for any $n$.  Of course this is not the only way to see this, as we could argue from symmetry, write $X_n$ as a sum of $n$ i.i.d. random variables, or attempt an explicit computation using the pdf in Lemma \ref{Lemma:SRWpdf}.  
    
    \item Similarly, we also see that (perhaps surprisingly) the process $Y_n = X_n^2-n$ in Example \ref{Eg:SecondMartingale} has constant average
    \begin{align*}
        \nE(Y_n) = \nE(Y_0) = \nE(X_0^2-0)= 0.
    \end{align*}
    Hence $\nE(X_n^2) = n$ for all $n$, and we thus have the nice consequence that the simple random walk after $n$ steps has variance
    \begin{align*}
        \Var(X_n) = \nE(X^2) - \nE(X)^2 = n-0 = n.
    \end{align*}
    Again, this is nothing profound, as $X_n$ is the sum of $n$ i.i.d. random variables, each with variance 1.  The point is that we get this with almost no work from the martingale.
    \item For the P\'{o}lya urn martingale $(Y_n)$ of Example \ref{Eg:PolyaMartingale}, we have
    \begin{align*}
        \nE(Y_n) = \nE(Y_0) = \frac{a}{a+b},
    \end{align*}
    the initial proportion of black balls.  Here the consequence is quite surprising: the mean proportion of black (or red) balls in the P\'{o}lya urn is always the same. 
\end{itemize}

\begin{exercise}
    Explain what Theorem \ref{Thm:MartingaleAverage} says about the martingale $h(X_n)$ of Example \ref{Eg:BuildHarmonic1}.
\end{exercise}

\section{Adapted processes, martingales from eigenvalues and eigenvectors, and the Markov operator}

Now that we have an initial taste for martingales through some preliminary examples and properties, we delve further into the theory.  We first revisit adapted processes and see an important example as well as a non-example.  Next, we ask if there are any reliable methods for constructing martingales.  We will answer affirmatively and see, perhaps surprisingly, that we can use the eigenvalues and eigenvectors of the probability transition matrix to build them.  (We will see even more robust methods later in this chapter, via harmonic functions, but this gives us a preliminary method.)  Finally, we step back and define the \textbf{Markov operator} which averages a fixed function on our vertices over one step of the chain, yielding a new function.  The Markov operator naturally arises when we consider the averaging property \eqref{Eq:Martingale} in the definition of a martingale.

\subsection{More on adapted processes}

Martingales are a very special subclass of all adapted processes $Y_n$ to $(X_n)$, and one should keep in mind that not all adapted processes are martingales, and furthermore that not all processes $(Y_n)$ are even adapted.  

\begin{example}\label{Eg:NonMartingale}
We first consider an example of an adapted process which is not a martingale.  Consider a proper subset $D \subsetneq \Omega$ and let $\tau$ be the hitting time of $D^c$,
\begin{align}\label{Eq:HittingTimeTau}
    \tau = \tau_{D^c} = \min\{ n \geq 0 \, : \, X_n \in D^c\},
\end{align}
and now define the process $(Y_n)_{n\geq 0}$ by
\begin{align*}
    Y_n = \I_{\{\tau \leq n\}}
\end{align*}
First, we show $(Y_n)$ is adapted.  This is intuitively clear, because the hitting time \eqref{Eq:HittingTimeTau} only depends on the first $n$ states, and so $Y_n$ also only depends on these states; knowledge of $(X_0, X_1, \ldots, X_n)$ determines $Y_n$.  For a formal proof, we verify the definition by constructing the functions $f_n$ in \eqref{Eq:AdaptedProcess}.  Indeed, writing
\begin{align*}
    \I_{D^c}(x) = \begin{cases} 
        1 & x \in D^c,\\
        0 & x \in D,
    \end{cases}
\end{align*}
we then have
\begin{align*}
    Y_n &= \I_{\{X_0 \in D^c \text{ or } X_1 \in D^c \text{ or } \cdots \text{ or } X_n \in D^c\}}\\
    &= \max\{\I_{D^c}(X_0), \I_{D^c}(X_1), \ldots, \I_{D^c}(X_n)\}\\
    &= f_n(X_0, X_1, \ldots, X_n),
\end{align*}
proving $(Y_n)$ is adapted.

However, $(Y_n)$ is not a martingale.  One way to see this is by writing
\begin{align*}
    \nE(Y_{n+1} \, &| \, X_0, \ldots, X_n)\\
    &= \nE(\I_{\{\tau \leq n+1\}} \, | \, X_0, \ldots, X_n)\\
    &= \nP(\tau \leq n+1 \,|\, X_0, \ldots, X_n)\\
    &= \I_{\{X_0 \in D^c \text{ or } \cdots \text{ or } X_n \in D^c\}} + \I_{\{X_0 \in D, \ldots, X_n \in D\}}\nP(X_{n+1} \in D^c \, | \, X_n).
\end{align*}
Since $Y_n$ equals just the first term in this sum, we do not always have
\begin{align*}
    \nE(Y_{n+1} \, &| \, X_0, \ldots, X_n) = Y_n,
\end{align*}
and so $(Y_n)$ is not a martingale. 
\end{example}

In fact, there is an easier way to see that $(Y_n)$ is not a martingale. For every realization of the process $(X_n)$, the path of $(Y_n)$ is decreasing in $n$. $Y_n=1$ for $n=0,1,2,\ldots, \tau$, and then, for all $n > \tau$, $Y_n=0$. Such a decreasing process cannot be a martingale since it cannot have constant expectation. The next exercise asks you to craft a rigorous argument based on this observation. 

\begin{exercise}\label{Ex:NotAMartingale}
    Use Theorem \ref{Thm:MartingaleAverage} to give another argument for why the process $(Y_n)$ in Example \ref{Eg:NonMartingale} is not a martingale.
\end{exercise}

%It is also helpful for the intuition to have negative examples in mind, as well.

\begin{example}\label{Eg:NonAdapted}
    Another process we could consider is
    \begin{align*}
        Z_n := \inf\{m > n\; : \; X_m \in D \},
    \end{align*}
    the \emph{next} time that our walk lies within $D$.  Is this adapted process?  While $(Z_n)$ is a function of $(X_n)$ and $n$, it clearly does not only depend on the first $n$ steps of the walk.  We need to know future steps of the walk to determine $Z_n$.  Thus $(Z_n)_{n \geq 0}$ is not adapted, and since martingales are always adapted, in particular it is not a martingale.
\end{example}

\subsection{Martingales from eigenvalues and eigenvectors}

Given an irreducible Markov chain on $\Omega = \{1,2,\ldots, N\}$ with $N \times N$ transition matrix $P$, recall from Theorem \ref{Thm:PF} that the eigenspace for eigenvalue $\lambda_1=1$ of $P$ is one-dimensional, and that all the other eigenvalues $\lambda_j$ satisfy $|\lambda_j|<1$.  Moreover, we know a left eigenvector of $\lambda_1$ is $\pi$,
\begin{align*}
    \pi P = \pi,
\end{align*}
and a right eigenvector is the column vector $\mathbf{1}^T = (1,1,\ldots,1)^T$,
\begin{align*}
    P \mathbf{1}^T = \mathbf{1}^T.
\end{align*}
(Note this is an example of different left- and right-eigenspaces for the same eigenvalue when $\pi$ is not uniform.)  The only use we've seen for the other eigenvalues, so far, is that the spectral gap $\gamma^* = 1 - \max_{2 \leq n \leq N}|\lambda_j|$ gives information about the mixing time for the chain; if the second-largest eigenvalue is close to 1 in absolute value, the chain is ``close'' to reducible and mixing takes longer, as we saw in \S\ref{Sec:MixingTimes}.  

As we alluded to in the introduction to this section, though, another use for the other eigenvalues and eigenvectors is to construct martingales, as laid out in the following theorem.

\begin{theorem}\label{Thm:BuildMartingale}
    Let $(X_n)_{n \geq 0}$ be an irreducible Markov chain with transition matrix $P$.  If $\lambda \neq 0$ is a eigenvalue of $P$ with right-eigenvector $v^T$, then the process
    \begin{align}\label{Eq:EvalEvecMartingale}
        Y_n = \lambda^{-n} v(X_n)
    \end{align}
    is a martingale.
\end{theorem}
\noindent So the setting is that we have a non-zero row vector $v \in \mathbb{R}^N$ which satisfies $Pv^T = \lambda v^T$. The notation $v(X_n)$ means that we choose the $X_n$-component of $v$ (or of $v^T$ - of course either way gives us the same component).  The theorem says that when we consider the random component $v(X_n)$ of our eigenvector, scaled by $\lambda^{-n}$, we obtain a martingale.  The scaling factor $\lambda^{-n}$ is entirely deterministic, as $\lambda$ is a fixed eigenvalue, and so the only random part in the definition of $Y_n$ is which component of the eigenvector we consider.  

It will be instructive to consider an example before diving into the proof.

\begin{example}\label{Eg:EhrenfestMartingale}
    Recall the Ehrenfest urn model of \S \ref{Sec:EhrenfestUrn}: we have $B$ identical balls split up between two urns, and at each step we pick one of the $B$ balls uniformly at random and switch its urn.  If $X_n \in \{0,1,\ldots, B\}$ is the number of balls in the first urn, then we have the transition probabilities
    \begin{align}\label{Eq:EhrenfestTrans}
        P(k,k+1) = \frac{B-k}{B}, \qquad P(k,k-1) = \frac{k}{B},
    \end{align}
    and $P(k,j) = 0$ when $j \notin \{k-1,k+1\}$.  For example, if $B=5$, our $6 \times 6$ transition matrix $P$ is
    \begin{align}\label{Eq:EhrenfestMatrixEg}
        P = \begin{bmatrix}
            0 & 1 & 0 & 0 & 0 & 0\\
            1/5 & 0 & 4/5 & 0 & 0 & 0\\
            0 & 2/5 & 0 & 3/5 & 0 & 0\\
            0 & 0 & 3/5 & 0 & 2/5 & 0\\
            0 & 0 & 0 & 4/5 & 0 & 1/5\\
            0 & 0 & 0 & 0 & 1 & 0
        \end{bmatrix}
    \end{align}
    Can we find an eigenvector-eigenvalue pair for $P$ to build a martingale?  Consider the vector $v = (B, B-2, B-4, \ldots, -B)$, which entry-wise is $v(k) = B-2k$ for $k \in \{0,1,\ldots, B\}$. In the above case $B=5$ above we have 
    \begin{align*}
        v = (5, 3, 1, -1, -3, -5).
    \end{align*}
    We claim that $v^T$ is a right eigenvector for $P$ with eigenvalue $\lambda = 1 - 2/B$.  
    \begin{exercise}
        Verify that this holds for the $B=5$ case, and then verify the general case.
    \end{exercise}
    \noindent We highly suggest you do this computation first on your own, but let us walk through it so you can check your work. We show the claimed identity $Pv^T = \lambda v^T$ entry-wise.  Indeed, for the first row of $P$, we have
    \begin{align*}
        (Pv^T)(1) = 1 \cdot v^T(2) = B-2 = \Big(1 - \frac{2}{B} \Big)B = \lambda v^T(1),
    \end{align*}
    as needed.  The computation for the last row $N=B+1$ is similar; for a general row $k \in \{2, \ldots, B\}$, we use \eqref{Eq:EhrenfestTrans} and obtain
    \begin{align*}
        (Pv^T)(k) &=  \frac{k}{B} \cdot v^T(k-1) + \frac{B-k}{B}\cdot v^T(k+1)\\
        &= \frac{k}{B} \cdot (B-2(k-1)) + \frac{B-k}{B}\cdot (B-2(k+1))\\
        &= \frac{1}{B}(B^2 - 2B(k+1)+4k) = \Big(1 - \frac{2}{B} \Big)(B-2k) = \lambda v^T(k),
    \end{align*}
    as claimed.  We conclude $v^T$ is indeed a right eigenvector for $P$ with eigenvalue $1-2/B$.
    
    Now we can build our martingale.  According to Theorem \ref{Thm:BuildMartingale}, the process
    \begin{align}\label{Eq:EhrenfestZn}
        Z_n := \lambda^{-n}v(X_n) &= \Big(1-\frac{2}{B} \Big)^{-n} (B-2X_n)
    \end{align}
    is a martingale (why we are calling it $Z_n$ instead of $Y_n$ will be apparent in a moment).  Note that 
    \begin{align}\label{Eq:EhrenfestEVect}
        v(X_n) = B-2X_n = B-X_n - X_n
    \end{align}
    is the difference of the number of balls in the second urn and the first.  Does that look familiar?  We saw in Problem \ref{Prob:A3.6} that the process $v(X_n)$, there called $Y_n$, satisfies
    \begin{align}\label{Eq:EhrenfestProblemReview}
        \nE(v(X_{n+1})\,|\,v(X_n)) = \nE(v(X_n)\,|\,X_n) = \Big(1 -\frac{2}{B}\Big)v(X_n),
    \end{align}
    where the second equality holds because knowing $v(X_n)$ is equivalent to knowing $X_n$ by \eqref{Eq:EhrenfestEVect}.  We can use this to verify that $Z_n$ is, indeed, a martingale.  We observe
    \begin{align*}
        \nE(Z_{n+1}|X_0, \ldots, X_n) &= \nE(Z_{n+1} \, | \, X_n)\\
        &= \Big(1-\frac{2}{B} \Big)^{-(n+1)}\nE(v(X_{n+1}) \,|\, X_n)\\
        &= \Big(1-\frac{2}{B} \Big)^{-(n+1)}\Big(1 -\frac{2}{B}\Big)v(X_n)\\
        &= \Big(1-\frac{2}{B} \Big)^{-n}v(X_n) = Z_n.
    \end{align*}
    Here for the first equality we use the Markov property, for the second we use the definition of $Z_n$ in \eqref{Eq:EhrenfestZn}, and for the third we use \eqref{Eq:EhrenfestProblemReview}.  We conclude $(Z_n)$ has the needed averaging property and so is indeed a martingale.
    
    Of course, once we know $v^T$ is a right eigenvector of $P$ with eigenvalue $\lambda$, we immediately know that the process defined by \eqref{Eq:EhrenfestZn} is a martingale.  The point here is that we can also verify this by hand using our earlier computations, which is satisfying and helps make everything more believable.
    
   This example illustrates the power of Theorem \ref{Thm:BuildMartingale}.  If you are ``playing'' the Ehrenfest urn with your friends (undoubtedly a common occurrence), how would you create a fair game?  The theorem says that the difference in balls between the two urns, scaled by $(5/3)^n$ if there are 5 total balls, is a fair game.  This is not super intuitive and would probably be hard to arrive at entirely on your own.
    
    % While we've built $(Z_n)$ through the construction in Theorem \ref{Thm:BuildMartingale} and thus know that it is a martingale, it may be a helpful as a ``sanity check'' to explicitly show that
    % \begin{align*}
    %     \nE(Z_{n+1} \, | \, X_0, \ldots, X_n) = Z_n.
    % \end{align*}
    % Indeed, plugging in our formula \eqref{Eq:ZnMartingale}, we have
    % \begin{align}
    %     \nE(Z_{n+1} \, | \, X_0, \ldots, X_n) &= \nE(Z_{n+1} \, | \, X_n) \notag\\
    %     &= \Big(1-\frac{2}{N} \Big)^{-(n+1)} (N-2\nE(X_{n+1} \, | \, X_n)).\label{Eq:ZnMartingaleVerification}
    % \end{align}
    % Using \eqref{Eq:EhrenfestTrans}, we have
    % \begin{align*}
    %     \nE(X_{n+1} \, | \, X_n) &= (X_n+1) \frac{N-X_n}{N} + (X_n-1)\frac{X_n}{N}\\
    %     &= 1 + \Big(1-\frac{2}{N}\Big) X_n,
    % \end{align*}
    % and hence \eqref{Eq:ZnMartingaleVerification} becomes
    % \begin{align*}
    %     \Big(1-\frac{2}{N} \Big)^{-(n+1)} &\Big(N-2-2\big(1-\frac{2}{N} \big)X_n \Big)\\
    %     &=\Big(1-\frac{2}{N} \Big)^{-(n+1)} \Big(N\big(1 - \frac{2}{N}\big)-2\big(1-\frac{2}{N} \big)X_n \Big)\\
    %     &=\Big(1-\frac{2}{N} \Big)^{-n} (N-2X_n) = Z_n,
    % \end{align*}
    % which verifies that $(Z_n)$ is, in fact, a martingale.
\end{example}
\begin{exercise}
    Explain what Theorem \ref{Thm:MartingaleAverage} says about the martingale \eqref{Eq:EhrenfestZn}.
\end{exercise}
\noindent Having an example under our belts, we proceed with the proof of Theorem \ref{Thm:BuildMartingale}.
\begin{proof}
    We recall our process is $Y_n = \lambda^{-n}v(X_n)$, where $\lambda$ is an eigenvalue of $P$ with corresponding right eigenvector $v^T$.  First, note that it is clear that $(Y_n)$ is adapted: given $X_n$, we know what $Y_n$ is because we have $\lambda$ and $v$.  To check the averaging property \eqref{Eq:Martingale}, we compute
    \begin{align}
        \nE(Y_{n+1} \, | \, X_0, \ldots, X_n) &= \nE(\lambda^{-(n+1)}v(X_{n+1}) \, | \, X_0, \ldots, X_n) \notag\\
        &= \lambda^{-(n+1)}\nE(v(X_{n+1}) \, | \, X_0, \ldots, X_n) \notag\\
        &= \lambda^{-(n+1)}\nE(v(X_{n+1}) \, | \, X_n) \label{Eq:MartingaleCheckAverage}\\
        &= \lambda^{-(n+1)}\sum_{j=1}^N v(j)P(X_n,j), \notag
    \end{align}
    where we used the Markov property in \eqref{Eq:MartingaleCheckAverage}.  But notice that the sum in the last line is the dot product of $v^T$ with the row of $P$ corresponding to $X_n$, which is exactly the $X_n$-entry of the matrix-vector product $Pv^T$.  Since $v^T$ is an eigenvector, we thus have
    \begin{align*}
        \lambda^{-(n+1)}\sum_{j=1}^N v(j)P(X_n,j)         &= \lambda^{-(n+1)} (\lambda v^T)(X_n)\\
        &= \lambda^{-n} v^T(X_n) = Y_n.
    \end{align*}
    We conclude $(Y_n)$ is a martingale.
\end{proof}

\subsection{The Markov operator $P$}
Given the averaging property of martingales, it is probably not hard to believe that we will frequently encounter steps like \eqref{Eq:MartingaleCheckAverage} (if you didn't read that proof, you should do so now).  The following formalism will help.  Let $f: \Omega \rightarrow \mathbb{R}$ be a function on our sample space.  We can simply think of $f$ as a row vector in $\mathbb{R}^N$, $f = (f_1, f_2, \ldots, f_N)$, with $f_j$ the value of $f$ at state $j$.  Given a chain $(X_n)$ with transition matrix $P$, we define the \textbf{Markov operator} 
\begin{align}\label{Eq:MarkovOperatorP}
	P: \mathbb{R}^N \rightarrow \mathbb{R}^N
\end{align}
as 
\begin{align}
	(Pf)(x) :=& \nE(f(X_{n+1}) \, | \, X_n=x) \label{Eq:MarkovOp1}\\
	=& \sum_{j=1}^N P(x,j)f_j \label{Eq:MarkovOp3}\\
	=& (Pf^T)(x), \label{Eq:MarkovOp2}
\end{align}
the $x$-entry of the matrix-vector product $Pf^T$.  Here \eqref{Eq:MarkovOp1} is the definition of the Markov operator, and \eqref{Eq:MarkovOp3} is the computation of the conditional expectation using the transition matrix.  

So what is happening?  We see the Markov operator takes a starting function $f$ and outputs a new function $Pf$ on our space, where the value of $Pf$ at each state $x$ is the $P$-average \eqref{Eq:MarkovOp3} of the values of $f$ of all adjacent states in the chain.  In other words, the Markov operator is averaging out our function according to $P$.  If we start at state $x$ and win $f_j$ dollars each time we reach state $j$, $Pf(x)$ computes our expected one-step winnings. 

\begin{example}
    Suppose we have the function $f=(-2,0,3,6,1,1)$ on the Ehrenfest urn with $B=5$ balls.  This means, for instance, that if there are two balls in the first urn, we ``win'' \$3.  If there are no balls, we lose \$2.  Applying the Markov operator means multiplying by the transition matrix $P$,
    \begin{align*}
        Pf = \begin{bmatrix}
            0 & 1 & 0 & 0 & 0 & 0\\
            1/5 & 0 & 4/5 & 0 & 0 & 0\\
            0 & 2/5 & 0 & 3/5 & 0 & 0\\
            0 & 0 & 3/5 & 0 & 2/5 & 0\\
            0 & 0 & 0 & 4/5 & 0 & 1/5\\
            0 & 0 & 0 & 0 & 1 & 0
        \end{bmatrix} \begin{bmatrix}  -2 \\ 0 \\ 3 \\ 6 \\ 1 \\ 1\end{bmatrix} = \begin{bmatrix} 0 \\ 2 \\ 3.6 \\ 2.2 \\ 5 \\ 1 \end{bmatrix},
    \end{align*}
    and we have averaged $f$ according to $P$.  So, for example, if we start with four balls in the first urn, we expect to win \$5 in one step.
\end{example}

\begin{exercise}
    What happens if you keep on averaging?  Compute $P^2f$, $P^3f$, $P^4f$ (use a computer if needed).  Does $P^nf$ appear to have a limit as $n \rightarrow \infty$? 
\end{exercise}

So concretely, the Markov operator is rather simple,  as it is just the matrix multiplication of $P$ with $f^T \in \mathbb{R}^N$.  We can also view it abstractly, though, as a function on functions (!), since its input is a function $f$ and its output a new function $Pf$.  We will interchangeably switch between these perspectives, and you should keep both in mind.  We observe that the concrete perspective immediately gives the following lemma.

\begin{lemma}
	The Markov operator $P$ is linear.  That is, given two functions $f,g : \Omega \rightarrow \mathbb{R}$ and $\alpha, \beta \in \mathbb{R}$, the function $P(\alpha f + \beta g)$ on $\Omega$ is the same as the function $\alpha (Pf) + \beta (Pg)$ on $\Omega$.
\end{lemma}
\begin{proof}
    Since we can consider the function $P(\alpha f + \beta g)$ as the column vector $P(\alpha f^T + \beta g^T)$, the result is clear by the linearity of matrix multiplication.
\end{proof}

We conclude with the following remark as an aside.  If $P$ is a linear operator on \emph{functions} via right-multiplication, what objects does it naturally act upon via left-multiplication? Recall that if $\mu = (q_1, \ldots, q_N)$ is a probability distribution on $\Omega$, then so is $\mu P$.  More concretely, if $X_0 \sim \mu$, then $\mu P$ is the distribution of $X_1$ (review the beginning of \S \ref{Sec:Stationary} if you are rusty here).  Thus, while $P$ takes functions $f$ to new functions $Pf$ through right-multiplication, it is a linear operator on \emph{probability distributions} on $\Omega$ through left-multiplication.

\section{Harmonic functions}\label{Sec:Harmonic}

We saw in the previous section how to construct martingales using eigenvalues and eigenvectors of $P$.  In this section we begin developing another method, using \emph{harmonic functions}.  Harmonic functions have a rich structure and we will spend the remainder of the chapter studying them and their relation to martingales.

\subsection{Space-time harmonic and harmonic}
All of our examples of martingales so far, \eqref{Eq:MartEg1}, \eqref{Eq:MartEg2}, \eqref{Eq:PolyaMartingale}, \eqref{Def:GraphHarmonicMart} and \eqref{Eq:EvalEvecMartingale}, have been functions of the form 
\begin{align}\label{Eq:MartLastState}
    Y_n = f_n(X_n) = f(n,X_n).
\end{align}
That is, even though the definition of a martingale permits $Y_n$ be a function of \emph{all} the preceeding states $X_0,X_1, \ldots, X_n$, so far our examples have only used the most-recent value $X_n$.  Seeing that processes of this form \eqref{Eq:MartLastState} are evidently common, we begin by asking when they are martingales.

It is obvious that such $Y_n$ are adapted, and so we just have to see when they satisfy the averaging property.  We compute
\begin{align}
	\nE(Y_{n+1} \, | \, X_0, \ldots, X_n) &= \nE(f_{n+1}(X_{n+1}) \, | \, X_0, \ldots, X_n)
	\notag\\
	&= \nE(f_{n+1}(X_{n+1}) \, | \, X_n) \notag\\
	&= (Pf_{n+1})(X_n). \label{Eq:STHarmonicMotivation}
\end{align}
That is, the conditional expectation is the average payout by $f_{n+1}$ after taking one step from $X_n$.  So in order for $(Y_n)$ to be a martingale, we need 
\begin{align*}
	\nE(Y_{n+1} \, | \, X_0, \ldots, X_n)=(Pf_{n+1})(X_n) = Y_n = f_n(X_n)
\end{align*}
for all $n$ and any state $X_n$.  As an identity of functions, this says $Pf_{n+1} = f_n$, and we make the following definition.

\begin{definition}
	A function $f: \mathbb{Z}_{\geq 0} \times \Omega \rightarrow \mathbb{R}$ is \textbf{space-time harmonic} if, for all $n \geq 0$,
	\begin{align}\label{Def:STHarmonic}
		Pf_{n+1} = f_n,
	\end{align}
	where $f_n(x) := f(n,x)$.  
\end{definition}
Note that this definition is closely related to our notion of a ``fair game,'' as component-wise it says $Pf_{n+1}(x) = f_n(x)$ for all $x \in \Omega$.  That is, if our payout at stage $n$ and state $x$ is $f_n(x)$, then our expected payout after taking one random step in our chain is exactly $f_n(x)$, what we began with.  That is, the process $Y_n=f_n(X_n)$ is a martingale. We record this for future reference as a theorem.

% We have defined space-time harmonic functions $f_n$ in such a way that processes $Y_n = f_n(X_n)$ automatically are martingales.  

% Our work in \eqref{Eq:STHarmonicMotivation} and the subsequent definition \eqref{Def:STHarmonic} show that, if $f_n$ is space-time harmonic, then $Y_n=f_n(X_n)$ is a martingale. 

\begin{theorem}\label{Thm:STHarmonicMartingale}
    If $f:\mathbb{Z}_{\geq 0} \times \Omega \rightarrow \mathbb{R}$ is space-time harmonic, then the process $(Y_n)$ defined by 
    \begin{align*}
        Y_n := f(n,X_n) = f_n(X_n)
    \end{align*}
    is a martingale.
\end{theorem}
\begin{exercise}\label{Ex:ProveSTHarmonicMartingale}
    Prove Theorem \ref{Thm:STHarmonicMartingale} by showing $\nE(Y_{n+1}\,|\,X_0,\ldots,X_n) = Y_n$ (this suffices for the proof since $(Y_n)$ is clearly adapted).
\end{exercise}

\begin{example}\label{Eg:SpaceTimeHarmonic}
    We have already seen many example of space-time harmonic functions.
    \begin{itemize}
        \item In Example \ref{Eg:SecondMartingale}, we saw that $X_n^2 -n$ is a martingale for the simple symmetric random walk $X_n$ on $\mathbb{Z}$.  Here the underlying space-time harmonic function is $f_n(x) = x^2-n$.  The fact that $Pf_{n+1} = f_n$ is exactly the computation carried out in that example.
        \item For the P\'{o}lya urn, Example \ref{Eg:PolyaMartingale}, we saw that the proportion $X_n/(n+a+b)$ of black balls is a martingale.  Here the space-time harmonic function is $f_n(x) = x/(n+a+b)$.  
        \item Our first martingale, Example \ref{Eg:FirstMartingale}, was based on the elementary observation that the simple walk $X_n$ on $\mathbb{Z}$ is itself a martingale.  We can also consider this as coming from a space-time harmonic function (albeit not a very interesting one), namely, $f_n(x) = x$ for all $n,x$.  Thus there is no dependence on $n$ and the function is always the same, the identity function $f(x)=x$.
        
        We can also view Example \ref{Eg:HarmonicMartingale1} as having an underlying space-time harmonic function, namely $f_n(x) = h(x)$, where again there is no dependency upon $n$.
    \end{itemize}
\end{example}

\begin{exercise}
    What is the underlying space-time harmonic function $f_n$ for the Ehrenfest urn martingale $Z_n$ of Example \ref{Eg:EhrenfestMartingale}?
\end{exercise}

% We can think of a space-time harmonic function $f$ as actually a \emph{sequence} of functions $f_n$ on $\Omega$ that satisfy the averaging property \eqref{Def:STHarmonic}; the Markov operator ``projects'' each subsequent function $f_{n+1}$ back to the previous one $f_n$.

As in the last bullet point in Example \ref{Eg:SpaceTimeHarmonic}, the simplest sub-class of space-time harmonic functions are those where the sequence $f_0, f_1, \ldots$ is constant in $n$; that is, where $f_n$ is just a fixed function $f:\Omega \rightarrow \mathbb{R}$ for all $n$. Even though it is an easier case, this turns out to correspond to an extremely important class of functions, the \emph{harmonic} functions.

\begin{definition}
	A function $h: \Omega \rightarrow \mathbb{R}$ is \textbf{harmonic} on $\Omega$ (with respect to $P$) if $Ph = h$, which is to say,
	\begin{align}\label{Def:Harmonic}
		(Ph)(x) = h(x)
	\end{align}
	for all $x \in \Omega$.  We say $h$ is \textbf{harmonic on a subset $D \subset \Omega$} if \eqref{Def:Harmonic} holds for all $x \in D$.
\end{definition}
Note that \eqref{Def:Harmonic} says
\begin{align}\label{Eq:Harmonic2}
	h(x) = \sum_{j=1}^N P(x,j)h(j) = (Ph)(x).
\end{align}
That is, when we start at $x$ and average the ``pay out'' $h(j)$ over one random step in our chain, the result is exactly the same as our starting value $h(x)$. We call this the \textbf{mean-value property} at $x$, and \eqref{Def:Harmonic} says that harmonic functions are precisely those functions on $\Omega$ which have the mean-value property at all $x \in \Omega$. From the abstract operator point of view, \eqref{Def:Harmonic} says $h$ is a \emph{fixed point} of the Markov operator $P$; applying the Markov operator gives the same function.
% Hence the function is unchanged when averaged with respect to the probability measure $P(x, \cdot)$ drawn from the corresponding row of $P$. , 

\begin{exercise}\label{Ex:HarmonicEigenvector}
	Interpret \eqref{Def:Harmonic} through the linear algebra lens.  What does it translate to in terms of the transition matrix $P$ and the column vector $h^T$?
\end{exercise}

Just as we can build martingales through space-time harmonic functions, Theorem \ref{Thm:STHarmonicMartingale}, so we also can through individual harmonic functions.  Composing any chain $(X_n)$ with a harmonic function $h$ creates a ``fair game'' $(h(X_n))$ since the value of a harmonic function equals its average over the next step in the chain, \eqref{Def:Harmonic}.

\begin{theorem}\label{Thm:HarmonicMartingale}
	If $(X_n)$ is a Markov chain on $\Omega$ and $h: \Omega \rightarrow \mathbb{R}$ is harmonic, then the process $(Y_n)$ defined by 
	\begin{align*}
		Y_n = h(X_n)
	\end{align*}
	is a martingale.
\end{theorem}
This theorem is already proved because it is a special case of Theorem \ref{Thm:STHarmonicMartingale}, where the space time harmonic function $f_n$ is just given by the harmonic function $h$ at every step $n$.  However, because martingales of this form are especially important, we review the steps.  We hope the following is very similar to what you already did for Exercise \ref{Ex:ProveSTHarmonicMartingale}.
\begin{proof}
    By definition, $(Y_n)$ is adapted; once we know $X_n$ we know the value of $Y_n$.  So we check the averaging property, noting
    \begin{align*}
        \nE(Y_{n+1} \, | \, X_0, \ldots, X_n) &= \nE(h(X_{n+1}) \, | \, X_n)\\
        &= \sum_{j=1}^N P(X_n,j)h(j)\\
        &= (Ph)(X_{n}) = h(X_n) = Y_n,
    \end{align*}
    where the second-to-last equality uses the harmonicity of $h$.  We conclude that $h(X_n)$ is a martingale.
\end{proof}

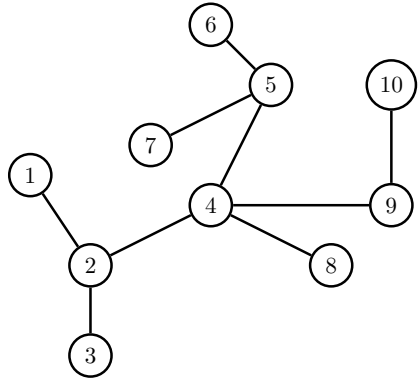
\begin{figure}
    \centering
         \scalebox{0.8}{\begin{tikzpicture}
        \tikzset{
            node/.style={circle, draw=black, very thick, minimum size=7mm},
            arrow/.style={->, black, very thick},
            line/.style={black, very thick}
            }
        \node[node] (one) at (4,0) {8};
        \node[node] (two) at (5,1) {9};
        \node[node] (three) at (3,3) {5};
        \node[node] (four) at (0,0) {2};
        \node[node] (five) at (2,4) {6};
        \node[node] (six) at (1,2) {7};
        \node[node] (seven) at (-1,1.5) {1};
        \node[node] (eight) at (0,-1.5) {3};
        \node[node] (nine) at (5,3) {10};
        \node[node] (x) at (2,1) {4};
        
        \draw[line] (one) -- (x) -- (two) -- (nine);
        \draw[line] (x) -- (three) -- (five);
        \draw[line] (three) -- (six);
        \draw[line] (x) -- (four) -- (seven);
        \draw[line] (four) -- (eight);
    \end{tikzpicture}}
    \caption{\small What are the harmonic functions $h$ on this graph?}
    \label{Fig:HarmonicEg3}
\end{figure}

\begin{example}
	Let's concretely look at what \eqref{Def:Harmonic} means for the graph in Figure \ref{Fig:HarmonicEg3}, where the transition matrix $P$ corresponds to the simple symmetric random walk on the graph.  Since $P(x,j) = 1/\deg(x)$, \eqref{Eq:Harmonic2} says that each function value $h(x)$ must equal the average of all its neighboring values.  Hence,
	\begin{align*}
		h(1) &= h(2),\\
        h(2) &= \frac{1}{3}(h(1) + h(3) + h(4) ),\\
        h(3) &= h(2),\\
        h(4) &= \frac{1}{4}(h(2) + h(8) + h(9) + h(5)),\\
        h(5) &= \frac{1}{3}(h(6) + h(7) + h(4) ),
	\end{align*}
	and so on.  One easy candidate would be a constant function, $h(x) \equiv c$ for all $x \in \Omega$ and some $c \in \mathbb{R}$.  It is less obvious how to form non-constant examples, or whether such $h$ even exist.
\end{example}

\begin{exercise}\label{Ex:NonConstantHarmonic}
	Construct a non-constant harmonic function $h$ for the 3-cycle with matrix $P$ given by the simple symmetric random walk, or show that no such $h$ exists.
\end{exercise}

What you see in Exercise \ref{Ex:NonConstantHarmonic} is characteristic of all harmonic functions for irreducible Markov chains.

\begin{theorem}\label{Thm:HarmonicConstant}
    If $P$ is the transition matrix for an irreducible Markov chain on $\Omega = \{1, \ldots, N\}$ and $h$ is harmonic with respect to $P$, then $h$ is constant.
\end{theorem}

\begin{proof}
    We give the proof idea for when the chain is a simple symmetric random walk on a graph $G = (V,E)$.  Suppose $h$ is harmonic on $G$, and let $x_0$ be a vertex such that
    \begin{align}\label{Eq:HarmonicMax}
        h(x_0) = M := \max_{x \in V} h(x).
    \end{align}
    By harmonicity, we also have
    \begin{align}
        h(x_0) &= \frac{1}{\deg(x_0)} \sum_{y \sim x_0}h(y) \label{Eq:HarmonicMaxAvg}\\
        &\leq \frac{1}{\deg(x_0)} \sum_{y \sim x_0} M = M, \label{Eq:HarmonicMax2}
    \end{align}
    since each $h(y) \leq M$ by \eqref{Eq:HarmonicMax}.  If we had some $y \sim x_0$ such that $h(y) < M$, then the inequality in \eqref{Eq:HarmonicMax2} would be strict, yielding $M=h(x_0) < M$, a contradiction.  Hence $h(y) \equiv M$ for all $y \sim x_0$.  That is, $h$ has the same (maximum) value on $x_0$ and all its neighbors.
    
    Now we iterate on each of these adjacent vertices, taking further steps away from $x_0$: the same argument shows that $h$ must also have the same value on all of \emph{their} neighbors (see Figure \ref{Fig:HarmonicConstantProof}).  Since the Markov chain is irreducible, $G$ is connected, and so after finitely-many iterations of this argument we see that $h(x) \equiv M$ for all $x \in V$. 
\end{proof}

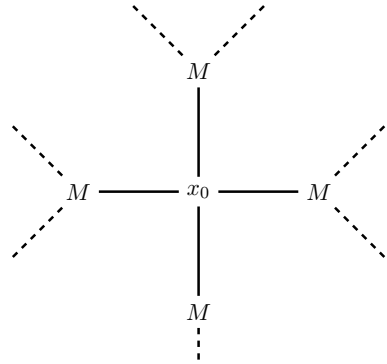
\begin{figure}
    \centering
    \scalebox{0.8}{
        \begin{tikzpicture}
    \tikzset{
            node/.style={draw=black, very thick, minimum size=7mm},
            arrow/.style={->, black, very thick},
            line/.style={black, very thick}
            }
        \node (X0) at (0,0) {$x_0$};
        \node (M1) at (2,0) {$M$};
        \node (M2) at (-2,0) {$M$};
        \node (M3) at (0,2) {$M$};
        \node (M4) at (3.25,-1.25) {};
        \node (M5) at (3.25,1.25) {};
        \node (M6) at (-3.25,-1.25) {};
        \node (M7) at (-3.25,1.25) {};
        \node (M8) at (-1.25,3.25) {};
        \node (M9) at (1.25,3.25) {};
        \node (M10) at (0,-2) {$M$};
        \node (M11) at (0,-3) {};
        
        \draw[line] (X0) -- (M1);
        \draw[line] (X0) -- (M2);
        \draw[line] (X0) -- (M3);
        \draw[line, dashed] (M1) -- (M4);
        \draw[line, dashed] (M1) -- (M5);
        \draw[line, dashed] (M2) -- (M6);
        \draw[line, dashed] (M2) -- (M7);
        \draw[line, dashed] (M3) -- (M8);
        \draw[line, dashed] (M3) -- (M9);
        \draw[line] (X0) -- (M10);
        \draw[line,dashed] (M10) -- (M11);
    \end{tikzpicture}
    }
    \caption{\small The proof idea for Theorem \ref{Thm:HarmonicConstant}: since $h$ attains its maximum $M$ at $x_0$, and it is also the average of all of its neighbors, all the neighbors must likewise have value $M$.  Then move out one step and apply this same argument to all $y \sim x_0$.  Iterating shows that $h$ is constant on $\Omega$.}
    \label{Fig:HarmonicConstantProof}
\end{figure}

\noindent As the last line of the proof suggests, the irreducibility assumption in Theorem \ref{Thm:HarmonicConstant} is necessary.  

\begin{exercise}
    Construct a non-constant harmonic function $h$ on the disconnected graph in Figure \ref{Fig:HexagonPlusLine}.
\end{exercise}

%Lecture 26
\subsection{A bit of dirty laundry}
Harmonic functions and Theorem \ref{Thm:HarmonicConstant} enable us to (finally) finish our self-contained proof that an irreducible Markov chain has a unique stationary distribution, Theorem \ref{Thm:ReturnTime}.  Recall that in Theorem \ref{Thm:ReturnTimeLite} we proved that any irreducible chain has at least one stationary distribution $\pi$, given by the formula
\begin{align}\label{Eq:StationaryFormulaAgain2}
    \pi(x) = \frac{1}{ \nE(\tau_x^+ \, | \, X_0 = x)}.
\end{align}
We used the Perron-Frobenius theorem in \S \ref{Sec:PF} to prove the uniqueness of $\pi$, completing the proof of Theorem \ref{Thm:ReturnTime}.  The unsatisfying part, though, is that we did not fully prove the Perron-Frobenius theorem.  Since we desire a self-contained argument, we proceed to show that any stationary distribution is, in fact, given by \eqref{Eq:StationaryFormulaAgain2}.

\begin{proof}[Proof of Theorem \ref{Thm:ReturnTime}]
    Any stationary distribution $\rho$ is a left-eigenvector of $P$ for $\lambda_1=1$, $\rho P = \rho$.  So it suffices to show that the dimension of the left-eigenspace $L_1$ for $\lambda_1$ is one, since then we must have
    \begin{align*}
        L_1 = \{ c \pi \, | \, c \in \mathbb{R}\}
    \end{align*}
    for $\pi$ defined by \eqref{Eq:StationaryFormulaAgain2}. As $c\pi$ is only a probability measure when $c=1$, we have uniqueness.  Hence we show $\dim(L_1)=1$.
    
    It follows from the fact that any matrix is similar to its transpose that the dimension of $L_1$ is the same as the dimension of the \emph{right} eigenspace $R_1$ for $\lambda_1=1$.  If $v^T\in R_1$, then
    \begin{align*}
        Pv^T = v^T,
    \end{align*}
    and $v^T$ is thus harmonic on $\Omega$ with respect to $P$ (compare Exercise \ref{Ex:HarmonicEigenvector}).  Hence by Theorem \ref{Thm:HarmonicConstant}, $v^T$ is a constant function/vector.  In other words,
    \begin{equation*}
        v^T = \begin{pmatrix}
            c\\c\\ \vdots \\c
        \end{pmatrix} = 
        c \begin{pmatrix}
            1\\1\\ \vdots \\1
        \end{pmatrix} \in \text{span}\{ \mathbf{1}^T\},
    \end{equation*}
    and $R_1$ is one-dimensional.
\end{proof}

Note again the difference between the left- and right-eigenspaces for $\lambda=1$ with respect to our transition matrix $P$.  The left eigenspace is spanned by the stationary distribution $\pi$ for $P$, whereas the right eigenspace consists of the harmonic functions on all of $\Omega$ (which are all constant functions by Theorem \ref{Thm:HarmonicConstant}).  

\section{Harmonic functions on subsets $D \subsetneq \Omega$}\label{Sec:HarmonicSubsets}

We learned in the previous section that we can build martingales using space-time harmonic and harmonic functions.  Namely, if $f_n$ is space-time harmonic for our chain, then $Y_n := f_n(X_n)$ is a martingale, and if $h$ is harmonic, then $Z_n := h(X_n)$ is also a martingale.

You might be wondering why even bother to mention that $h(X_n)$ is a martingale, when the only harmonic functions on an irreducible chain are constant by Theorem \ref{Thm:HarmonicConstant}.  The reason is that if we relax the condition that $h$ is harmonic on \emph{all} of $\Omega$ and instead just focus on a proper subset $D \subsetneq \Omega$, the situation dramatically changes.  We study such harmonic functions in this section and the next section.  We will see this class of harmonic functions is quite rich, and we will work towards classifying all such harmonic functions.

Recall that we say $h:\Omega \rightarrow \mathbb{R}$ is harmonic on $D \subsetneq \Omega$ if it has the mean-value property 
\begin{align}\label{Eq:Harmonic3}
    (Ph)(x) = \sum_{j=1}^N P(x,j)h(j) = h(x)
\end{align}
for all $x \in D$.  In particular, note that we have no requirement on the values of $h$ at states $y$ outside of $D$.  Even though $h$ is defined for $y \in D^c$, no averaging property needs to hold there.

\begin{example}\label{Eg:Harmonic4Cycle}
    To see that this makes a difference, consider the simple symmetric random walk on the 4-cycle in Figure \ref{Fig:HarmonicSubsetEg}, and take $D$ to be vertices 1 and 2.  That is, $h$ only must satisfy the mean-value property at vertices 1 and 2.  We (arbitrarily) assign the values $h(3)=1$ and $h(4)=0$ for the two vertices outside of $D$.  What values does $h$ take on $D$? From \eqref{Eq:Harmonic3}, we must have
    \begin{align*}
        h(1) &= \frac{1}{2}(h(2) + h(4)) = \frac{1}{2} h(2),\\
        h(2) &= \frac{1}{2}(h(1) + h(3)) = \frac{1}{2}(h(1) + 1).
    \end{align*}
    Solving this linear system yields the (unique) values $h(1) = 1/3$ and $h(2) = 2/3$.  So our function $h$ is harmonic on $D$ but is no longer constant, even though the 4-cycle is irreducible.

    It is also instructive to note that the mean value property \emph{does not} hold at vertices $y \in \Omega \backslash D = \{3,4\}$.  Indeed, we find the average at $x=3$ is
    \begin{align*}
        \frac{1}{2}(h(2)+h(4)) = \frac{1}{3} \neq 1=h(3),
    \end{align*}
    and at $x=4$ we have
    \begin{align*}
        \frac{1}{2}(h(1) + h(3)) = \frac{2}{3} \neq 0=h(4).
    \end{align*}
    This is not any problem, though, as we only asked $h$ to be harmonic on $D$.
    
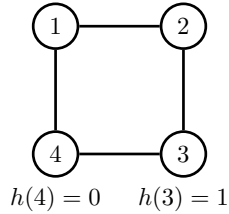
\begin{figure}
    \centering
    \scalebox{0.85}{
        \begin{tikzpicture}
    \tikzset{
        node/.style={circle, draw=black, very thick, minimum size=7mm},
        line/.style={black, very thick}
    }
    
    \node[node] (one) at (0,0) {1};
    \node[node] (two) at (2,0) {2};
    \node[node] (three) [label=below:{$h(3)=1$}] at (2,-2) {3};
    \node[node] (four) [label=below:{$h(4)=0$}] at (0, -2) {4};
    
    \draw[line] (one) -- (two);
    \draw[line] (two) -- (three);
    \draw[line] (three) -- (four);
    \draw[line] (four) -- (one);
    \end{tikzpicture}}
    \caption{\small Set $D=\{1,2\}$ and assign values 0 and 1 for $h$ at $x=4$ and $x=3$, respectively.  What are $h$'s values at vertices 1 and 2, if $h$ is harmonic there?}
    \label{Fig:HarmonicSubsetEg}
\end{figure}
\end{example}

This example immediately raises a number of questions. How can we think about such harmonic functions?  Is there any meaning to the specific values for $h$ in this example?  Do we have a method for building other harmonic functions on $D$?  Can we describe all such functions?  We spend the remainder of this section and the next section answering these questions.

As a staring point of our analysis, write $\Omega$ as the disjoint union $\Omega = D \cup D^c$, and note that we can partition the transition matrix $P$ in terms of vertices in $D$ and vertices outside of $D$,
\begin{align*}
    P = \begin{pNiceMatrix}[first-row,first-col]
        & D & D^c       \\
    D   & A   & B \\
    D^c & C   & Q  \\
\end{pNiceMatrix}
\end{align*}
If $\#\Omega = N$ and $\#D=m$, then, for instance, $A$ is the $m \times m$ submatrix of probabilities for starting and ending in $D$, and $B$ is the $m \times (N-m)$ submatrix of probabilities for starting in $D$ and leaving $D$.

Now consider tweaking $P$ to
\begin{align}\label{Eq:TildeP}
    \tilde{P} = \begin{pNiceMatrix}[first-row,first-col]
        & D & D^c       \\
    D   & A   & B \\
    D^c & \mathbf{0}   & I  \\
\end{pNiceMatrix}.
\end{align}
Here $\mathbf{0}$ is a $(N-m) \times m$ matrix of all zeros, and $I$ is the $(N-m)\times (N-m)$ identity matrix.  It is not immediately obvious why we would want to do this, but we claim that $\tilde{P}$ is still a transition matrix.  Clearly $\tilde{P}(x,y) \geq 0$ for all $x,y \in \Omega$, and so we need to verify that each row sums to one.  A moment's thought shows this is obvious: if $x \in D$, then
\begin{align*}
    \sum_{j=1}^N \tilde{P}(x,j) = \sum_{j=1}^N P(x,j) = 1,
\end{align*}
while if $x \notin D$, then the only non-zero entry in row $x$ of $\tilde{P}$ is that coming from the identity matrix block, which is 1.  That is,
\begin{align*}
    \sum_{j=1}^N \tilde{P}(x,j) = \tilde{P}(x,x) = 1,
\end{align*}
and so $\tilde{P}$ is indeed a transition matrix.  

What is the Markov chain $\tilde{X}$ corresponding to $\tilde{P}$?  Since $\tilde{P}(x,y) = P(x,y)$ for $x \in D$, $\tilde{X}$ behaves the same as $X$ when in $D$.  However, for $x \in D^c$,
\begin{align*}
    \nP(\tilde{X}_{n+1}=y \, | \, \tilde{X}_n =x) = \I_{\{x=y\}},
\end{align*}
and so $\tilde{X}$ is permanently fixed at the first vertex it reaches outside of $D$.  In other words, $\tilde{P}$ is the transition matrix for the chain which is \emph{absorbed in $D^c$}.  Another way to write this is \begin{align}\label{Eq:AbsorbedChain}
    \tilde{X}_n = X_{n \wedge \tau},
\end{align}
where $\tau = \tau_{D^c}$ is the hitting time of $D^c$ and \begin{align*}
    n \wedge \tau := \min\{n, \tau\}.
\end{align*}
So \eqref{Eq:AbsorbedChain} says that $\tilde{X}$ is the same as $X$ until we reach $D^c$, at which point $\tilde{X}$ becomes frozen in the same state. Our modified transition matrix $\tilde{P}$ thus turns each $y\in D^c$ into an absorbing boundary state for the Markov chain.

What we are really interested in, though, is harmonic functions.  How do harmonic functions with respect to $\tilde{P}$ relate to harmonic functions on $D$ for the original chain?  The following lemma tells us that harmonic functions on $D$ for $P$ become harmonic \emph{on all of $\Omega$} with respect to $\tilde{P}$.

\begin{lemma}\label{Lemma:HarmonicEverywhere}
    If $h$ is a harmonic function on $D$ with respect to $P$, then $h$ is harmonic on all of $\Omega$ with respect to the modified transition matrix $\tilde{P}$ in \eqref{Eq:TildeP}.
\end{lemma}
We encourage you to think through what is happening here through the following exercises.

\begin{exercise}
    Explain why the lemma is intuitively clear.
\end{exercise}

\begin{exercise}
    Prove the lemma by showing \eqref{Eq:Harmonic3} holds for all $x \in \Omega$ when $P$ is replaced by $\tilde{P}$.
\end{exercise}

\begin{exercise}
    We had a non-constant harmonic function $h$ on $D$ in Example \ref{Eg:Harmonic4Cycle}, and Lemma \ref{Lemma:HarmonicEverywhere} says that this $h$ is harmonic on all of $\Omega$ with respect to $\tilde{P}$.  Why is this not a contradiction of Theorem \ref{Thm:HarmonicConstant}?
\end{exercise}
% \begin{proof}
%     We need to show \eqref{Eq:Harmonic3} holds for all $x \in \Omega$ when $P$ is replaced by $\tilde{P}$.  However, this is immediate from the definition of $\tilde{P}$.  Indeed, if $x \in D$, then
%     \begin{align*}
%         (\tilde{P}h)(x) &= \sum_{j=1}^N \tilde{P}(x,j)h(j)\\
%         &= \sum_{j=1}^N P(x,j)h(j) = h(x).
%     \end{align*}
%     The first equality on the second line is because the rows of $\tilde{P}$ for $x \in D$ are the same as the rows of $P$, and the second equality is because $h$ is harmonic on $D$ with respect to $P$.
    
%     If $x \notin D$, examing the form of the corresponding row \eqref{Eq:TildeP} of $\tilde{P}$ yields
%     \begin{align*}
%         (\tilde{P}h)(x) &= \sum_{j=1}^N \tilde{P}(x,j)h(j) = \tilde{P}(x,x) h(x) = 1\cdot h(x).
%     \end{align*}
% \end{proof}

So we now understand harmonic functions on $D \subset \Omega$ are harmonic on all of $\Omega$ for the absorbing chain $X_{n \wedge \tau}$.  We still do not understand, however, how to build these harmonic functions, other than solving a linear system of equations (as in Example \ref{Eg:Harmonic4Cycle}).  Is there a more systematic approach?  The following theorem says one method is to take averages of function values upon hitting $D^c$.

% We still do not understand, however, how the specific values of such $h$ arise.  Is there any more intuition behind saying $h(1) = 1/3$ and $h(2) = 2/3$ in Example \ref{Eg:Harmonic4Cycle}, for instance?
\begin{theorem}\label{Thm:BuildHarmonic}
    Let $P$ be a transition matrix on $\Omega$ and let $D \subsetneq \Omega$ be a proper subset of states.  Let $\tau = \tau_{D^c}$ be the hitting time of $D^c$ and let $f: D^c \rightarrow \mathbb{R}$ be any function.   Then the function $h: \Omega \rightarrow \mathbb{R}$ defined as
    \begin{align}\label{Eq:BuildHarmonic}
        h(x) := \nE_x(f(X_\tau)), \qquad x \in \Omega,
    \end{align}
    is harmonic with respect to $P$ on $D$ and satisfies $h(x) = f(x)$ for all $x \in D^c$.
\end{theorem}

\noindent In this context we often say that $h$ is harmonic on $D$ with \emph{boundary values} $f$.

The intuition here is to think of $f$ as the ``pay-out'' function: we start by running the chain from some $x \in D$, and we are playing the game that we ``earn'' $f(y)$ dollars if we first exit $D$ at state $y \in D^c$. Theorem \ref{Thm:BuildHarmonic} says that the function $h(x)$ which gives our average earnings from starting at $x$ is harmonic on $D$.

\begin{proof}
    It is clear from the definition \eqref{Eq:BuildHarmonic} that $h(x)=f(x)$ for $x \in D^c$, since $\tau=0$ for such $x$.  
    
    To show $h$ is harmonic elsewhere, we need to show \eqref{Eq:Harmonic3} holds in $D$.  Suppose first that $f:D^c \rightarrow \mathbb{R}$ is an indicator of a single vertex $y \in D^c$,
    \begin{align*}
        f(z) = \I_y(z) =  \begin{cases}
            1 & z = y,\\
            0 & z \neq y.
        \end{cases}
    \end{align*}
    Then $f(X_\tau) = \I_y(X_\tau)$, and to show that $\nE_x(\I_y(X_\tau))$ is harmonic we condition on the first step of the walk.  Indeed, for $x \in D$,
    \begin{align*}
        h(x) = \nE_x(\I_y(X_\tau)) &= \nP_x(X_\tau = y)\\
        &= \sum_{j=1}^N \nP_x(X_\tau =y \, | \, X_1=j)\nP_x(X_1=j)\\
        &= \sum_{j=1}^N \nP_j(X_\tau = y) P(x,j)\\
        &= \sum_{j=1}^N \nE_j(\I_y(X_\tau)) P(x,j)\\
        &= \sum_{j=1}^N h(j) P(x,j),
    \end{align*}
    which is exactly the mean-value property, showing $h$ is harmonic on $D$ as claimed.
    
    For functions $f:D^c \rightarrow \mathbb{R}$ which are not necessarily indicators, we can write $f$ as a linear combination of indicators,
    \begin{align*}
        f(z) = \sum_{j \in D^c} f(j) \I_j(z).
    \end{align*}
    Then our function $h$ is 
    \begin{align*}
        h(x) = \nE_x(f(X_\tau)) = \sum_{j \in D^c} f(j) \nE_x(\I_j(X_\tau))
    \end{align*}
    since expectation is linear.  By our above argument, each of the functions $\nE_x(\I_j(X_\tau))$ in this sum is harmonic on $D$.  Since harmonic functions are closed under linear combinations (see Problem \ref{Ex:LinearCombosHarmonic}), we conclude that $h$ is itself harmonic.
\end{proof}

\begin{exercise}\label{Ex:Harmonic4Cycle}
    Revisit the 4-cycle in Figure \ref{Fig:HarmonicSubsetEg} with $D = \{1,2\}$ and the boundary values $f$ given by $f(3)=1$ and $f(4)=0$.  Explicitly compute $\nE_x(f(X_\tau))$ for $x \in D$.  How does this compare with the function $h$ in Example \ref{Eg:Harmonic4Cycle}?
\end{exercise}

%Lecture 27

Let's see Theorem \ref{Thm:BuildHarmonic} at work in a couple of examples.

\begin{example}\label{Eg:BuildHarmonic1}
    Consider the simple symmetric random walk $(X_n)$ on $\{0,1,\ldots, N\}$ with $D=\{1,\ldots, N-1 \}$ and let $\tau$ be the hitting time of $D^c = \{0,N\}$.  Suppose we start the walk at some $x \in D$.  What is 
    \begin{align*}
        \nP_x(X_\tau = N)?
    \end{align*}
    We already know the answer, of course, from our considerations of gambler's ruin in \S \ref{Sec:GamblersRuin}; Theorem \ref{Thm:GamblersRuin} says $\nP_x(X_\tau = N) = x/N$.  The new insight we gain from Theorem \ref{Thm:BuildHarmonic} is that this function 
    \begin{align*}
        h(x) = \nP_x(X_\tau = N) = \nE_x(\I_N(X_\tau))
    \end{align*}
    is actually harmonic.  While the theorem guarantees this, it is also easy to explicitly verify: indeed, for $x \in D$,
    \begin{align*}
        (Ph)(x) &= \frac{1}{2}h(x-1) + \frac{1}{2}h(x+1)\\
        &= \frac{x-1}{2N} + \frac{x+1}{2N}\\
        &= \frac{x}{N} = h(x),
    \end{align*}
    and hence $Ph=h$.  Note also that
    \begin{align*}
        h(0) = 0 = \I_N(0) \quad \text{ and } h(N) = 1 = \I_N(N),
    \end{align*}
    and so $h$ agrees with the boundary values given by $\I_N$ on $D^c$.
\end{example}

\begin{example}
    Just as at the end of the proof of Theorem \ref{Thm:BuildHarmonic}, we can use linearity to compute our expected ``earnings'' $\nE_x(f(X_\tau))$ for other ``pay-out'' functions $f$ on $D^c$.  For the same chain as in Example \ref{Eg:BuildHarmonic1}, consider the function $f:D^c \rightarrow \mathbb{R}$ given by 
    \begin{align*}
        f(0) = -10 \quad \text{ and } \quad f(N) = 40.
    \end{align*}
    That is, we lose 10 dollars if we first exit $\{1,\ldots, N-1\}$ at $0$, and win 40 dollars if we exit at $N$.  What is our expected earnings $\tilde{h}(x) = \nE_x(f(X_\tau))$ now?  We have
    \begin{align*}
        \tilde{h}(x) &= \nE_x(f(X_\tau))\\
        &= -10 \cdot \nP_x(X_\tau =0) + 30 \cdot \nP_x(X_\tau = N)\\
        &= -10 \cdot (1-h(x)) + 30 \cdot h(x)\\
        &= -10\Big(1 - \frac{x}{N}\Big) + \frac{30x}{N} = -10 + \frac{40x}{N}.
    \end{align*}
    It is easy to explicitly verify that this is a harmonic function on $D$ which satisfies $\tilde{h}(0) =-10$ and $\tilde{h}(N)=40$.
\end{example}

In Problem \ref{Prob:HarmonicZInterval} you will characterize all harmonic functions on intervals of the integers for the simple random walk on $\mathbb{Z}$.

\section{Optional sampling and harmonic functions}\label{Sec:Optional}

We are well on our way to understanding all harmonic functions on $\Omega = \{1,\ldots, N\}$ with respect to transition probabilities $P$.  Functions harmonic on all of $\Omega$ for irreducible chains are simply the constant functions, Theorem \ref{Thm:HarmonicConstant}.  Functions harmonic on a subset $D\subsetneq \Omega$ form a much richer class, and we know that we can build such $h$ with arbitrary ``boundary values'' $f:D^c \rightarrow \mathbb{R}$ through the formula \eqref{Eq:BuildHarmonic}, that is,
\begin{align}\label{Eq:ExpectedPayoutLast}
    h(x) = \nE_x(f(X_\tau))
\end{align} 
for $\tau=\tau_{D^c}$.  Note that we also recover the constant functions with this construction when $f$ is constant on $D^c$.

But does \eqref{Eq:ExpectedPayoutLast} give \emph{all} the harmonic functions on $D$?  Or are there others that we cannot write as the expected ``pay-out'' of our Markov chain at the hitting time?  In this section we bring all of our tools together to answer in the negative: there are no harmonic function on $D$ that are not of the form \eqref{Eq:ExpectedPayoutLast} for some $f:D^c \rightarrow \mathbb{R}$.  This is a beautiful and important result, and is a fitting capstone to our text.

We will need an important piece of machinery, the \emph{optional sampling theorem}, to prove this.  This will be a generalization of the averaging property for martingales to random times.  So while we initially introduced harmonic functions to help us study martingales, in this last section we turn the tables and use properties of martingales to help us classify harmonic functions.

We begin by looking at how simple examples make the universality of \eqref{Eq:ExpectedPayoutLast} very plausible.  The upshot will be that this is a direct consequence of the mean-value property $Ph=h$ of harmonic functions.
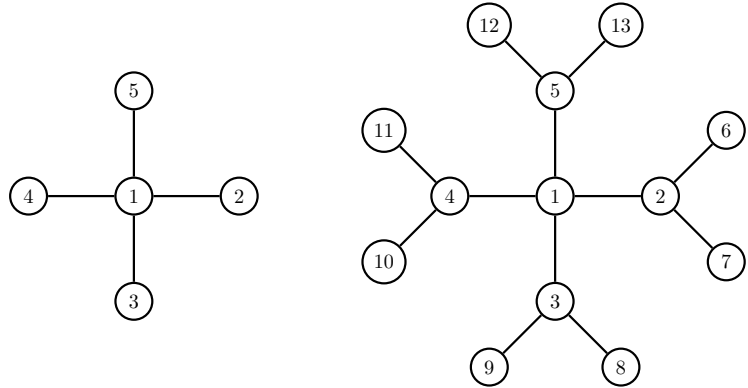
\begin{figure}
    \centering
    \scalebox{0.7}{
        \begin{tikzpicture}
    \tikzset{
            node/.style={circle, draw=black, very thick, minimum size=7mm},
            arrow/.style={->, black, very thick},
            line/.style={black, very thick}
            }
        \node[node] (1) at (-2,0) {1};
        \node[node] (2) at (0,0) {2};
        \node[node] (3) at (-2,-2) {3};
        \node[node] (4) at (-4,0) {4};
        \node[node] (5) at (-2,2) {5};
        
        \node[node] (1b) at (6,0) {1};
        \node[node] (2b) at (8,0) {2};
        \node[node] (3b) at (6,-2) {3};
        \node[node] (4b) at (4,0) {4};
        \node[node] (5b) at (6,2) {5};
        
        \node[node] (6b) at (9.25,1.25) {6};
        \node[node] (7b) at (9.25,-1.25) {7};
        \node[node] (8b) at (7.25,-3.25) {8};
        \node[node] (9b) at (4.75,-3.25) {9};
        \node[node] (10b) at (2.75,-1.25) {10};
        \node[node] (11b) at (2.75,1.25) {11};
        \node[node] (12b) at (4.75,3.25) {12};
        \node[node] (13b) at (7.25,3.25) {13};
        
        \draw[line] (1) -- (2);
        \draw[line] (1) -- (3);
        \draw[line] (1) -- (4);
        \draw[line] (1) -- (5);
        
        \draw[line] (1b) -- (2b);
        \draw[line] (1b) -- (3b);
        \draw[line] (1b) -- (4b);
        \draw[line] (1b) -- (5b);
        \draw[line] (2b) -- (6b);
        \draw[line] (2b) -- (7b);
        \draw[line] (3b) -- (8b);
        \draw[line] (3b) -- (9b);
        \draw[line] (4b) -- (10b);
        \draw[line] (4b) -- (11b);
        \draw[line] (5b) -- (12b);
        \draw[line] (5b) -- (13b);
    \end{tikzpicture}
    }
    \caption{\small Two examples where we can explicitly see that the mean-value property implies the averaging formula $h(x) = \nE_x(f(X_\tau))$.}
    \label{Fig:HarmonicAverage}
\end{figure}

\begin{example}
    Consider the two graphs in Figure \ref{Fig:HarmonicAverage} with transition probabilities given by the simple symmetric random walk.  Start with the graph $G_1$ on the left, and suppose $D_1=\{1\}$, just the center vertex.  If $h$ is harmonic on $D_1$, then 
    \begin{align*}
        h(1) &= \sum_{j=2}^5\frac{1}{4}h(j) = \sum_{j=2}^5 \nP_1(X_\tau =j)h(j) = \nE_1(f(X_\tau))
    \end{align*}
    if $f=h$ on $D_1^c=\{2,3,4\}$ is the boundary-value function.  The second equality is since $\nP_1(X_\tau=j) = 1/4$ for each $j \in D_1^c$.  So we see that the mean value property immediately forces our function to be of the form \eqref{Eq:ExpectedPayoutLast}.
    
    Now expand the graph as in the right of Figure \ref{Fig:HarmonicAverage}; call this $G_2$.  Expand the set of harmonic states to include all the neighbors, $D_2 = \{1,2,3,4,5\}$, and suppose that $h$ is any harmonic function on $D_2$.  Then, as above,
    \begin{align}\label{Eq:HarmonicAvgEg2}
         h(1) &= \frac{1}{4}h(2) + \frac{1}{4}h(3) + \frac{1}{4}h(4) + \frac{1}{4}h(5).
    \end{align}
    Since $h$ is also harmonic at 2, we have $h(2) = \frac{1}{3}(h(1)+h(6)+h(7))$, with similar expressions at vertices 2,4 and 5. Plugging these all back into the right-hand side of \eqref{Eq:HarmonicAvgEg2} yields
    \begin{align*}
        h(1) = \frac{4}{12}h(1) + \frac{1}{12}\sum_{j=6}^{13}h(j) \quad \Leftrightarrow \quad h(1) = \sum_{j=6}^{13}\frac{1}{8} h(j).
    \end{align*}
    By symmetry of $G_2$ with reference to our starting point $x=1$, hitting $D^c$ at any of the vertices $6,7,\ldots,13$ should be equally likely.  Hence the above says
    \begin{align*}
        h(1) = \sum_{j=6}^{13}\nP_1(X_\tau =j) h(j) = \nE_1(f(X_\tau))
    \end{align*}
    for $f=h$ on $D^c$.  Thus we again see that (iterations of) the mean-value property lead to \eqref{Eq:ExpectedPayoutLast}.
\end{example}
Of course, this is no proof that all harmonic functions satisfy \eqref{Eq:ExpectedPayoutLast}.  And we were significantly helped by the symmetry of $G_1$ and $G_2$ in Figure \ref{Fig:HarmonicAverage} to see our expressions gave the average $\nE_x(f(X_\tau))$.  We will need a more robust approach for a general proof, provided via the \emph{optional sampling} (or \emph{optional stopping}) theorem.

\begin{theorem}[Optional sampling]\label{Thm:OptionalSampling}
    Let $(X_n)_{n \geq 0}$ be an irreducible Markov chain on $\Omega=\{1,\ldots, N\}$, and let $(Y_n)_{n \geq 0}$ be a martingale for $(X_n)$.  Then if $\tau$ is the hitting time of any non-empty subset $D^c \subset \Omega$, 
    \begin{align}\label{Eq:OptionalSampling}
        \nE(Y_\tau) = \nE(Y_0).
    \end{align}
\end{theorem}

We saw in Theorem \ref{Thm:MartingaleAverage} that for any \emph{fixed}, \emph{deterministic} time $n$, $\nE(Y_n) = \nE(Y_0)$.  Theorem \ref{Thm:OptionalSampling} says something much stronger: we may replace $n$ with a random time $\tau$, thus obtaining the composition of random variables $Y_\tau$, and still have the identity \eqref{Eq:OptionalSampling}.  To gain an appreciation for some of the subtlties involved here, we start by considering two examples.

\begin{example}
We first note that we need the assumption of irreducibility in order for the left-hand side of \eqref{Eq:OptionalSampling} to make sense.  For instance, consider a simple symmetric random walk on the reducible graph of Figure \ref{Fig:HexagonPlusLine} (a six-cycle plus a disconnected two-cycle).  Start at a vertex in the six-cycle and let $\tau$ be the hitting time of the two-cycle.  Since we will never reach the two-cycle, $\tau =\infty$ and $Y_\tau$ is not well defined.   You will show in Problem \ref{Prob:taufiniteas} that if $(X_n)$ is an irreducible chain, then $\tau<\infty$ with probability one.  Hence in this case it makes sense to talk about the random variable $Y_\tau$. 
\end{example}

\begin{example}
    We secondly note that the finiteness assumption on $\Omega$ is also necessary.  Indeed, consider the simple symmetric random walk $(X_n)$ on $\mathbb{Z}$, started from $x=0$.  We saw in Example \ref{Eg:FirstMartingale} that $Y_n = X_n$ is a martingale.  Let $\tau = \tau_{\{1\}}$ be the hitting time of the vertex 1.  If \eqref{Eq:OptionalSampling} held in this case, we would have
    \begin{align*}
        1 = \nE(Y_\tau) = \nE(Y_0) = 0,
    \end{align*}
    which is absurd.  The reason for the failure here is that the infiniteness of the state space $\mathbb{Z}$ makes $\tau$ a ``nasty'' stopping time.\footnote{We note that there are versions of Theorem \ref{Thm:OptionalSampling} for infinite state spaces, but they include assumptions on the nature of $\tau$.  For instance, it suffices to have some fixed $M<\infty$ such that $\tau \leq M$ with probability 1, or to have $|Y_{t \wedge \tau}| \leq M$ with probability 1.  That is, the hitting time is bounded or the martingale up until time $\tau$ is bounded.  Both of these fail for the $\tau$ in this example: $\tau$ can be arbitrarily large, and $Y_n$ may become very large negatively before coming back and hitting 1.}  This also helps us appreciate that it is not obvious that \eqref{Eq:OptionalSampling} should hold.  Unlike for deterministic times, this is not always the case, and so this is a non-trivial statement.
\end{example}

\begin{proof}
    Our proof is based on the fact \eqref{Eq:ConditionalIncrementZero} that the conditional increments have zero mean.  Note that for a deterministic time $n$, we can use a telescoping sum to write
    \begin{align*}
        Y_n - Y_0 &= \sum_{j=1}^{n} (Y_{j}-Y_{j-1}) =\sum_{j=1}^\infty (Y_{j}-Y_{j-1}) \I_{\{j \leq n\}}.
    \end{align*}
    Similarly, for our random time $\tau$,
    \begin{align*}
        Y_\tau - Y_0 = \sum_{j=1}^{\tau} (Y_{j}-Y_{j-1}) =\sum_{j=1}^\infty (Y_{j}-Y_{j-1}) \I_{\{j \leq \tau\}}.
    \end{align*}
    % where we are implicitly using the fact that the chain is irreducible and $D^c$ non-empty to conclude that $\tau < \infty$ with probability one.  
    Taking expectations on both sides of the above yields
    \begin{align}\label{Eq:OptionalStoppingSetup}
        \nE(Y_\tau) - \nE(Y_0) = \sum_{j=1}^\infty \nE((Y_{j}-Y_{j-1}) \I_{\{j \leq \tau\}}).
    \end{align}
    (In general, some care needs to be taken when moving an expectation through an infinite sum.  Careful treatment of this is beyond the scope of our text, but the interchange is justified here.)  We claim that each of the expectations in the sum is zero,
    \begin{align}\label{Eq:OptionalStoppingClaim}
        \nE((Y_{j}-Y_{j-1}) \I_{\{j \leq \tau\}})=0
    \end{align}
    for all $j \geq 1$.  To see this, note that
    \begin{align*}
        \nE((Y_{j}-Y_{j-1}) \I_{\{j \leq \tau\}}) = \nE\big(\nE((Y_{j}-Y_{j-1}) \I_{\{j \leq \tau\}} \, | \, X_0, \ldots, X_{j-1}) \big),
    \end{align*}
    and so it suffices to show that
    \begin{align*}
        \nE((Y_{j}-Y_{j-1}) \I_{\{j \leq \tau\}} \, | \, X_0, \ldots, X_{j-1}) =0
    \end{align*}
    for each $j$.  If we are given the values of $X_0, \ldots, X_{j-1}$, then we know the value of $\I_{\{j \leq \tau\}}$ (it is 1 if each $X_k \in D$, $0 \leq k \leq j-1$, and 0 otherwise) and may treat it as a constant, pulling it outside the expectation:
    \begin{align*}
        \nE((Y_{j}-Y_{j-1}) \I_{\{j \leq \tau\}} \, | \, X_0, \ldots, X_{j-1}) &=  \I_{\{j \leq \tau\}}\nE(Y_{j}-Y_{j-1} \, | \, X_0, \ldots, X_{j-1})\\
        &=  \I_{\{j \leq \tau\}} \cdot 0 = 0
    \end{align*}
    by the conditional martingale increments, \eqref{Eq:ConditionalIncrementZero}.  We conclude the right-hand side of \eqref{Eq:OptionalStoppingSetup} is zero, completing the proof.
\end{proof}

With the optional stopping theorem in hand, we can prove that all harmonic functions are averaged ``pay outs'' as in \eqref{Eq:ExpectedPayoutLast}.

\begin{theorem}\label{Thm:BuildHarmonic2}
    If $h$ is a function on $\Omega$ that is harmonic on $D \subsetneq \Omega$ with respect to some transition matrix $P$, where $D^c$ includes states in each irreducible component of the chain, then there exists a function $f:D^c \rightarrow \mathbb{R}$ such that
    \begin{align}\label{Eq:BuildHarmonic2}
        h(x) = \nE_x(f(X_\tau))
    \end{align}
    where $\tau=\tau_{D^c}$.
\end{theorem}

There is no mystery about what the function $f$ is; we will see it is simply $h$ on $D^c$.  That is, the theorem is saying that a harmonic function on $D$ is always equal to its expected pay-out upon leaving $D$!

Before the proof, a word about the technical condition on $D^c$: we need a state in each irreducible component because of chains like that in Figure \ref{Fig:HexagonPlusLine}.  Indeed, suppose for that graph we start a random walk in the six-cycle and $\tau$ is the hitting time of the two-cycle.  Then $\tau=\infty$ and it does not make sense to talk about $X_\tau$, and so the right-hand side of \eqref{Eq:BuildHarmonic2} doesn't make sense.  Our assumption on $D^c$ in the theorem statement prevents this; $D^c$ would have to have a vertex in both the six-cycle and the two-cycle in this case (that is, in each of the two irreducible components of the chain).

\begin{proof}
    Fix an initial state $x \in \Omega$.  By considering the irreducible component of the chain containing $x$, we see that it suffices to assume that our chain is irreducible to begin with.  (For example, if $x$ belonged to the six-cycle in Figure \ref{Fig:HexagonPlusLine}, we could proceed by ignoring the two-cycle, as it is impossible to enter.  We would separately consider the two-cycle when we started at one of its two states.)  The point is that we want to be able to apply the optional sampling theorem, and by restricting to the irreducible component containing $x$ we may do so. 
    
    Recall from \S \ref{Sec:HarmonicSubsets} that if $h$ is harmonic on $D$ for $(X_n)$, then $h$ is harmonic on all of $\Omega$ for the chain $(\tilde{X}_n)$ where $D^c$ becomes an absorbing boundary, defined by
    \begin{align*}
        \tilde{X}_n = X_{\tau \wedge n}.
    \end{align*}
    Hence $Y_n := h(\tilde{X}_n)$ is a martingale, Theorem \ref{Thm:HarmonicMartingale}, and we may apply the optional sampling theorem. Starting the chain $(X_n)$ at $X_0=x \in D$, this yields
    \begin{align*}
        \nE(Y_\tau) &= \nE(Y_0)\\
        &= \nE_x(h(\tilde{X}_0)) = \nE_x(h(X_0)) = \nE_x(h(x)) = h(x),
    \end{align*}
    as $h(x)$ is a non-random constant.  On the other hand, however,
    \begin{align*}
        \nE(Y_\tau) = \nE_x(h(\tilde{X}_\tau)) = \nE_x(h(X_\tau)) = \nE_x(f(X_\tau)),
    \end{align*}
    where $f(y):=h(y)$ for $y \in D^c$.  Putting the last two equations together yields \eqref{Eq:BuildHarmonic2}.
\end{proof}

We close with two applications of Theorem \ref{Thm:BuildHarmonic2}.  The first is a new proof of Theorem \ref{Thm:HarmonicConstant}.

\begin{proof}[Alternative proof of Theorem \ref{Thm:HarmonicConstant}] Let $h$ be a function harmonic on all of $\Omega$; we wish to show that $h$ is constant.  Fix an $x \in \Omega$ and let $y$ be any other state; we show $h(x) = h(y)$.  Indeed, we have that $Y_n := h(X_n)$ is a martingale, and for $D^c = \{y\}$, $\tau=\tau_{D^c}$, the optional sampling theorem gives
\begin{align*}
    h(y) = \nE_x(h(X_\tau)) = \nE_x(Y_\tau) = \nE_x(Y_0) = \nE_x(h(X_0)) = h(x),
\end{align*}
and hence $h$ is constant.
\end{proof}

The second application is that harmonic functions with given ``boundary values'' $f$ on a subset $D^c$ are unique. That is, the boundary values of a harmonic function determine its values everywhere.  Initially this seems like a very strong statement, but in light of \eqref{Eq:BuildHarmonic2} we see that it has to be the case.  This formula says that the boundary values ``communicate inwards,'' and thus determine all the other values of the harmonic function.  
\begin{theorem}\label{Thm:HarmonicUnique}
    Let $h_1$ and $h_2$ be functions on $\Omega$ that are harmonic on $D \subsetneq \Omega$.  If $h_1(y) = h_2(y)$ for all $y \in D^c$, then $h_1=h_2$.
\end{theorem}

\begin{proof}
    We need to show $h_1(x) = h_2(x)$ for all $x \in D$.  Since $X_\tau \in D^c$, \eqref{Eq:BuildHarmonic2} immediately gives
    \begin{equation*}
        h_1(x) = \nE_x(h_1(X_\tau)) = \nE_x(h_2(X_\tau)) = h_2(x).\qedhere
    \end{equation*}
\end{proof}

It is instructive to revisit Example \ref{Eg:Harmonic4Cycle} and Exercise \ref{Ex:Harmonic4Cycle} in light of Theorem \ref{Thm:HarmonicUnique}.  In the exercise, you were asked to compute $\nE_x(f(X_\tau))$ where $f$ has the same values on $D^c$ as $h$ in Example \ref{Eg:Harmonic4Cycle}.  So by Theorem \ref{Thm:HarmonicUnique}, we immediately know these are the same functions.  In the example we found $h(1)=1/3$, and hence
\begin{align*}
    \frac{1}{3} = \nE_1(f(X_\tau)) = 0 \cdot \nP_1(X_\tau = 4) + 1 \cdot \nP_1(X_\tau =3) = \nP_1(X_\tau = 3).
\end{align*}
Similarly, we had $h(2)=2/3$, and the same logic shows that $\nP_2(X_\tau=3)=2/3$.  We can thus see that harmonic functions give us a powerful method to attack the gambler's ruin problem of \S \ref{Sec:GamblersRuin} for general graphs, using boundary values of 0 and 1 and \eqref{Eq:BuildHarmonic2}.  

\vfill
\pagebreak

\section*{Problems for chapter \ref{Ch:Martingales}}
\addcontentsline{toc}{section}{Problems for chapter \ref{Ch:Martingales}}

%Add a problem about coming up with the general nth-order martingale for the simple random walk.
\begin{problem}\label{Prob:ExpMartingale}
	Let $(X_n)_{n \geq 0}$ be a simple random random walk on $\mathbb{Z}$ with $X_0 = 0$.
	\begin{enumerate}[$(a)$]
	    \item Show that the process
    \begin{align*}
        Y_n := X_n^3 - 3nX_n
    \end{align*}
    is a martingale.  This is the \emph{cubic martingale} for the simple symmetric random walk.
	    \item Show that 
	\begin{align*}
	    Y_n := e^{\alpha X_n - \frac{\alpha^2}{2}n}
	\end{align*}
	is a martingale, where $\alpha \in \nR$ is a constant.  This process $(Y_n)$ is the \emph{exponential martingale} for the simple symmetric random walk. 
	\item For any $n \in \mathbb{N}$ and $t \in \mathbb{R}$, compute the moment generating function $M_{X_n}(t) = \nE(e^{tX_n})$ for the simple walk  after $n$ steps.
	\end{enumerate}
\end{problem}

\begin{problem}
    Fix a proper subset $D \subsetneq \Omega$ of the state space $\Omega$ of a chain and let $\tau = \tau_D$ be the first hitting time of $D$.  We saw in Examples \ref{Eg:NonMartingale} and \ref{Eg:NonAdapted} that the processes defined by $Y_n := \I_{\{\tau \leq n\}}$ and $Z_n := \inf\{m>n \;: \; X_m \in D\}$ are both not martingales.  Show that we can fix this by taking our ``best guess'' for these random variables with the information on hand through step $n$.  That is, show:
    \begin{enumerate}[$(i)$]
        \item The process $(\tilde{Y}_n)_{n \geq 0}$ defined by $\tilde{Y}_n:= \nP(\tau \leq n \; | \; X_0, \ldots, X_n)$ is a martingale.
        \item The process $(\tilde{Z}_n)_{n \geq 0}$ defined by $$\tilde{Z}_n := \nE(\,\inf\{m>n \,: \, X_m \in D \} \; | \; X_0, \ldots, X_n)$$ is a martingale.
    \end{enumerate}
\end{problem}

\begin{problem}\label{Ex:LinearCombosHarmonic}
	Suppose $h_1, \ldots, h_n$ are harmonic functions on $D \subset \Omega$ with respect to $P$.  Show that, for any constants $c_1, \ldots, c_n$, the linear combination $\sum_{j=1}^n c_j h_j$ is likewise harmonic on $D$.
\end{problem}

\begin{problem}\label{Prob:HarmonicZInterval}
In this problem you will characterize all harmonic functions on an interval of integers $D=\{M,M+1,\ldots, M+N\}$ for the simple symmetric random walk on $\mathbb{Z}$.
\begin{enumerate}[$(a)$]
    \item Show that any linear function $h(x) = mx+b$, where $m,b \in \mathbb{R}$, is harmonic on $D$.  Explain intuitively why this is the case.
    \item Conversely, show that if $h$ is harmonic on $D$, then there exist $m,b \in \mathbb{R}$ such that $h(x)=mx +b$.  Hence all harmonic functions for the simple symmetric random walk on an interval of $\mathbb{Z}$ are linear.
\end{enumerate}
{\footnotesize \emph{Comment:} Note that the arguments here do not depend on the interval $\{M,M+1,\ldots, M+N\}$ in question.  That is, linear functions are harmonic on \emph{all} of $\mathbb{Z}$.  If you've been reading the text carefully, this might sound a bit suspicious: didn't we say in Theorem \ref{Thm:HarmonicConstant} that functions $h$ that are harmonic on all of $\Omega$ are constant?  The key here is that $\mathbb{Z}$ is not a \emph{finite} state space, as we were assuming in Theorem \ref{Thm:HarmonicConstant} (and is usually the case for our Markov chains).  In the proof of that theorem, it is instructive to note that we started by choosing a state $x_0$ where $h$ achieves a maximum.  That isn't necessarily possible if $\Omega$ is infinite, as for $\mathbb{Z}$ and (non-constant) linear functions $h$!  So there is no contradiction between that theorem and this problem.} 
\end{problem}

%A5B5.3
\begin{problem}\label{Prob:AB5.3} Consider the following Markov chain on the state space $\Omega=\{ 2^k, \; k=\ldots, -2,-1,-0,1,2,\ldots\}$. If your current position is $x=2^k$, then the next step you jump to $2x=2^{k+1}$ or $x/2=2^{k-1}$ with equal probabilities. 

\begin{enumerate}[$(a)$]
\item Let $D=\{2^{-N}, 2^{-N+1}, \ldots,1,2,\ldots, 2^N\}$, for some $N\ge 1$. Suppose you absorb the chain when it hits $D^c$. Show that $f(x)=\log(x)$ is a harmonic function at any $x$ in $D$.  
\item Let $\tau$ denote the hitting time of $D^c$. Find $\E_2\left( \tau\right)$ (i.e., the expected hitting time of $D^c$ starting from $2$). 
\end{enumerate}
\end{problem}

%A5B5.1
\begin{problem}\label{Prob:AB5.1} Consider a connected graph $G=(V,E)$ and a function $f:V\rightarrow \rr$. Assume that the random walk on this graph is aperiodic. Define a sequence of functions $\{f_0, f_1, \ldots\}$ on the vertices in the following way. Start with $f_0(v)=f(v)$ for all vertices $v$. Then, inductively, 
\[
f_{k+1}(v) = \frac{1}{\mathrm{deg}(v)} \sum_{y \sim v} f_k(y), \quad k=0,1,2,\ldots.
\]
That is, at each step replace the value of the function at any vertex by the average value of all its neighbors in the previous step. 
Find $\lim_{k\rightarrow \infty} f_k(v)$ for each $v$.
\end{problem}

\begin{problem}\label{Prob:5Graph1}
    For the graph $G$ in Figure \ref{Fig:MartingaleProbability}, let $D=\{1,2,3,4,5\}$ and let $\tau=\tau_{D^c}$ be the hitting time of $D^c$ for the simple symmetric random walk $(X_n)$ on $G$.  
    \begin{enumerate}[$(a)$]
        \item Find $\nP_x(X_\tau=8)$ for all vertices $x \in D$.
        \item Find $\nP_x(X_\tau \in \{6,7,9,10\})$ for all vertices $x \in D$.
        \item Find $\nP_x(X_\tau \in \{6,7,8\})$ for all vertices $x \in D$.
    \end{enumerate}
\begin{figure}[ht]
    \centering
    \scalebox{0.6}{
        \begin{tikzpicture}
    \tikzset{
            node/.style={circle, draw=black, very thick, minimum size=7mm},
            arrow/.style={->, black, very thick},
            line/.style={black, very thick}
            }
        
        \node[node] (1b) at (6,0) {1};
        \node[node] (2b) at (8,0) {2};
        \node[node] (3b) at (6,-2) {3};
        \node[node] (4b) at (4,0) {4};
        \node[node] (5b) at (6,2) {5};
        
        \node[node] (6b) at (9.25,1.25) {6};
        \node[node] (7b) at (9.25,-1.25) {7};
        \node[node] (8b) at (6,-3.25) {8};
        \node[node] (10b) at (2.75,-1.25) {9};
        \node[node] (11b) at (2.75,1.25) {10};
        \node[node] (12b) at (4.75,3.25) {11};
        \node[node] (13b) at (7.25,3.25) {12};
        
        \draw[line] (1b) -- (2b);
        \draw[line] (1b) -- (3b);
        \draw[line] (1b) -- (4b);
        \draw[line] (1b) -- (5b);
        \draw[line] (2b) -- (6b);
        \draw[line] (2b) -- (7b);
        \draw[line] (3b) -- (8b);
        \draw[line] (4b) -- (10b);
        \draw[line] (4b) -- (11b);
        \draw[line] (5b) -- (12b);
        \draw[line] (5b) -- (13b);
    \end{tikzpicture}
    }
    \caption{\small Find the hitting probabilities for certain subsets of the boundary.}
    \label{Fig:MartingaleProbability}
\end{figure}
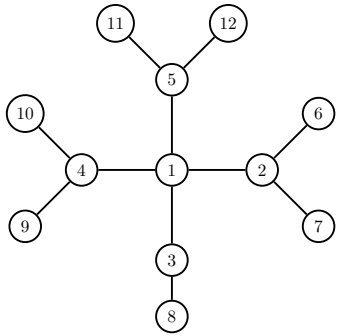
\end{problem}

\begin{problem}\label{Prob:5Opt1}
    Use optional sampling for the simple symmetric random walk $(X_n)$ on integers $\{a,a+1, \ldots, b\}$ to re-derive the gambler's ruin hitting probabilities of Theorem \ref{Thm:GeneralGamblersRuin}.
\end{problem}

\begin{problem}\label{Prob:5Opt2}
    Use optional sampling and the quadratic martingale $Y_n = X_n^2-n$ for the simple symmetric random walk on integers $\{a,a+1, \ldots, b\}$ to re-derive the gambler's ruin expected hitting times of Theorem \ref{Thm:GeneralGamblersRuinTimes}.
\end{problem}

\begin{problem}\label{Prob:taufiniteas}
     Let $(X_n)_{n \geq 0}$ be an irreducible Markov chain on $\Omega=\{1,\ldots, N\}$, and let $\tau$ be the hitting time of any non-empty subset $D^c \subset \Omega$.  Start the chain at some $x \in \Omega$.  Show that, with probability 1, $\tau < \infty$.
\end{problem}

\begin{problem}\label{Prob:5Gauntlet}
    A casino is designing a new game called ``The Gauntlet,'' and, knowing your expertise in Markov chains, hires you as a consultant.  The game is based on the graph in Figure \ref{Fig:Gauntlet} and works as follows.  The player starts at $S$, and in the first step immediately moves right one step.  Then they roll a die.  If they roll a 1,2 or 3, they move one step left.  Getting a 4,5 or 6 moves them one step right.  The same rule applies whenever they are at a vertex with only two neighbors.  If they are at a vertex with three neighbors, then a roll of 1 or 2 moves them left, a 3 or 4 moves them vertically, and a 5 or 6 moves them right.  If they ever return to $S$, they automatically go right one step in the next turn.
    
    The game is free to play.  It ends when the player first reaches a vertex marked $L$ or $W$.  If they first reach $W$, they win \$1,000,000.  If they first reach an $L$ vertex, they must pay the casino some amount of money $x$.  
    \begin{enumerate}[$(a)$]
        \item What is the probability of winning?
        \item The casino would like to make \$1, on average, each time someone plays this game.  As the consultant, what do you tell them for the value of $x$?
        \item The casino is worried that if $x$ is too large then the game will be unpopular.  They think it will make a big difference if they only want to earn $\$0.01$ on average for this game.  What should the new value of $x$ be here?  Does it make a big difference?
    \end{enumerate}
\end{problem}

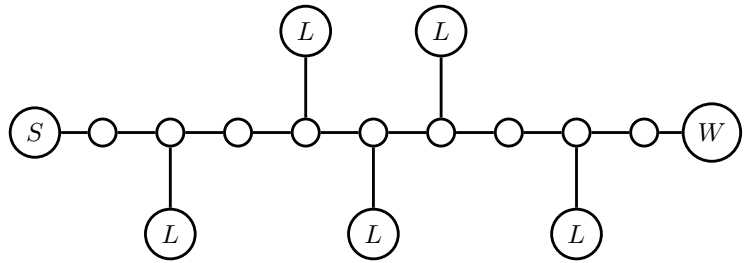
\begin{figure}[ht]
    \centering
    \scalebox{0.9}{
        \begin{tikzpicture}
    \tikzset{
            node/.style={circle, draw=black, very thick, minimum size=4mm},
            arrow/.style={->, black, very thick},
            line/.style={black, very thick}
            }
        
        \node[node] (1) at (1,0) {$S$};
        \node[node] (2) at (2,0) {};
        \node[node] (3) at (3,0) {};
        \node[node] (4) at (4,0) {};
        \node[node] (5) at (5,0) {};
        \node[node] (6) at (6,0) {};
        \node[node] (7) at (7,0) {};
        \node[node] (8) at (8,0) {};
        \node[node] (9) at (9,0) {};
        \node[node] (10) at (10,0) {};
        \node[node] (11) at (11,0) {$W$};
        \node[node] (3L) at (3,-1.5) {$L$};
        \node[node] (5L) at (5,1.5) {$L$};
        \node[node] (6L) at (6,-1.5) {$L$};
        \node[node] (7L) at (7,1.5) {$L$};
        \node[node] (9L) at (9,-1.5) {$L$};
        
        \draw[line] (1) -- (2);
        \draw[line] (2) -- (3);
        \draw[line] (3) -- (4);
        \draw[line] (4) -- (5);
        \draw[line] (5) -- (6);
        \draw[line] (6) -- (7);
        \draw[line] (7) -- (8);
        \draw[line] (8) -- (9);
        \draw[line] (9) -- (10);
        \draw[line] (10) -- (11);
        
        \draw[line] (3) -- (3L);
        \draw[line] (5) -- (5L);
        \draw[line] (6) -- (6L);
        \draw[line] (7) -- (7L);
        \draw[line] (9) -- (9L);
    \end{tikzpicture}
    }
    \caption{\small \emph{The Gauntlet} casino game.}
    \label{Fig:Gauntlet}
\end{figure}

\chapter*{Appendix 1: Notation}
\addcontentsline{toc}{chapter}{Appendix 1: Notation}

\renewcommand{\arraystretch}{1.3}
\begin{longtable}{|c|p{9cm}|c|}
    \hline
    \textit{Notation} & \textit{Description} & \textit{First location in text}\\
    \hline
    $\mathbb{N}$ & The natural numbers $\{1,2,3,\ldots\}$ & Theorem \ref{Thm:PowersOfP}\\
    \hline
    $\mathbb{Z}$ & The integers $\{\ldots,-2,-1,0,1,2,\ldots\}$ & Example \ref{Eg:WalkOn4Cycle}\\
     \hline
    iid & Independent and identically distributed random variables.  That is, $X_1, X_2, \ldots$ are iid if they are mutually independent and each $X_j$ has the same distribution.  &  Review problem \\
    \hline
    Exp$(\lambda)$ & The exponential $\lambda$ distribution.  $X \sim \text{Exp}(\lambda)$ if $X$ has pdf $f(x) = \lambda e^{-\lambda x}$ for $x \geq 0$. &  Review problem\\
    \hline
    Geo$(p)$ & A geometric random variable $X$ with parameter $p$.  $X$ is the number of trials necessary to obtain the first success, where the probability of success in each trial is $p$. $X$ has pmf $p_X(k) = (1-p)^{k-1}p$, $k \in \mathbb{N}$. & Example \ref{Eg:ReturnTime}\\
    \hline
    Unif$(a,b)$ & The uniform distribution on the interval $(a,b)$ for some $a < b$.  $X \sim \Unif(a,b)$ if $X$ has pdf $f(x) = \frac{1}{b-a}$ for $x \in (a,b)$. & Review problem\\
    \hline
    $:=$ & The symbol on the left is \emph{defined} by what follows on the right.  For example, $\nP_{ij}:=\nP(X_1=j \, | \, X_0=i)$. & \S \ref{SSec:SRW}\\
    \hline
    $u \sim v$ & Vertices $u$ and $v$ in a graph are adjacent. & \S \ref{SSec:SRW}\\
    \hline
    $P_{ij}, P(i,j)$ & The probability of moving from state $i$ to $j$. & \\
    \hline
    $\hat{P}_{ij}, \hat{P}(i,j)$ & The probability of moving from state $i$ to $j$ in the time-reversed Markov chain. & \eqref{Eq:TimeReversalTransProb} \\
    \hline
    $\tau_A, \tau_x$ & The hitting time of $A$ and $x$, respectively. This is the first non-negative time $k\geq 0$ at which the Markov chain satisfies $X_k \in A$ or $X_k = x$, respectively & \eqref{Def:tau_x}\\
    \hline
    $\tau_A^+, \tau_x^+$ & The return time of $A$ and $x$, respectively, or the first positive time $k\geq 1$ that the Markov chain is in the set $A$ or state $x$. & \eqref{Def:tau_x+} \\
    \hline
    $\nP_x(A)$ & The conditional probability $\nP(A \, | \, X_0=x)$. & \eqref{Notation:ConditionalSubscriptP}\\
    \hline
    $\nP_\mu(A)$ & The conditional probability of $A$, given $X_0$ is distributed according to the probability distribution $\mu$. & \eqref{Notation:ConditionalSubscriptPiP}\\
    \hline
    $\nE_x(A)$ & The conditional expectation of $A$, given $X_0=x$. & \eqref{Notation:ConditionalSubscriptE}\\
    \hline
    $\nE_\mu(A)$ & The conditional expectation of $A$, given $X_0$ is distributed according to the probability distribution $\mu$. & \eqref{Notation:ConditionalSubscriptPiE}\\
    \hline
    $\partial \Omega$ & The boundary of the set $\Omega$. If $\Omega = \{0,1,\ldots, n\}$, then $\partial \Omega = \{0,n\}$. & \eqref{Eq:Boundary}\\
    \hline
    $a \wedge b$ & The minimum $\min\{a,b\}$ of $a$ and $b$ & \eqref{Eq:AbsorbedChain}\\
    \hline
\end{longtable}
\thispagestyle{plain}

\chapter*{Appendix 2: Suggested homework sets}
\addcontentsline{toc}{chapter}{Appendix 2: Suggested homework sets}

The following are suggestions for (approximately) weekly homework sets.  Problems numbered 6.$x$ are from the review problems in Appendix 3.  

Note that there also a number of exercises within the text of each chapter.  Students should be encouraged to do \emph{all} of these as they are reading to ensure they are internalizing the concepts.  Although they vary in difficulty, most of these exercises are straightforward and give an opportunity to immediately engage with the new ideas.

\thispagestyle{plain}
\begin{itemize}
    \item Homework set 1: Problems \ref{Prob:B1.1}, \ref{Prob:B1.2}, \ref{Prob:B1.3}, \ref{Prob:B1.4}, \ref{Prob:AbsorbingBarriers1}, \ref{Prob:B1.6}.
    
    \item Homework set 2: Problems \ref{Prob:ExpMin2}, \ref{Prob:A1.2}, \ref{Prob:B1.3b}, \ref{Prob:ReflectingBoundary}, \ref{Prob:A1.3}, \ref{Prob:A1.5}, \ref{Prob:A1.6}. 
    
    \item Homework set 3: Problems \ref{Prob:B2.1} to \ref{Prob:B2.6} and \ref{Prob:1Sociologist}.
    
    \item Homework set 4: Problem \ref{Prob:5CycleHittingTimes} and \ref{Prob:2Hitting} to \ref{Prob:A4.5}.
    
    \item Homework set 5: Problems \ref{Prob:B3.1}, \ref{Prob:B3.3}, \ref{Prob:B3.4}, \ref{Prob:B3.5} and \ref{Prob:B3.6}
    
    \item Homework set 6: Problems  \ref{Prob:B4.3}, \ref{Prob:B4.4}, \ref{Prob:A3.2}, \ref{Prob:A3.5} and \ref{Prob:A3.6}
    
    \item Homework set 7: Problems \ref{Prob:A4.4} to \ref{Prob:A3.4}
    
    % \item Homework set 4: Problems \ref{Prob:A2.1} to \ref{Prob:A2.6}, \ref{Prob:A2Urn}.
    
    %\item Homework set 5: Problems \ref{Prob:B3.1}, \ref{Prob:B3.2}, \ref{Prob:B3.3}, \ref{Prob:B3.4}, \ref{Prob:B3.5} and \ref{Prob:B3.6}
    
    % \item Homework set 6: Problems \ref{Prob:A3.2} and \ref{Prob:A3.1}  to \ref{Prob:A3.6}
    
    % \item Homework set 7: Problems \ref{Prob:B4.3}, \ref{Prob:B4.4}, \ref{Prob:B4.1}, \ref{Prob:B4.2}, \ref{Prob:B4.5}, \ref{Prob:B4.6}
    
    % \item Homework set 8: Problems \ref{Prob:A4.1}, \ref{Prob:A4.5},  \ref{Prob:A4.4}, \ref{Prob:A4.2}, \ref{Prob:A4.3}, \ref{Prob:A4.6}
    
    % \item Homework set 9: Problems \ref{Prob:AB5.2} to \ref{Prob:AB5.6}, \ref{Prob:AB5.3}, \ref{Prob:AB5.1}
    
    % \item Homework set 10: Problems \ref{Ex:LinearCombosHarmonic}, \ref{Prob:5Graph1} to \ref{Prob:5Gauntlet} 
\end{itemize}
\thispagestyle{plain}

\chapter*{Appendix 3: Review problems}
\addcontentsline{toc}{chapter}{Appendix 3: Review problems}
The following problems help you brush up on familiar distributions, conditional expectations, and linear algebra.

\thispagestyle{plain}
\bigskip
\stepcounter{chapter}
\begin{problem}\label{Prob:B1.1} Let $N$ be a geometric random variable with parameter $p$. That is, the p.m.f. of $N$ is given by 
\[
\nP(N=n)=(1-p)^{n-1}p,\quad n=1,2,3,\ldots.
\] 
Given $N=n$, generate $n$ many i.i.d. exponentials with rate $1$: $X_1, X_2, \ldots, X_n \sim \mathrm{Exp}(1)$. Let $W=\min\{X_1, X_2, \ldots, X_n\}$ be their minimum. 
\begin{enumerate}[$(a)$]
\item Find the conditional c.d.f. of $W$ by computing $\nP(W > t \mid N=n)$, n=1,2,\ldots.
\item Use part $(a)$ to compute $\nP(W > t)$.   
\item Compute the conditional probability mass function of $N$, given $\{W>t\}$, and identify it as one of the named distributions (binomial, Poisson, geometric, exponential, normal etc.) and specify all the parameters.
\end{enumerate}
\end{problem}

\begin{problem}\label{Prob:ExpMin2} Let $N$ be a Poisson random variable with mean $\lambda$. Given $N=n$, generate $X_1, X_2, \ldots, X_n, X_{n+1}$  i.i.d. Exp$(1)$ random variables. 
\begin{enumerate}[$(a)$]
	\item Let $Y =\min \{X_1, \ldots, X_{n+1}\}$ when $N=n$. Compute $P(Y > a)$ for all $a>0$.
	\item Find $E(Y)$. (\emph{Hint:} conditional expectation and the tower property.)
\end{enumerate}  
\end{problem}

%HW5. Problem 9.

\begin{problem}\label{Prob:B1.2} Let $X$ be the number of rolls of a fair die until I see the first six. Given $X=x$, I choose a sample, with replacement, of size $x$ from an urn with $5$ red balls and $4$ green balls. Let $Y$ be the number of green balls in my sample. 
\begin{enumerate}[$(a)$]
\item Find $\nE(Y\mid X=x)$.
\item Use part $(a)$ to find $\nE(Y)$.
\end{enumerate}
\end{problem}
\thispagestyle{plain}
\medskip

%midterm 2 in 395
\begin{problem}\label{Prob:A1.1} Take a stick of length one and denote it by the unit interval $[0,1]$. Pick two i.i.d. Unif$(0,1)$ random variables $U$ and $V$ and mark them on the stick. Break the stick at those two marks so that you get three broken pieces of the stick. Choose one these three pieces uniformly at random. Let $L$ be the length of the chosen piece. What is the expected value of $L$?   
\end{problem}

%HW5. Problem 9.

\begin{problem}\label{Prob:A1.2} Start with an urn with a red ball and a black ball. At each turn, pick a ball at random from the urn, return it to the urn along with another ball of the same color. 
\begin{enumerate}[$(a)$]
\item Find the conditional probability that the second ball picked is red, given that the third ball picked is red.
\item Find the expected number of red balls in the urn after the fourth turn (i.e., when there are six balls in the urn).
\end{enumerate}
\end{problem}

\bibliographystyle{acm}
\bibliography{mcmcbib}

\end{document}